\documentclass[a4paper, 11pt]{article}
\usepackage[utf8]{inputenc}
\usepackage[T1]{fontenc}

\usepackage[xindy]{imakeidx}

\usepackage[
	sorting=nyt,
	backend = biber,
	hyperref = auto,
	backref = true,
	doi = false,
	url = false
]{biblatex}

\usepackage{amsthm}
\usepackage{amssymb}
\usepackage{mathtools}
\usepackage{stmaryrd}	\SetSymbolFont{stmry}{bold}{U}{stmry}{m}{n}	
\usepackage{bm}
\usepackage{nccrules}
\usepackage{mathrsfs}
\usepackage{euscript}
\usepackage{upgreek}
\usepackage{tikz-cd}
\usetikzlibrary{positioning, shapes.geometric, patterns, graphs, decorations.pathmorphing}

\usepackage{paralist}

\usepackage[colorlinks]{hyperref}
\hypersetup{
	linkcolor=blue,
	citecolor=red,
	urlcolor=teal,
	pdftitle = {Arthur packets for metaplectic groups},
	pdfauthor = {Wen-Wei Li}
}

\newcommand{\Z}{\ensuremath{\mathbb{Z}}}
\newcommand{\Q}{\ensuremath{\mathbb{Q}}}
\newcommand{\R}{\ensuremath{\mathbb{R}}}
\newcommand{\CC}{\ensuremath{\mathbb{C}}}
\newcommand{\A}{\ensuremath{\mathbb{A}}}
\newcommand{\F}{\ensuremath{\mathbb{F}}}

\newcommand{\Tr}{\operatorname{tr}}

\newcommand{\Ind}{\operatorname{Ind}}

\newcommand{\Sym}{\operatorname{Sym}}

\newcommand{\Frob}{\operatorname{Frob}}
\newcommand{\Weil}[1]{\ensuremath{\mathrm{Weil}_{#1}}}	

\newcommand{\mes}{\operatorname{mes}}	

\newcommand{\lrangle}[1]{\ensuremath{\langle #1 \rangle}}

\newcommand{\Stab}{\operatorname{Stab}}

\newcommand{\identity}{\ensuremath{\mathrm{id}}}

\newcommand{\rightiso}{\ensuremath{\stackrel{\sim}{\rightarrow}}}

\newcommand{\dcate}[1]{\ensuremath{\text{-}\mathsf{#1}}}	

\newcommand{\Image}{\operatorname{im}}
\newcommand{\dotimes}[1]{\ensuremath{\underset{#1}{\otimes}}}

\newcommand{\Ad}{\operatorname{Ad}}

\newcommand{\Gm}{\ensuremath{\mathbb{G}_\mathrm{m}}}

\newcommand{\GL}{\operatorname{GL}}
\newcommand{\PGL}{\operatorname{PGL}}
\newcommand{\SO}{\operatorname{SO}}
\newcommand{\Or}{\operatorname{O}}

\newcommand{\SL}{\operatorname{SL}}
\newcommand{\Sp}{\operatorname{Sp}}
\newcommand{\GSp}{\operatorname{GSp}}
\newcommand{\Lgrp}[1]{\ensuremath{{}^{\mathrm{L}} #1}}	
\newcommand{\Resprod}{\ensuremath{\prod\nolimits'}}	

\newcommand{\Mp}{\ensuremath{\widetilde{\mathrm{Sp}}}}
\newcommand{\MMp}{\operatorname{Mp}}

\newcommand{\bmu}{\ensuremath{\bm\mu}}
\newcommand{\bpsi}{\ensuremath{\uppsi}}
\newcommand{\Endo}{\ensuremath{\mathcal{E}}}
\newcommand{\orbI}{\ensuremath{\mathcal{I}}}

\newcommand{\elli}{\operatorname{ell}}
\newcommand{\asp}{\ensuremath{\dashrule[.7ex]{2 2 2 2}{.4}}} 
\newcommand{\rev}{\ensuremath{\mathbf{p}}} 
\newcommand{\Trans}{\ensuremath{\mathcal{T}}}	
\newcommand{\trans}{\ensuremath{\check{\mathcal{T}}}}	

\theoremstyle{plain}
\newtheorem{proposition}{Proposition}
\newtheorem{lemma}[proposition]{Lemma}
\newtheorem{theorem}[proposition]{Theorem}
\newtheorem{corollary}[proposition]{Corollary}
\theoremstyle{definition}
\newtheorem{definition}[proposition]{Definition}
\newtheorem{definition-theorem}[proposition]{Definition--Theorem}
\newtheorem{definition-proposition}[proposition]{Definition--Proposition}
\newtheorem{remark}[proposition]{Remark}

\theoremstyle{plain}
\newtheorem{Thm}{Theorem}
\theoremstyle{definition}
\newtheorem{Rmk}{Remark}

\numberwithin{equation}{section}
\numberwithin{proposition}{subsection}
\numberwithin{Thm}{section}	
\numberwithin{Rmk}{section}	

\renewcommand{\emptyset}{\ensuremath{\varnothing}}	

\usepackage{amsmath}
\usepackage[nottoc]{tocbibind}

\usepackage{geometry}
\title{Arthur packets for metaplectic groups}
\author{Wen-Wei Li}
\date{}

\DeclareFieldFormat{postnote}{#1}	
\makeatletter
\renewcommand{\l@section}{\@dottedtocline{1}{1.5em}{2.0em}}
\renewcommand{\l@subsection}{\@dottedtocline{2}{4.0em}{3.0em}}
\makeatother

\makeindex[
	title={Index}]

\begin{document}

\maketitle

\begin{abstract}
	For metaplectic groups over a local field of characteristic zero, we define the Arthur packet attached to any Arthur parameter $\psi$ as a multi-set of unitary genuine irreducible representations, characterized by endoscopic character relations. Over number fields, we obtain a multiplicity formula for the genuine discrete $L^2$-automorphic spectrum in terms of global Arthur parameters and $\epsilon$-factors, by leveraging the trace formula for metaplectic groups. This confirms a conjecture of Gan, and extends earlier results of Gan--Ichino on the Shimura--Waldspurger correspondences, whereas their works play a critical role in our proof. Furthermore, all these are shown to be compatible with existing results in rank one (Waldspurger) and two (Gan--Ichino).
\end{abstract}

{\scriptsize
\begin{tabular}{ll}
	\textbf{MSC (2020)} & Primary 22E50; Secondary 11F70, 11F72 \\
	\textbf{Keywords} & endoscopy, metaplectic group, Arthur packet, Arthur-Selberg trace formula
\end{tabular}}

\setcounter{tocdepth}{2}
\tableofcontents

\section{Introduction}\label{sec:intro}
\subsection*{Overview}
Put $\bmu_m := \left\{z \in \CC^{\times}: z^m = 1 \right\}$ for all $m \in \Z_{\geq 1}$, and denote the Pontryagin dual of any finite abelian group $A$ by $A^\vee$. After A.\ Weil \cite[\S 34]{Weil64}, for a local field $F$ of characteristic zero and a symplectic $F$-vector space $W$, the metaplectic group is a central extension of locally compact groups
\[ 1 \to \bmu_2 \to \MMp(W) \to \Sp(W) \to 1, \]
which splits if and only if $F = \CC$. For a number field $\dot{F}$ with adèle ring $\A = \A_{\dot{F}}$ and a symplectic $\dot{F}$-vector space $\dot{W}$, the adélic metaplectic group is a central extension
\[ 1 \to \bmu_2 \to \MMp(\dot{W}, \A) \to \Sp(\dot{W}, \A) \to 1, \]
which splits uniquely over $\Sp(\dot{W})$ and can be identified with the restricted product of $\MMp(\dot{W}_v)$ modulo $\left\{ (z_v)_v \in \bigoplus_v \bmu_2 : \prod_v z_v = 1 \right\}$, where $v$ ranges over the places of $\dot{F}$.

Although these groups are not $F$ or $\A$-points of a linear algebraic group, thus lie outside the scope of the ordinary Langlands program, they play a prominent role in the study of automorphic forms, representations and number theory. The aim of this article is to settle \emph{Arthur's conjectures} for metaplectic groups, in both the local and global aspects.

To make sense of these statements, we have to define the Langlands dual group of $\MMp(W)$; further still, we need a theory of endoscopy. Let us fix the symplectic form $\lrangle{\cdot|\cdot}: W \times W \to F$ and an additive character $\bpsi \neq \mathbf{1}$ of $F$ in the local setting. Put $G := \Sp(W)$ and $n := \frac{1}{2} \dim W$. It will be convenient to enlarge the twofold covering $\MMp(W)$ to an \emph{eightfold} one by pushing out along $\bmu_2 \hookrightarrow \bmu_8$, denoted by $\tilde{G} = \Mp(W)$, which has a simpler description via Schrödinger models of the Weil representation $\omega_{\bpsi}$; the covering depends on $\bpsi$, but its twofold counterpart does not (see \S\ref{sec:MMp}). Define the Langlands dual group of $\tilde{G}$ as
\[ \tilde{G}^\vee := \Sp(2n, \CC) \quad \text{with trivial Galois action.} \]
The passage to eightfold coverings does not affect the genuine representation theory; a representation of $\tilde{G}$ (resp.\ $\MMp(W)$) is said to be genuine if every $z \in \bmu_8$ (resp.\ $z \in \bmu_2$) acts as $z \cdot \identity$. A basic example of genuine representations is the Weil representation $\omega_{\bpsi} = \omega^+_{\bpsi} \oplus \omega^-_{\bpsi}$, where $\omega_{\bpsi}^+$ and $\omega_{\bpsi}^-$ denote its even and odd parts, respectively.

Similarly, in the global setting we fix $(\dot{W}, \lrangle{\cdot|\cdot})$ over $\dot{F}$ and an additive character $\dot{\bpsi} = \prod_v \bpsi_v \neq \mathbf{1}$ of $\dot{F} \backslash \A$, enlarge $\MMp(\dot{W}, \A)$ to an eightfold covering $\tilde{G} := \Mp(\dot{W}, \A)$ of $G(\A)$, and study the genuine automorphic forms/representations of $\tilde{G}$, such as the adélic Weil representation $\omega_{\dot{\bpsi}, \A}$.

Endoscopy for $\tilde{G}$ was first defined by J.\ Adams \cite{Ad98} and D.\ Renard \cite{Re99} over $F = \R$, and was then extended to arbitrary local or global fields of characteristic zero in \cite{Li11} (with minor mistakes corrected in \cite[Remarque 2.3.1]{Li15}). In particular, there are notions of endoscopic data, transfer factors and transfer of orbital integrals. All these can be viewed as the first instance of Weissman's Langlands program for covering groups \cite{Weis18}; note that the formalism is sensitive to the choice of $(W, \lrangle{\cdot|\cdot})$ and $\bpsi$.

In particular, local L-parameters and Arthur parameters for $\tilde{G}$ are defined. Following Arthur's approach, global parameters are defined unconditionally in terms of cuspidal automorphic representations of general linear groups and L-functions. Our aim can now be phrased as follows.
\begin{description}
	\item[Local goal] Define the local Langlands group of $F$ as
	\[ \mathcal{L}_F := \begin{cases}
		\Weil{F}, & \text{if $F$ is Archimedean,} \\
		\Weil{F} \times \SL(2, \CC), & \text{if $F$ is non-Archimedean.}
	\end{cases}\]
	
	For each Arthur parameter $\psi: \mathcal{L}_F \times \SL(2, \CC) \to \tilde{G}^\vee$ and $\chi \in \EuScript{S}_\psi^\vee$, where $\EuScript{S}_\psi$ is the abelian component group of $S_\psi := Z_{\tilde{G}^\vee}(\Image(\psi))$, we want to define an invariant distribution $\pi_{\psi, \chi}$ on $\tilde{G}$ that is a linear combination of unitary genuine irreducible characters with coefficients in $\Z_{\geq 0}$, characterized by suitable endoscopic character relations.
	
	Once this is done, the \emph{Arthur packet} $\Pi_\psi$ can be defined as a multi-set of unitary genuine irreducibles fibered over $\EuScript{S}_\psi^\vee$, by collecting the irreducible summands in each $\pi_{\psi, \chi}$.
	
	\item[Global goal] For an adélic metaplectic group $\tilde{G}$, we want to
	\begin{itemize}
		\item decompose the genuine discrete $L^2$-automorphic spectrum $L^2_{\mathrm{disc}, -}$ of $\tilde{G}$ into an orthogonal direct sum $\widehat{\bigoplus}_{\dot{\psi}} L^2_{\dot{\psi}}$, where $\dot{\psi}$ ranges over discrete Arthur parameters for $\tilde{G}$;
		\item for each $\dot{\psi}$, decompose $L^2_{\dot{\psi}}$ in terms of the local distributions $\pi_{\dot{\psi}_v, \chi_v}$ and certain $\epsilon$-factors, where the global information resides.
	\end{itemize}
\end{description}

For tempered genuine representations, a local Langlands correspondence for metaplectic groups has been obtained by Adams--Barbasch \cite{AB98} for $F = \R$, and by Gan--Savin \cite{GS1} for non-Archimedean $F$; the case $F = \CC$ is much simpler, see \cite[\S 7.6]{Li19}. From these one can parameterize all genuine irreducibles via Langlands quotients, but this approach is not enough for global questions.

For generic discrete global Arthur parameters $\dot{\psi}$, i.e.\ those trivial on $\SL(2, \CC)$, Gan--Ichino \cite{GI18} obtained a multiplicity formula for $L^2_{\dot{\psi}}$ by using $\Theta$-correspondences.

The aforementioned works are all based on $\Theta$-correspondence. In Arthur's vision, carried out splendidly in \cite{Ar13, Mok15} for quasisplit classical groups (see also \cite{KMSW, Is24}), these statements should be deduced from the \emph{stable trace formula}. In this article, we shall follow a similar, yet somewhat different route. In fact, in a series of works that culminates in \cite{Li26}, the stable trace formula for $\tilde{G}$ has been established in the global setting. We will see how it can be combined with $\Theta$-correspondence to yield the desired results.

\subsection*{Main local results}
Let $F$ be a local field of characteristic zero, and consider the metaplectic covering $\tilde{G} \xrightarrow{\rev} G(F)$ associated with $(W, \lrangle{\cdot|\cdot})$ and $\bpsi$. Denote the set of Arthur parameters of $\tilde{G}$ by $\Psi(\tilde{G})$, taken up to $\tilde{G}^\vee$-conjugacy. We have to consider a larger set $\Psi^+(\tilde{G})$ by dropping the boundedness condition of $\psi|_{\Weil{F}}$, in order to accommodate possible local components of global Arthur parameters.

Let $n = \frac{1}{2} \dim W$. The elliptic endoscopic data of $\tilde{G}$ are in bijection with conjugacy classes of elements $s \in \tilde{G}^\vee$ satisfying $s^2 = 1$. In turn, these conjugacy classes are in bijection with pairs $(n', n'') \in \Z_{\geq 0}^2$ such that $n' + n'' = n$, by counting the eigenvalues $\pm 1$ of $s$. If an elliptic endoscopic datum $\mathbf{G}^!$ corresponds to $(n', n'')$, then the underlying endoscopic group is
\[ G^! = \SO(2n' + 1) \times \SO(2n'' +1), \]
where $\SO(m)$ means the split form; the boldface symbol is used to distinguish endoscopic data from endoscopic groups. The transfer $\Trans_{\mathbf{G}^!, \tilde{G}}$ of orbital integrals dualizes to yield the spectral transfer map $\trans_{\mathbf{G}^!, \tilde{G}}$, which sends stable virtual characters on $G^!(F)$ to genuine virtual characters on $\tilde{G}$.

If $\mathbf{G}^!$ corresponds to the conjugacy class of $s$, then $(G^!)^\vee \simeq Z_{\tilde{G}^\vee}(s) \subset \tilde{G}^\vee$. Given a pair $(\mathbf{G}^!, \psi^!)$ where $\mathbf{G}^!$ is an elliptic endoscopic datum of $\tilde{G}$ and $\psi^! \in \Psi^+(G^!)$, we obtain $(\psi, s)$ where $\psi \in \Psi^+(\tilde{G})$ is the image of $\psi^!$ (see \S\ref{sec:local-parameters}). By \S\ref{sec:basic-bijection}, this is actually a bijection: every pair $(\psi, s)$ arises uniquely in this way, where $s \in S_\psi$ satisfies $s^2 = 1$, considered up to $S_\psi$-conjugacy.

We are now ready to define $\pi_{\psi, \chi}$. Given $(\psi, s)$, take its preimage $(\mathbf{G}^!, \psi^!)$. The theory of Arthur packets for $G^!$ is available \cite[Chapters 1--2]{Ar13}; in particular, there is a stable virtual character $S\Theta^{G^!}_{\psi^!}$, independent of all choices since $G^!$ is adjoint. Consider the $(-1)$-eigenspace $\psi^{s = -1}$ of $\psi$ under $s$ and set
\[ \epsilon(\psi^{s = -1}) := \epsilon\left(\frac{1}{2}, \psi^{s = -1}|_{\mathcal{L}_F}, \bpsi \right); \]
this is a symplectic local root number, hence independent of $\bpsi$ by \cite[\S 5]{GGP1}. Set
\begin{equation*}
	T_{\psi, s} := \epsilon\left( \psi^{s = -1} \right) \cdot \trans_{\mathbf{G}^!, \tilde{G}}\left( S\Theta^{G^!}_{\psi^!} \right).
\end{equation*}
It will be shown in Lemma \ref{prop:T-psi-s} that $T_{\psi, s}$ depends only on $\psi$ and the image $x$ of $s$ in $\EuScript{S}_\psi$, although the individual terms on the right depend on the conjugacy class of $s$. Every $x \in \EuScript{S}_\psi$ arises from some $s \in S_\psi$ with $s^2 = 1$, hence we may write $T_{\psi, x} = T_{\psi, s}$, and consider the Fourier expansion of $x \mapsto T_{\psi, x}$ or its translates.

\begin{Thm}[= Theorem \ref{prop:local-desiderata}]
	Given $\psi \in \Psi^+(\tilde{G})$, set $s_\psi := \psi\left(1, \bigl(\begin{smallmatrix} -1 & \\ & -1 \end{smallmatrix}\bigr) \right) \in S_\psi$ and let $x_\psi \in \EuScript{S}_\psi$ be its image. Then
	\[ \pi_{\psi, \chi} := |\EuScript{S}_\psi|^{-1} \sum_{x \in \EuScript{S}_\psi} \chi(x_\psi x) T_{\psi, x} \]
	is a linear combination (possibly zero) of genuine irreducible characters of $\tilde{G}$ with coefficients in $\Z_{\geq 0}$, for all $\chi \in \EuScript{S}_\psi^\vee$. If $\psi \in \Psi(\tilde{G})$, then these irreducible characters arise from unitary representations.
\end{Thm}

Based on this theorem, one defines the Arthur packet $\Pi_\psi$ as a multi-set of genuine irreducible representations. Rigorously, $\Pi_\psi$ is a finite set equipped with two maps
\[ \Pi_-(\tilde{G}) \leftarrow \Pi_\psi \rightarrow \EuScript{S}_\psi^\vee \]
where $\Pi_-(\tilde{G})$ is the set of isomorphism classes of genuine irreducible representations of $\tilde{G}$. For details, we refer to \S\ref{sec:A-packets}, especially Definition \ref{def:A-packet}.

\begin{Rmk}
	The $\epsilon$-factor in $T_{\psi, s}$ is a metaplectic feature: it does not appear in Arthur's original version. Another feature is that we consider all characters of $\EuScript{S}_\psi$, with no condition imposed on their pullbacks to $Z_{\tilde{G}^\vee} = \{\pm 1\}$, whereas Arthur considered only those trivial on the image of the center of the dual group (see \cite[\S 1.4]{Ar13}).
\end{Rmk}

\begin{Rmk}
	If $\psi$ is a bounded L-parameter of $\tilde{G}$, then the local Langlands correspondence of Adams--Barbasch and Gan--Savin is recovered: we have $s_\psi = 1$, the Arthur packet $\Pi_\psi$ reduces to the tempered L-packet, which is multiplicity-free, and the definition of $\pi_{\psi, \chi}$ becomes the endoscopic character relation of C.\ Luo \cite[\S 1, Main Theorem]{Luo20} after Fourier inversion.
\end{Rmk}

Furthermore, $\pi_{\psi, \chi}$ has the following properties.
\begin{description}
	\item[Reduction to good parity (Proposition \ref{prop:pi-psi-gp})] We say $\psi$ is of good parity (Definition \ref{def:good-parity}) if, as a $2n$-dimensional representation of $\mathcal{L}_F \times \SL(2, \CC)$, all its simple constituents are self-dual of symplectic type. General $\psi$ can be expressed as the image of some parameter $\psi_M$ from a Levi $\tilde{M} \subset \tilde{G}$ via $\tilde{M}^\vee = \Sp(2n^\flat, \CC) \times \prod_i \GL(n_i, \CC)\hookrightarrow \Sp(2n,\CC)$, whose $\Sp$-component is of good parity, such that $\EuScript{S}_{\psi_M} \simeq \EuScript{S}_\psi$ and $\pi_{\psi, \chi}$ equals the full parabolic induction of $\pi_{\psi_M, \chi}$.
	
	If we write $M = \prod_i \GL(n_i) \times \Sp(W^\flat)$, then $\tilde{M} = \prod_i \GL(n_i, F) \times \Mp(W^\flat)$ canonically since $(W, \lrangle{\cdot|\cdot}, \bpsi)$ is chosen, and we are using eightfold coverings. This reduces the theory for $\tilde{M}$ to a smaller metaplectic group.
	
	\item[Infinitesimal character (Proposition \ref{prop:pi-psi-inf-char})] When $F$ is Archimedean, $\pi_{\psi, \chi}$ has an explicitly specified infinitesimal character $\lambda(\psi)$ (see \S\ref{sec:inf-character}).
	
	\item[Unramified case (Proposition \ref{prop:pi-psi-nr})] In the unramified situation (see \S\ref{sec:unramified-situation} for the precise definition), one has the anti-genuine spherical Hecke algebra $\mathcal{H}_{\asp}(K \backslash \tilde{G} / K)$ and the notion of $K$-spherical representations, where $K := G(\mathfrak{o}_F)$; a function $f$ on $\tilde{G}$ being anti-genuine means that $f(z \tilde{g}) = z^{-1} f(\tilde{g})$ for all $z \in \bmu_8$ and $\tilde{g} \in \tilde{G}$. If $\psi$ is unramified (Definition \ref{def:unramified-psi}), then $\Pi_\psi$ has a unique $K$-spherical element, which is multiplicity-free and parametrized by $\chi = \mathbf{1}$; if $\psi$ is not unramified then $\Pi_\psi$ has no $K$-spherical elements.
	
	\item[Central character (Proposition \ref{prop:p-psi-negation})] Using the Weil representations, one obtains a canonical preimage of $-1 \in G(F)$ in $\tilde{G}$, which is central and still denoted by $-1$ (Definition \ref{def:minus-1}). For all $\psi \in \Psi^+(\tilde{G})$, the translation by $-1 \in \tilde{G}$ acts by $\chi(-1) \epsilon(\psi|_{\mathcal{L}_F})$, where we take $-1 \in Z_{\tilde{G}^\vee}$ in $\chi(-1)$. This determines the central character of $\pi_{\psi, \chi}$ completely.
	
	\item[The L-packet within (Proposition \ref{prop:L-within})] Let $\psi \in \Psi(\tilde{G})$, and denote by $\Pi_\psi^{\mathrm{mf}}$ the multiplicity-free part (Definition \ref{def:multiplicity-free}) of $\Pi_\psi$. Define the L-parameter $\phi_\psi: \mathcal{L}_F \to \tilde{G}^\vee$ by
	\[ \phi_\psi(w) = \psi\left( w, \begin{pmatrix} |w|^{\frac{1}{2}} & \\ & |w|^{-\frac{1}{2}} \end{pmatrix} \right). \]
	There is a unique injection $\Pi_{\phi_\psi} \to \Pi_\psi^{\mathrm{mf}}$ making the diagram
	\[\begin{tikzcd}
		\Pi_{\phi_\psi} \arrow[hookrightarrow, r] \arrow[d, "{1:1}"'] & \Pi_\psi^{\mathrm{mf}} \arrow[d] \\
		\EuScript{S}_{\phi_\psi}^\vee \arrow[hookrightarrow, r] & \EuScript{S}_\psi^\vee
	\end{tikzcd}\]
	commute; therefore, all members of $\Pi_{\phi_\psi}$ are unitary.
	
	\item[Variation of $\bpsi$ (Proposition \ref{prop:variation-pi})] For every $c \in F^{\times}$, let $\bpsi_c(x) = \bpsi(cx)$. Write $\tilde{G}^{(2)} := \MMp(W) \subset \tilde{G}$. It is well-known that this twofold covering of $G(F)$ does not depend on $\bpsi$ (Proposition \ref{prop:MMp-uniqueness}). Hence we can compare $\pi^{\bpsi_c}_{\psi, \chi}$ associated with various $c \in F^{\times}$, as genuine distributions on $\tilde{G}^{(2)}$; in fact they depend only on $c F^{\times 2}$. We have
	\[ \pi^{\bpsi_c}_{\psi\zeta, \chi \delta_c} = \pi^{\bpsi}_{\psi, \chi}, \]
	where $\zeta = \zeta_c: \Weil{F} \to \{\pm 1\} \simeq Z_{\tilde{G}^\vee}$ is associated with $cF^{\times 2}$ by class field theory, and $\delta_c \in \EuScript{S}_\psi^\vee$ is explicitly given in Definition \ref{def:delta-c}.
	
	This extends \cite[Theorem 12.1]{GS1} in the tempered case.
\end{description}

\subsection*{Main global results}
Let $\dot{F}$ be a number field and $\dot{\bpsi} = \prod_v \bpsi_v \neq \mathbf{1}$ be an additive character of $\dot{F} \backslash \A$. Let $(\dot{W}, \lrangle{\cdot|\cdot})$ be a $2n$-dimensional symplectic $\dot{F}$-vector space, $G := \Sp(\dot{W})$ and consider the eightfold adélic metaplectic covering
\[ 1 \to \bmu_8 \to \left( \tilde{G} := \Mp(\dot{W}, \A) \right) \to G(\A) \to 1. \]

At almost all places $v$, we are in the unramified situation, and there is a Satake isomorphism $\mathcal{H}_{\asp}(K_v \backslash \tilde{G}_v / K_v) \simeq \CC[ \tilde{G}^\vee /\!/ \tilde{G}^\vee ]$ (see \eqref{eqn:Satake-Mp}), where $K_v := G(\mathfrak{o}_v)$ for these places.

Let $\dot{\Psi}(\tilde{G})$ be the set of global Arthur parameters for $\tilde{G}$. For each $\dot{\psi} \in \dot{\Psi}(\tilde{G})$, one can define its centralizer and component groups $S_{\dot{\psi}}$ and $\EuScript{S}_{\dot{\psi}}$, related to their local avatars by localization maps, as well as the element $s_{\dot{\psi}} \in S_{\dot{\psi}}$.

For almost all places $v$, one attaches a Satake parameter to $\dot{\psi}_v$. In this way, one extracts from the genuine discrete $L^2$-automorphic spectrum the summands (as unitary representations)
\[ L^2_{\dot{\psi}} \subset L^2_{\mathrm{disc}, -} := L^2_{\mathrm{disc}, -}(G(\dot{F}) \backslash \tilde{G}). \]

\begin{Thm}[= Theorem \ref{prop:GI-disc}]
	Let $\dot{\Psi}_2(\tilde{G})$ denote the set of discrete global Arthur parameters of $\tilde{G}$, i.e.\ the parameters which cannot factor through any proper Levi. Then
	\begin{equation*}
		L^2_{\mathrm{disc}, -} = \widehat{\bigoplus}_{\dot{\psi} \in \dot{\Psi}_2(\tilde{G})} L^2_{\dot{\psi}}.
	\end{equation*}
\end{Thm}

The result above is actually due to Gan--Ichino \cite[Theorem 1.1]{GI18}. The hardest part is to show that only discrete parameters contribute. In \cite[Theorem 5]{Ar11} this is called the ``no embedded Hecke eigenvalues'' property.

Let $\dot{\psi} \in \dot{\Psi}_2(\tilde{G})$. The summand $L^2_{\dot{\psi}}$ may be of infinite length. To decompose it, take a representative $s \in S_{\dot{\psi}}$ of $x \in \EuScript{S}_{\dot{\psi}}$ satisfying $s^2 = 1$; it makes sense to consider the $(-1)$-eigenspace of $\dot{\psi}$ as in the local setting, although global parameters are defined differently, and consider the global root number (see \S\ref{sec:global-root-number}):
\[ \nu_{\dot{\psi}}(x) := \epsilon\left( \dot{\psi}^{s = -1} \right) = \prod_v \epsilon\left( \dot{\psi}_v^{s = -1} \right). \]
Proposition \ref{prop:nu-factorization} says that $\nu_{\dot{\psi}}(x)$ is independent of the choice of $s$, and gives a character $\nu_{\dot{\psi}}$ of $\EuScript{S}_{\dot{\psi}}$; this is specific to the global case.

Denote by $\epsilon^{\mathrm{Art}}_{\dot{\psi}} \in \overline{\EuScript{S}}_{\dot{\psi}}^\vee$ Arthur's sign character for $\SO(2n+1)$ in \cite[Theorem 1.5.2]{Ar13}, where $\overline{\EuScript{S}}_{\dot{\psi}}$ is the quotient of $\EuScript{S}_{\dot{\psi}}$ by the image of $Z_{\tilde{G}^\vee}$.

Put $\epsilon_{\dot{\psi}} := \epsilon^{\mathrm{Art}}_{\dot{\psi}} \nu_{\dot{\psi}} \in \EuScript{S}_{\dot{\psi}}^\vee$. Denote by $\mathrm{diag}$ the diagonal map $\EuScript{S}_{\dot{\psi}} \to \prod_v \EuScript{S}_{\dot{\psi}_v}$. Define
\begin{align*}
	\Pi_{\dot{\psi}} & := \left\{\begin{array}{r|l}
		\dot{\pi} = (\dot{\pi}_v)_v & \dot{\pi}_v \in \Pi_{\dot{\psi}_v}, \\
		& K_v\text{-spherical for almost all}\; v
	\end{array}\right\}, \\
	\Pi_{\dot{\psi}}(\epsilon_{\dot{\psi}}) & := \left\{ \dot{\pi} \in \Pi_{\dot{\psi}} \; \middle| \; \mathrm{diag}^* \prod_v \lrangle{\cdot, \dot{\pi}_v} = \epsilon_{\dot{\psi}} \right\} \quad \text{(notation of \S\ref{sec:A-packets})}.
\end{align*}
We will show in Proposition \ref{prop:unramified-mf} (i) that $K_v$-sphericity implies $\lrangle{\cdot, \dot{\pi}_v} = \mathbf{1}$ at each unramified place $v$, thus the definition of $\Pi_{\dot{\psi}}(\epsilon_{\dot{\psi}})$ makes sense.

\begin{Thm}[= Theorem \ref{prop:global-multiplicity} + Remark \ref{rem:global-packet}]\label{prop:main-global}
	Let $\dot{\psi} \in \dot{\Psi}_2(\tilde{G})$. With the notations above, one has
	\[ L^2_{\dot{\psi}} \simeq \widehat{\bigoplus}_{\dot{\pi} \in \Pi_{\dot{\psi}}(\epsilon_{\dot{\psi}})} \dot{\pi}. \]
	Here $\dot{\pi}_v \in \Pi_{\dot{\psi}_v}$ (resp.\ $\dot{\pi}$) is identified with the corresponding genuine irreducible representation of $\tilde{G}_v$ (resp.\ of $\tilde{G}$, given by restricted tensor product), by abusing notation.
\end{Thm}

The direct sum above ranges over the multi-set $\Pi_{\dot{\psi}}(\epsilon_{\dot{\psi}})$, thus it need not be multiplicity-free. This generalizes the multiplicity formula in \cite[Theorem 1.4]{GI18}; it also confirms Gan's Arthur conjecture for $\tilde{G}$ in \cite[Conjecture 8.2]{Gan14} as explicated in Remark \ref{rem:ArthurMp-Gan}. In Theorem \ref{prop:global-multiplicity}, we phrase it as a global character relation involving $\pi_{\dot{\psi}_v, \chi_v}$.

\begin{Rmk}
	The character $\nu_{\dot{\psi}}$ is also a metaplectic feature since it does not appear in Arthur's scenario. It can be ``explained'' by the local root numbers appearing in $T_{\dot{\psi}_v, s}$.
\end{Rmk}

Finally, when $n = 1, 2$, a similar description of local packets and global multiplicities has been obtained by J.-L.\ Waldspurger \cite{Wa80, Wa91} and Gan--Ichino \cite{GI21}, respectively. Their approaches are mainly based on $\Theta$-correspondence. We have the following compatibility.

\begin{Thm}[= \S\ref{sec:Waldspurger} + Theorem \ref{prop:Mp4} + Lemma \ref{prop:GI-epsilon}]
	Suppose that $n = 1$ (resp.\ $n = 2$). In both the local and global settings, the aforementioned results are compatible with those of Waldspurger (resp.\ Gan--Ichino).
\end{Thm}

\subsection*{Idea of the proofs}
Following Arthur's philosophy, one would expect to prove the local and global desiderata simultaneously through the stable trace formula, by a long local-global argument. This appears reasonable since the endoscopic groups in question are products of split odd $\SO$, for which Arthur's results are available. In particular, one can apply the \emph{stable multiplicity formula} to the endoscopic side of the stable trace formula in \cite{Li26}. We refer to \cite[\S 4.8]{Ar13} for the general methodology.

However, Arthur needs the \emph{local intertwining relations} (LIR) to exclude non-discrete parameters in the discrete $L^2$-automorphic spectrum, see \cite[\S 2.4]{Ar13}, which is one of the hard cores of \textit{loc.\ cit.}
\index{local intertwining relation}

Instead of trying to establish LIR first, we take the shortcut through $\Theta$-correspondence in \cite[Theorem 1.1]{GI18}. The ``no embedded Hecke eigenvalues'' property is thus directly available to us. The strategy in the article can be summarized as follows.

\begin{enumerate}
	\item In the global setting with $\dot{\psi} \in \dot{\Psi}_2(\tilde{G})$, by combining Arthur's stable multiplicity formula \cite[Theorem 4.1.2]{Ar13} for odd $\SO$, certain sign lemmas in \cite[Chapter 4]{Ar13} and the stable trace formula for $\tilde{G}$ in \cite{Li26}, one obtains an expression for $\Tr L^2_{\dot{\psi}}$ (as a formal sum) in terms of $\pi_{\dot{\psi}_v, \chi_v}$ (Theorem \ref{prop:global-multiplicity}) --- this is the global multiplicity formula stated as a character relation, without assuming Theorem \ref{prop:local-desiderata}.

	\item In the local non-Archimedean setting, we take Luo's tempered character relations \cite{Luo20} as input to construct anti-tempered Arthur packets of good parity (or quasi-good parity) via Aubert--Zelevinsky involution, reducing the statement of Theorem \ref{prop:local-desiderata} except unitarity to a sign equality (Lemma \ref{prop:anti-tempered-bp}) concerning tempered representations. Ingredients for this step include:
	\begin{itemize}
		\item the commutation between endoscopic transfer of $\tilde{G}$ and Aubert--Zelevinsky involution \cite{Chen26};
		\item the shifting character $\mu_\psi$ due to Liu--Lo--Shahidi \cite{LLS24} (Proposition \ref{prop:beta-variance}) for odd $\SO$, which is equal to Xu's character $\epsilon^{\mathrm{M}/\mathrm{MW}}_\psi$ in \cite{Xu17, Xu21}.
	\end{itemize}
	Note that a similar approach has been adopted in \cite[\S 7.1]{Ar13}, and for $\tilde{G}$ we need a variant $\tilde{\mu}_\psi$ of the shifting character $\mu_\psi$, see Definition \ref{def:mu-tilde}.
	
	\item We prove the aforementioned sign equality for $\psi$ in the set $\Psi(\tilde{G})^\star$ in Definition \ref{def:Psi-star}, by reducing it to the case for odd $\SO$ (split or non-split) settled in \cite{LLS24}. In this step we invoke $\Theta$-correspondence, and use Kudla's results \cite{Ku86} to match cuspidal supports.
	
	\item For the general local setting, we prove Theorem \ref{prop:local-desiderata} by reducing to good parity and globalization, using the toolkit in \cite[Chapter 6]{Ar13}. The globalized parameter will have local components in $\Psi(\tilde{G}_v)^{\star}$ at auxiliary places $v$ in some finite set $V_0$, but $\dot{\psi}_v$ could lie outside $\Psi(\tilde{G}_v)^{\star}$ at some ramified $v \nmid \infty$, although they are still anti-tempered of quasi-good parity. The latter case is addressed by Proposition \ref{prop:L-within-0} about the L-packet within an Arthur packet, engineered so that it also applies to anti-tempered parameters without granting the validity of Theorem \ref{prop:local-desiderata}.
\end{enumerate}

In Proposition \ref{prop:principal}, we will describe the principal Arthur packets precisely. Over a local $F$, an Arthur parameter $\psi$ for $\tilde{G}$ is said to be \emph{principal} if
\[ \psi = \zeta \boxtimes r(2n), \]
where $\zeta$ is a quadratic character of $\Weil{F}$, inflated to $\mathcal{L}_F$, and $r(b)$ indicates the $b$-dimensional irreducible representation of $\SL(2, \CC)$ for every $b \in \Z_{\geq 1}$. They are the most degenerate Arthur parameters. In this case we have
\[ \EuScript{S}_\psi^\vee = \{\pm 1 \}, \quad \pi^{\bpsi}_{\psi, \pm} = \Tr \omega^{\pm}_{\bpsi_c} , \]
where $c F^{\times 2}$ corresponds to $\zeta$ by local class field theory, and we work over $\tilde{G}^{(2)}$ to make sense of the statement above. This idea already appears in Adams' foundational work \cite{Ad98}.

Principal global parameters are defined in the same manner, and their contribution to $L^2_{\mathrm{disc}, -}$ is given by elementary $\vartheta$-series; see Proposition \ref{prop:principal-theta-series}.

Let us say a few words about $n = 1, 2$. When $n=1$, the only non-tempered Arthur parameters are the principal ones, and it is easy to see that our results are compatible with Waldspurger's. The situation is considerably more complicated when $n=2$. To compare with Gan--Ichino \cite{GI21}, we proceed as follows.
\begin{enumerate}
	\item Show that the local Gan--Ichino packets enjoy the same basic properties (Lemma \ref{prop:GI-local}), except character relations.
	\item Check that the global multiplicity formula of Gan--Ichino involves the same sign character as ours (Lemma \ref{prop:GI-epsilon}).
	\item Show the compatibility when $F$ is non-Archimedean and $\psi \in \Psi(\tilde{G})^{\star}$. For non-principal $\psi$, this is achieved through a case-by-case analysis of Aubert--Zelevinsky duals of the representations listed in \cite[Appendix C]{GI21}.
	\item Conclude by re-running the local-global argument for Theorem \ref{prop:local-desiderata}, by matching everything at the auxiliary places.
\end{enumerate}

\begin{Rmk}
	Arthur's endoscopic classification \cite{Ar13} for $\SO(2n+1)$, together with its counterparts for generic parameters of non-split odd orthogonal groups, lurks behind all these arguments and the upstream references including \cite{GI18, LLS24, Luo20, Xu17}. Due to this logical dependency, three inputs are needed in the present work.
	
	First, Arthur's work relies on the stabilization of twisted trace formulas. As a result, we have to grant the validity of the \emph{twisted weighted fundamental lemma} in large enough residual characteristic, whose full proof has not appeared in published form; see however the recent preprint \cite{HD26} of C.\ Halleck-Dubé.
	
	Secondly, the LIR is not proved in Arthur's work. This substantial gap is now filled by the joint work \cite{AGIKMS} of Atobe--Gan--Ichino--Kaletha--Mínguez--Shin, which also provides the anti-tempered seed results required in his inductive arguments. Modulo the twisted weighted fundamental lemma, this puts Arthur's work on firm ground, and also justifies the ``no embedded Hecke eigenvalues'' property \cite{GI18} for $\MMp(2n)$.
	
	Finally, the endoscopic classification for generic parameters of non-split odd orthogonal groups has been established in \cite{Is24}, taking the split theory as an input. For generic parameters, Ishimoto also obtained the LIR for these groups. Together with the split case and the assumed twisted weighted fundamental lemma, this completes the Gan--Savin construction \cite{GS1} of the local L-packets and the Gan--Ichino generic multiplicity formula \cite{GI18}. In turn, the results of C.\ Luo \cite{Luo20} are also justified, as they are deduced from the Gan--Ichino formula combined with the endoscopy for $\MMp(2n)$.
\end{Rmk}

\subsection*{Prospects}
Several issues remain unaddressed in this article.

First of all, our proof gives no concrete information about local Arthur packets beyond the tempered and principal cases. In particular, we do not know whether the packets are multiplicity-free in general, and we do not have a criterion for $\pi_{\psi, \chi} \neq 0$. This limits the applicability of our global multiplicity formula.

For non-Archimedean $F$, it would be desirable to have a concrete construction of Arthur packets, similar to the works of H.\ Atobe \cite{At22} or B.\ Xu \cite{Xu17, Xu21} which are closely related to Moeglin's works. In particular, one might hope to prove multiplicity-one by studying Jacquet modules (cf.\ the case of classical groups in \cite{Moe11}), and to study the intersection, membership or non-vanishing problems for Arthur packets, for example; multiplicity-one means that $\Pi_\psi$ is actually a set (Definition \ref{def:multiplicity-free}). The multiplicity-one property is recently established by J.\ Chen \cite{ChenJ26}.

For $F = \R$, one hopes to adapt the works of Moeglin--Renard, e.g.\ \cite{MR18}, to gain more information about Arthur packets. What we do in \S\ref{sec:quadratic-unipotent} about quadratic unipotent parameters serves as a first step in this direction. Theorem \ref{prop:real-unip-packet} asserts that $\Pi_\psi$ is multiplicity-free if $\psi$ is quadratic unipotent.

The case $F = \CC$ ought to be simpler in view of \cite{MR17}. Parameters over $\CC$ of good parity are always unipotent. Following \cite{Moe17}, we show in Theorem \ref{prop:cplx-packet} that $\pi_{\psi, \chi}$ is either irreducible or zero, and the non-zero ones are distinct as $\chi$ varies; in particular, $\Pi_\psi$ is multiplicity-free. There should be a description à la Barbasch--Vogan, cf.\ the metaplectic case discussed in \cite{MR17}, but we have to leave it to a future work. Note that genuine representations of $\tilde{G}$ are the same as representations of $\Sp(2n, \CC)$, but the packets differ.
 
It will also be important to explore the relation between non-tempered Arthur packets and $\Theta$-correspondences.

Recall that we circumvented the usage of LIR in the proof. It would be desirable to establish LIR for metaplectic groups with the help of results obtained in this article. The tempered case has been studied by H.\ Ishimoto \cite{Is20}. The results in \S\ref{sec:normalizing} on normalized intertwining operators provide some preparations toward the non-tempered case.

Over a global field $\dot{F}$, our results should be compared with other explicit constructions of automorphic forms or representations, such as automorphic descent \cite{GRS11}, or the Ikeda--Yamana lifting \cite{IY20} for odd $n$.

\subsection*{Acknowledgement}
The author is deeply indebted to Jean-Loup Waldspurger for his guidance and encouragement since the very beginning \cite{Li11} of this project. Several crucial arguments in the present article were also suggested by him.

The author also wants to express his sincere gratitude to Hiraku Atobe, Fei Chen, Wee Teck Gan, Atsushi Ichino, Sandeep Varma and Bin Xu for helpful comments. The author thanks the referees for their meticulous reading and valuable suggestions, which have helped to repair several genuine gaps in an earlier version.

This work is supported by NSFC, Grant No.\ 11922101 and 12321001.

\subsection*{AI use declaration}
Large language models were used in revising the manuscript and checking the mathematical soundness of the revision. Generative AI contributed no ideas or proofs. The author bears full responsibility for the correctness of the results and proofs, and for the integrity and accuracy of the references.

\subsection*{Organization}
This article is structured as follows.

After a summary of conventions, in \S\ref{sec:review} we give a brief review of metaplectic groups and endoscopy, most of which is taken from \cite{Li26}.

In \S\ref{sec:parameters}, we study local and global Arthur parameters for the metaplectic group $\tilde{G} = \Mp(W)$ where $\dim W = 2n$. Since $\tilde{G}$ and $H := \SO(2n+1)$ have the same dual group, this is merely a recap of known facts due to Arthur. Readers familiar with these materials can safely skip \S\S\ref{sec:review}--\ref{sec:parameters}.

The local theory is expounded in \S\ref{sec:local}. We define the genuine distributions $T_{\psi, s}$ in terms of local root numbers and endoscopic transfer, show that they depend only on the Arthur parameter $\psi$ and the image of $s$ in $\EuScript{S}_\psi$, then define $\pi_{\psi, \chi}$ by Fourier inversion. We state the main local Theorem \ref{prop:local-desiderata} and prove certain basic properties of $\pi_{\psi, \chi}$. Granting the validity of Theorem \ref{prop:local-desiderata}, we define the Arthur packet $\Pi_\psi$ as a multi-set.

In \S\ref{sec:global}, which is of a global nature, the stable trace formula in \cite{Li26} and results of Gan--Ichino are combined to yield a multiplicity formula (Theorem \ref{prop:global-multiplicity}) for the discrete genuine $L^2$-automorphic spectrum of adélic metaplectic groups.

In \S\ref{sec:L-packets}, we review the local Langlands correspondence for metaplectic groups, and the endoscopic character relations in the tempered case due to C.\ Luo. We then study the embedding of the L-packet of $\phi_\psi$ into the Arthur packet of $\psi$ in terms of character relations.

The materials in \S\S\ref{sec:anti-tempered}--\ref{sec:proof-local} are more technical, and can be skipped at first reading. We define anti-tempered Arthur parameters over non-Archimedean $F$ in \S\ref{sec:anti-tempered}. For anti-tempered $\psi$ of quasi-good parity, we reduce the statement of Theorem \ref{prop:local-desiderata} (except unitarity) to a sign equality \eqref{eqn:anti-tempered-bp}, which we resolve in the special case of $\psi \in \Psi(\tilde{G})^\star$. By putting these parameters at auxiliary places, we complete the proof of Theorem \ref{prop:local-desiderata} in \S\ref{sec:proof-local} by a globalization argument.

In \S\ref{sec:further-properties}, we give further properties of $\pi_{\psi, \chi}$, including their dependence on the additive character $\bpsi$, the normalization of intertwining operators through $\psi$, and some supplements on Moeglin's results \cite{Moe17} on quadratic unipotent Arthur parameters for metaplectic groups over $F \in \{\R, \CC \}$.

Concrete examples of $\pi_{\psi, \chi}$ are given in \S\ref{sec:examples}. We determine the Arthur packets attached to principal local Arthur parameters and their contribution to the discrete genuine $L^2$-automorphic spectrum. Our local and global results are then shown to be compatible with Waldspurger's when $n=1$, and with Gan--Ichino's when $n=2$. The arguments for $n=2$ are relatively lengthy, involving local-global methods and Xu's characters.

\subsection*{Conventions}
\paragraph{Fields}
The Weil group of a local or global field $E$ is denoted by $\Weil{E}$. The normalized absolute value of a local field $F$ is denoted by $|\cdot|_F$. If $F$ is non-Archimedean, we have the inertia (resp.\ wild inertia) subgroup $I_F$ (resp.\ $I_F^{\mathrm{wild}}$) in $\Weil{F}$; we denote any representative of the Frobenius by $\Frob$.

Global objects will often be decorated with a dot, except in \S\ref{sec:global}. Given a global field $\dot{F}$, its ring of adèles is denoted by $\A = \A_{\dot{F}}$; places of $\dot{F}$ are denoted by $u, v, w$ and so on; $v \mid \infty$ means that $v$ is Archimedean.

For a non-Archimedean local field $F$ (resp.\ number field $\dot{F}$), we denote its ring of integers by $\mathfrak{o}_F$ (resp.\ $\mathfrak{o}_{\dot{F}}$). The $v$-completion of $\mathfrak{o}_{\dot{F}}$ is denoted by $\mathfrak{o}_v$ if $v \nmid \infty$.

An additive character of a local field $F$ means a non-trivial continuous unitary character $\bpsi: F \to \CC^{\times}$. We set $\bpsi_c: x \mapsto \bpsi(cx)$ for all $c \in F^{\times}$.

Similarly, for a global field $\dot{F}$, additive characters $\dot{\bpsi}: \dot{F} \backslash \A \to \CC^{\times}$ are non-trivial, continuous and unitary, and we set $\dot{\bpsi}_c(x) = \dot{\bpsi}(cx)$ for all $c \in \dot{F}^{\times}$. They factorize as $\dot{\bpsi} = \prod_v \bpsi_v$.
\index{psi-character@$\bpsi, \dot{\bpsi}$}

If $F$ is non-Archimedean and $\bpsi$ satisfies
$\bpsi|_{\mathfrak{o}_F} = \mathbf{1}$ and $\bpsi|_{\mathfrak{p}_F^{-1}} \neq \mathbf{1}$, we say $\bpsi$ is of conductor $\mathfrak{o}_F$. Here $\mathfrak{p}_F$ is the maximal ideal of $\mathfrak{o}_F$.
\index{conductor}

\paragraph{Varieties and groups}
For a scheme $X$ defined over a commutative ring $k$ and a commutative $k$-algebra $R$, let $X(R)$ denote the set of $R$-points of $X$.

Consider a field $F$. For an algebraic $F$-group $G$, denote its identity connected component by $G^\circ$, and its Lie algebra by $\mathfrak{g}$. Denote by $\mathfrak{g}^*$ the linear dual of $\mathfrak{g}$. Given a linear algebraic $F$-group $P$, its Levi decompositions are written in the form $P = MU$, where $M$ is a Levi factor and $U := R_{\mathrm{u}}(P)$ is the unipotent radical.

Let $H$ be a connected reductive $F$-group. The center of $H$ is denoted by $Z_H$; the centralizer of $S \subset H$ is denoted by $Z_H(S)$. The Langlands dual group of $H$ is denoted by $H^\vee$, viewed as a pinned $\CC$-group where a pinning is fixed once and for all, and it will be conveniently identified with the group of its $\CC$-points.

When $F$ is local, the set of L-parameters (resp.\ Arthur parameters) of $H$ will be denoted by $\Phi(H)$ (resp.\ $\Psi(H)$). Detailed discussions will be given in \S\ref{sec:parameters} (see \S\ref{sec:global-parameters} for the global case as well).

For all $m \in \Z_{\geq 1}$, let $\bmu_m := \mu_m(\CC)$ be the group of $m$-th roots of unity in $\CC^{\times}$.
\index{mu-m@$\bmu_m$}

\paragraph{Parabolic subgroups}
Let $G$ be a connected reductive group over a field $F$.

Suppose a minimal parabolic subgroup $P_0 \subset G$ and a Levi decomposition $P_0 = M_0 U_0$ are given. Parabolic subgroups $P$ with $P \supset P_0$ are said to be standard; they have unique Levi factors $M$ satisfying $M \supset M_0$, and such Levi subgroups $M$ of $G$ are called standard. More generally, parabolic subgroups with Levi factors $M \supset M_0$ are said to be semi-standard; such Levi factors are unique, and Levi subgroups of $G$ arising in this way are called semi-standard.

\paragraph{Haar measures}
For each locally compact group $R$, let $\mes(R)$ be the $1$-dimensional $\CC$-vector space spanned by Haar measures of $R$, and denote its linear dual by $\mes(R)^\vee$. They serve to keep track of the dependence on Haar measures in various constructions.

For a connected reductive $F$-group $G$ over a local field $F$, we write $\mes(G) := \mes(G(F))$.
\index{mesG@$\mes(G)$}

We occasionally write $\mathrm{vol}(C) := \mu(C) \in [0, +\infty]$ for a measurable subset $C \subset X$ in a measure space $(X, \mu)$.

\paragraph{Representations}
All representations in this article are over $\CC$. The trivial representation of a group is denoted by $\mathbf{1}$.
\index{1@$\mathbf{1}$}

For a finite Abelian group $A$, its Pontryagin dual is denoted by $A^\vee$.

We will work with representations of $G(F)$ where $G$ is a connected reductive $F$-group and $F$ is a local field, or of finite coverings $\tilde{G}$ of $G(F)$ defined as central extensions of locally compact groups
\[ 1 \to \bmu_m \to \tilde{G} \xrightarrow{\rev} G(F) \to 1, \quad m \in \Z_{\geq 1}. \]
Representations in the adélic context are also considered. Their precise meanings are explained as follows.

\begin{itemize}
	\item For non-Archimedean $F$, representations are understood as the smooth admissible ones.
	\item For Archimedean $F$, representations are understood either as Harish-Chandra modules where a maximal compact subgroup $K \subset G(F)$ is chosen implicitly, or equivalently as smooth admissible Fréchet representations of moderate growth of $\tilde{G}$.
	\item Occasionally, unitary representations on Hilbert spaces which are not necessarily admissible will appear, for example the $L^2$-automorphic spectra in the adélic setting.
\end{itemize}
The notion of irreducibility must be interpreted accordingly. The precise choice is always clear in context, hence the categories in use will not be specified unless needed.

For similar reasons, we do not bother to distinguish unitary and unitarizable irreducible representations in this article.

In all cases, the Harish-Chandra character of a representation $\pi$ of finite length is denoted by $\Theta_\pi := \Tr(\pi)$.
\index{Theta-pi@$\Theta_\pi$}

Parabolic induction and Jacquet modules of representations are always assumed to be normalized.

Arthur's notation $\Pi_-(\tilde{G})$ will be adopted to denote the set of isomorphism classes of genuine irreducible representations of $\tilde{G}$, and $\Pi_{\mathrm{unit}, -}(\tilde{G})$ denotes the subset of unitary ones. Further variants and explanations will be given in the main text.

\paragraph{Symplectic and orthogonal groups}
Let $F$ be any field with $\mathrm{char}(F) \neq 2$. A symplectic $F$-vector space $(W, \lrangle{\cdot|\cdot})$ is a finite-dimensional $F$-vector space $W$ together with a non-degenerate alternating form $\lrangle{\cdot|\cdot}: W \times W \to F$. The corresponding symplectic group is denoted by $\Sp(W)$ or $\Sp(2n)$ where $n := \frac{1}{2} \dim W$.

A quadratic $F$-vector space $(V, q)$ or simply $V$ is a finite-dimensional $F$-vector space $V$ together with a quadratic map $q: V \to F$, such that the associated bilinear form $V \times V \to F$ is non-degenerate. The corresponding orthogonal group (resp.\ special orthogonal group) is denoted by $\Or(V, q) = \Or(V)$ (resp.\ $\SO(V, q) = \SO(V)$).

\index{SO(m)@$\SO(m)$}
The notation $\SO(m)$ denotes the split special orthogonal group associated with the $m$-dimensional quadratic $F$-vector space whose Gram matrix is
\[\begin{pmatrix}
	& & 1 \\
	& \reflectbox{$\ddots$} & \\
	1 & &
\end{pmatrix}.\]

These classical groups will often be identified with the groups of their $F$-points.

\paragraph{Twisted general linear groups}
\index{twisted $\GL(N)$}
The general linear group of rank $N$ over a given field is denoted by $\GL(N)$. Throughout this article, the twisted $\GL(N)$, denoted by $\widetilde{\GL}(N)$, means the reductive group $\GL(N)$ endowed with the involution
\begin{gather*}
	\tilde{\theta}(g) = \tilde{J} \cdot {}^{\mathrm{t}} g^{-1} \cdot \tilde{J}^{-1}, \\
	\tilde{J} := \begin{pmatrix}
		& & & 1 \\
		& & -1 & \\
		& \reflectbox{$\ddots$} & & \\
		(-1)^{N+1} & & &
	\end{pmatrix}.
\end{gather*}

Over a local field $F$, there is a fully fledged theory of twisted representations and characters on the twisted $\GL(N, F)$, which we will make free use of.

\section{Groups of metaplectic type and endoscopy}\label{sec:review}
Below is a general summary of the metaplectic groups, their representations and endoscopy. We refer to \cite[Chapters 2--3]{Li26} for a more detailed review.

\subsection{Local metaplectic coverings}\label{sec:local-Mp}
Let $(W, \lrangle{\cdot|\cdot})$ be a symplectic $F$-vector space, where $F$ is a field, $\mathrm{char}(F) \neq 2$. Fixing a symplectic basis
\[ e_1, \ldots, e_n, e_{-n}, \ldots, e_{-1} \]
of $W$, there is a corresponding standard Borel pair $(B, T)$ of $\Sp(W)$ where $B$ is the stabilizer of the flag $\{0\} \subsetneq Fe_1 \subsetneq \cdots \subsetneq \bigoplus_{i=1}^n Fe_i$, and $T \simeq \Gm^n$ corresponds to the decomposition $W = \bigoplus_{i=1}^n (Fe_i \oplus Fe_{-i})$, namely $(t_j)_{j=1}^n$ acts on $e_{\pm i}$ by $t_i^{\pm 1}$, for each $i$. The standard (resp.\ semi-standard) parabolic and Levi subgroups are thus defined. The Weyl group of $T$ is identified with $\mathfrak{S}_n \ltimes \{\pm 1\}^n$, acting by signed permutations on $\Gm^n$.

Now assume $F$ is local and $\mathrm{char}(F) = 0$. Given an additive character $\bpsi$ of $F$, the \emph{metaplectic group} associated with $(W, \lrangle{\cdot|\cdot})$ and $\bpsi$ in this work is the central extension of locally compact groups
\[ 1 \to \bmu_8 \to \Mp(W) \xrightarrow{\rev} \Sp(W) \to 1. \]
We sometimes denote $\Mp(W)$ by $\Mp(2n)$, and the covering splits if and only if $F = \CC$. It is described in terms of the Schrödinger model for the irreducible representations of the Heisenberg group $\mathrm{H}(W)$ with central character $\bpsi$; the appearance of $\bmu_8$ stems from the fact that Weil's constants $\gamma_\bpsi(\cdot)$ are $\bmu_8$-valued. Note that $\Mp(\{0\}) = \bmu_8$. We refer to \cite[\S 2.2, \S 2.4]{Li11} for details about the Heisenberg group, the Schrödinger model and the definition of $\Mp(W)$ adopted in this article.
\index{Sp-tilde-W@$\Mp(W)$}
\index{metaplectic group}

Write $G := \Sp(W)$ and denote the metaplectic covering by $\rev: \tilde{G} \to G(F)$. The group $G(F)$ acts on $\tilde{G}$ by conjugation. For every subgroup $L \subset G$ we set $\tilde{L} := \rev^{-1}(L(F))$. The Haar measures on $\tilde{G}$ and $G(F)$ are assumed to be compatibly chosen so that
\begin{equation*}
	\mathrm{vol}(C) = \mathrm{vol}(\rev^{-1}(C)), \quad C \subset G(F): \text{any compact subset}.
\end{equation*}

\begin{itemize}
	\item A representation $(\pi, V)$ of $\tilde{G}$ is said to be \emph{genuine} if $\pi(z) = z \cdot \identity_V$ for all $z \in \bmu_8$.
	\item A function $f: \tilde{G} \to \CC$ is said to be \emph{genuine} (resp.\ \emph{anti-genuine}) if $f(z\tilde{x}) = z f(\tilde{x})$ (resp.\ $f(z\tilde{x}) = z^{-1} f(\tilde{x})$) for all $\tilde{x} \in \tilde{G}$ and $z \in \bmu_8$.
\end{itemize}
\index{genuine}
\index{anti-genuine}

The \emph{Weil representation} of $\tilde{G}$ is a genuine admissible representation $\omega_\bpsi = \omega_\bpsi^+ \oplus \omega_\bpsi^-$, where $\omega_\bpsi^{\pm}$ are both unitary and irreducible, called the even ($+$) and odd ($-$) parts of $\omega_\bpsi$. These constructions are canonical once $(W, \lrangle{\cdot|\cdot})$ and $\bpsi$ are given.
\index{omega-psi@$\omega_{\bpsi}, \omega_{\bpsi}^{\pm}$}

The parabolic and Levi subgroups of $\tilde{G}$ are defined as preimages of parabolic and Levi subgroups of $G(F)$, respectively. We will need the following splittings of $\rev: \tilde{G} \to G(F)$.

\begin{itemize}
	\item Let $P$ be a parabolic subgroup of $G$ with unipotent radical $U$. There is a unique equivariant splitting of $\rev|_{\tilde{U}}: \tilde{U} \to U(F)$ relative to $P(F)$-conjugation. By viewing $U(F)$ as a closed subgroup of $\tilde{G}$, every Levi decomposition $P = MU$ lifts canonically to coverings as
	\[ \tilde{P} = \tilde{M} \cdot U(F). \]
	These properties hold for covering groups in general, see \cite[\S 2.2]{Li26}.
	
	\item Suppose we are given an orthogonal decomposition
	\[ W = W^\flat \oplus \bigoplus_{i=1}^r (\ell_i \oplus \ell'_i) \]
	where $W^\flat$ and all $\ell_i \oplus \ell'_i$ are non-degenerate subspaces of $W$, and $\ell_i, \ell'_i$ are totally isotropic and nonzero for all $i$. To such a decomposition are attached
	\begin{align*}
		M & := \Sp(W^\flat) \times \prod_{i=1}^r \GL(\ell_i), \\
		P & := \text{stabilizer of the flag}\; \ell_1 \subset \ell_1 \oplus \ell_2 \subset \cdots \subset \bigoplus_{i=1}^r \ell_i.
	\end{align*}
	Then $P$ is a parabolic subgroup with Levi component $M$, and all $P$ and $M$ arise in this manner; standard parabolic and Levi subgroups are so obtained by defining $W^\flat$ and $\ell_i, \ell'_i$ from the standard symplectic (ordered) basis.
	
	We have a canonical decomposition
	\[ \tilde{M} = \Mp(W^\flat) \times \prod_{i=1}^r \GL(\ell_i) \]
	in $\tilde{G}$. The splitting over $\GL$ factors is constructed in \cite[\S 3.1]{Li12a} via the Schrödinger model, hence depends on $\bpsi$, and is only available in the eightfold (or larger) metaplectic coverings.
\end{itemize}

Recall that $Z_G = \{\pm 1\}$, where $\pm 1$ is a shorthand for $\pm\identity \in G(F)$.

\begin{definition}\label{def:minus-1}
	\index{$-1 \in \tilde{G}$}
	There is a unique preimage of $-1$ in $\tilde{G}$, denoted by the same symbol $-1$, such that $\omega_\bpsi^{\pm}(-1) = \pm \identity$.
\end{definition}

See \cite[Définition 2.8, Proposition 2.9]{Li11}. Also recall that $-1$ is central in $\tilde{G}$.

\subsection{Unramified situation}\label{sec:unramified-situation}
Assume $F$ is a non-Archimedean local field with $F \supset \Q_p$ and $p > 2$. For an $\mathfrak{o}_F$-lattice $L \subset W$, define the lattice
\[ L^* := \left\{x \in W: \forall y \in L, \; \bpsi(\lrangle{x|y}) = 1 \right\}. \]
We say $L$ is \emph{self-dual} if $L = L^*$. Since $p > 2$, every self-dual $L$ determines an $\mathfrak{o}_F$-structure for $G$ such that
\begin{itemize}
	\item $K := G(\mathfrak{o}_F) = \Stab_{G(F)}(L)$ is a hyperspecial subgroup of $G(F)$;
	\item the lattice model for $\omega_\bpsi$ furnishes a splitting $s: K \to \tilde{G}$ of the metaplectic covering $\rev$ over $K$, see \cite[\S 2.4.2]{Li11};
	\item the element $s(-1)$ coincides with the $-1 \in \tilde{G}$ in Definition \ref{def:minus-1}.
\end{itemize}

With the data $(K, s)$ above, we say that we are in the \emph{unramified situation}. The Haar measure on $G(F)$ with $\mathrm{vol}(K) = 1$ is said to be \emph{unramified}.
\index{unramified situation}

\begin{definition}
	\index{$f_K$}
	\index{Hecke@$\mathcal{H}_{\asp}(K \backslash \tilde{G} / K)$}
	In the unramified situation, let $\mathcal{H}_{\asp}(K \backslash \tilde{G} / K)$ be the space of $K$-bi-invariant (via $s$) anti-genuine functions on $\tilde{G}$ of compact support. It forms a $\CC$-algebra under convolution with respect to the unramified Haar measure, whose unit is
	\[ f_K : \tilde{G} \to \CC, \quad f_K(\tilde{x}) = \begin{cases}
		z^{-1}, & \tilde{x} = zs(k), \; z \in \bmu_8, \; k \in K, \\
		0, & \tilde{x} \notin \tilde{K}.
	\end{cases}\]
	We call $\mathcal{H}_{\asp}(K \backslash \tilde{G} / K)$ the \emph{anti-genuine spherical Hecke algebra} of $\tilde{G}$.
\end{definition}

For a more precise description of $\mathcal{H}_{\asp}(K \backslash \tilde{G} / K)$ à la Satake, see \eqref{eqn:Satake-Mp}. In particular, this algebra is commutative.

\begin{definition}
	\index{$K$-spherical}
	Let $\pi$ be an irreducible genuine representation of $\tilde{G}$ in the unramified situation. If $\pi$ contains a nonzero $K$-invariant vector, we say $\pi$ is \emph{$K$-spherical}.
\end{definition}

\subsection{Adélic metaplectic coverings}
Now let $\dot{F}$ be a number field and $(\dot{W}, \lrangle{\cdot|\cdot})$ a symplectic $\dot{F}$-vector space. Choose an additive character $\dot{\bpsi} = \prod_v \bpsi_v$ of $\dot{F} \backslash \A$.

Fix an $\mathfrak{o}_{\dot{F}}$-lattice $L \subset \dot{W}$, which endows $\dot{G} := \Sp(\dot{W})$ with an $\mathfrak{o}_{\dot{F}}$-structure. For every place $v$, we have the covering $\rev_v: \dot{\tilde{G}}_v \to \dot{G}(\dot{F}_v)$ associated with $\dot{W}_v$ and $\bpsi_v$.

When $v \nmid \infty$, the $v$-completion of $L$ gives rise to $L_v \subset \dot{W}_v$. Define the finite set
\begin{equation}\label{eqn:V-ram}
	V_{\mathrm{ram}} := \left\{\begin{array}{r|l}
		v: \text{place of}\; \dot{F} & v \mid 2, \;\text{or}\; v \mid \infty, \\
		& \text{or}\; L_v \neq L_v^*, \\
		& \text{or}\; \bpsi_v \;\text{is not of conductor $\mathfrak{o}_v$}
	\end{array}\right\}.
\end{equation}
\index{Vram@$V_{\mathrm{ram}}$}

When $v \notin V_{\mathrm{ram}}$, we are in the unramified situation and $\rev_v$ splits over the hyperspecial subgroup $K_v = \Stab_{\dot{G}(\dot{F}_v)}(L_v) = \dot{G}( \mathfrak{o}_{\dot{F}_v})$. Define
\index{Sp-tilde-W-A@$\Mp(\dot{W}, \A)$}
\begin{align*}
	\mathbf{N} & := \left\{ (z_v)_v \in \bigoplus_v \bmu_8 : \prod_v z_v = 1 \right\}, \\
	\dot{\tilde{G}} = \Mp(\dot{W}, \A) & := \Resprod_v \dot{\tilde{G}}_v \big/ \mathbf{N}
\end{align*}

where $\Resprod_v$ is taken relative to $(K_v)_{v \notin V_{\mathrm{ram}}}$. Then:
\begin{itemize}
	\item $\Resprod_v \rev_v$ is trivial on $\mathbf{N}$ and yields a central extension
	\[ 1 \to \bmu_8 \to \dot{\tilde{G}} \xrightarrow{\rev} \dot{G}(\A) \to 1 \]
	which is independent of the choice of $L$;
	\item there exists a unique splitting $i: \dot{G}(\dot{F}) \to \dot{\tilde{G}}$ of $\rev$ over the rational points;
	\item as in the local setting, there are splittings over unipotent radicals and $\GL$ factors of Levi subgroups;
	\item with $-1_v \in \dot{\tilde{G}}_v$ as in Definition \ref{def:minus-1}, the $\mathbf{N}$-coset of $(-1_v)_v$ equals $i(-1)$, see \cite[Corollaire 2.17]{Li11}.
\end{itemize}

For every place $v$, the preimage of $\dot{G}(\dot{F}_v)$ in $\dot{\tilde{G}}$ is identified with $\dot{\tilde{G}}_v$. More generally, for any set $V$ of places of $\dot{F}$, the preimage $\dot{\tilde{G}}_V$ of $\Resprod_{v \in V} \dot{G}(\dot{F}_v)$ is the image of $\Resprod_{v \in V} \dot{\tilde{G}}_v$ in $\dot{\tilde{G}}$.

We call $\rev: \dot{\tilde{G}} \to \dot{G}(\A)$ the \emph{adélic metaplectic covering} associated with $(\dot{W}, \lrangle{\cdot|\cdot})$ and $\dot{\bpsi}$. There is still a notion of genuine representations and genuine (resp.\ anti-genuine) functions. The theory of genuine automorphic forms and automorphic representations \cite{MW94} applies in such a setting.

By the commutativity of anti-genuine spherical Hecke algebras, every irreducible genuine representation $\dot{\pi}$ of $\dot{\tilde{G}}$ factors uniquely into $\bigotimes'_v \dot{\pi}_v$, where $\dot{\pi}_v$ is a genuine irreducible representation of $\dot{\tilde{G}}_v$, and $\dot{\pi}_v$ is $K_v$-spherical for almost all $v \notin V_{\mathrm{ram}}$.

To give an example for later use in \S\ref{sec:global-principal}, define the adélic Weil representation as the genuine representation
\begin{equation}\label{eqn:adelic-Weil}
	\omega_{\dot{\bpsi}, \A} := \bigotimes\nolimits'_v \omega_{\bpsi_v}
\end{equation}
of $\dot{\tilde{G}}$; although $\omega_{\bpsi_v}$ is reducible, the restricted tensor product makes sense since $\omega^+_{\bpsi_v}$ is unramified whereas $\omega^-_{\bpsi_v}$ is not, for almost all $v$.
\index{omega-psi-A@$\omega_{\dot{\bpsi}, \A}$}

Denote by $\mathcal{A}_-(\dot{G}(\dot{F}) \backslash \dot{\tilde{G}})$ the space of genuine automorphic forms on $\dot{\tilde{G}}$. There is a canonical $\dot{\tilde{G}}$-equivariant linear map
\begin{equation}\label{eqn:theta-series}
	\vartheta_{\bpsi}: \omega_{\dot{\bpsi}, \A} \to \mathcal{A}_-(\dot{G}(\dot{F}) \backslash \dot{\tilde{G}}).
\end{equation}
To construct it, choose a Lagrangian subspace $\dot{\ell}$ of $\dot{W}$ and realize $\omega_{\dot{\bpsi}, \A}$ by the Schrödinger model on the Schwartz space of $(\dot{W} / \dot{\ell})(\A)$. Then one may take the $\vartheta$-series:
\[ \vartheta_{\bpsi}(f)(\tilde{g}) = \sum_{x \in \dot{W}/\dot{\ell}} \left( \omega_{\dot{\bpsi}, \A}(\tilde{g}) f \right)(x), \quad \tilde{g} \in \dot{\tilde{G}}. \]
Convergence and moderate growth in $\tilde{g}$ are guaranteed by the definition of Schwartz spaces; on the other hand Schrödinger models associated with different Lagrangian subspaces are canonically identified by partial Fourier transforms (see \cite[\S 2.4]{Li11}). Thus $\vartheta_{\bpsi}$ is independent of $\dot{\ell}$, in view of the Poisson summation formula.
\index{theta@$\vartheta_{\bpsi}$}

The elements of $\vartheta_{\bpsi}\left( \omega_{\dot{\bpsi}, \A} \right)$ are called \emph{elementary $\vartheta$-series}.
\index{elementary $\vartheta$-series}

\subsection{Endoscopy}\label{sec:endoscopy}
To begin with, suppose $F$ is local. Given $\tilde{G} = \Mp(W)$ with $\dim W = 2n$, the dual group of $\tilde{G}$ is
\index{G-tilde-vee@$\tilde{G}^\vee$}
\begin{equation*}
	\tilde{G}^\vee := \Sp(2n, \CC).
\end{equation*}
Equip $\tilde{G}^\vee$ with the standard pinning, and let the Galois group of $F$ act trivially on $\tilde{G}^\vee$. It is the same as the dual group of $\SO(2n+1)$.

\begin{remark}
	According to Weissman's formalism \cite{Weis18}, the Langlands correspondence involves not just $\tilde{G}^\vee$, but the L-group $\Lgrp{\tilde{G}}$ (Weil form) sitting in an extension of locally compact groups
	\[ 1 \to \tilde{G}^\vee \to \Lgrp{\tilde{G}} \to \Weil{F} \to 1. \]
	However, $\Lgrp{\tilde{G}}$ is canonically identified with $\tilde{G}^\vee \times \Weil{F}$ since $(W, \lrangle{\cdot|\cdot})$ and $\bpsi$ are chosen; see \cite[\S 11.6]{Weis18}. Hence it suffices to consider $\tilde{G}^\vee$.
\end{remark}

\begin{remark}\label{rem:reduced-center}
	As testified by the stabilization of trace formula \cite{Li26}, the role of $Z_{\tilde{G}^\vee}$ in endoscopy should be replaced by that of $Z_{\tilde{G}^\vee}^\circ = \{1\}$ in the metaplectic setting. We will see more manifestations of this principle.
\end{remark}

\begin{definition}
	\index{Endo-ell@$\Endo_{\elli}(\tilde{G})$}
	\index{G-shrek-bold@$\mathbf{G}^{"!}$}
	Elements of the set
	\[ \Endo_{\elli}(\tilde{G}) := \left\{ s \in \tilde{G}^\vee : s^2 = 1 \right\} \bigg/ \tilde{G}^\vee\text{-conj.} \]
	are called \emph{elliptic endoscopic data} of $\tilde{G}$. There is a natural identification
	\[ \Endo_{\elli}(\tilde{G}) \xleftrightarrow{1:1} \left\{ (n', n'') \in \Z_{\geq 0}^2 : n' + n'' = n \right\}. \]
	Indeed, $2n'$ (resp.\ $2n''$) is the multiplicity of $+1$ (resp.\ $-1$) as an eigenvalue of $s$. The endoscopic group corresponding to the class of $s$ is
	\[ G^! := \SO(2n' + 1) \times \SO(2n''+1). \]
	Note that $(G^!)^\vee$ is naturally identified with $Z_{\tilde{G}^\vee}(s)$.
	
	The notation $\mathbf{G}^!$ will be frequently used to designate the endoscopic datum that underlies the group $G^!$.
\end{definition}

In contrast with the endoscopy for $\SO(2n+1)$, we distinguish the endoscopic data $(n', n'')$ and $(n'', n')$, and we do not consider automorphisms of elliptic endoscopic data.

Consider $\mathbf{G}^! \in \Endo_{\elli}(\tilde{G})$. Define $\orbI_{\asp}(\tilde{G})$ (resp.\ $S\orbI(G^!)$) as in \cite[Definitions 2.3.4, 2.3.5]{Li26}, namely the quotient of the space of anti-genuine $C^\infty_c$-functions on $\tilde{G}$ (resp.\ $C^\infty_c$-functions on $G^!(F)$) modulo those with zero orbital integrals (resp.\ stable orbital integrals) along all strongly regular semisimple orbits.

The \emph{geometric transfer} in \cite[Theorem 3.8.1]{Li26} is a linear map
\begin{equation*}
	\Trans_{\mathbf{G}^!, \tilde{G}}: \orbI_{\asp}(\tilde{G}) \otimes \mes(G) \to S\orbI(G^!) \otimes \mes(G^!)
\end{equation*}
characterized by matching orbital integrals; see \cite[\S 3.8]{Li11} for an overview. Furthermore,
\begin{itemize}
	\item in the Archimedean case, $\orbI_{\asp}(\tilde{G})$ and $S\orbI(G^!)$ are naturally LF-spaces (i.e.\ locally convex inductive limits of sequences $\cdots \to V_n \xrightarrow{i_n} V_{n+1} \cdots$ of Fréchet spaces, such that each $i_n$ is a homeomorphism onto its image) and $\Trans_{\mathbf{G}^!, \tilde{G}}$ is continuous;
	\item when maximal compact groups $K \subset G(F)$ and $K^! \subset G^!(F)$ are chosen in the Archimedean case, every $\tilde{K}$-bi-finite $f \in C^\infty_{c, \asp}(\tilde{G})$ admits a $K^!$-bi-finite transfer $f^! \in C^\infty_c(G^!(F))$;
	\item in the unramified situation, one has the \emph{spherical fundamental lemma} \cite{Luo18}: for every hyperspecial subgroup $K^! \subset G^!(F)$, there is a natural homomorphism between spherical Hecke algebras
	\[ b_{\mathbf{G}^!, \tilde{G}}: \mathcal{H}_{\asp}(K \backslash \tilde{G} / K) \to \mathcal{H}(K^! \backslash G^!(F) / K^!) \]
	making the following diagram commute:
	\[\begin{tikzcd}
		\mathcal{H}_{\asp}(K \backslash \tilde{G} / K) \arrow[r, "{b_{\mathbf{G}^!, \tilde{G}}}"] \arrow[d] & \mathcal{H}(K^! \backslash G^!(F) / K^!) \arrow[d] \\
		\orbI_{\asp}(\tilde{G}) \arrow[r, "{\Trans_{\mathbf{G}^!, \tilde{G}}}"'] & S\orbI(G^!)
	\end{tikzcd}\]
	where $\mes(G)$ (resp.\ $\mes(G^!)$) is trivialized by the unramified Haar measures. In particular, the transfer of $f_K$ is represented by $\mathbf{1}_{K^!}$.
\end{itemize}
\index{Trans@$\Trans_{\mathbf{G}^{"!}, \tilde{G}}$}
\index{fundamental lemma}

Here we used the basic fact that for non-Archimedean $F$ with residual characteristic $p \neq 2$, all hyperspecial subgroups of $G^!(F)$ are conjugate; the Haar measure characterized by $\mathrm{vol}(K^!) = 1$ is said to be unramified.

When $(n', n'') = (n, 0)$ in the unramified situation, the map $b_{\mathbf{G}^!, \tilde{G}}$ is an isomorphism onto $\mathcal{H}(K^! \backslash H / K^!)$ where $H := \SO(2n+1)$; we denote the composition of $b_{\mathbf{G}^!, \tilde{G}}$ with the Satake isomorphism for $H$ by
\begin{equation}\label{eqn:Satake-Mp}
	\mathcal{S}_{\tilde{G}}: \mathcal{H}_{\asp}(K \backslash \tilde{G} / K) \rightiso \CC\left[ H^\vee /\!/ H^\vee \right] = \CC\left[ \tilde{G}^\vee /\!/ \tilde{G}^\vee \right].
\end{equation}

The transfer factor lurking behind $\Trans_{\mathbf{G}^!, \tilde{G}}$ is canonically defined: this is because we have fixed $(W, \lrangle{\cdot|\cdot})$ and $\bpsi$.

Denote by $D_-(\tilde{G})$ (resp.\ $SD(G^!)$) the dual spaces of $\mathcal{I}_{\asp}(\tilde{G})$ (resp.\ $S\orbI(G^!)$), taken in the topological sense when $F$ is Archimedean. They are the spaces of genuine invariant (resp.\ stably invariant) distributions on $\tilde{G}$ (resp.\ $G^!(F)$). By dualizing $\Trans_{\mathbf{G}^!, \tilde{G}}$, we obtain the \emph{spectral transfer}
\begin{equation*}
	\trans_{\mathbf{G}^!, \tilde{G}}: SD(G^!) \otimes \mes(G^!)^\vee \to D_-(\tilde{G}) \otimes \mes(G)^\vee .
\end{equation*}
\index{D@$D_-(\tilde{G})$}
\index{SD@$SD(G^{"!})$}
\index{trans@$\trans_{\mathbf{G}^{"!}, \tilde{G}}$}

For every genuine admissible representation $\pi$ of $\tilde{G}$ of finite length, its character $\Theta_\pi$ belongs to $D_-(\tilde{G}) \otimes \mes(G)^\vee$. Let $D_{\mathrm{spec}, -}(\tilde{G}) \subset D_-(\tilde{G})$ be the subspace spanned by all these $\Theta_\pi$ above after trivializing $\mes(G)$; the subspace $SD_{\mathrm{spec}}(G^!) \subset SD(G^!)$ is defined similarly, using stable characters. See \cite[\S 2.6]{Li26}.
\index{D-spec@$D_{\mathrm{spec}, -}(\tilde{G})$}
\index{SD-spec@$SD_{\mathrm{spec}}(G^{"!})$}

\begin{theorem}\label{prop:spec-trans}
	The spectral transfer $\trans_{\mathbf{G}^!, \tilde{G}}$ restricts to a linear map
	\[ SD_{\mathrm{spec}}(G^!) \otimes \mes(G^!)^\vee \to D_{\mathrm{spec}, -}(\tilde{G}) \otimes \mes(G)^\vee. \]
\end{theorem}
\begin{proof}
	This follows from \cite[(2.6.2) and Theorem 3.9.2]{Li26}, together with the fact that the transfer commutes with parabolic induction up to a ``central twist'', see \cite[Proposition 3.8.4]{Li26}.
\end{proof}

Finally, the definition of $\tilde{G}^\vee$ and $\Endo_{\elli}(\tilde{G})$ carries over verbatim to the case of adélic coverings $\dot{\tilde{G}} = \Mp(\dot{W}, \A)$. In view of the fundamental lemma for spherical units, the geometric transfer
\[ \Trans_{\dot{\mathbf{G}}^!, \dot{\tilde{G}}}: \bigotimes\nolimits'_v \orbI_{\asp}(\dot{\tilde{G}}_v) \otimes \mes(\dot{G}_v) \to \bigotimes\nolimits'_v S\orbI(\dot{G}^!_v) \otimes \mes(\dot{G}^!_v) \]
is well-defined, where the $\bigotimes\nolimits'_v$ are defined with respect to $f_{K_v}$, $\mathbf{1}_{\dot{G}^!(\mathfrak{o}_v)}$ and the unramified Haar measures for almost all $v \nmid \infty$.

\subsection{Groups of metaplectic type}\label{sec:adelic-type}
Let $F$ (resp.\ $\dot{F}$) be a local field of characteristic zero (resp.\ number field).

\begin{definition}
	\index{group of metaplectic type}
	Following \cite[Definition 3.3.1]{Li26}, a \emph{group of metaplectic type} is
	\begin{itemize}
		\item (local case) a covering group of the form $\Mp(W) \times \prod_{i=1}^r \GL(n_i, F)$;
		\item (global case) a covering group of the form $\Mp(W, \A) \times \prod_{i=1}^r \GL(n_i, \A)$.
	\end{itemize}
	Here $(W, \lrangle{\cdot|\cdot})$ is a chosen symplectic vector space over $F$ (resp.\ $\dot{F}$), possibly zero, and an additive character $\bpsi$ of $F$ (resp.\ $\dot{F} \backslash \A$) is chosen as well.
\end{definition}

By \S\ref{sec:local-Mp}, the class of groups of metaplectic type is closed under passing to Levi subgroups, in both the local and global cases.

For a group of metaplectic type $\tilde{M} = \Mp(W) \times \prod_{i=1}^r \GL(n_i, F)$ or $\Mp(W, \A) \times \prod_{i=1}^r \GL(n_i, \A)$, where $\dim W = 2n$, define
\[ \tilde{M}^\vee = \Sp(2n, \CC) \times \prod_{i=1}^r \GL(n_i, \CC). \]
The elliptic endoscopic data of groups of metaplectic type are defined factor-wise: they are tautological on the $\GL$ factors. The transfer of orbital integrals or distributions extends immediately to the groups of metaplectic type.

One can then define the endoscopic data of a metaplectic covering $\Mp(W)$ or $\Mp(W, \A)$ to be the elliptic endoscopic data of Levi subgroups, taken up to conjugation; the transfer also extends to this setting. For details, see \cite[\S 3.4]{Li26}. No direct use of non-elliptic endoscopic data will be made in this article.

\subsection{Passing to twofold coverings}\label{sec:MMp}
Consider a local metaplectic covering $\tilde{G} = \Mp(W)$ first. Its derived subgroup $\tilde{G}_{\mathrm{der}}$ is a twofold non-split covering of $G(F)$ if $F \neq \CC$, and $\rev$ induces $\tilde{G}_{\mathrm{der}} \rightiso G(F)$ when $F = \CC$. Define the following subgroup of $\Mp(W)$:
\[ \MMp(W) := \begin{cases}
	\Mp(W)_{\mathrm{der}}, & F \neq \CC, \\
	\bmu_2 \times \Sp(W), & F = \CC;
\end{cases}\]
it always sits in the topological central extension
\[ 1 \to \bmu_2 \to \MMp(W) \xrightarrow{\rev} \Sp(W) \to 1, \]
and $\rev: \Mp(W) \to \Sp(W)$ is the push-out of $\MMp(W) \to \Sp(W)$ via $\bmu_2 \hookrightarrow \bmu_8$.
\index{MMp-W@$\MMp(W)$}

The twofold covering $\MMp(W)$ is commonly used in the literature. The description of $\MMp(W)$ in terms of $2$-cocycles of Rao or Lion--Perrin, as summarized in \cite[\S 2.4]{Li11}, involves not only $(W, \lrangle{\cdot|\cdot})$ but also $\bpsi$. Nevertheless, the covering is independent of $\bpsi$ by the following well-known fact.

\begin{proposition}\label{prop:MMp-uniqueness}
	Given additive characters $\bpsi$ and $\bpsi'$, denote the corresponding twofold coverings by $\MMp(W)^{\bpsi}$ and $\MMp(W)^{\bpsi'}$, respectively. There exists a unique isomorphism $\MMp(W)^{\bpsi} \rightiso \MMp(W)^{\bpsi'}$ that lifts $\identity_{\Sp(W)}$ and induces $\identity$ on $\bmu_2$.
\end{proposition}
\begin{proof}
	This follows from the classification of coverings of $G(F)$ in \cite[Theorem 10.4]{Mo68}.
\end{proof}

In general, the splitting over $\GL$ factors of the Levi subgroups does not exist within $\MMp(W)$, and the element $-1$ in Definition \ref{def:minus-1} does not land in $\MMp(W)$.

Our earlier conventions on Haar measures, genuine functions, distributions and representations work for $\MMp(W)$ as well; there is no distinction between genuine and anti-genuine objects. Since $\Mp(W)$ is the push-out of $\MMp(W)$, the genuine representation theories of $\MMp(W)$ and $\Mp(W)$ are the same, as summarized below.

\begin{proposition}\label{prop:MMp-dist}
	Denoting by $D_-(\MMp(W))$ the dual of genuine $C^\infty_c$-functions on $\MMp(W)$, we have a canonical isomorphism
	\[ D_-(\Mp(W)) \rightiso D_-(\MMp(W)); \]
	this map can be viewed as restriction of distributions. On the other hand, restriction of representations induces an equivalence of categories
	\[ \left\{ \text{genuine smooth representations of}\; \Mp(W) \right\} \rightiso \left\{ \text{those of}\; \MMp(W) \right\}, \]
	and similarly for genuine admissible representations, and so on.
\end{proposition}

Via these isomorphisms, all $\omega_{\bpsi}^{\pm}$ for various $\bpsi$ can co-exist as genuine representations on $\MMp(W)$.

In the unramified situation, the section $s: K \to \Mp(W)$ factors through $\MMp(W)$ by \cite[Chapitre 2, Lemme II.10]{MVW87}.

Finally, in the global setting one defines the twofold covering
\[ 1 \to \bmu_2 \to \MMp(\dot{W}, \A) \to \Sp(\dot{W}, \A) \to 1 \]
by $\MMp(\dot{W}, \A) := \Mp(\dot{W}, \A)_{\mathrm{der}}$; equivalently, $\MMp(\dot{W}, \A)$ is a quotient of $\Resprod_v \MMp(\dot{W}_v)$. Therefore $\Mp(\dot{W}, \A)$ is the push-out of $\MMp(\dot{W}, \A)$ via $\bmu_2 \hookrightarrow \bmu_8$. The unique section $i: \Sp(\dot{W}, \dot{F}) \to \Mp(\dot{W}, \A)$ factors through $\MMp(\dot{W}, \A)$.
\index{MMp-W-A@$\MMp(\dot{W}, \A)$}

As in the local setting, the theory of genuine automorphic forms and representations on $\MMp(\dot{W}, \A)$ in the sense of \cite{MW94} is the same as that on $\Mp(\dot{W}, \A)$.

The analog of Proposition \ref{prop:MMp-uniqueness} also holds for $\MMp(\dot{W}, \A)$: see \cite[Theorems 12.2, 12.3]{Mo68}.

\section{Arthur parameters}\label{sec:parameters}
\subsection{Local Arthur parameters}\label{sec:local-parameters}
For a local field $F$ of characteristic zero, define the local Langlands group as
\begin{equation*}
	\mathcal{L}_F := \begin{cases}
		\Weil{F}, & \text{if $F$ is Archimedean,} \\
		\Weil{F} \times \SL(2, \CC), & \text{if $F$ is non-Archimedean.}
	\end{cases}
\end{equation*}
\index{L-F@$\mathcal{L}_F$}

Let $N \in \Z_{\geq 1}$. Following Arthur \cite[Chapter 1]{Ar13}, consider the set $\Psi(N)$ of equivalence classes of representations
\[ \psi: \mathcal{L}_F \times \SL(2, \CC) \to \GL(V), \]
where $V = V_\psi$ is an $N$-dimensional $\CC$-vector space, such that $\psi|_{\Weil{F}}$ is continuous semisimple with bounded image and $\psi|_{\SL(2, \CC)}$ is algebraic. Write $\dim \psi := \dim V$.

Define $\Phi_{\mathrm{bdd}}(N)$ by demanding furthermore that $\psi|_{\SL(2, \CC)}$ is trivial. The elements $\phi \in \Phi_{\mathrm{bdd}}(N)$ will be viewed as representations of $\mathcal{L}_F$, taken up to equivalence.

Dropping the condition of bounded image, we obtain larger sets $\Psi^+(N) \supset \Psi(N)$ and $\Phi(N) \supset \Phi_{\mathrm{bdd}}(N)$.
\index{Psi-N@$\Psi(N), \Psi^+(N), \tilde{\Psi}(N), \tilde{\Psi}^+(N)$}
\index{Phi-N@$\Phi_{\mathrm{bdd}}(N), \Psi(N), \tilde{\Phi}_{\mathrm{bdd}}(N), \tilde{\Phi}(N)$}

Denote the contragredient of $\psi$ by $\psi^\vee$. We say $\psi$ is self-dual if $\psi^\vee \simeq \psi$. The self-dual condition cuts out subsets $\widetilde{\Psi}^+(N) \supset \widetilde{\Psi}(N)$ and $\widetilde{\Phi}(N) \supset \widetilde{\Phi}_{\mathrm{bdd}}(N)$. In general, a self-dual representation on an $N$-dimensional $\CC$-vector space $V$ is said to be of orthogonal (resp.\ symplectic) type if there exists a non-degenerate symmetric (resp.\ alternating) invariant form on $V$.

\begin{definition}
	\index{rb@$r(b)$}
	For every $b \in \Z_{\geq 1}$, let $r(b)$ denote the unique $b$-dimensional irreducible algebraic representation of $\SL(2, \CC)$. Specifically, let $\mathrm{std}$ denote the $2$-dimensional standard representation of $\SL(2, \CC)$. Then
	\[ r(b) = \Sym^{b-1}(\mathrm{std}). \]
	It follows that $r(b)$ is always self-dual, of symplectic (resp.\ orthogonal) type when $b$ is even (resp.\ odd).
\end{definition}

Every $\psi \in \Psi(N)$ (resp.\ $\Psi^+(N)$) has a unique decomposition
\begin{equation}\label{eqn:psi-decomp}
	\psi = \bigoplus_{i \in I} m_i \psi_i
\end{equation}
where $I = I_\psi$ is a finite set, $\psi_i \in \Psi(N_i)$ (resp.\ $\Psi^+(N_i)$) is simple, $m_i \in \Z_{\geq 1}$, such that $i \neq j \implies \psi_i \not\simeq \psi_j$ and $\sum_{i \in I} m_i N_i = N$. We may further write
\[ \psi_i = \phi_i \boxtimes r(b_i) \]
for a unique pair $(\phi_i, b_i)$, where $b_i \in \Z_{\geq 1}$ and $\phi_i \in \Phi_{\mathrm{bdd}}(N_i/b_i)$ (resp.\ $\Phi(N_i/b_i)$) is simple.

Let $\Psi_{\mathrm{symp}}(N) \subset \widetilde{\Psi}(N)$ be the subset of those $\psi$ of symplectic type, and define $\Psi^+_{\mathrm{symp}}(N) \subset \tilde{\Psi}^+(N)$ similarly. For every $\psi \in \Psi^+_{\mathrm{symp}}(N)$, the set $I$ in \eqref{eqn:psi-decomp} decomposes into a disjoint union
\[ I = I^+ \sqcup I^- \sqcup J \sqcup J' \]
where $J$ and $J'$ are related by a bijection $j \leftrightarrow j'$, such that
\index{Ipm@$I^+, I^-, J$}
\begin{itemize}
	\item if $i \in I^+$ then $\psi_i = \phi_i \boxtimes r(b_i)$ is of symplectic type, i.e.
	\begin{itemize}
		\item either $\phi_i$ is symplectic and $b_i$ is odd,
		\item or $\phi_i$ is orthogonal and $b_i$ is even;
	\end{itemize}
	\item if $i \in I^-$ then $\psi_i = \phi_i \boxtimes r(b_i)$ is of orthogonal type, i.e.
	\begin{itemize}
		\item either $\phi_i$ is orthogonal and $b_i$ is odd,
		\item or $\phi_i$ is symplectic and $b_i$ is even,
	\end{itemize}
	moreover $m_i$ is even in this case;
	\item if $j \in J$, then $\psi_j = \phi_j \boxtimes r(b_j)$ is not self-dual and
	\[ \phi_{j'} \simeq \phi_j^\vee, \quad m_{j'} = m_j. \]
\end{itemize}
The subsets $I^{\pm}$ are uniquely determined, though $J$ is not. Hence \eqref{eqn:psi-decomp} becomes
\begin{equation}\label{eqn:psi-decomp-2}
	\begin{aligned}
		\psi & = \bigoplus_{i \in I^+ \sqcup I^-} m_i \left( \phi_i \boxtimes r(b_i) \right) \oplus \bigoplus_{j \in J} m_j \left( \psi_j \oplus \psi_j^\vee \right) \\
		& = \bigoplus_{i \in I^+ \sqcup I^-} m_i \left( \phi_i \boxtimes r(b_i) \right) \oplus \bigoplus_{j \in J} m_j \left( \phi_j \oplus \phi_j^\vee \right) \boxtimes r(b_j).
	\end{aligned}
\end{equation}

\begin{definition}\label{def:good-parity}
	\index{good parity}
	\index{quasi-good parity}
	Given $\psi \in \Psi^+_{\mathrm{symp}}(N)$, if $I = I^+$ then $\psi$ is said to be of \emph{good parity}. If $(\det\phi_i)(-1)^{b_i} = 1$ for all $i \in I^- \sqcup J$, then $\psi$ is said to be of \emph{quasi-good parity}.
\end{definition}

The definition of quasi-good parity is independent of the choice of $J$. Note that $(\det\phi_i)(-1)^{b_i} = 1$ holds automatically when $i \in I^+$.

Now consider a metaplectic group $\Mp(W)$ with $\dim W = 2n$.

\begin{definition}
	\index{Psi-G-tilde@$\Psi_2(\tilde{G}), \Psi_{\mathrm{gp}}(\tilde{G}), \Psi(\tilde{G}), \Psi^+(\tilde{G})$}
	\index{Phi-G-tilde@$\Phi_{2, \mathrm{bdd}}(\tilde{G}), \Phi_{\mathrm{gp, bdd}}(\tilde{G}), \Phi_{\mathrm{bdd}}(\tilde{G}), \Phi(\tilde{G})$}
	Set $\Psi(\tilde{G}) := \Psi_{\mathrm{symp}}(2n)$; its elements are called \emph{Arthur parameters} for $\tilde{G}$, which can be viewed as homomorphisms $\psi: \mathcal{L}_F \times \SL(2, \CC) \to \Sp(2n, \CC) = \tilde{G}^\vee$ taken up to $\tilde{G}^\vee$-conjugacy. In a similar vein, define the sets
	\[\begin{tikzcd}[row sep=small, column sep=small]
		\Psi^+(\tilde{G}) \arrow[phantom, r, "\supset" description] \arrow[phantom, d, "\supset" description, sloped] & \Psi(\tilde{G}) \arrow[phantom, d, "\supset" description, sloped] \\
		\Phi(\tilde{G}) \arrow[phantom, r, "\supset" description] & \Phi_{\mathrm{bdd}}(\tilde{G}).
	\end{tikzcd}\]
	Elements of $\Phi(\tilde{G})$ (resp.\ $\Phi_{\mathrm{bdd}}(\tilde{G})$) are called \emph{L-parameters} (resp.\ \emph{bounded L-parameters}) of $\tilde{G}$.
	
	An Arthur parameter $\psi$ is said to be \emph{discrete} if it does not factor through any proper Levi subgroup of $\tilde{G}^\vee$; equivalently, $I = I^+$ and $m_i = 1$ for all $i$ in \eqref{eqn:psi-decomp-2}. The same notion pertains to L-parameters as well.
	
	We put the subscript $2$ (resp.\ $\mathrm{gp}$) to designate sets of discrete parameters (resp.\ parameters of good parity). In this manner we obtain
	\begin{gather*}
		\Psi(\tilde{G}) \supset \Psi_{\mathrm{gp}}(\tilde{G}) \supset \Psi_2(\tilde{G}), \\
		\Phi_{\mathrm{bdd}}(\tilde{G}) \supset \Phi_{\mathrm{gp, bdd}}(\tilde{G}) \supset \Phi_{2, \mathrm{bdd}}(\tilde{G}),
	\end{gather*}
	and so forth. Note that $\Psi^+_{\mathrm{gp}}(\tilde{G}) = \Psi_{\mathrm{gp}}(\tilde{G})$.
\end{definition}

All these notions involve only $\tilde{G}^\vee$, therefore they coincide with the same-named sets defined for $H := \SO(2n+1)$ in \cite[Chapter 1]{Ar13}, namely
\[ \Psi(\tilde{G}) = \Psi(H), \quad \Phi(\tilde{G}) = \Phi(H), \quad \text{and so forth.} \]

More generally, for each $\mathbf{G}^! \in \Endo_{\elli}(\tilde{G})$ we have the evident map $\Psi^+(G^!) \to \Psi^+(\tilde{G})$ induced by $(G^!)^\vee \hookrightarrow \tilde{G}^\vee$. It simply merges the formal combinations \eqref{eqn:psi-decomp} for $\SO(2n'+1)$ and $\SO(2n''+1)$ when $\mathbf{G}^!$ corresponds to $(n', n'')$. It restricts to $\Psi(G^!) \to \Psi(\tilde{G})$ and $\Psi_{\mathrm{gp}}(G^!) \to \Psi_{\mathrm{gp}}(\tilde{G})$.

Finally, consider a group of metaplectic type
\[ \tilde{M} = \Mp(W^\flat) \times \prod_{i=1}^r \GL(n_i, F), \quad \dim W^\flat = 2n^\flat . \]
The sets of parameters $\Psi(\tilde{M}) \supset \Psi_2(\tilde{M})$, etc., are defined in the obvious way. For example, letting $\tilde{G}^\flat := \Mp(W^\flat)$, then $\psi \in \Psi(\tilde{M})$ consists of a pair $(\psi_0, \psi_{\GL})$ where $\psi_0 \in \Psi(\tilde{G}^\flat)$ and $\psi_{\GL} = (\psi_{\GL, i})_{i=1}^r \in \prod_{i=1}^r \Psi(n_i)$; it is discrete if and only if all $\psi_{\GL, i}$ are simple and $\psi_0 \in \Psi_2(\tilde{G}^\flat)$.

When $\tilde{M}$ embeds as a Levi subgroup of $\tilde{G} = \Mp(W)$ (equivalently, $\dim W^\flat + 2\sum_{i=1}^r n_i = \dim W$), there is a canonical map
\begin{equation}\label{eqn:psi-induction}
	\begin{tikzcd}[row sep=tiny]
		\Psi^+(\tilde{M}) \arrow[r] & \Psi^+(\tilde{G}) \\
		\Psi(\tilde{M}) \arrow[r] \arrow[phantom, u, "\subset" description, sloped] & \Psi(\tilde{G}) \arrow[phantom, u, "\subset" description, sloped] \\
		(\psi_0, \psi_{\GL}) \arrow[mapsto, r] & \psi_0 \oplus \bigoplus_{i=1}^r (\psi_{\GL, i} \oplus \psi_{\GL, i}^\vee).
	\end{tikzcd}
\end{equation}
The same is true for L-parameters. All these are identical to the story for $H$ and its Levi subgroups, treated in \cite{Ar13}.

\subsection{Centralizers and component groups}
Let $\psi \in \Psi^+(\tilde{G})$. Define its centralizer group and the corresponding component group as
\begin{align*}
	\index{S-psi@$S_\psi, \EuScript{S}_\psi$}
	S_\psi & := Z_{\tilde{G}^\vee}(\Image(\psi)), \\
	\EuScript{S}_\psi & := \pi_0(S_\psi).
\end{align*}

For $\psi$ in the form of \eqref{eqn:psi-decomp-2}, by \cite[\S 1.4]{Ar13} we have canonical isomorphisms
\begin{equation}\label{eqn:S-psi}
	\begin{aligned}
		S_\psi & \simeq \prod_{i \in I^+} \Or(m_i, \CC) \times \prod_{i \in I^-} \Sp(m_i, \CC) \times \prod_{j \in J} \GL(m_j, \CC), \\
		\EuScript{S}_\psi & \simeq \bmu_2^{I^+};
	\end{aligned}
\end{equation}
the quotient map $S_\psi \to \EuScript{S}_\psi$ is given by taking determinants in $\Or(m_i, \CC)$.

Note that in the case of $\SO(2n+1)$, although the dual group is the same $\Sp(2n, \CC)$ with trivial Galois action, what one considers in \cite{Ar13} are the quotients of these groups by the center of $\Sp(2n, \CC)$, generated by $(-1, \ldots, -1) \in S_\psi$, whose image in $\EuScript{S}_\psi \simeq \bmu_2^{I^+}$ is $((-1)^{m_i})_{i \in I^+}$; cf.\ Remark \ref{rem:reduced-center}.

\begin{definition}\label{def:s-psi}
	\index{s-psi@$s_\psi$}
	Set $s_\psi \in S_\psi$ to be
	\[ s_\psi = \psi\left( 1, \begin{pmatrix} -1 & \\ & -1 \end{pmatrix} \right). \]
\end{definition}

Writing each $\psi_i$ in \eqref{eqn:psi-decomp-2} as $\phi_i \boxtimes r(b_i)$, it follows from the description of $r(b_i)$ that for all $i \in I^{\pm}$ (resp.\ $j \in J$), the projection of $s_\psi$ to the corresponding direct factors of $S_\psi$ equals $1$ if $b_i$ (resp.\ $b_j$) is odd, $-1$ if $b_i$ (resp.\ $b_j$) is even.

From this, one easily verifies that if $\mathbf{G}^! \in \Endo_{\elli}(\tilde{G})$ and $\psi^! \mapsto \psi$ under $\Psi^+(G^!) \to \Psi^+(\tilde{G})$, then $(G^!)^\vee \hookrightarrow \tilde{G}^\vee$ induces
\begin{equation}\label{eqn:S-functoriality}
	S_{\psi^!} \hookrightarrow S_\psi, \quad \text{such that}\; s_{\psi^!} \mapsto s_\psi .
\end{equation}

\subsection{Global Arthur parameters}\label{sec:global-parameters}
Consider now a number field $\dot{F}$ and $\dot{\tilde{G}} = \Mp(\dot{W}, \A)$ with respect to an additive character $\dot{\bpsi} = \prod_v \bpsi_v$ of $\dot{F} \backslash \A$. Let $n := \frac{1}{2}\dim \dot{W}$.

The automorphic Langlands group $\mathcal{L}_{\dot{F}}$ is currently out of reach, thus we follow Arthur \cite[\S 1.4]{Ar13} to define the global version $\dot{\Psi}(N)$ as the set of formal direct sums of pairs $(\dot{\phi}, b)$ where $\dot{\phi}$ is a unitary irreducible cuspidal automorphic representation of $\GL(N_{\dot{\phi}}, \A)$ for some $N_{\dot{\phi}} \geq 1$ (we also write $\dim \dot{\phi} := N_{\dot{\phi}}$) and $b \in \Z_{\geq 1}$, written as
\begin{equation}\label{eqn:dot-psi-decomp}
	\dot{\psi} = \bigoplus_{i \in I} m_i \dot{\phi}_i \boxtimes r(b_i), \quad m_i \in \Z_{\geq 1},
\end{equation}
where $\sum_{i \in I} m_i b_i \dim\dot{\phi}_i = N$ and the pairs $(\dot{\phi}_i, b_i)$ are all distinct.
\index{Psi-N-dot@$\dot{\Psi}(N)$}

We shall use the strong multiplicity one theorem \cite{JS81} for $\GL(N)$, namely: $\dot{\psi}_v = \dot{\psi}'_v$ for almost all $v$ implies $\dot{\psi} = \dot{\psi}'$.

The dual $\dot{\psi}^\vee = \bigoplus_i m_i \dot{\phi}_i^\vee \boxtimes r(b_i)$ is defined by taking the dual of cuspidal automorphic representations $\dot{\phi}_i$. A self-dual cuspidal automorphic representation $\dot{\phi}$ of $\GL(N, \A)$ is said to be of symplectic (resp.\ orthogonal) type if $L(s, \dot{\phi}, \wedge^2)$ (resp.\ $L(s, \dot{\phi}, \Sym^2)$) has a pole at $s=1$; by the discussions before \cite[Theorem 1.5.3]{Ar13}, exactly one of these $L$-functions has a pole at $s=1$, and that pole is simple.

Therefore, by turning the discussions before \eqref{eqn:psi-decomp-2} into definitions, one defines the subset $\dot{\Psi}_{\mathrm{symp}}(N)$ of self-dual parameters of symplectic type. It consists of $\bigoplus_{i \in I} m_i \dot{\phi}_i \boxtimes r(b_i) \in \dot{\Psi}(N)$ such that $I = I^+ \sqcup I^- \sqcup J \sqcup J'$ where $J$ and $J'$ are related by a bijection $j \leftrightarrow j'$: for $i \in I^+$, $\dot{\phi}_i$ is self-dual and $\dot{\phi}_i \boxtimes r(b_i)$ is of symplectic type (i.e.\ $\dot{\phi}_i$ symplectic and $b_i$ odd, or $\dot{\phi}_i$ orthogonal and $b_i$ even); for $i \in I^-$ the opposite types hold and moreover $m_i$ is even; for $j \in J$, $\dot{\phi}_j$ is not self-dual, $\dot{\phi}_{j'} \simeq \dot{\phi}_j^\vee$ and $(m_{j'}, b_{j'}) = (m_j, b_j)$. One then defines $\dot{\Psi}(\dot{\tilde{G}}) := \dot{\Psi}_{\mathrm{symp}}(2n)$ and its subsets
\begin{equation*}
	\dot{\Psi}(\dot{\tilde{G}}) \supset \underbracket{\dot{\Psi}_{\mathrm{gp}}(\dot{\tilde{G}})}_{I^+ = I} \supset \underbracket{\dot{\Psi}_2(\dot{\tilde{G}})}_{\substack{I^+ = I \\ \forall i, \; m_i = 1}}.
\end{equation*}
Elements of $\dot{\Psi}(\dot{\tilde{G}})$ are said to be the \emph{global Arthur parameters} for $\dot{\tilde{G}}$.
\index{Psi-dot-tilde-G@$\dot{\Psi}(\dot{\tilde{G}}), \dot{\Psi}_{\mathrm{gp}}(\dot{\tilde{G}}), \dot{\Psi}_2(\dot{\tilde{G}})$}

For every $\dot{\mathbf{G}}^! \in \Endo_{\elli}(\dot{\tilde{G}})$, there is an evident map $\dot{\Psi}(\dot{G}^!) \to \dot{\Psi}(\dot{\tilde{G}})$ that merges two formal direct sums \eqref{eqn:dot-psi-decomp} as in the local setting.

The description \eqref{eqn:S-psi} of the centralizer is turned into a definition
\[ S_{\dot{\psi}} = \prod_{i \in I^+} \Or(m_i, \CC) \times \prod_{i \in I^-} \Sp(m_i, \CC) \times \prod_{j \in J} \GL(m_j, \CC), \]
and $\EuScript{S}_{\dot{\psi}} = \pi_0(S_{\dot{\psi}}) \rightiso \bmu_2^{I^+}$ accordingly.

Alternatively, one can use Arthur's definition \cite[(1.4.4)]{Ar13} of the makeshift Langlands group $\mathcal{L}_{\dot{\psi}}$ and view $\dot{\psi} \in \dot{\Psi}(\dot{\tilde{G}})$ as a homomorphism
\[ \dot{\psi}^{\sim}: \mathcal{L}_{\dot{\psi}} \times \SL(2, \CC) \to \dot{\tilde{G}}^\vee. \]
Then $S_{\dot{\psi}}$ can be understood as the \textit{bona fide} centralizer of $\Image(\dot{\psi}^{\sim})$ in $\dot{\tilde{G}}^\vee$. Likewise, one defines $s_{\dot{\psi}} \in S_{\dot{\psi}}$ as
\begin{equation}\label{eqn:dot-s-psi}
	\index{s-psi-dot@$s_{\dot{\psi}}$}
	s_{\dot{\psi}} = \dot{\psi}^{\sim} \left(1, \begin{pmatrix} -1 & \\ & -1 \end{pmatrix}\right),
\end{equation}
with the same concrete description as in the discussions after Definition \ref{def:s-psi}. In particular, \eqref{eqn:S-functoriality} holds for $\dot{\Psi}(\dot{G}^!) \to \dot{\Psi}(\dot{\tilde{G}})$.

Given $\dot{\psi} \in \dot{\Psi}(N)$ in the form \eqref{eqn:dot-psi-decomp} and any place $v$ of $\dot{F}$, its localization is defined near the end of \cite[\S 1.4]{Ar13}, namely
\[ \dot{\psi}_v = \sum_{i \in I} m_i \dot{\phi}_{i, v} \boxtimes r(b_i) \;\in \Psi^+(N). \]
In \textit{loc.\ cit.}, it is established that
\begin{itemize}
	\item $\dot{\phi}_{i, v}$ is the L-parameter for the $v$-component of the cuspidal automorphic representation $\dot{\phi}_i$ (via the local Langlands correspondence for general linear groups);
	\item if $\dot{\psi} \in \dot{\Psi}(\dot{\tilde{G}})$ then $\dot{\psi}_v \in \Psi^+(\dot{\tilde{G}}_v)$ (see \cite[Theorem 1.4.2]{Ar13}).
\end{itemize}

These facts are most conveniently summarized by considering $\dot{\psi}^{\sim}: \mathcal{L}_{\dot{\psi}} \times \SL(2, \CC) \to \dot{\tilde{G}}^\vee$ and using the canonical homomorphism
\begin{equation*}
	\iota_v: \mathcal{L}_{\dot{F}_v} \to \mathcal{L}_{\dot{\psi}}
\end{equation*}
in \cite[(1.4.14)]{Ar13}. Then $\dot{\psi}_v$ is characterized by the commutative diagram
\[\begin{tikzcd}
	\mathcal{L}_{\dot{F}_v} \times \SL(2, \CC) \arrow[r, "{\dot{\psi}_v}"] \arrow[d, "{\iota_v \times \identity}"'] & \dot{\tilde{G}}_v^\vee \arrow[d, "{\identity}"] \\
	\mathcal{L}_{\dot{\psi}} \times \SL(2, \CC) \arrow[r, "{\dot{\psi}^{\sim}}"'] & \dot{\tilde{G}}^\vee .
\end{tikzcd}\]

For $\dot{\psi} \in \dot{\Psi}(\dot{\tilde{G}})$, these discussions also lead to the localization homomorphisms
\begin{equation}\label{eqn:S-localization}
	S_{\dot{\psi}} \to S_{\dot{\psi}_v}, \quad \EuScript{S}_{\dot{\psi}} \to \EuScript{S}_{\dot{\psi}_v}.
\end{equation}


All these definitions and results extend to adélic groups of metaplectic type introduced in \S\ref{sec:adelic-type}, in the obvious manner. If $\dot{\tilde{M}}$ is a Levi subgroup of $\dot{\tilde{G}} = \Mp(\dot{W}, \A)$, thus of metaplectic type, then we have a natural map $\dot{\Psi}(\dot{\tilde{M}}) \to \dot{\Psi}(\dot{\tilde{G}})$ as described in \eqref{eqn:psi-induction}. If $\dot{\psi}_M \in \dot{\Psi}(\dot{\tilde{M}})$ has image $\dot{\psi} \in \dot{\Psi}(\dot{\tilde{G}})$, we obtain natural homomorphisms
\begin{equation}\label{eqn:centralizer-induction}
	S_{\dot{\psi}_M} \to S_{\dot{\psi}}, \quad \EuScript{S}_{\dot{\psi}_M} \to \EuScript{S}_{\dot{\psi}}.
\end{equation}

\subsection{The L-parameter attached to an Arthur parameter}
In this article, we use two ways to attach an L-parameter to an Arthur parameter $\psi$ of $\tilde{G}$. To begin with, let us consider the local setting.

The first construction is to take
\begin{equation}\label{eqn:psi-restriction}\begin{tikzcd}[row sep=tiny]
	\Psi^+(\tilde{G}) \arrow[r] & \Phi(\tilde{G}) \\
	\Psi(\tilde{G}) \arrow[phantom, u, sloped, "\subset" description] \arrow[r] & \Phi_{\mathrm{bdd}}(\tilde{G}) \arrow[phantom, u, sloped, "\subset" description] \\
	\psi \arrow[mapsto, r] \arrow[phantom, u, sloped, "\in" description] & {\psi|_{\mathcal{L}_F}} \arrow[phantom, u, sloped, "\in" description] .
\end{tikzcd}\end{equation}
If $\psi = \bigoplus_{i \in I} m_i \phi_i \boxtimes r(b_i)$ as in \eqref{eqn:psi-decomp}, then $\psi|_{\mathcal{L}_F} = \bigoplus_{i \in I} m_i b_i \phi_i$. This construction will be used in \S\S\ref{sec:T}--\ref{sec:T-properties}.

The second construction, also recorded in \cite[Chapter 1]{Ar13}, goes as follows.

\begin{definition}\label{def:phi-psi}
	\index{phi-psi@$\phi_\psi$}
	For $\psi \in \Psi(\tilde{G})$, define $\phi_\psi \in \Phi(\tilde{G})$ as the homomorphism $\mathcal{L}_F \to \tilde{G}^\vee$ given by
	\[ \phi_\psi(w) = \psi\left( w, \begin{pmatrix} |w|^{\frac{1}{2}} & \\ & |w|^{-\frac{1}{2}} \end{pmatrix} \right), \]
	where $|\cdot|$ stands for the composition of $\mathcal{L}_F \twoheadrightarrow \Weil{F}$ with $|\cdot|_F: \Weil{F} \to \R_{> 0}^{\times}$.
\end{definition}

Specifically, if $\psi = \bigoplus_{i \in I} m_i \phi_i \boxtimes r(b_i) \in \Psi(\tilde{G})$ then
\begin{equation}\label{eqn:phi-psi}
	\phi_\psi = \bigoplus_{i \in I} m_i \left( \bigoplus_{h=0}^{b_i - 1} \phi_i |\cdot|^{\frac{b_i - 1}{2} - h} \right),
\end{equation}
where
$\phi |\cdot|^x$ means the twist of an L-parameter $\phi$ by $\mathcal{L}_F \twoheadrightarrow \Weil{F} \xrightarrow{|\cdot|^x} \R_{> 0}^{\times}$.

Here we normalize $|\cdot|: \Weil{F} \to \R_{> 0}^{\times}$ so that $|\Frob| = q^{-1}$ for non-Archimedean $F$ with residual cardinality $q$.

The same construction applies to a number field $\dot{F}$ and global Arthur parameters in \S\ref{sec:global-parameters}. To see this, it suffices to declare $\bigoplus_{i \in I} m_i b_i \dot{\phi}_i$ and the global analogue of \eqref{eqn:phi-psi} to be the definitions of $\dot{\psi}|_{\mathcal{L}_{\dot{F}}}$ and $\dot{\phi}_{\dot{\psi}}$, respectively, where $\dot{\phi}_i |\cdot|^{x}$ is now the twist of the cuspidal automorphic representation $\dot{\phi}_i$ by the $x$-th power of the global absolute value.

\subsection{Infinitesimal characters}\label{sec:inf-character}
Here $F$ is an Archimedean local field, so that $\CC^{\times} \subset \Weil{F}$. The following recipe is based on \cite[\S 9, \S 11]{Bo79}; see also \cite[\S 3.2]{MR17} for the complex case.

Let $\psi \in \Psi^+(\tilde{G})$. Consider $\phi_\psi |_{\CC^{\times}}$, a $\tilde{G}^\vee$-conjugacy class of homomorphisms $\CC^{\times} \to \tilde{G}^\vee$. We use the following facts from \textit{loc.\ cit.}
\begin{itemize}
	\item There is a representative $C_\psi$ of $\phi_\psi |_{\CC^{\times}}$ landing in the standard maximal torus $\check{T}$ of $\tilde{G}^\vee$, identified with the dual of the standard maximal torus $T$ of $G$;
	\item One may express $C_\psi(z) = z^\mu \overline{z}^\nu$ for a unique pair $\mu, \nu \in X^*(T) \otimes \CC$ with $\mu - \nu \in X^*(T)$; this means that $\lrangle{\xi, C_\psi(z)} = |z|^{\lrangle{\xi, \mu + \nu}} (z/|z|)^{\lrangle{\xi, \mu - \nu}}$ for all $\xi \in X_*(T)$.
	\item Regard $\mu$ (resp.\ $(\mu, \nu)$) as an element of $\mathfrak{t}^* \dotimes{\R} \CC$ (resp.\ $\mathfrak{t}^* \times \mathfrak{t}^*$) when $F = \R$ (resp.\ $F = \CC$). It turns out its orbit under the Weyl group (resp.\ two copies of the Weyl group) is independent of the representative $C_\psi$.
\end{itemize}

Define the complex Lie algebra $\mathfrak{g}_{\CC} := \mathfrak{g} \dotimes{\R} \CC$ (resp.\ $\mathfrak{g}_{\CC} := \mathfrak{g} \times \mathfrak{g}$) when $F = \R$ (resp.\ $F = \CC$). Let $\mathcal{Z}(\mathfrak{g}_{\CC})$ denote the center of $\mathcal{U}(\mathfrak{g}_{\CC})$. By the Harish-Chandra isomorphism, the Weyl orbit of $\mu$ (resp.\ $(\mu, \nu)$) furnishes a homomorphism of $\CC$-algebras
\begin{equation}\label{eqn:lambda-psi}
	\index{lambda-psi@$\lambda(\psi)$}
	\lambda(\psi): \mathcal{Z}(\mathfrak{g}_{\CC}) \to \CC
\end{equation}
that depends only on $\phi_\psi$. We will see in Proposition \ref{prop:pi-psi-inf-char} that $\lambda(\psi)$ is the common infinitesimal character of members in the Arthur packet associated with $\psi$.

In the case $F = \CC$, a more explicit representative of $\lambda(\psi)$ in $\mathfrak{t}^* \times \mathfrak{t}^*$ can be found in \cite[(3.6)]{MR17}.

\section{Local theory}\label{sec:local}
Unless otherwise specified, we work with a local field $F$ with $\mathrm{char}(F) = 0$ and a chosen additive character $\bpsi$ of $F$ throughout this section.

\subsection{Local root numbers}\label{sec:local-root-numbers}
For all $N \in \Z_{\geq 1}$ and $\phi \in \Phi(N)$, define the local root number
\begin{equation}\label{eqn:local-root-number}
	\epsilon(\phi, \bpsi) := \epsilon\left( \frac{1}{2}, \phi, \bpsi\right)
\end{equation}
following the standard recipe in \cite[\S 2.2]{GR10}; note that the representations of $\Weil{F} \times \SL(2, \CC)$ can be converted to Weil--Deligne representations, see \cite[Proposition 2.2]{GR10}. We will make extensive use of the properties below from \cite[\S 5]{GGP1}.
\index{epsilon-phi@$\epsilon(\phi, \bpsi), \epsilon(\phi)$}

\begin{itemize}
	\item $\epsilon(\phi_1 \oplus \phi_2, \bpsi) = \epsilon(\phi_1, \bpsi) \epsilon(\phi_2, \bpsi)$;
	
	\item $\epsilon\left( \phi, \bpsi_c \right) = (\det\phi)(c) \cdot \epsilon(\phi, \bpsi)$ for all $c \in F^{\times}$, where we identify $\det\phi$ with a character of $F^{\times}$ via local class field theory;

	\item $\epsilon\left( \phi, \bpsi \right) \epsilon\left( \phi^\vee, \bpsi \right) = (\det\phi)(-1)$;
	
	\item It follows that when $\phi$ is self-dual of symplectic type, $\epsilon(\phi) := \epsilon(\phi, \bpsi)$ is independent of the choice of $\bpsi$ and satisfies $\epsilon(\phi)^2 = 1$;
	
	\item Suppose $F$ is non-Archimedean so that $\mathcal{L}_F = \Weil{F} \times \SL(2, \CC)$, if $\phi$ is trivial on $I_F \times \SL(2, \CC)$ and $\bpsi$ is of conductor $\mathfrak{o}_F$, then $\epsilon(\phi, \bpsi) = 1$.
\end{itemize}

Furthermore, for non-Archimedean $F$, consider a simple $\rho \boxtimes r(a) \in \Phi(N)$ where $\rho$ has underlying space $V_\rho$ and $a \in \Z_{\geq 1}$, we have
\begin{equation}\label{eqn:epsilon-SL}
	\begin{aligned}
		\epsilon(\rho \boxtimes r(a), \bpsi) & = \epsilon( \rho \boxtimes r(1), \bpsi)^a \cdot \det\left(-\rho(\Frob) \middle| V_\rho^{I_F} \right)^{a-1} \\
		& = \begin{cases}
			\epsilon(\rho \boxtimes r(1), \bpsi)^a (-\rho(\Frob))^{a-1}, & \rho\;\text{is an unramified character,} \\
			\epsilon(\rho \boxtimes r(1), \bpsi)^a, & \text{otherwise.} 
		\end{cases}
	\end{aligned}
\end{equation}
Indeed, the first equality is recorded in \textit{loc.\ cit.}; as for the second, note that $V_\rho^{I_F} \neq \{0\}$ implies $\rho$ is an unramified character of $\Weil{F}$ by irreducibility. Here $\Frob$ denotes the geometric Frobenius.

The property \eqref{eqn:epsilon-SL} reduces the computation of $\epsilon(\phi, \bpsi)$ to the case where $\phi$ factors through $\Weil{F}$, i.e. $\phi$ is trivial on the factor $\SL(2, \CC)$.

Now consider Arthur parameters for the metaplectic group $\tilde{G} = \Mp(W)$.

\begin{definition}\label{def:nu-psi}
	\index{nu-psi@$\nu_{\psi}$}
	Suppose $\psi \in \Psi^+(\tilde{G})$ with a decomposition as in \eqref{eqn:psi-decomp-2}. Identify $\EuScript{S}_\psi$ with $(\Z/2\Z)^{I^+}$. Define the character $\nu_\psi \in \EuScript{S}_\psi^\vee \simeq \bmu_2^{I^+}$ by prescribing its components
	\[ \nu_{\psi, i} := \epsilon(\phi_i)^{b_i}, \quad i \in I^+ . \]
\end{definition}

The character $\nu_\psi$ will be used in the study of anti-tempered Arthur packets (see \S\ref{sec:anti-tempered}).

\subsection{The basic bijection}\label{sec:basic-bijection}
The following construction is a simple variant of (1.4.11) and the discussions before (4.8.10) of \cite{Ar13}.

\begin{proposition}\label{prop:basic-bijection}
	\index{S-psi-2@$S_{\psi, 2}$}
	Consider $\tilde{G} = \Mp(W)$ with $\dim W = 2n$. For all $\psi \in \Psi^+(\tilde{G})$, define
	\begin{equation*}
		S_{\psi, 2} := \{s \in S_\psi: s^2 = 1 \}.
	\end{equation*}
	
	\begin{enumerate}[(i)]
		\item There is a canonical $\tilde{G}^\vee$-equivariant bijection
		\[\left\{\begin{array}{r|l} (s, \psi^!) & s \in \tilde{G}^\vee, \; s^2 = 1 \\
			& \psi^!: \mathcal{L}_F \times \SL(2, \CC) \to (G^!)^\vee \\
			\end{array}\right\} \xleftrightarrow{1:1}
		\left\{\begin{array}{r|l}
			(\psi, s) & \psi: \mathcal{L}_F \times \SL(2, \CC) \to \tilde{G}^\vee \\
			& s \in S_{\psi, 2}
		\end{array} \right\}, \]
		where $\tilde{G}^\vee$ acts by conjugation on both sides, $\mathbf{G}^!$ is the endoscopic datum determined by the conjugacy class of $s$, and $\psi^!$, $\psi$ are required to satisfy the conditions for $\Psi^+(G^!)$, $\Psi^+(\tilde{G})$ respectively.
	
		\item Taking quotients by $\tilde{G}^\vee$ in (i) yields a bijection
		\[\left\{\begin{array}{r|l}
			(\mathbf{G}^!, \psi^!) & \mathbf{G}^! \in \Endo_{\elli}(\tilde{G}) \\
			& \psi^! \in \Psi^+(G^!)
		\end{array}\right\} \xleftrightarrow{1:1}
		\left\{\begin{array}{r|l}
			(\psi, s) & \psi \in \Psi^+(\tilde{G}) \\
			& s \in S_{\psi, 2}/\text{conj}
		\end{array}\right\};\]
		on the right hand side we implicitly fix a representative of $\psi$ in the $\tilde{G}^\vee$-conjugacy class, but this does not affect the conjugacy class in $S_{\psi, 2}$.
	\end{enumerate}
	Ditto if $\Psi^+(\cdots)$ is replaced by $\Psi(\cdots)$.
\end{proposition}
\begin{proof}
	Consider (i). Given $(s, \psi^!)$, define $\psi$ as the composition of $\psi^!$ with $(G^!)^\vee \hookrightarrow \tilde{G}^\vee$, so that $s \in S_{\psi, 2}$ since $(G^!)^\vee = Z_{\tilde{G}^\vee}(s)$. Given $(\psi, s)$ define $\psi^!$ by factoring $\psi$ through $Z_{\tilde{G}^\vee}(s) = (G^!)^\vee$. The conditions for Arthur parameters are preserved throughout. This proves (i).
	
	Consider the quotients by $\tilde{G}^\vee$ in (i). On the left hand side, the element $s$ (resp.\ $\psi^!$) gets replaced by its conjugacy class in $\tilde{G}^\vee$, i.e.\ $\mathbf{G}^!$ (resp.\ its conjugacy class by $Z_{\tilde{G}^\vee}(s) = (G^!)^\vee$, i.e.\ the class in $\Psi^+(G^!)$). On the right hand side, $\psi$ gets replaced by its class in $\Psi^+(\tilde{G})$; once a representative $\psi: \mathcal{L}_F \times \SL(2, \CC) \to \tilde{G}^\vee$ is fixed, the only ambiguity of $s$ comes from $S_\psi$-conjugation. This proves (ii).
\end{proof}

The result above is simpler than its counterparts for $\SO(2n+1)$ in \cite{Ar13}, since we do not divide by the center of $\Sp(2n, \CC)$; see Remark \ref{rem:reduced-center}.

\begin{corollary}\label{prop:basic-bijection-2}
	The bijection in Proposition \ref{prop:basic-bijection} (ii) restricts to
	\[\left\{\begin{array}{r|l}
		(\mathbf{G}^!, \psi^!) & \mathbf{G}^! \in \Endo_{\elli}(\tilde{G}) \\
		& \psi^! \in \Psi_2(G^!) \;\text{with image in}\; \Psi_2(\tilde{G})
	\end{array}\right\} \xleftrightarrow{1:1}
	\left\{\begin{array}{r|l}
		(\psi, x) & \psi \in \Psi_2(\tilde{G}) \\
		& x \in \EuScript{S}_\psi
	\end{array}\right\}.\]
\end{corollary}
\begin{proof}
	The condition $\psi \in \Psi_2(\tilde{G})$ implies $S_\psi$ is a product of copies of $\Or(1, \CC)$, hence $S_\psi = \EuScript{S}_\psi$ is abelian of $2$-torsion.
\end{proof}

The results above extend easily to groups of metaplectic type, i.e.\ to Levi subgroups of metaplectic groups. It suffices to require that $s$ is trivial on the $\GL$ factors.

\begin{remark}\label{rem:basic-bijection}
	The bijections in Proposition \ref{prop:basic-bijection} (ii) and Corollary \ref{prop:basic-bijection-2} work in the global setting, as explained below.
	
	For the former one, observe that given $\dot{\psi} \in \dot{\Psi}(\dot{\tilde{G}})$, prescribing a conjugacy class of $2$-torsion elements $s \in S_{\dot{\psi}}$ amounts to dividing the multi-set of simple summands in \eqref{eqn:dot-psi-decomp} into two piles $\dot{\psi}'$ and $\dot{\psi}''$ (ordered, according to the eigenvalues $\pm 1$ of $s$), and $\dot{\psi}^! := (\dot{\psi}', \dot{\psi}'')$ lies automatically in $\dot{\Psi}(\dot{G}^!)$ and $\dot{\psi}^! \mapsto \dot{\psi}$, where $\dot{\mathbf{G}}^!$ is determined by the conjugacy class of $s$. This also gives another proof of Proposition \ref{prop:basic-bijection} (ii) in the local setting.
	
	For the second one, simply repeat the argument for Corollary \ref{prop:basic-bijection-2}.
\end{remark}

\subsection{The distribution \texorpdfstring{$T_{\psi, s}$}{Tpsis}}\label{sec:T}
Consider $\tilde{G} = \Mp(W)$. For all $s \in S_{\psi, 2}$, the underlying $\CC$-vector space of $\psi$ decomposes into
\[ V_\psi = V_\psi^+ \oplus V_\psi^-, \]
namely the $(\pm 1)$-eigenspaces of $s$; the sum is orthogonal with respect to the symplectic structure on $V_\psi$. Denote by $\psi^{s = \pm 1}$ the action of $\mathcal{L}_F \times \SL(2, \CC)$ on $V_\psi^{\pm}$. Using \eqref{eqn:psi-restriction} and \eqref{eqn:local-root-number} one defines
\begin{equation*}
	\index{epsilon-pm@$\epsilon(\psi^{s=\pm 1})$}
	\epsilon\left( \psi^{s=\pm 1} \right) := \epsilon\left( \psi^{s=\pm 1}|_{\mathcal{L}_F} \right),
\end{equation*}
which belong to $\{\pm 1\}$ and depend only on the conjugacy class of $s$ in $S_\psi$.

Hereafter, only the conjugacy class of $s$ matters. From $(\psi, s)$ and the bijection in Proposition \ref{prop:basic-bijection} (ii), one obtains the pair $(\mathbf{G}^!, \psi^!)$. Suppose that $\mathbf{G}^!$ corresponds to $(n', n'')$. To proceed, we need another ingredient from Arthur's work.

\begin{definition}\label{def:STheta}
	\index{STheta@$S\Theta^{G^{"!}}_{\psi^{"!}}$}
	Given $G^! = \SO(2n' + 1) \times \SO(2n'' + 1)$ and $\psi^! \in \Psi^+(G^!)$, denote the stable distribution on $G^!(F)$ in \cite[Theorem 2.2.1]{Ar13} by
	\[ S\Theta^{G^!}_{\psi^!} \in SD_{\mathrm{spec}}(G^!) \otimes \mes(G^!)^\vee . \]
\end{definition}

When $\psi^! \in \Psi(G^!)$, the characterization in \textit{loc.\ cit.}\ is based on endoscopic transfer to twisted $\GL(2n)$, Whittaker-normalized as in \cite[\S 2.1]{Ar13}. The choice of Whittaker data does not affect $S\Theta^{G^!}_{\psi^!}$. We also include the dependence on the choice of Haar measures, whence the term $\mes(G^!)^\vee$.

For general $\psi^! \in \Psi^+(G^!)$, one defines $S\Theta^{G^!}_{\psi^!}$ by expressing it as the full parabolic induction of an unbounded twist of the stable distribution attached to some $\psi_{M^!, \mathrm{bdd}} \in \Psi(M^!)$, for a suitable Levi subgroup $M^!$. For details, we refer to \eqref{eqn:STheta-gen} and the discussions preceding it.

\begin{definition}\label{def:T-dist}
	\index{T-psi-s@$T_{\psi, s}$}
	Given $(\psi, s) \leftrightarrow (\mathbf{G}^!, \psi^!)$ as in Proposition \ref{prop:basic-bijection} (ii), set
	\[ T_{\psi, s} := \epsilon\left( \psi^{s = -1} \right) \cdot \trans_{\mathbf{G}^!, \tilde{G}}\left( S\Theta^{G^!}_{\psi^!} \right) \; \in D_{\mathrm{spec}, -}(\tilde{G}) \otimes \mes(G)^\vee ; \]
	see Theorem \ref{prop:spec-trans}.
\end{definition}

Observe that $S_{\psi, 2}/\text{conj} \to \EuScript{S}_\psi$ is surjective.

\begin{lemma}\label{prop:T-psi-s}
	Take $(\psi, s)$ as above, $\psi \in \Psi^+(\tilde{G})$. Let $x$ denote the image of $s$ in $\EuScript{S}_\psi$. Then $T_{\psi, s}$ depends only on $(\psi, x)$.
\end{lemma}
\begin{proof}
	Choose a representative of $s$ in $S_{\psi, 2}$, and write the decomposition \eqref{eqn:psi-decomp-2} as
	\begin{gather*}
		\psi = \bigoplus_{i \in I^+ \sqcup I^-} m_i \phi_i \boxtimes r(b_i) \oplus \bigoplus_{i \in J} m_i (\phi_i \oplus \phi_i^\vee) \boxtimes r(b_i), \\
		s = (s_i)_i \in \prod_{i \in I^+} \Or(m_i, \CC) \times \prod_{i \in I^-} \Sp(m_i, \CC) \times \prod_{i \in J} \GL(m_i, \CC).
	\end{gather*}
	
	The conjugacy class of $s_i$ amounts to a pair
	\[ (m'_i, m''_i) \in \Z_{\geq 0}^2, \quad m'_i + m''_i = m_i, \]
	corresponding to the multiplicities of $1$ and $-1$ as eigenvalues of $s_i$, such that $m'_i $ and $m''_i$ are both even when $i \in I^-$. For $i \in I^+$, the determinant in $\Or(m_i, \CC)$ satisfies
	\[ \det s_i = \begin{cases}
		1, & m''_i \;\text{is even,} \\
		-1, & m''_i \;\text{is odd.}
	\end{cases} \]
	Therefore $x \in \EuScript{S}_\psi$ only keeps record of the parities $(m''_i \;\bmod 2)_{i \in I^+}$.
	
	The corresponding $\psi^! = (\psi', \psi'')$ is given by
	\begin{align*}
		\psi' & = \bigoplus_{i \in I^+ \sqcup I^-} m'_i \phi_i \boxtimes r(b_i) \oplus \bigoplus_{i \in J} m'_i (\phi_i \oplus \phi_i^\vee) \boxtimes r(b_i), \\
		\psi'' & = \bigoplus_{i \in I^+ \sqcup I^-} m''_i \phi_i \boxtimes r(b_i) \oplus \bigoplus_{i \in J} m''_i (\phi_i \oplus \phi_i^\vee) \boxtimes r(b_i).
	\end{align*}
	
	Consider $s, \underline{s} \in S_{\psi, 2}/\text{conj}$, described by $(m'_i, m''_i)_i$ and $(\underline{m}'_i, \underline{m}''_i)_i$ respectively. They are in the same fiber of $S_{\psi, 2}/\text{conj} \to \EuScript{S}_\psi$ if and only if they are connected by a sequence of the following three cases.
	\begin{enumerate}[(a)]
		\item There exists $i_0 \in I^+$ such that $(\underline{m}'_{i_0}, \underline{m}''_{i_0}) = (m'_{i_0} + 2, m''_{i_0} - 2)$, and they coincide for the other indices. Reason: $m''_i \equiv \underline{m}''_i \pmod{2}$ for all $i \in I^+$.
		\item There exists $i_0 \in I^-$ such that $(\underline{m}'_{i_0}, \underline{m}''_{i_0}) = (m'_{i_0} + 2, m''_{i_0} - 2)$, and they coincide for the other indices. Reason: $m''_i$ and $\underline{m}''_i$ are both even for all $i \in I^-$.
		\item There exists $i_0 \in J$ such that $(\underline{m}'_{i_0}, \underline{m}''_{i_0}) = (m'_{i_0} + 1, m''_{i_0} - 1)$, and they coincide for the other indices.
	\end{enumerate}

	It remains to show $T_{\psi, s} = T_{\psi, \underline{s}}$ in each case. Consider the data $(\mathbf{G}^!, \psi^! = (\psi', \psi''))$ and $(\underline{\mathbf{G}}^!, \underline{\psi}^! = (\underline{\psi}', \underline{\psi}''))$ so obtained, and let $\pi$ be the irreducible representation of $\GL(\ell, F)$ parameterized by $\phi_{i_0} \boxtimes r(b_{i_0}) \in \Psi^+(\ell)$ where $\ell := b_{i_0} \dim \phi_{i_0}$. Let $\Theta_\pi \in D(\GL(\ell)) \otimes \mes(\GL(\ell))^\vee$ be its character.
	
	In each case, it is easy to see that
	\[\begin{tikzcd}[column sep=small, row sep=large]
		G^! \arrow[equal, rr] & & \SO(2n' + 1) \arrow[phantom, r, "\times" description] \arrow[hookleftarrow, d] & \SO(2n'' + 1) \arrow[hookleftarrow, d] \\
		R \arrow[equal, r] \arrow[hookrightarrow, u, "\text{Levi}"] \arrow[hookrightarrow, d, "\text{Levi}"'] & \GL(\ell) \arrow[phantom, r, "\times" description] \arrow[hookrightarrow, rru, crossing over] \arrow[hookrightarrow, rd] & \SO(2n' + 1) \arrow[hookrightarrow, d] \arrow[phantom, r, "\times" description] & \SO(2(n'' - \ell) + 1) \arrow[hookrightarrow, d] \\
		\underline{G}^! \arrow[equal, rr] & & \SO(2\underline{n}' + 1) \arrow[phantom, r, "\times" description] & \SO(2\underline{n}'' + 1)
	\end{tikzcd}\]
	with self-evident notations. Both $S\Theta^{G^!}_{\psi^!}$ and $S\Theta^{\underline{G}^!}_{\underline{\psi}^!}$ are obtained by parabolic induction from $\Theta_\pi \boxtimes S\Theta^\flat$ on $R(F)$, where $S\Theta^\flat = S\Theta^{G^{!, \flat}}_{\psi^{!, \flat}}$, where $\psi^{!, \flat} = (\psi', \psi^{'', \flat}) \in \Psi^+(G^{!, \flat})$ is obtained by decreasing the multiplicity $m''_{i_0}$ by $2$ (resp.\ by $1$) in cases (a) and (b) (resp.\ in the case (c)), the other summands being unchanged.
	
	By the commutation of parabolic induction and transfer up to a central twist \cite[Proposition 3.8.4]{Li26},
	\begin{align*}
		\trans_{\mathbf{G}^!, \tilde{G}} \left( S\Theta^{G^!}_{\psi^!} \right) & = \omega_\pi(-1) \Ind\left( \Theta_\pi \boxtimes \trans_{\mathbf{G}^{!, \flat}, \tilde{G}^\flat} S\Theta^\flat \right), \\
		\trans_{\underline{\mathbf{G}}^!, \tilde{G}} \left( S\Theta^{\underline{G}^!}_{\underline{\psi}^!} \right) & = \Ind\left( \Theta_\pi \boxtimes \trans_{\mathbf{G}^{!, \flat}, \tilde{G}^\flat} S\Theta^\flat \right).
	\end{align*}
	Here $\tilde{G}^\flat = \Mp(W^\flat)$ with $\dim W^{\flat} = 2(n' + n'' - \ell)$ and $\mathbf{G}^{!, \flat}$ corresponds to the pair $(n', n'' - \ell) = (\underline{n}' - \ell, \underline{n}'')$, whereas $\Ind$ is the induction from $\GL(\ell, F) \times \tilde{G}^{\flat}$ to $\tilde{G}$, see \cite[(2.4.1)]{Li26}. By realizing $\pi$ as Langlands quotient, one sees
	\[ \omega_\pi(-1) = \left( \det\phi_{i_0}(-1)\right)^{b_{i_0}}. \]
	On the other hand,
	\begin{align*}
		\epsilon\left( \psi^{s = -1} \right) & = \epsilon\left( \psi^{\underline{s} = -1} \right) \left(\epsilon(\phi_{i_0}, \bpsi) \epsilon(\phi^\vee_{i_0}, \bpsi)\right)^{b_{i_0}} \\
		& = \epsilon\left( \psi^{\underline{s} = -1} \right) \left( \det\phi_{i_0}(-1)\right)^{b_{i_0}}.
	\end{align*}
	
	It follows that $T_{\psi, s} = T_{\psi, \underline{s}}$ in all the cases (a)--(c), as desired.
\end{proof}

We record a by-product of the proof above.

\begin{proposition}\label{prop:nu-good-parity}
	If $\psi \in \Psi^+(\tilde{G})$ is of quasi-good parity, for example when $\psi \in \Psi_{\mathrm{gp}}(\tilde{G})$, then the map
	\begin{align*}
		S_{\psi, 2}/\mathrm{conj} & \to \{\pm 1\} \\
		s \bmod\; \mathrm{conj} & \mapsto \epsilon\left( \psi^{s=-1} \right)
	\end{align*}
	descends to $\EuScript{S}_\psi \to \{\pm 1\}$, and equals the character $\nu_\psi$ in Definition \ref{def:nu-psi}.
\end{proposition}
\begin{proof}
	Since $\psi$ is of quasi-good parity, in the proof of Lemma \ref{prop:T-psi-s} we have
	\[ \left(  \det \phi_{i_0}(-1) \right)^{b_{i_0}} = 1 \]
	in cases (b) and (c); it also holds for (a) since $\phi_{i_0} \boxtimes r(b_{i_0})$ is then of symplectic type. The last paragraph of that proof therefore shows that the map $S_{\psi, 2}/\text{conj} \to \{\pm 1\}$ descends to $\EuScript{S}_\psi$. Moreover, to compute $\epsilon\left( \psi^{s=-1} \right)$, we may modify $s$ in the fiber to ensure that $m''_i \in \{0, 1\}$ for all $i \in I^+$ and $m''_i = 0$ for $i \in I^- \sqcup J$, in the notation of that proof. For such $s$, we clearly have
	\[ \epsilon\left( \psi^{s=-1} \right) = \prod_{\substack{i \in I^+ \\ m''_i = 1}} \epsilon\left( \phi_i \right)^{b_i} = \nu_\psi(x) \]
	where $x$ stands for the image of $s$ in $\EuScript{S}_\psi$.
\end{proof}

All these definitions and results generalize to groups of metaplectic type; it suffices to take $s$ from
\begin{equation}\label{eqn:S-psi-2-M}
	S_{\psi, 2} := \left\{ s \in S_\psi: s^2 = 1, \; \text{trivial on $\GL$ factors} \right\},
\end{equation}
and only the local root numbers from the metaplectic factor intervene in the formation of $T_{\psi, s}$. We omit the details.

\subsection{First properties of \texorpdfstring{$T_{\psi, s}$}{Tpsis}}\label{sec:T-properties}
We begin with the compatibility between $T_{\psi, s}$ and parabolic induction. Consider a Levi subgroup of $\tilde{G}$ of the form
\[ \tilde{M} = \Mp(W^\flat) \times \prod_{i=1}^r \GL(n_i, F). \]
Given $\psi_M \in \Psi^+(\tilde{M})$ with image $\psi \in \Psi^+(\tilde{G})$, write the natural map $\EuScript{S}_{\psi_M} \to \EuScript{S}_\psi$ as $x_M \mapsto x$.

In view of Lemma \ref{prop:T-psi-s}, we may write $T_{\psi, s}$ as $T_{\psi, x}$, where $x \in \EuScript{S}_\psi$ is the image of $s \in S_{\psi, 2}$. Ditto for $\psi_M$; the definition of $S_{\psi_M, 2}$ is given after Corollary \ref{prop:basic-bijection-2}.
\index{T-psi-x@$T_{\psi, x}$}

Let $\Ind^{\tilde{G}}_{\tilde{M}} : D_-(\tilde{M}) \otimes \mes(M)^\vee \to D_-(\tilde{G}) \otimes \mes(G)^\vee$ denote the parabolic induction on the level of invariant distributions in \cite[(2.4.1)]{Li26}.
\index{Ind-G-M@$\Ind^{\tilde{G}}_{\tilde{M}}$}

\begin{proposition}\label{prop:T-psi-Ind}
	In the circumstance above, $T_{\psi, x} = \Ind^{\tilde{G}}_{\tilde{M}} \left( T_{\psi_M, x_M} \right)$ for all $x_M \in \EuScript{S}_{\psi_M}$. 
\end{proposition}
\begin{proof}
	Take any preimage $s_M \in S_{\psi_M, 2}$ of $x_M$, see \eqref{eqn:S-psi-2-M}. The conjugacy class of $s_M$ determines $\mathbf{M}^! \in \Endo_{\elli}(\tilde{M})$. This can be completed into
	\[\begin{tikzcd}
		G^![\mathbf{s}] \arrow[dashed, leftrightarrow, r, "\text{ell.}", "\text{endo.}"'] & \tilde{G} \\
		M^! \arrow[dashed, leftrightarrow, r, "\text{ell.}", "\text{endo.}"'] \arrow[hookrightarrow, u, "\text{Levi}"] & \tilde{M} \arrow[hookrightarrow, u, "\text{Levi}"']
	\end{tikzcd} \quad \text{where} \quad \mathbf{G}^![\mathbf{s}] \in \Endo_{\elli}(\tilde{G}), \]
	by taking $\mathbf{s} \in \Endo_{\mathbf{M}^!}(\tilde{G})$ (see \cite[Definition 3.4.6]{Li26}) to be the partition $\{1, \ldots, r\} = I' \sqcup I''$ with $I'' = \emptyset$; roughly speaking, this means that we map all $\GL$ factors in $M^!$ into the first $\SO$ factor of $G^![\mathbf{s}]$.
	
	By this choice of $\mathbf{s}$ and \cite[Proposition 3.8.4]{Li26}, transfer commutes with induction in the diagram above.
	
	Let $s$ be the image of $s_M$ in $S_{\psi, 2}$, which determines $\mathbf{G}^! \in \Endo_{\elli}(\tilde{G})$. Unwinding the description of $s_M$ and $S_{\psi_M} \to S_\psi$, we see $\mathbf{G}^! = \mathbf{G}^![\mathbf{s}]$ and
	\[ \epsilon\left( \psi^{s=-1} \right) = \epsilon\left( \psi_M^{s_M = -1} \right) \]
	since $s_M$ is trivial on the $\GL$-factors. Hence $T_{\psi, x}$ equals the induction of $T_{\psi_M, x_M}$ by applying \cite[Proposition 3.8.4]{Li26} once again.
\end{proof}

Consider the case of Archimedean $F$. Define $\lambda(\psi): \mathcal{Z}(\mathfrak{g}_{\CC}) \to \CC$ as in \S\ref{sec:inf-character}.

\begin{proposition}\label{prop:T-psi-inf-char}
	For each homomorphism of $\CC$-algebras $\lambda: \mathcal{Z}(\mathfrak{g}_{\CC}) \to \CC$, let $D_{\mathrm{spec}, \lambda, -}(\tilde{G}) \otimes \mes(G)^\vee$ be the subspace of $D_{\mathrm{spec}, -}(\tilde{G}) \otimes \mes(G)^\vee$ spanned by genuine characters with infinitesimal character $\lambda$. Then $T_{\psi, s} \in D_{\mathrm{spec}, \lambda(\psi), -}(\tilde{G}) \otimes \mes(G)^\vee$.
\end{proposition}
\begin{proof}
	The spectral transfer for $\tilde{G}$ preserves infinitesimal characters: see \cite[Proposition 7.3.1]{Li19} and the discussions preceding it for details. The assertion thus reduces to the version for $\SO(2m+1)$ for various $m$, which is in \cite[Lemme 3.4]{BC13}, and can be seen by transferring from $\SO(2m+1)$ to twisted $\GL(2m)$, since the twisted spectral transfer also preserves infinitesimal characters by \cite[I.2.8 Corollaire]{MW16-1} and the local Langlands correspondence for $\GL(2m, F)$ is explicit.
\end{proof}

Next, consider the unramified situation. Thus $F$ is non-Archimedean and we have $K \subset G(F)$, a section $K \to \tilde{G}$ as well as the unit $f_K$ of $\mathcal{H}_{\asp}(K \backslash \tilde{G} / K)$. Unramified Haar measures will be used throughout.

\begin{definition}\label{def:unramified-psi}
	\index{unramified Arthur parameter}
	In the unramified situation, if $\psi \in \Psi^+(\tilde{G})$ is trivial on $I_F \times \SL(2, \CC) \subset \mathcal{L}_F$, we say $\psi$ is \emph{unramified}.
\end{definition}

\begin{proposition}\label{prop:T-psi-nr}
	For all $\psi$ and $s$ in the unramified situation, we have
	\[ T_{\psi, s}(f_K) = \begin{cases}
		1, & \text{if $\psi$ is unramified,} \\
		0, & \text{otherwise.}
	\end{cases}\]
\end{proposition}
\begin{proof}
	Assume that $\psi$ is unramified. As reviewed in \S\ref{sec:local-root-numbers}, the factor $\epsilon(\psi^{s=-1})$ is trivial. Thus it suffices to show
	\[ \trans_{\mathbf{G}^!, \tilde{G}}\left( S\Theta^{G^!}_{\psi^!} \right)(f_K) = 1. \]
	
	By the fundamental lemma for $\tilde{G}$, the above is equivalent to $S\Theta^{G^!}_{\psi^!}(\mathbf{1}_{K^!}) = 1$ where $K^!$ is any hyperspecial subgroup of $G^!(F)$. Note that $\psi^!$ is still unramified, and the assertion follows from Arthur's work; see, e.g., the proof of \cite[Lemma 7.3.4]{Ar13} when $\psi \in \Psi(\tilde{G})$, and the general case follows by parabolic induction.
	
	Assume that $\psi$ is not unramified. By the fundamental lemma again, it suffices to show $S\Theta^{G^!}_{\psi^!}(\mathbf{1}_{K^!}) = 0$ when $\psi^! \in \Psi^+(G^!)$ is not unramified. Via the twisted fundamental lemma \cite{LMW18}, it reduces to a similar assertion on twisted $\GL(2n, F)$, which is clear.
\end{proof}

Now step back to the general local situation.

\begin{proposition}\label{prop:T-psi-negation}
	\index{tau-minus-1@$\tau_{-1}$}
	Let $\tau_{-1}$ be the automorphism of $D_-(\tilde{G}) \otimes \mes(G)^\vee$ induced by translation by $-1 \in \tilde{G}$ (Definition \ref{def:minus-1}). Then
	\[ T_{\psi, -s} = \epsilon(\psi|_{\mathcal{L}_F}) \tau_{-1}\left( T_{\psi, s} \right) \]
	where $-s$ denotes the product of $s$ with the central element $-1 \in \tilde{G}^\vee$.
\end{proposition}
\begin{proof}
	Take $\mathbf{G}^! \in \Endo_{\elli}(\tilde{G})$ corresponding to $s$, which in turn corresponds to a pair $(n', n'')$. Passing to $-s$ amounts to replacing $(n', n'')$ by $(n'', n')$. Applying \cite[Proposition 5.16]{Li11} to the transfer, we obtain
	\[ \epsilon\left( \psi^{s=1} \right)^{-1} T_{\psi, -s} = \tau_{-1} \left( \epsilon\left(\psi^{s=-1} \right)^{-1} T_{\psi, s}\right). \]
	Furthermore, $\epsilon\left( \psi^{s=1} \right) \epsilon\left(\psi^{s=-1} \right)^{-1} = \epsilon\left( \psi^{s=1} \right) \epsilon\left(\psi^{s=-1}\right) = \epsilon(\psi|_{\mathcal{L}_F})$.
\end{proof}

Once again, these results generalize from $\tilde{G} = \Mp(W)$ to groups of metaplectic type.

\subsection{Local desiderata}\label{sec:local-desiderata}
For $\psi \in \Psi^+(\tilde{G})$, Lemma \ref{prop:T-psi-s} furnishes a $D_{\mathrm{spec}, -}(\tilde{G}) \otimes \mes(G)^\vee$-valued function $x \mapsto T_{\psi, x}$ on $\EuScript{S}_\psi$. Denote by $x_\psi \in \EuScript{S}_\psi$ the image of $s_\psi \in S_\psi$ (Definition \ref{def:s-psi}). Note that $\EuScript{S}_\psi$ is a finite $2$-torsion abelian group. The construction below extracts the Fourier coefficients of $x \mapsto T_{\psi, x_\psi x}$.
\index{x-psi@$x_\psi$}

\begin{definition}\label{def:pi-psi-chi}
	\index{pi-psi-chi@$\pi_{\psi, \chi}$}
	For all $\chi \in \EuScript{S}_\psi^\vee$, put
	\[ \pi_{\psi, \chi} := \left| \EuScript{S}_\psi \right|^{-1} \sum_{x \in \EuScript{S}_\psi} \chi(x_\psi x) T_{\psi, x}. \]
	An equivalent characterization is
	\begin{equation*}
		T_{\psi, x} = \sum_{\chi \in \EuScript{S}_\psi^\vee} \chi(x_\psi x) \pi_{\psi, \chi}, \quad x \in \EuScript{S}_\psi.
	\end{equation*}
\end{definition}

A priori, $\pi_{\psi, \chi}$ is only a $\CC$-linear combination of irreducible genuine characters. It may also happen that $\pi_{\psi, \chi} = 0$. The main local result in this article is stated as follows.

\begin{theorem}\label{prop:local-desiderata}
	Let $\psi \in \Psi^+(\tilde{G})$ and $\chi \in \EuScript{S}_\psi^\vee$. The distribution $\pi_{\psi, \chi}$ is equal to either zero or the character of a finite length genuine representation of $\tilde{G}$. If $\psi \in \Psi(\tilde{G})$, then this representation is unitary.
\end{theorem}

The proof will be completed in \S\ref{sec:proof-local}. The description of $\pi_{\psi, \chi}$ for $F = \CC$ will be strengthened in Theorem \ref{prop:cplx-packet}.

Before embarking on the proof, we shall collect certain simple properties of $\pi_{\psi, \chi}$ that do not depend on Theorem \ref{prop:local-desiderata}.

\begin{proposition}[Reduction to good parity case]\label{prop:pi-psi-gp}
	Express $\psi \in \Psi^+(\tilde{G})$ as
	\[ \psi = \bigoplus_{i \in I^+ \sqcup I^-} m_i \psi_i \oplus \bigoplus_{j \in J} m_j (\psi_j \oplus \psi^\vee_j) \]
	as in \eqref{eqn:psi-decomp-2}, with each $\psi_i$ and $\psi_j$ simple. Set
	\[ \psi_0 := \bigoplus_{i \in I^+} m_i \psi_i, \quad \psi_{\mathrm{GL}} = \bigoplus_{i \in I^-} \frac{m_i}{2} \psi_i \oplus \bigoplus_{j \in J} m_j \psi_j . \]
	Then $\psi_M := (\psi_0, \psi_{\GL}) \in \Psi^+(\tilde{M})$ and $\psi_0 \in \Psi_{\mathrm{gp}}(\Mp(W^\flat))$, where
	\[ \tilde{M} := \Mp(W^\flat) \times \GL(m), \quad \dim W^\flat = \dim \psi_0, \quad m = \dim \psi_{\GL} \]
	is a Levi subgroup of $\tilde{G}$, and
	\[ \psi_M \mapsto \psi, \quad \EuScript{S}_{\psi_0} \simeq \EuScript{S}_{\psi_M} \rightiso \EuScript{S}_\psi . \]
	
	We have $\psi_M \in \Psi(\tilde{M})$ if $\psi \in \Psi(\tilde{G})$. Moreover, for all $\chi \in \EuScript{S}_\psi^\vee \simeq \EuScript{S}_{\psi_M}^\vee$ we have
	\[ \pi_{\psi, \chi} = \Ind^{\tilde{G}}_{\tilde{M}}\left(\pi_{\psi_M, \chi}\right). \]
\end{proposition}
\begin{proof}
	Only the last equality requires proof. Observe that $s_{\psi_M} \in S_{\psi_M}$ maps to $s_\psi \in S_\psi$, hence $x_{\psi_M} \mapsto x_\psi$. We deduce
	\begin{align*}
		\pi_{\psi, \chi} & = |\EuScript{S}_\psi|^{-1} \sum_{x \in \EuScript{S}_\psi} \chi(x_\psi x) T_{\psi, x} \\
		& = |\EuScript{S}_{\psi_M}|^{-1} \sum_{x \in \EuScript{S}_{\psi_M}} \chi(x_{\psi_M} x) \Ind^{\tilde{G}}_{\tilde{M}} \left( T_{\psi_M, x} \right) \\
		& = \Ind^{\tilde{G}}_{\tilde{M}} \left( \pi_{\psi_M, \chi} \right),
	\end{align*}
	where Proposition \ref{prop:T-psi-Ind} is applied in the second equality.
\end{proof}

\begin{proposition}[Archimedean infinitesimal characters]\label{prop:pi-psi-inf-char}
	When $F$ is Archimedean, we have
	\[ \pi_{\psi, \chi} \in D_{\mathrm{spec}, \lambda(\psi), -}(\tilde{G}) \otimes \mes(G)^\vee \]
	for all $\chi$, where $\lambda(\psi)$ (resp.\ $D_{\mathrm{spec}, \lambda(\psi), -}(\cdots)$) is defined as in \S\ref{sec:inf-character} (resp.\ Proposition \ref{prop:T-psi-inf-char}).
\end{proposition}
\begin{proof}
	Immediate from Proposition \ref{prop:T-psi-inf-char}.
\end{proof}

\begin{proposition}[Unramified case]\label{prop:pi-psi-nr}
	In the unramified situation, use the unramified Haar measure on $G(F)$ and denote the hyperspecial subgroup in question by $K$. If $\psi$ is unramified (Definition \ref{def:unramified-psi}), then
	\[ \pi_{\psi, \chi}(f_K) = \begin{cases}
		1, & \text{if $\chi = \mathbf{1}$,} \\
		0, & \text{otherwise.}
	\end{cases}\]
	
	If $\psi$ is not unramified, then $\pi_{\psi, \chi}(f_K) = 0$ for all $\chi$.
\end{proposition}
\begin{proof}
	We have $\pi_{\psi, \chi}(f_K) = |\EuScript{S}_\psi|^{-1} \sum_x \chi(x_\psi x) T_{\psi, x}(f_K)$. Suppose $\psi$ is unramified. By Proposition \ref{prop:T-psi-nr}, the above is equal to
	\[ |\EuScript{S}_\psi|^{-1} \sum_x \chi(x_\psi x) = \begin{cases}
		1, & \text{if $\chi = \mathbf{1}$,} \\
		0, & \text{otherwise.}
	\end{cases}\]
	
	Suppose $\psi$ is not unramified, then Proposition \ref{prop:T-psi-nr} implies $T_{\psi, x}(f_K) = 0$ for all $x$, thus $\pi_{\psi, \chi}(f_K) = 0$.
\end{proof}

For the next result, we use the image of $-1 \in \tilde{G}^\vee$ in $\EuScript{S}_\psi$ to define the involution $x \mapsto -x$ of $\EuScript{S}_\psi$.

\begin{proposition}[Translation by $-1$]\label{prop:p-psi-negation}
	Let $\tau_{-1}$ be as in Proposition \ref{prop:T-psi-negation}. For all $\psi \in \Psi^+(\tilde{G})$ and $\chi \in \EuScript{S}_\psi^\vee$, we have
	\[ \tau_{-1}\left( \pi_{\psi, \chi} \right) = \chi(-1) \epsilon(\psi|_{\mathcal{L}_F}) \pi_{\psi, \chi}. \]
\end{proposition}
\begin{proof}
	We have
	\begin{align*}
		\pi_{\psi, \chi} & = |\EuScript{S}_\psi|^{-1} \sum_x \chi(x_\psi x) T_{\psi, x} = |\EuScript{S}_\psi|^{-1} \sum_x \chi(-x_\psi x) T_{\psi, -x}. \\
		& = \chi(-1) |\EuScript{S}_\psi|^{-1} \sum_x \chi(x_\psi x) \epsilon(\psi|_{\mathcal{L}_F}) \tau_{-1} \left( T_{\psi, x} \right) \\
		& = \chi(-1) \epsilon(\psi|_{\mathcal{L}_F}) \tau_{-1}\left( \pi_{\psi, \chi} \right)
	\end{align*}
	by Proposition \ref{prop:T-psi-negation}.
\end{proof}

\subsection{Arthur packets}\label{sec:A-packets}
Let $\Pi_-(\tilde{G})$ be the set of isomorphism classes of genuine irreducible representations of $\tilde{G}$, and $\Pi_{\mathrm{unit}, -}(\tilde{G})$ the subset of unitary ones.
\index{Pi-G-tilde@$\Pi_-(\tilde{G}), \Pi_{\mathrm{unit}, -}(\tilde{G})$}

The validity of Theorem \ref{prop:local-desiderata} is assumed in this subsection. Adopting Arthur's point of view (cf.\ \cite[Theorem 1.5.1]{Ar13}), the packets associated with Arthur parameters of $\tilde{G}$ are defined as multi-sets as follows.

\begin{definition}\label{def:A-packet}
	\index{Pi-psi@$\Pi_\psi$}
	\index{Arthur packet}
	Let $\psi \in \Psi(\tilde{G})$. For every $\chi \in \EuScript{S}_\psi^\vee$, expand $\pi_{\psi, \chi}$ into
	\[ \pi_{\psi, \chi} = \sum_{\pi \in \Pi_{\mathrm{unit}, -}(\tilde{G})} m_{\psi, \chi}(\pi) \Theta_\pi \]
	where $m_{\psi, \chi}(\pi) \in \Z_{\geq 0}$. Let $\Pi_{\psi}$ be the finite set consisting of all triplets $(\chi, \pi, k)$ with
	\[ \chi \in \EuScript{S}_\psi^\vee, \quad \pi \in \Pi_{\mathrm{unit}, -}(\tilde{G}), \quad k \in \Z, \quad 1 \leq k \leq m_{\psi, \chi}(\pi). \]
	
	This set will be called the \emph{Arthur packet} associated with $\psi$. It admits two maps
	\[\begin{tikzcd}[row sep=tiny]
		\EuScript{S}_\psi^\vee & \Pi_\psi \arrow[l] \arrow[r] & \Pi_{\mathrm{unit}, -}(\tilde{G}) \\
		\chi & (\chi, \pi, k) \arrow[mapsto, l] \arrow[mapsto, r] & \pi .
	\end{tikzcd}\]
\end{definition}

In Arthur's terminology \cite[\S 1.5]{Ar13}, one has the following:
\begin{itemize}
	\item $\Pi_\psi$ together with the map to $\Pi_{\mathrm{unit}, -}(\tilde{G})$ is said to be a set over $\Pi_{\mathrm{unit}, -}(\tilde{G})$;
	\item the cardinality of the fiber of $\Pi_\psi \to \Pi_{\mathrm{unit}, -}(\tilde{G})$ over $\pi$ is called the multiplicity of $\pi$ in $\Pi_\psi$, which equals $\sum_\chi m_{\psi, \chi}(\pi)$;
	\item abusing notation, one often denotes an element of $\Pi_\psi$ by $\pi$, although in reality it carries extra labels $\chi$ and $k$;
	\item the map $\Pi_\psi \to \EuScript{S}_\psi^\vee$ is denoted by $\pi \mapsto \lrangle{\cdot, \pi}$ in the notation of \cite[Theorem 1.5.1]{Ar13}.
\end{itemize}
\index{$\lrangle{\cdot, \pi}$}

For all $x \in \EuScript{S}_\psi$, the expansion of $T_{\psi, x}$ becomes
\begin{equation*}
	T_{\psi, x} = \sum_{\pi \in \Pi_\psi} \lrangle{x_\psi x, \pi} \Theta_\pi .
\end{equation*}
This is akin to \cite[(2.2.6)]{Ar13}, but one has to keep in mind that $T_{\psi, x}$ involves the local root number $\epsilon(\psi^{s = -1})$ where $s$ is a representative of $x$.

\begin{definition}\label{def:multiplicity-free}
	\index{multiplicity-free}
	\index{Pi-psi-mf@$\Pi_\psi^{\mathrm{mf}}$}
	The multiplicity-free part $\Pi_\psi^{\mathrm{mf}} \subset \Pi_\psi$ is the preimage under $\Pi_\psi \to \Pi_{\mathrm{unit}, -}(\tilde{G})$ of $\{\pi : \text{multiplicity $=1$ in $\Pi_\psi$} \}$.
	
	We say $\Pi_\psi$ is \emph{multiplicity-free} if $\Pi_\psi^{\mathrm{mf}} = \Pi_\psi$, i.e.\ if each $\pi \in \Pi_{\mathrm{unit}, -}(\tilde{G})$ has multiplicity $0$ or $1$ in $\Pi_\psi$.
\end{definition}

The same constructions work for $\psi \in \Psi^+(\tilde{G})$. The only difference is that one cannot assert unitarity anymore: the Arthur packet $\Pi_\psi$ now sits between $\EuScript{S}_\psi^\vee \leftarrow \Pi_\psi \to \Pi_-(\tilde{G})$.

\begin{proposition}\label{prop:unramified-mf}
	In the unramified situation, consider $\psi \in \Psi^+(\tilde{G})$.
	\begin{enumerate}[(i)]
		\item If $\psi$ is unramified, there is a unique $K$-spherical $\pi \in \Pi_-(\tilde{G})$ that has non-zero multiplicity in $\Pi_\psi$; in fact, $\lrangle{\cdot, \pi} = \mathbf{1}$ and $\pi \in \Pi_\psi^{\mathrm{mf}}$. If $\psi$ is not unramified, then there are no $K$-spherical members in $\Pi_\psi$ with nonzero multiplicity.
		\item Suppose $\psi, \psi' \in \Psi^+(\tilde{G})$ are both unramified. If $\Pi_\psi, \Pi_{\psi'}$ have a $K$-spherical member in common, then $\psi = \psi'$.
		\item Suppose $\psi \in \Psi(\tilde{G})$, then the $\pi$ in (i) has Satake parameter equal to $\phi_\psi(\Frob) \in \tilde{G}^\vee_{\mathrm{ss}} / \mathrm{conj}$.
	\end{enumerate}
\end{proposition}
\begin{proof}
	For (i), the sum of multiplicities in $\Pi_\psi$ of $K$-spherical irreducible representations is $\sum_\chi \pi_{\psi, \chi}(f_K)$, hence the assertion follows immediately from Proposition \ref{prop:pi-psi-nr}. Also recall that $H := \SO(2n+1)$ satisfies a similar property.
	
	For (ii), use the elliptic endoscopic datum corresponding to $(n, 0)$ to obtain
	\[ \sum_\chi \chi(s_\psi) \pi_{\psi, \chi} = \trans_{(n, 0)} \left( S\Theta^H_\psi \right), \]
	and restrict both sides to $\mathcal{H}_{\asp}(K \backslash \tilde{G} / K)$ tensored with the unramified Haar measure. By (i), the left hand side contains a unique $K$-spherical constituent with coefficient one, and similarly for $S\Theta^H_\psi$. By the spherical fundamental lemma for $\tilde{G}$ (see \S\ref{sec:endoscopy} or \cite{Luo18}), the Satake parameters of these spherical constituents must match.
	
	The same is also true for $\psi'$, with the same Satake parameter as before. To prove (ii), we are thus reduced to the corresponding statement for $H$, which can be seen by transferring to the twisted $\GL(2n)$ via twisted spherical fundamental lemma \cite{LMW18}.

	For the same reason, (iii) also reduces to $H$. We then consider the L-packet within the Arthur packet \cite[Proposition 7.4.1]{Ar13} and apply the fact that spherical representations survive in Langlands quotients.
\end{proof}

For a direct argument for Proposition \ref{prop:unramified-mf} (iii) without transfer to $H$, see \S\ref{sec:L-within}.

\section{Global theory}\label{sec:global}
In this section, $\dot{F}$ is a number field, and $\dot{\bpsi} = \prod_v \bpsi_v$ is an additive character of $\dot{F} \backslash \A$.

Since we are mainly dealing with global objects, and also for aesthetic concerns, we will not decorate the $\dot{F}$-groups or coverings of their adélic points with a dot. Therefore, we shall consider a symplectic $\dot{F}$-vector space $(W, \lrangle{\cdot|\cdot})$ of dimension $2n$ and the covering $\tilde{G} := \Mp(W, \A) \xrightarrow{\rev} G(\A) := \Sp(W, \A)$. Identify $G(\dot{F})$ with its image in $\tilde{G}$.

\subsection{The discrete automorphic spectrum}\label{sec:disc}
Decompose the genuine $L^2$-automorphic spectrum of $\tilde{G}$ into the orthogonal direct sum of discrete and continuous parts
\[ L^2_-(G(\dot{F}) \backslash \tilde{G}) = L^2_{\mathrm{disc}, -} \oplus L^2_{\mathrm{cont}, -} . \]
These are genuine unitary representations of $\tilde{G}$.
\index{L2-disc-cont@$L^2_{\mathrm{disc}, -}, L^2_{\mathrm{cont}, -}$}

Fix an $\mathfrak{o}_{\dot{F}}$-lattice $L \subset W$ to obtain the finite set $V_{\mathrm{ram}}$ of places of $\dot{F}$ from \eqref{eqn:V-ram}, which contains all $v \mid \infty$. Using $L$, we define $K_v \subset G(\dot{F}_v)$ and embed it into $\tilde{G}_v$ for each $v \notin V_{\mathrm{ram}}$.

When $v \notin V_{\mathrm{ram}}$, we have the anti-genuine spherical Hecke algebra $\mathcal{H}_v := \mathcal{H}_{\asp}(K_v \backslash \tilde{G}_v / K_v)$. For every finite set $V \supset V_{\mathrm{ram}}$ of places, define $K^V := \prod_{v \notin V} K_v$, and let $\mathcal{C}^V$ be the set of all $c^V = (c_v)_{v \notin V}$ where $c_v$ is a Satake parameter for $\tilde{G}_v$, i.e.\ a semisimple conjugacy class in $\tilde{G}^\vee$. When $V' \supset V$, we have the projection $\mathcal{C}^V \to \mathcal{C}^{V'}$. The set of all such $V \supset V_{\mathrm{ram}}$ is directed by inclusion. Define
\[ \mathcal{C} := \varinjlim_{V \supset V_{\mathrm{ram}}} \mathcal{C}^V. \]

From the actions of the commutative $\CC$-algebras $\bigotimes'_{v \notin V} \mathcal{H}_v$ on $K^V$-invariants, for various $V$, we obtain an orthogonal decomposition
\begin{equation}\label{eqn:L2-disc-c}
	L^2_{\mathrm{disc}, -} = \widehat{\bigoplus}_{c \in \mathcal{C}} L^2_{\mathrm{disc}, c}.
\end{equation}

Set $H = \SO(2n+1)$ and $H_v$ its base change to $\dot{F}_v$. Since $H^\vee = \tilde{G}^\vee$, the Satake parameters for $\tilde{G}_v$ and $H_v$ are the same; they are also the same as elements of $\Psi_{\mathrm{symp}}(2n)$ that are trivial on $I_{\dot{F}_v} \times \SL(2, \CC)$.

Global Arthur parameters have been reviewed in \S\ref{sec:global-parameters}: they are certain formal linear combinations of cuspidal automorphic representations of linear groups. There is a natural map from $\dot{\Psi}(H) = \dot{\Psi}(\tilde{G})$ to $\mathcal{C}$, namely by taking the unramified local components, based on Arthur's crucial result \cite[Theorem 1.4.1]{Ar13}; it is injective by Jacquet--Shalika \cite[Theorem 4.4]{JS81}. We conclude that
\begin{equation}\label{eqn:JS-psi}
	\begin{tikzcd}[column sep=small]
		\{c \in \mathcal{C}: L^2_{\mathrm{disc}, c} \neq 0 \} \arrow[phantom, r, "\subset" description] & \mathcal{C} & \arrow[hookrightarrow, l] \dot{\Psi}(H) = \dot{\Psi}(\tilde{G}) \arrow[phantom, r, "\supset" description] & \dot{\Psi}_2(\tilde{G}).
	\end{tikzcd}
\end{equation}

The following deep result by Gan and Ichino \cite{GI18} relates the two extremities in \eqref{eqn:JS-psi}.

\begin{theorem}[Gan--Ichino {\cite[Theorem 1.1]{GI18}}]\label{prop:GI-disc}
	\index{L2-psi@$L^2_{\dot{\psi}}$}
	If $c \in \mathcal{C}$ satisfies $L^2_{\mathrm{disc}, c} \neq 0$, then $c$ belongs to the image of $\dot{\Psi}_2(\tilde{G})$. In other words,
	\begin{equation*}
		L^2_{\mathrm{disc}, -} = \widehat{\bigoplus}_{\dot{\psi} \in \dot{\Psi}_2(\tilde{G})} L^2_{\dot{\psi}},
	\end{equation*}
	where $L^2_{\dot{\psi}} := L^2_{\mathrm{disc}, c}$ with $c \in \mathcal{C}$ being the image of $\dot{\psi}$.
\end{theorem}

To describe the continuous spectrum $L^2_{\mathrm{cont}, -}$, one can pass to the Levi subgroups $\tilde{M}$ of $\tilde{G}$, which are of metaplectic type, and apply Theorem \ref{prop:GI-disc} to the metaplectic factor of $\tilde{M}$ together with the known results for $\GL$ factors. Indeed, the treatment of Langlands' theory of Eisenstein series in \cite{MW94} applies to covering groups as well.

\begin{remark}
	As discussed in \cite[Remark 1.2]{GI18}, Theorem \ref{prop:GI-disc} entails the \emph{no embedded Hecke eigenvalues} property \cite[Theorem 5]{Ar11} for $\tilde{G}$: the discrete and continuous spectra are separated by Hecke eigenvalues, i.e.\ by their Arthur parameters. In Arthur's endoscopic classification \cite{Ar13} for quasisplit classical groups, this property is one of the end products of his long arguments, whereas in our setting this property is transported from $\SO(2r+1)$ to $\Mp(2n)$ by Gan--Ichino, so it still rests on Arthur's work. The subsequent arguments will be crucially based on this property.
\end{remark}

We proceed to relate the decomposition in Theorem \ref{prop:GI-disc} to trace formula. For every $v \mid \infty$, fix a maximal compact subgroup $K_v$ of $G(\dot{F}_v)$ in good position relative to the standard maximal torus of $G$.

\begin{definition}
	\index{H-asp@$\mathcal{H}_{\asp}$}
	In what follows, the subscript $\asp$ means ``anti-genuine''. Denote by $\mathcal{H}_{\asp}$ the subspace of $C^\infty_c(\tilde{G}) \otimes \mes(G(\A))$ spanned by elements $\dot{f} = \prod_v f_v$ where
	\begin{itemize}
		\item $f_v \in C^\infty_{c, \asp}(\tilde{G}_v) \otimes \mes(G_v)$;
		\item $f_v$ is $\tilde{K}_v$-bi-finite for all $v \mid \infty$;
		\item there exists a finite set $V \supset V_{\mathrm{ram}}$ of places of $\dot{F}$, such that $f_v$ is $f_{K_v}$ tensored with the unramified Haar measure when $v \notin V$.
	\end{itemize}
	Elements of the form $\prod_v f_v$ are said to be \emph{factorizable}.
\end{definition}

The notion of \emph{semi-finite} genuine distributions on $\tilde{G}$ in the sense of \cite[Definition 13.5.3]{Li26} or \cite[X.5.3]{MW16-2} will be convenient. Roughly speaking, these are formal $\CC$-linear combinations of genuine irreducible characters, and can be evaluated at any $\dot{f} \in \mathcal{H}_{\asp}$, whose value is given by a convergent sum. Moreover, one can extract the part with prescribed infinitesimal characters and Satake parameters at almost all places.

The invariant trace formula for $\tilde{G}$ features the invariant distributions $I_{\mathrm{disc}, \lambda}$, see \cite[\S 13.5]{Li26}, where $\lambda$ prescribes the infinitesimal characters at Archimedean places. This is expressed as a formal infinite sum of genuine irreducible characters. Nonetheless, it is semi-finite since it is extracted from the distribution $I_{\mathrm{disc}, t}$ which is semi-finite by \textit{loc.\ cit.}

Given $c \in \mathcal{C}$, as in \textit{loc.\ cit.}\ or \cite[X.5.3]{MW16-2}, we can further extract from $I_{\mathrm{disc}, \lambda}$ the part with Satake parameter $c$, denoted by $I_{\mathrm{disc}, \lambda, c}$, which is also semi-finite.

Now consider $\dot{\psi} \in \dot{\Psi}(\tilde{G})$; it determines
\begin{itemize}
	\item the infinitesimal character $\lambda := \lambda(\dot{\psi})$ at Archimedean places, as in \S\ref{sec:inf-character};
	\item the image $c := c(\dot{\psi}) \in \mathcal{C}$ of $\dot{\psi}$ via \eqref{eqn:JS-psi}.
\end{itemize}
Define
\begin{equation*}
	I_{\mathrm{disc}, \dot{\psi}} := I_{\mathrm{disc}, \lambda, c}.
\end{equation*}

The same construction applies to Levi subgroups $\tilde{M} \subset \tilde{G}$, except that we work with $L^2_{\mathrm{disc}, -}(M(\dot{F}) \backslash \tilde{M} / A_{M, \infty})$, where $A_{M, \infty}$ is the familiar connected central subgroup satisfying
\[ A_{M, \infty} \simeq (\R_{> 0}^{\times})^{\dim Z_M}, \quad \mathrm{vol}\left( M(\dot{F}) \backslash \tilde{M} / A_{M, \infty}\right) < +\infty. \]

\begin{lemma}\label{prop:tr-I-disc}
	Let $\dot{\psi} \in \dot{\Psi}_2(\tilde{G})$. By decomposing $L^2_{\dot{\psi}}$ into irreducibles, its character also makes sense as a semi-finite distribution $\Tr L^2_{\dot{\psi}}$ on $\tilde{G}$, and
	\begin{equation*}
		\Tr L^2_{\dot{\psi}} = I_{\mathrm{disc}, \dot{\psi}}.
	\end{equation*}
\end{lemma}
\begin{proof}
	Take any $\dot{f} \in \mathcal{H}_{\asp}$. We shall apply the formula \cite[(13.5.2)]{Li26} for $I_{\mathrm{disc}, t}(\dot{f})$ where $t = \| \operatorname{Im}(\lambda) \|$. This is a finite linear combination of terms
	\[ \Tr\left( M_{P|P}(s, 0) \mathcal{I}_{\tilde{P}, \mathrm{disc}, t}(0, \dot{f}) \right) \]
	with non-negative coefficients, where
	\begin{itemize}
		\item $P = MU \subset G$ is a semi-standard parabolic subgroup with $M$ semi-standard;
		\item $s$ is a regular element in the group $W^G(M)$;
		\item $\mathcal{I}_{\tilde{P}, \mathrm{disc}, t}(0, \cdot)$ stands for the parabolic induction of the $t$-part of $L^2_{\mathrm{disc}, -}(M(\dot{F}) \backslash  \tilde{M} / A_{M, \infty})$;
		\item $M_{P|P}(s, 0)$ is the standard self-intertwining operator, and $\Tr\left( M_{P|P}(s, 0) \mathcal{I}_{\tilde{P}, \mathrm{disc}, t}(0, \dot{f}) \right)$ makes sense as a semi-finite distribution.
	\end{itemize}
	See \textit{loc.\ cit.}\ for a full explanation.
	
	Parabolic induction respects infinitesimal characters and Satake parameters. To extract the part contributed by $\dot{\psi}$, it suffices to replace the $\mathcal{I}_{\tilde{P}, \mathrm{disc}, t}(0, \dot{f})$ above with
	\[ \sum_{\substack{\dot{\psi}_M \in \dot{\Psi}_2(\tilde{M}) \\ \dot{\psi}_M \mapsto \dot{\psi}}} \mathcal{I}_{\tilde{P}, \mathrm{disc}, \dot{\psi}_M}(0, \dot{f}) \]
	in self-evident notations. This is nonzero only when $\dot{\psi}$ factors through $\tilde{M}^\vee$, i.e.\ when $M = G$.
	
	Moreover, the terms corresponding to $M=G$ give exactly $\Tr L^2_{\dot{\psi}}(\dot{f})$. Indeed, one then has $P = G$, $W^G(M) = \{1\}$, $s = 1$, and $M_{P|P}(1, 0) = \identity$; the induced representation $\mathcal{I}_{\tilde{P}, \mathrm{disc}, \dot{\psi}}(0, \cdot)$ reduces then to the $\dot{f}$-action on $L^2_{\dot{\psi}}$ itself, so the corresponding term is exactly $\Tr L^2_{\dot{\psi}}(\dot{f})$. This proves the desired assertions.
\end{proof}

\subsection{An endoscopic formula}
Continue the discussions in \S\ref{sec:disc}. In what follows, $\dot{f} = \prod_v f_v$ is a factorizable element in $\mathcal{H}_{\asp}$. We also fix $\dot{\psi} \in \dot{\Psi}_2(\tilde{G})$.

\begin{lemma}\label{prop:disc-spec-TF}
	We have
	\[ \Tr L^2_{\dot{\psi}}(\dot{f}) = \sum_{\mathbf{G}^! \in \Endo_{\elli}(\tilde{G})} \iota(\tilde{G}, G^!) S^{G^!}_{\mathrm{disc}, \dot{\psi}}(\dot{f}^!), \]
	where
	\begin{itemize}
		\item $\iota(\tilde{G}, G^!) := \left| Z_{(G^!)^\vee} \right|^{-1}$ as in \cite[Definition 3.7.1]{Li26};
		\item $S^{G^!}_{\mathrm{disc}, \dot{\psi}} := \sum_{\substack{\dot{\psi}^! \in \dot{\Psi}_2(G^!) \\ \dot{\psi}^! \mapsto \dot{\psi}}} S^{G^!}_{\mathrm{disc}, \dot{\psi}^!}$, with $S^{G^!}_{\mathrm{disc}, \dot{\psi}^!}$ being the stable distribution on $G^!(\A)$ defined in \cite[(3.3.13)]{Ar13};
		\item $\dot{f}^!$ is a transfer of $\dot{f}$.
	\end{itemize}
\end{lemma}
\begin{proof}
	Combine Lemma \ref{prop:tr-I-disc} with the stabilization of $I_{\mathrm{disc}, \lambda(\dot{\psi})}$ in \cite[Theorem 13.6.3]{Li26}, by extracting $c(\dot{\psi})$-parts in the latter result, as justified by the spherical fundamental lemma and the fact \cite[X.5.8]{MW16-2} that $S^{G^!}_{\mathrm{disc}}$ is semi-finite.
\end{proof}

Let $\overline{\EuScript{S}}_{\dot{\psi}}$ be the quotient of $\EuScript{S}_{\dot{\psi}}$ by the image of $Z_{\tilde{G}^\vee} = \{\pm 1\}$.
\index{S-psi-curly-overline@$\overline{\EuScript{S}}_{\dot{\psi}}$}

\begin{definition}\label{def:epsilon-Art}
	\index{epsilon-Art@$\epsilon_{\dot{\psi}}^{\mathrm{Art}}$}
	Let $\epsilon_{\dot{\psi}}^{\mathrm{Art}}$ be the character of $\EuScript{S}_{\dot{\psi}}$ defined by Arthur \cite[Theorem 1.5.2]{Ar13} for $H := \SO(2n+1)$. It factors through $\overline{\EuScript{S}}_{\dot{\psi}}$.
	
	Similarly, when $G^! = \SO(2n'+1) \times \SO(2n''+1)$ and $\dot{\psi}^! \in \dot{\Psi}_2(G^!)$, one has the character $\epsilon_{\dot{\psi}^!}^{G^!, \mathrm{Art}}$ of the quotient $\overline{\EuScript{S}}_{\dot{\psi}^!}$ of $\EuScript{S}_{\dot{\psi}^!}$ by the image of $Z_{(G^!)^\vee}$.
\end{definition}

Note that $\dot{\psi} \in \dot{\Psi}_2(\tilde{G})$ implies $S_{\dot{\psi}}$ is finite, $S_{\dot{\psi}} \rightiso \EuScript{S}_{\dot{\psi}}$.

\begin{lemma}\label{prop:disc-spec-TF-Endo}
	We have
	\[ \Tr L^2_{\dot{\psi}}(\dot{f}) = \left| \EuScript{S}_{\dot{\psi}} \right|^{-1} \sum_{\substack{ (\mathbf{G}^!, \dot{\psi}^!) \\ \dot{\psi}^! \mapsto \dot{\psi} }} \epsilon^{G^!}(\dot{\psi}^!) S\Theta^{G^!}_{\dot{\psi}^!}(\dot{f}^!), \]
	where
	\begin{itemize}
		\item $\mathbf{G}^! \in \Endo_{\elli}(\tilde{G})$ and $\dot{\psi}^! \in \dot{\Psi}_2(G^!)$;
		\item $\dot{f}^!$ is the transfer of $\dot{f}$;
		\item $S\Theta^{G^!}_{\dot{\psi}^!}$ is the product of the local distributions in Definition \ref{def:STheta};
		\item $\epsilon^{G^!}(\dot{\psi}^!) := \epsilon_{\dot{\psi}^!}^{G^!, \mathrm{Art}}(s_{\dot{\psi}^!})$ as in \cite[(4.1.4)]{Ar13}, noting that
		\[ s_{\dot{\psi}^!} \mapsto s_{\dot{\psi}} \quad\text{under}\quad S_{\dot{\psi}^!} \rightiso S_{\dot{\psi}}; \]
		see the description after \eqref{eqn:dot-s-psi}.
	\end{itemize}
\end{lemma}
\begin{proof}
	For all $\dot{\psi}^! \in \dot{\Psi}_2(G^!)$, Arthur's stable multiplicity formula \cite[Theorem 4.1.2]{Ar13} for $G^!$ gives
	\[ S^{G^!}_{\mathrm{disc}, \dot{\psi}^!} = \left| \overline{S}_{\dot{\psi}^!} \right|^{-1} \sigma\left( \overline{S}_{\dot{\psi}^!}^\circ \right) \epsilon^{G^!}(\dot{\psi}^!) S\Theta^{G^!}_{\dot{\psi}^!} \]
	as $m_{\dot{\psi}^!} = 1$ in \textit{loc.\ cit.}, where $\sigma(\cdot)$ is that defined in \cite[Proposition 4.1.1]{Ar13}; here $\overline{S}_{\dot{\psi}^!} := S_{\dot{\psi}^!} \big/ Z_{(G^!)^\vee}$ is finite. Substituting into Lemma \ref{prop:disc-spec-TF}, we obtain
	\begin{equation}\label{eqn:disc-spec-TF-Endo}
		\Tr L^2_{\dot{\psi}}(\dot{f}) =
		\sum_{\substack{ (\mathbf{G}^!, \dot{\psi}^!) \\ \dot{\psi}^! \mapsto \dot{\psi} }} \iota(\tilde{G}, G^!) \left| \overline{S}_{\dot{\psi}^!} \right|^{-1} \sigma\left( \overline{S}_{\dot{\psi}^!}^\circ \right) \epsilon^{G^!}(\dot{\psi}^!) S\Theta^{G^!}_{\dot{\psi}^!}(\dot{f}^!).
	\end{equation}
	
	Since $S_{\dot{\psi}^!}$ is finite, $\sigma\left( \overline{S}_{\dot{\psi}^!}^\circ \right) = 1$ by its definition in \cite[Proposition 4.1.1]{Ar13}. On the other hand, from $\dot{\psi} \in \dot{\Psi}_2(\tilde{G})$ one readily sees $S_{\dot{\psi}^!} \rightiso S_{\dot{\psi}}$, so there is a short exact sequence
	\[ 1 \to Z_{(G^!)^\vee} \to S_{\dot{\psi}} \to \overline{S}_{\dot{\psi}^!} \to 1. \]
	
	It follows that $\left| S_{\dot{\psi}} \right| =  \left| \overline{S}_{\dot{\psi}^!} \right| \cdot \left| Z_{(G^!)^\vee} \right| = \left| \overline{S}_{\dot{\psi}^!} \right| \cdot \iota(\tilde{G}, G^!)^{-1}$. The right hand side of \eqref{eqn:disc-spec-TF-Endo} has the desired form.
\end{proof}

\subsection{Global root numbers}\label{sec:global-root-number}
Let $\dot{\phi}$ be a formal sum $\bigoplus_i m_i \dot{\phi}_i$ of cuspidal automorphic representations $\dot{\phi}_i$ of $\GL(d_i, \A)$, where $m_i \in \Z_{\geq 1}$. Its dual $\dot{\phi}^\vee$ is $\bigoplus_i m_i \dot{\phi}_i^\vee$. Define its \emph{global root number} as
\[ \epsilon(\dot{\phi}) := \epsilon\left(\frac{1}{2}, \dot{\phi} \right) := \prod_v \epsilon( \dot{\phi}_v, \bpsi_v ). \]
\index{epsilon-phi-dot@$\epsilon(\dot{\phi})$}

From \S\ref{sec:local-root-numbers} we deduce the following properties.
\begin{itemize}
	\item $\epsilon(\dot{\phi}_1 \oplus \dot{\phi}_2) = \epsilon( \dot{\phi}_1) \epsilon(\dot{\phi}_2)$;
	\item $\epsilon(\dot{\phi})$ is independent of the additive character $\dot{\bpsi} = \prod_v \bpsi_v$ of $\dot{F} \backslash \A$;
	\item $\epsilon(\dot{\phi}) \epsilon(\dot{\phi}^\vee) = 1$.
\end{itemize}

It follows that $\epsilon(\dot{\phi})^2 = 1$ if $\dot{\phi}$ is self-dual.

Now consider $\dot{\psi} \in \dot{\Psi}(\tilde{G})$, and decompose it into $\bigoplus_{i \in I} m_i \dot{\phi}_i \boxtimes r(b_i)$ as in \eqref{eqn:dot-psi-decomp}.

\begin{definition}\label{def:nu-global}
	\index{nu-psi-dot@$\nu_{\dot{\psi}}$}
	Identify $\EuScript{S}_{\dot{\psi}}^\vee$ with $\bmu_2^{I^+}$. Define $\nu_{\dot{\psi}} \in \EuScript{S}_{\dot{\psi}}^\vee$ so that for every $i \in I^+$, its $i$-th component equals
	\[ \nu_{\dot{\psi}, i} := \epsilon(\dot{\phi_i})^{b_i}. \]
\end{definition}

Next, let $s \in S_{\dot{\psi}, 2}$ with image $x \in \EuScript{S}_{\dot{\psi}}$. Define $\epsilon\left( \dot{\psi}^{s = -1} \right)$ as in the local case \S\ref{sec:T}: namely, if the $(-1)$-eigenspace of $s$ gives $\bigoplus_{i \in I} m''_i \dot{\phi}_i \boxtimes r(b_i)$, then $\epsilon\left( \dot{\psi}^{s = -1} \right) := \epsilon\left( \bigoplus_{i \in I} m''_i b_i \dot{\phi}_i \right)$.

\begin{proposition}\label{prop:nu-factorization}
	Let $s \in S_{\dot{\psi}, 2}$. With the notations above,
	\begin{enumerate}[(i)]
		\item $\nu_{\dot{\psi}}(x) =  \epsilon\left( \dot{\psi}^{s = -1} \right)$;
		\item if $\dot{\psi}$ is unramified outside a finite set $V$ of places such that $V \supset V_{\mathrm{ram}}$, then
		\[ \nu_{\dot{\psi}}(x) = \prod_{v \in V} \epsilon\left( \dot{\psi}_v^{s = -1} \right) \]
		where we identify $s$ with its image in $S_{\dot{\psi}_v}$.
	\end{enumerate}
\end{proposition}
\begin{proof}
	The idea for (i) is similar to the proof of Proposition \ref{prop:nu-good-parity}, by replacing local root numbers by global ones. The global case is simpler since $\epsilon(\dot{\phi}) \epsilon( \dot{\phi}^\vee) = 1$, whilst in the local case we get $(\det\phi)(-1)$ instead. This explains why $\dot{\psi}$ is not required to be of quasi-good parity here.
	
	To deduce (ii), factorize $\epsilon\left( \dot{\psi}^{s = -1} \right)$.
\end{proof}

\begin{lemma}\label{prop:nu-global}
	For all $\dot{\psi} \in \dot{\Psi}(\tilde{G})$, we have $\nu_{\dot{\psi}}(s_{\dot{\psi}}) = 1$.
\end{lemma}
\begin{proof}
	For all $i \in I^+$, the component of $s_{\dot{\psi}}$ at the factor $\Or(m_i, \CC)$ of $S_{\dot{\psi}}$ is $\identity$ (resp.\ $-\identity$) when $b_i$ is odd (resp.\ even). Moreover, for even $b_i$ the image of $s_{\dot{\psi}}$ in $\EuScript{S}_{\dot{\psi}} \simeq \bmu_2^{I^+}$ has $i$-component equal to $\det(-\identity_{m_i}) = (-1)^{m_i}$. Hence
	\[ \nu_{\dot{\psi}}(s_{\dot{\psi}}) = \prod_{\substack{i \in I^+ \\ b_i \;\text{even} \\ m_i \;\text{odd}}} \epsilon( \dot{\phi}_i )^{b_i}. \]
	Since $\epsilon( \dot{\phi}_i )^2 = 1$ for all $i \in I^+$, the proof is complete.
\end{proof}

\subsection{The multiplicity formula}\label{sec:multiplicity-formula}
Assume $\dot{\psi} \in \dot{\Psi}_2(\tilde{G})$. Decompose $\dot{\psi}$ into $\bigoplus_{i \in I} \dot{\phi}_i \boxtimes r(b_i)$ where the $\dot{\phi}_i \boxtimes r(b_i)$ are simple and distinct. Note that $S_{\dot{\psi}} \rightiso \EuScript{S}_{\dot{\psi}}$.

Fix $V \supset V_{\mathrm{ram}}$ hereafter. From the homomorphisms \eqref{eqn:S-localization} we obtain
\begin{equation}\label{eqn:diag-map}
	\index{diag@$\mathrm{diag}_V, \mathrm{diag}$}
	\begin{aligned}
		\mathrm{diag}_V: \EuScript{S}_{\dot{\psi}} & \to \EuScript{S}_{\dot{\psi}, V} := \prod_{v \in V} \EuScript{S}_{\dot{\psi}_v}, \\
		\mathrm{diag}: \EuScript{S}_{\dot{\psi}} & \to \prod_{\text{all}\; v} \EuScript{S}_{\dot{\psi}_v}.
	\end{aligned}
\end{equation}

We are ready to state and prove the global multiplicity formula as a trace identity.

\begin{theorem}\label{prop:global-multiplicity}
	\index{epsilon-psi-dot@$\epsilon_{\dot{\psi}}$}
	\index{X-psi-dot-V@$\mathcal{X}(\dot{\psi}, V)$}
	Assume that $\dot{\psi}_v$ is unramified for all $v \notin V$. Set
	\begin{align*}
		\epsilon_{\dot{\psi}} & := \epsilon^{\mathrm{Art}}_{\dot{\psi}} \nu_{\dot{\psi}} \; \in \EuScript{S}_{\dot{\psi}}^\vee, \\
		\mathcal{X}(\dot{\psi}, V) & := \left\{ \chi_V = (\chi_v)_{v \in V} \in \prod_{v \in V} \EuScript{S}_{\dot{\psi}_v}^\vee : \mathrm{diag}_V^* \left( \prod_{v \in V} \chi_v \right) = \epsilon_{\dot{\psi}} \right\},
	\end{align*}
	where $\epsilon^{\mathrm{Art}}_{\dot{\psi}}$ and $\nu_{\dot{\psi}}$ are as in Definitions \ref{def:epsilon-Art} and \ref{def:nu-global}. For all factorizable $\dot{f} = \prod_v f_v \in \mathcal{H}_{\asp}$ such that $f_v$ equals $f_{K_v}$ tensored with the unramified Haar measure for all $v \notin V$, we have
	\[ \Tr L^2_{\dot{\psi}}(\dot{f}) = \sum_{\chi_V \in \mathcal{X}(\dot{\psi}, V)} \prod_{v \in V} \pi_{\dot{\psi}_v, \chi_v}(f_v) \]
	with $\pi_{\dot{\psi}_v, \chi_v}$ as in Definition \ref{def:pi-psi-chi}.
\end{theorem}
\begin{proof}
	There is no need to distinguish elements $s \in S_{\dot{\psi}}$ from their image $x \in \EuScript{S}_{\dot{\psi}}$. Suppose that $(\mathbf{G}^!, \dot{\psi}^!) \leftrightarrow (\dot{\psi}, s)$ via the basic bijection (Remark \ref{rem:basic-bijection}). Note that $s$ also induces an elliptic endoscopic datum of $H := \SO(2n+1)$. By \cite[Lemma 4.4.1]{Ar13},
	\[ \epsilon^{G^!}(\dot{\psi}^!) := \epsilon^{G^!, \mathrm{Art}}_{\dot{\psi}^!}(s_{\dot{\psi}^!}) = \epsilon^{H, \mathrm{Art}}_{\dot{\psi}}(s_{\dot{\psi}}s). \]
	
	Lemma \ref{prop:disc-spec-TF-Endo} thus implies
	\[ \Tr L^2_{\dot{\psi}}(\dot{f}) = \left| \EuScript{S}_{\dot{\psi}} \right|^{-1} \sum_{s \in \EuScript{S}_{\dot{\psi}}} \epsilon^{H, \mathrm{Art}}_{\dot{\psi}}(s_{\dot{\psi}}s) S\Theta^{G^!}_{\dot{\psi}^!}(\dot{f}^!) \]
	where the sum over $s \in \EuScript{S}_{\dot{\psi}}$ is identified with the sum over the pairs $(\mathbf{G}^!, \dot{\psi}^!)$ mapping to $\dot\psi$ via the basic bijection (Remark \ref{rem:basic-bijection}).	As $\dot{f} = \prod_v f_v$ and everything is unramified off $V$, in view of the spherical fundamental lemma, we obtain
	\begin{align*}
		\Tr L^2_{\dot{\psi}}(\dot{f}) & = \left| \EuScript{S}_{\dot{\psi}} \right|^{-1} \sum_{s \in \EuScript{S}_{\dot{\psi}}} \epsilon^{H, \mathrm{Art}}_{\dot{\psi}}(s_{\dot{\psi}}s) \prod_{v \in V} \left( \trans_{\mathbf{G}^!_v, \tilde{G}_v} S\Theta^{G^!_v}_{\dot{\psi}^!_v}\right)(f_v) \\
		& = \left| \EuScript{S}_{\dot{\psi}} \right|^{-1} \sum_{s \in \EuScript{S}_{\dot{\psi}}} \epsilon^{H, \mathrm{Art}}_{\dot{\psi}}(s_{\dot{\psi}}s) \prod_{v \in V} \epsilon\left( \dot{\psi}_v^{s = -1} \right) \sum_{\chi_v \in \EuScript{S}_{\dot{\psi}_v}^\vee} \chi_v(s_{\dot{\psi}} s) \pi_{\dot{\psi}_v, \chi_v}(f_v),
	\end{align*}
	where the maps $S_{\dot{\psi}} \to S_{\dot{\psi}_v}$ are omitted from the notation, and the easy fact $s_{\dot{\psi}} \mapsto s_{\dot{\psi}_v}$ has been used; the two equalities follow from Definitions \ref{def:T-dist} and \ref{def:pi-psi-chi}, respectively.
	
	Regroup the summation to obtain
	\begin{equation*}
		\Tr L^2_{\dot{\psi}}(\dot{f}) = \sum_{\chi_V} x(\chi_V) \prod_{v \in V} \pi_{\dot{\psi}_v, \chi_v}(f_v)
	\end{equation*}
	where $\chi_V = (\chi_v)_{v \in V}$ ranges over $\prod_{v \in V} \EuScript{S}_{\dot{\psi}_v}^\vee$, with
	\begin{align*}
		x(\chi_V) & := \left| \EuScript{S}_{\dot{\psi}} \right|^{-1} \sum_{s \in \EuScript{S}_{\dot{\psi}}} \epsilon^{H, \mathrm{Art}}_{\dot{\psi}}(s_{\dot{\psi}}s) \prod_{v \in V} \epsilon\left( \dot{\psi}_v^{s = -1} \right) \chi_v(s_{\dot{\psi}} s) \\
		& = \left| \EuScript{S}_{\dot{\psi}} \right|^{-1} \sum_{s \in \EuScript{S}_{\dot{\psi}}} \epsilon^{H, \mathrm{Art}}_{\dot{\psi}}(s_{\dot{\psi}}s) \nu_{\dot{\psi}}(s) \prod_{v \in V} \chi_v(s_{\dot{\psi}} s) \quad \text{by Proposition \ref{prop:nu-factorization} (ii)} \\
		& = \left| \EuScript{S}_{\dot{\psi}} \right|^{-1} \sum_{s \in \EuScript{S}_{\dot{\psi}}} \left(\epsilon^{H, \mathrm{Art}}_{\dot{\psi}} \nu_{\dot{\psi}}\right)(s_{\dot{\psi}}s) \prod_{v \in V} \chi_v(s_{\dot{\psi}} s) \quad \text{by Lemma \ref{prop:nu-global}} \\
		& = \left| \EuScript{S}_{\dot{\psi}} \right|^{-1} \sum_{s \in \EuScript{S}_{\dot{\psi}}} \left(\epsilon^{H, \mathrm{Art}}_{\dot{\psi}} \nu_{\dot{\psi}} \chi_V \right)(s).
	\end{align*}
	
	By Fourier inversion, $x$ is seen to be the indicator function of the subset $\mathcal{X}(\dot{\psi}, V)$, by recalling that $\epsilon^{\mathrm{Art}}_{\dot{\psi}} = \epsilon^{H, \mathrm{Art}}_{\dot{\psi}}$.
\end{proof}

The result above generalizes \cite[Theorem 1.4]{GI18}, which concerns the case of generic $\dot{\psi}$ (i.e.\ trivial on $\SL(2, \CC)$), in which case $\epsilon^{\mathrm{Art}}_{\dot{\psi}} = \mathbf{1}$ and $s_{\dot{\psi}} = 1$.

\begin{remark}
	In the right hand side of Theorem \ref{prop:global-multiplicity}, one can enlarge $V$ to any larger finite set $V'$; the indices $\chi_{V'}$ with nonzero contribution will satisfy $\chi_v = \mathbf{1}$ for all $v \in V' \smallsetminus V$, by Proposition \ref{prop:pi-psi-nr}.
\end{remark}

\begin{remark}\label{rem:global-packet}
	The proof above does not presume the validity of Theorem \ref{prop:local-desiderata}, although the formula would have little use if the $\pi_{\dot{\psi}_v, \chi_v}$ were merely virtual genuine characters. Granting Theorem \ref{prop:local-desiderata} and the formalism of packets as multi-sets in \S\ref{sec:A-packets}, the multiplicity formula can be written in the form of \cite[Theorem 1.5.2]{Ar13}: define
	\begin{align*}
		\Pi_{\dot{\psi}} & := \left\{ \dot{\pi} = (\dot{\pi}_v)_v : \dot{\pi}_v \in \Pi_{\dot{\psi}_v}, \; K_v\text{-spherical} \; \text{for almost all}\; v \notin V_{\mathrm{ram}} \right\}, \\
		\Pi_{\dot{\psi}}(\epsilon_{\dot{\psi}}) & := \left\{ \dot{\pi} \in \Pi_{\dot{\psi}}: \mathrm{diag}^* \prod_v \lrangle{\cdot, \dot{\pi}_v} = \epsilon_{\dot{\psi}} \right\}.
	\end{align*}
	Note that if $v \notin V_{\mathrm{ram}}$ and $\dot{\pi}_v \in \Pi_{\dot{\psi}_v}$ is $K_v$-spherical, then $\lrangle{\cdot, \dot{\pi}_v} = \mathbf{1}$ by Proposition \ref{prop:unramified-mf} (i).
	
	We then have $\Tr L^2_{\dot{\psi}} = \displaystyle\sum_{\dot{\pi} \in \Pi_{\dot{\psi}}(\epsilon_{\dot{\psi}})} \prod_v \Theta_{\dot{\pi}_v}$ as an equality of semi-finite distributions on $\tilde{G}$. Equivalently, by identifying $\dot{\pi}$ with the unitary genuine irreducible representation $\bigotimes'_v \dot{\pi}_v$ of $\tilde{G}$, we have
	\[ L^2_{\dot{\psi}} \simeq \widehat{\bigoplus}_{\dot{\pi} \in \Pi_{\dot{\psi}}(\epsilon_{\dot{\psi}})} \dot{\pi}. \]
	Note that by abuse of notation, $\dot{\pi}_v \in \Pi_{\dot{\psi}_v}$ (resp.\ $\dot{\pi}$) has been identified with the corresponding genuine irreducible representation of $\tilde{G}_v$ (resp.\ of $\tilde{G}$, given by restricted tensor product). This is the statement of Theorem \ref{prop:main-global}.
\end{remark}

\begin{corollary}\label{prop:global-mf}
	Granting the validity of Theorem \ref{prop:local-desiderata} for each $\dot{F}_v$, we employ the formalism in \S\ref{sec:A-packets}. Fix $\dot{\psi} \in \dot{\Psi}_2(\tilde{G})$. Let $\dot{\pi} = \bigotimes'_v \dot{\pi}_v$ be a genuine irreducible representation of $\tilde{G}$ such that $\dot{\pi}_v$ appears in $\Pi_{\dot{\psi}_v}$ for all places $v$.
	\begin{enumerate}[(i)]
		\item The multiplicity of $\dot{\pi}$ in $L^2_{\mathrm{disc}, -}$ is equal to its multiplicity in $L^2_{\dot{\psi}}$.
		\item If $\dot{\pi}_v$ appears in $\Pi_{\dot{\psi}_v}^{\mathrm{mf}}$ (Definition \ref{def:multiplicity-free}) for all $v$, then $\dot{\pi}$ has multiplicity $\leq 1$ in $L^2_{\dot{\psi}}$.
	\end{enumerate}
\end{corollary}
\begin{proof}
	Consider (i). Given $\alpha: \dot{\pi} \hookrightarrow L^2_{\mathrm{disc}, -}$, there exists $\dot{\xi} \in \dot{\Psi}_2(\tilde{G})$ such that $\Image(\alpha) \subset L^2_{\dot{\xi}}$ by Theorem \ref{prop:GI-disc}. This step uses multiplicity-one for $\GL(N)$ and the injectivity of the map from parameters of $\SO(2n+1)$ to those of $\GL(2n)$, a specific fact that fails for $\SO(2n)$.
	
	For almost all $v$, Proposition \ref{prop:unramified-mf} (ii) implies that $\dot{\psi}_v = \dot{\xi}_v$ since their Arthur packets both contain the $K_v$-spherical representation $\dot{\pi}_v$. Therefore $\dot{\psi} = \dot{\xi}$.
	
	In the setting of (ii), observe that in the formula for $L^2_{\dot{\psi}}$ in Theorem \ref{prop:global-multiplicity}, the trace of $\dot{\pi}$ has coefficient $0$ or $1$ in the right-hand side, by the definition of $\Pi_{\dot{\psi}_v}^{\mathrm{mf}}$.
\end{proof}

\begin{remark}
	If $\Pi_{\dot{\psi}_v}$ is multiplicity-free for all $v, \dot{\psi}$, then $L^2_{\mathrm{disc}, -}$ is multiplicity-free by Corollary \ref{prop:global-mf}. On the other hand, \cite[Remark 3.3]{GI18} shows that if the discrete $L^2$-automorphic spectrum of $\SO(2r+1, \A)$ is multiplicity-free for all $r > 2n+1$, then $L^2_{\mathrm{disc}, -}$ is multiplicity-free as well. In particular, multiplicity-freeness is known for totally imaginary $\dot{F}$, since local multiplicity-one for Arthur packets of $\SO(2r+1)$ is available at all finite and complex places.
\end{remark}

\begin{remark}\label{rem:ArthurMp-Gan}
	Theorem \ref{prop:global-multiplicity} (or in the form of Remark \ref{rem:global-packet}) confirms Gan's Arthur conjecture for metaplectic groups in \cite[Conjecture 8.2]{Gan14}. To see this, rename the character denoted by $\mu_\Psi$ in \textit{loc.\ cit.}\ to $\nu_{\dot{\psi}}^{\mathrm{Gan}}$, where $\dot{\psi} \in \dot{\Psi}_2(\tilde{G})$. It suffices to show $\nu_{\dot{\psi}}^{\mathrm{Gan}} = \nu_{\dot{\psi}}$. Indeed, Proposition \ref{prop:nu-factorization} (i) says $\nu_{\dot{\psi}}(s) = \epsilon(\dot{\psi}^{s=-1})$, whilst $\nu_{\dot{\psi}}^{\mathrm{Gan}}(s) = \epsilon(\phi_{\dot{\psi}}^{s = -1})$, cf.\ \eqref{eqn:phi-psi} for the definition of $\phi_{\dot{\psi}}$; to show $\nu_{\dot{\psi}}^{\mathrm{Gan}}(s) = \nu_{\dot{\psi}}(s)$, we express both $\epsilon$-factors as Euler products and apply the upcoming Lemma \ref{prop:epsilon-phi-psi} at each place of $\dot{F}$.
\end{remark}

\section{L-packets}\label{sec:L-packets}
In this section, $F$ is a local field of characteristic zero, and we fix $(W, \lrangle{\cdot|\cdot})$ and $\bpsi$ to define $\tilde{G} := \Mp(W)$.

\subsection{Tempered L-packets}
We have the following sets
\[ \Pi_-(\tilde{G}) \supset \Pi_{\mathrm{unit}, -}(\tilde{G}) \supset \Pi_{\mathrm{temp}, -}(\tilde{G}) \supset \Pi_{2, -}(\tilde{G}) \]
of irreducible representations of $\tilde{G}$: the first two already appeared in \S\ref{sec:A-packets}, whilst $\Pi_{\mathrm{temp}, -}(\tilde{G})$ (resp.\ $\Pi_{2, -}(\tilde{G})$) is the subset consisting of tempered (resp.\ square-integrable) representations.
\index{Pi-temp-2@$\Pi_{\mathrm{temp}, -}(\tilde{G}), \Pi_{2, -}(\tilde{G})$}

\begin{theorem}[C.\ Luo {\cite{Luo20}}]\label{prop:Luo}
	\index{L-packet}
	\index{Pi-phi@$\Pi_\phi$}
	\index{psi-generic@$\bpsi$-generic}
	Let $\phi \in \Phi_{\mathrm{bdd}}(\tilde{G})$.
	\begin{enumerate}[(i)]
		\item For every $\chi \in \EuScript{S}_\phi^\vee$, the $\pi_{\phi, \chi}$ is the character of some element of $\Pi_{\mathrm{temp}, -}(\tilde{G})$.
		\item If $\chi \neq \eta$, then $\pi_{\phi, \chi} \neq \pi_{\phi, \eta}$.
		\item Hereafter, identify $\pi_{\phi, \chi}$ with an element of $\Pi_{\mathrm{temp}, -}(\tilde{G})$. Then $\pi_{\phi, \chi}$ is parameterized by the pair $(\phi, \chi)$ via the local Langlands correspondence for $\tilde{G}$ of Adams--Barbasch \cite{AB98} and Gan--Savin \cite{GS1}, in the real and non-Archimedean cases respectively; for the easy case $F = \CC$, see \cite[\S 7.6]{Li19}.
		\item The representation $\pi_{\phi, \chi}$ is $\bpsi$-generic if and only if $\chi = \mathbf{1}$; see \cite[Proposition 12.1]{GGP1} and \cite[\S 9]{GS1} for the notion of $\bpsi$-genericity.
	\end{enumerate}
	Moreover, defining the L-packet indexed by $\phi$ as
	\[ \Pi_\phi = \Pi^{\tilde{G}}_\phi := \{ \pi_{\phi, \chi} \in \Pi_{\mathrm{temp}, -}(\tilde{G}) : \chi \in \EuScript{S}_\phi^\vee \}, \]
	we have
	\[\begin{tikzcd}[row sep=tiny, column sep=tiny]
		\Pi_{\mathrm{temp}, -}(\tilde{G}) \arrow[equal, r] & \bigsqcup_{\phi \in \Phi_{\mathrm{bdd}}(\tilde{G})} \Pi_\phi , \\
		\Pi_{2, -}(\tilde{G}) \arrow[phantom, u, "\subset" description, sloped] \arrow[equal, r] & \bigsqcup_{\phi \in \Phi_{2, \mathrm{bdd}}(\tilde{G})} \Pi_\phi . \arrow[phantom, u, "\subset" description, sloped]
	\end{tikzcd}\]
\end{theorem}
\begin{proof}
	All are in \cite{Luo20} except (iv) in the non-Archimedean case. The remaining piece follows by combining \cite[\S D.1]{AGIKMS} (see also \cite[Corollary 6.16]{Va17}) for $\SO(2n+1)$ and \cite[\S 9]{GS1}.
\end{proof}

The L-packets are singletons in the complex case; see \cite[\S 7.6]{Li19}. Luo's proof is a blend of endoscopy of $\tilde{G}$ (spectral transfer and spherical fundamental lemma, both depending on the stabilization of simple trace formulas) and the multiplicity formula of Gan--Ichino \cite{GI18}.

These results generalize immediately to groups of metaplectic type, say
\[ \tilde{M} = \Mp(W^\flat) \times \prod_{i=1}^r \GL(n_i, F), \]
such as the Levi subgroups of $\tilde{G}$. The L-packet associated with $\phi_M \in \Phi_{\mathrm{bdd}}(\tilde{M})$ will be denoted by $\Pi^{\tilde{M}}_{\phi_M}$.

For later use, we record the endoscopic character relations as follows. Decompose every $\phi_M \in \Phi(\tilde{M})$ and $s_M \in S_{\phi_M}$ into
\begin{equation}\label{eqn:phi-M-decomp}
	\begin{aligned}
		\phi_M & = (\underline{\phi}, \phi_1, \ldots, \phi_r), \quad \underline{\phi} \in \Phi(\Mp(W^\flat)), \; \phi_i \in \Phi(n_i), \\
		s_M & = (\underline{s}, s_1, \ldots, s_r).
	\end{aligned}
\end{equation}
Assume $\phi_M \in \Phi_{\mathrm{bdd}}(\tilde{M})$, $s_M^2 = 1$ and $s_1 = \cdots = s_r = 1$; let $(\mathbf{M}^!, \phi^!_{M^!})$ be the preimage of $(\phi_M, s_M)$ under the basic bijection (the variant of Proposition \ref{prop:basic-bijection} (ii) for $\tilde{M}$). Then
\begin{equation}\label{eqn:endo-char-M}
	\epsilon\left( \underline{\phi}^{\underline{s} = -1} \right) \cdot \trans_{\mathbf{M}^!, \tilde{M}} \left( S\Theta^{M^!}_{\phi^!_{M^!}} \right) = \sum_{\chi \in \EuScript{S}_{\phi_M}^\vee} \chi(s_M) \pi_{\phi_M, \chi}.
\end{equation}
Indeed, this follows immediately from $\EuScript{S}_{\phi_M} \simeq \EuScript{S}_{\underline{\phi}}$ and the characterization of the $\Mp(W^\flat)$-part of $\pi_{\phi_M, \chi}$.

\subsection{L-packets in general}\label{sec:L-general}
For every $\phi \in \Phi(\tilde{G})$, there exist
\begin{itemize}
	\item a standard parabolic subgroup $\tilde{P}^\vee = \tilde{M}^\vee U^\vee \subset \tilde{G}^\vee$, corresponding to a standard parabolic $P = MU \subset G$ where
	\[ M = \Sp(W^\flat) \times \prod_{i=1}^r \GL(n_i); \]
	\item $\lambda \in \mathfrak{a}^*_M := \R^r$ whose values under the coroots outside $M$ are all positive, i.e.\
	\[ \lambda_1 > \cdots > \lambda_r > 0; \]
	\item $\phi_M \in \Phi(\tilde{M})$ of the form $\phi_M = \left( \phi_{M, \mathrm{bdd}} \right)_\lambda$ such that
	\[ \phi_{M, \mathrm{bdd}} \in \Phi_{\mathrm{bdd}}(\tilde{M}), \quad \phi_M \mapsto \phi, \]
	with $\phi_M = \phi_{M,\mathrm{bdd}} \cdot a_\lambda$ where $a_\lambda: \Weil{F} \to Z_{\tilde{M}^\vee}$ parametrizes the character
	\[ c_\lambda: (g_1, \ldots, g_r) \mapsto \prod_{i=1}^r |\det g_i|_F^{\lambda_i}, \quad (g_1, \ldots, g_r) \in \prod_{i=1}^r \GL(n_i, F). \]
\end{itemize}
Furthermore, $(M, \phi_{M, \mathrm{bdd}}, \lambda)$ is uniquely determined by $\phi$, and
\begin{equation}\label{eqn:S-classification}
	S_\phi \simeq S_{\phi_M} = S_{\phi_{M, \mathrm{bdd}}}
\end{equation}
where the second and third centralizers are taken within $\tilde{M}^\vee$. The mechanism is general: by the assumption on $\lambda$, the centralizer in $\tilde{G}^\vee$ of $a_\lambda$ is exactly $\tilde{M}^\vee$, whence the isomorphism; as $a_\lambda$ is central in $\tilde{M}^\vee$, twisting by $a_\lambda$ does not change centralizers within $\tilde{M}^\vee$, whence the equality. In terms of \eqref{eqn:S-psi}, this means that all $\GL$ factors in $S_\phi$ fit into the $\GL$ factors of $M$, due to the $a_\lambda$-twists on them.

The description above of $\phi$ is clearly parallel to the classification of representations in terms of tempered ones via Langlands quotient. To be explicit, denote by $I_{\tilde{P}}(\cdot)$ the parabolic induction functor from $\tilde{P} = \tilde{M}U(F)$ to $\tilde{G}$; for every $\pi \in \Pi_{\mathrm{temp}, -}(\tilde{M})$ and $\lambda \in \mathfrak{a}^*_M$ with the positivity condition above, put $\pi_\lambda := \pi \otimes c_\lambda$; the corresponding standard module is $I_{\tilde{P}}(\pi_\lambda)$.
\index{IP@$I_{\tilde{P}}$}

\begin{definition}
	\index{L-packet}
	\index{Pi-phi}
	For all $\phi \in \Phi(\tilde{G})$, say arising from $(M, \phi_{M, \mathrm{bdd}}, \lambda)$, define
	\begin{align*}
		\Pi^{\tilde{M}}_{\phi_M} & := \left\{ \pi_\lambda \middle| \pi \in \Pi^{\tilde{M}}_{\phi_{M, \mathrm{bdd}}} \right\}, \\
		\Pi_\phi & := \left\{ \text{Langlands quotient of}\; I_{\tilde{P}}(\pi) \middle| \pi \in \Pi^{\tilde{M}}_{\phi_M} \right\}.
	\end{align*}
\end{definition}

We call $\Pi_\phi$ the L-packet associated with $\phi$. By \eqref{eqn:S-classification}, it is still in bijection with $\EuScript{S}_\phi^\vee$, written as $\pi_{\phi, \chi} \leftrightarrow \chi$. For bounded $\phi$, this $\Pi_\phi$ coincides with the multi-set of Definition \ref{def:A-packet} by Theorem \ref{prop:Luo}. Furthermore, Theorem \ref{prop:Luo} and Langlands' classification lead to
\begin{equation}\label{eqn:LLC-general}
	\Pi_-(\tilde{G}) = \bigsqcup_{\phi \in \Phi(\tilde{G})} \Pi_\phi.
\end{equation}


All these recipes are general. In particular, to all $\mathbf{G}^! \in \Endo_{\elli}(\tilde{G})$ and $\phi^! \in \Phi(G^!)$, we may attach the datum $(M^!, \phi^!_{M^!, \mathrm{bdd}}, \lambda)$. Put $\phi^!_{M^!} := (\phi^!_{M^!, \mathrm{bdd}})_\lambda$ and define the stable distributions
\begin{equation}\label{eqn:STheta-gen}
	S\Theta^{M^!}_{\phi^!_{M^!}} := \left( S\Theta^{M^!}_{\phi^!_{M^!, \mathrm{bdd}}}\right)_\lambda , \quad
	S\Theta^{G^!}_{\phi^!} := \Ind^{G^!}_{M^!} \left( S\Theta^{M^!}_{\phi^!_{M^!}} \right),
\end{equation}
with self-evident meaning of $(\cdots)_\lambda$. Below is an endoscopic character relation for $\phi \in \Phi(\tilde{G})$, stated using $S\Theta^{G^!}_{\phi^!}$ and the standard modules rather than Langlands quotients.

As a preparation, note that $\Phi(\tilde{G}) \subset \Psi^+(\tilde{G})$ and $\Phi(G^!) \subset \Psi^+(G^!)$. Given $\phi \in \Phi(\tilde{G})$ and $s \in S_{\phi, 2}/\mathrm{conj}$, let $(\mathbf{G}^!, \phi^!)$ be the preimage of $(\phi, s)$ under the basic bijection (Proposition \ref{prop:basic-bijection} (ii)). As a special instance of Definition \ref{def:T-dist}, applicable to all parameters in $\Psi^+(\tilde{G})$, put
\[ T_{\phi, s} := \epsilon\left( \phi^{s = -1} \right) \cdot \trans_{\mathbf{G}^!, \tilde{G}} \left( S\Theta^{G^!}_{\phi^!} \right). \]

\begin{proposition}\label{prop:endo-char-gen}
	Suppose $\phi \in \Phi(\tilde{G})$ and $s \in S_{\phi, 2}/\mathrm{conj}$, and $(\mathbf{G}^!, \phi^!)$ is the preimage of $(\phi, s)$ under the basic bijection. Then $T_{\phi, s} = \sum_\chi \chi(s) \Theta_{\rho_\chi}$, where
	\begin{itemize}
		\item $\chi$ ranges over $\EuScript{S}_\phi^\vee = \EuScript{S}_{\phi_M}^\vee$,
		\item $\rho_\chi$ is the standard module induced from $\pi_{\phi_M, \chi} \in \Pi^{\tilde{M}}_{\phi_M}$.
	\end{itemize}
\end{proposition}
\begin{proof}
	Take the $(M^!, \phi^!_{M^!, \mathrm{bdd}}, \lambda)$ associated with $\phi^!$. Since $T_{\phi, s}$ is known to depend only on the image $x \in \EuScript{S}_\phi$ of $s$, we will write it as $T_{\phi, x}$. Use \eqref{eqn:S-classification} to take the $x_M \in \EuScript{S}_{\phi_M}$ such that $x_M \mapsto x$. Proposition \ref{prop:T-psi-Ind} entails
	\[ T_{\phi, x} = \Ind^{\tilde{G}}_{\tilde{M}} \left( T_{\phi_M, x_M} \right). \]
	
	From \eqref{eqn:endo-char-M} for $\phi_{M, \mathrm{bdd}}$ plus a $\lambda$-twist on the $\GL$ part, we get
	\[ T_{\phi_M, x_M} = \sum_{\chi \in \EuScript{S}_{\phi_M}^\vee} \chi(x_M) \pi_{\phi_M, \chi}. \]
	One replaces $\chi(x_M)$ by $\chi(x) = \chi(s)$ by viewing $\chi$ as an element of $\EuScript{S}_\phi^\vee$. After applying $\Ind^{\tilde{G}}_{\tilde{M}}$, the right hand side becomes $\sum_\chi \chi(s) \Theta_{\rho_\chi}$.
\end{proof}

The identity above is obtained by reduction to the tempered case via standard modules. The mechanism is thus different from that for Arthur parameters discussed in \S\ref{sec:local-desiderata}.

\subsection{The L-packet within an Arthur packet}\label{sec:L-within}
Let $\psi \in \Psi(\tilde{G})$. The results below are based on \cite[Proposition 7.4.1]{Ar13}, with extra complications from local root numbers.

\begin{lemma}\label{prop:s-phi-psi}
	The image of $s_\psi$ in $\EuScript{S}_{\phi_\psi}$ is trivial.
\end{lemma}
\begin{proof}
	Since $\psi\left(1, \bigl(\begin{smallmatrix} z & \\ & z^{-1} \end{smallmatrix} \bigr) \right)$ centralizes $\phi_\psi$ for all $z \in \CC^{\times}$, they all have trivial image in $\EuScript{S}_{\phi_\psi}$.
\end{proof}

\begin{lemma}\label{prop:epsilon-phi-psi}
	Let $s \in S_{\psi, 2}$. Then $\epsilon(\psi^{s = -1}) = \epsilon(\phi_\psi^{s = -1})$.
\end{lemma}
\begin{proof}
	Decompose $\psi$ (resp.\ $\phi_\psi$) as in \eqref{eqn:psi-decomp-2} (resp.\ \eqref{eqn:phi-psi}).
	
	As seen in the proof of Lemma \ref{prop:T-psi-s}, the conjugacy class of $s$ amounts to a decomposition $m_i = m'_i + m''_i$ (ordered) for all $i \in I^+ \sqcup I^- \sqcup J$, such that $m''_i$ is even if $i \in I^-$.

	Specifically, for $i \in I^+ \sqcup I^-$, in $m_i (\phi_i \boxtimes r(b_i))$ there are exactly $m'_i$ (resp.\ $m''_i$) summands on which $s$ acts by $1$ (resp.\ $-1$). Same for $i \in J$ and $m_i (\phi_i \oplus \phi_i^\vee) \boxtimes r(b_i)$. Therefore,
	\begin{equation*}
		\epsilon(\psi^{s = -1}) = \prod_{i \in I^+ \sqcup I^-} \epsilon(\phi_i, \bpsi)^{b_i m''_i} \prod_{j \in J} (\det\phi_j)(-1)^{b_j m''_j}.
	\end{equation*}
	
	For $\phi_\psi$, note that in \eqref{eqn:phi-psi} there are pairs $(\phi_i |\cdot|^k , \phi_i |\cdot|^{-k})$ for various $k \in \frac{1}{2}\Z$, leaving one unpaired $\phi_i$ exactly when $b_i$ is odd. Consider the contribution of these pairs in $\epsilon(\phi_\psi^{s = -1})$.
	\begin{itemize}
		\item When $i \in I^+ \sqcup I^-$, elements in each pair are in duality, contributing
		\[ \left(\det (\phi_i |\cdot|^k)\right)(-1) = (\det \phi_i)(-1). \]
		\item For $j \in J$, both $\phi_j$ and $\phi^\vee_j$ intervene; we get $(\phi_j |\cdot|^k, \phi_j |\cdot|^{-k}, \phi_j^\vee |\cdot|^k , \phi_j^\vee |\cdot|^{-k} )$ from pairs. Taking local root numbers yields $(\det \phi_j)(-1) (\det \phi^\vee_j)(-1) = 1$. 
	\end{itemize}
	
	All in all,
	\begin{multline*}
		\epsilon(\phi_\psi^{s = -1}) = \prod_{i \in I^+ \sqcup I^-} (\det\phi_i)(-1)^{m''_i \lfloor b_i/2 \rfloor} \prod_{\substack{i \in I^+ \sqcup I^- \\ b_i \;\text{odd}}} \epsilon(\phi_i, \bpsi)^{m''_i} \prod_{\substack{j \in J \\ b_j \;\text{odd}}} \epsilon(\phi_j, \bpsi)^{m''_j} \epsilon(\phi^\vee_j, \bpsi)^{m''_j} \\
		= \prod_{i \in I^+ \sqcup I^-} (\det\phi_i)(-1)^{m''_i \lfloor b_i/2 \rfloor} \prod_{\substack{i \in I^+ \sqcup I^- \\ b_i \;\text{odd}}} \epsilon(\phi_i, \bpsi)^{m''_i} \prod_{j \in J} (\det\phi_j)(-1)^{b_j m''_j} \\
		= \prod_{\substack{i \in I^+ \\ b_i \;\text{even}}} (\det\phi_i)(-1)^{m''_i b_i/2} \prod_{\substack{i \in I^+ \sqcup I^- \\ b_i \;\text{odd}}} \epsilon(\phi_i, \bpsi)^{m''_i} \prod_{j \in J} (\det\phi_j)(-1)^{b_j m''_j},
	\end{multline*}
	the last equality follows since $m''_i$ is even for $i \in I^-$, whilst for $i \in I^+$ with $b_i$ odd, $\phi_i$ is symplectic and thus $\det \phi_i = \mathbf{1}$; the corresponding factor is $1$ in either case. For the remaining indices (namely $i \in I^+$, $b_i$ even), $\lfloor b_i/2 \rfloor = b_i/2$; finally, the factors indexed by $J$ are unchanged.
	
	Comparing with $\epsilon(\psi^{s = -1})$, it remains to show
	\[ \prod_{\substack{i \in I^+ \sqcup I^- \\ b_i \;\text{even}}} \epsilon(\phi_i, \bpsi)^{m''_i b_i} = \prod_{\substack{i \in I^+ \\ b_i \;\text{even}}} (\det \phi_i)(-1)^{m''_i b_i / 2}. \]
	If $i \in I^-$ and $b_i$ is even, then $\phi_i$ is symplectic, $\epsilon(\phi_i, \bpsi)^2 = 1$, thus $\epsilon(\phi_i, \bpsi)^{m''_i b_i} = 1$. The problem reduces to showing
	\[ \prod_{\substack{i \in I^+ \\ b_i \;\text{even}}} \epsilon(\phi_i, \bpsi)^{m''_i b_i} = \prod_{\substack{i \in I^+ \\ b_i \;\text{even}}} (\det \phi_i)(-1)^{m''_i b_i / 2}. \]
	We know that $\epsilon(\phi_i, \bpsi)^{b_i} = (\det \phi_i)(-1)^{b_i / 2}$ for each $i$ above. The proof is complete.
\end{proof}

For the following result, note that the natural inclusion $S_\psi \subset S_{\phi_\psi}$ induces a surjection $\EuScript{S}_\psi \to \EuScript{S}_{\phi_\psi}$ by \cite[p.38]{Ar89-Unip} or a direct check; thus $\EuScript{S}_{\phi_\psi}^\vee$ embeds into $\EuScript{S}_\psi^\vee$. We also use the formalism from \S\ref{sec:L-general} for L-packets.

The result below is a key step in proving Theorem \ref{prop:local-desiderata}; together with the theorem, it yields Proposition \ref{prop:L-within}. Condition (a) is the non-negativity asserted in that theorem, hence automatic once it is proved, but is assumed at this inductive stage. Conditions (a) and (b) are used at the Archimedean and anti-tempered auxiliary places in \S\S\ref{sec:proof-0}--\ref{sec:proof-1}, respectively, with (b) holding by Lemma \ref{prop:anti-tempered-bp}.

\begin{proposition}\label{prop:L-within-0}
	\index{Pi-phi-psi@$\Pi_{\phi_\psi}$}
	Given $\psi \in \Psi(\tilde{G})$, make one of the following two assumptions:
	\begin{enumerate}[(a)]
		\item $\pi_{\psi, \chi}$ is a linear combination of irreducible characters with coefficients in $\Z_{\geq 0}$ for all $\chi \in \EuScript{S}_\psi^\vee$;
		\item for each $\chi \in \EuScript{S}_\psi^\vee$ there exist $\tau_\chi \in \Pi_-(\tilde{G})$ and $c_\chi \in \CC^{\times}$ such that
		\begin{itemize}
			\item $\pi_{\psi, \chi} = c_\chi \Theta_{\tau_\chi}$, and
			\item $\chi \neq \eta \implies \tau_\chi \not\simeq \tau_\eta$.
		\end{itemize}
	\end{enumerate}
	
	Let $\pi \in \Pi_{\phi_\psi}$ be parameterized by $\chi_\pi \in \EuScript{S}_{\phi_\psi}^\vee$. Then, for each $\chi \in \EuScript{S}_\psi^\vee$,
	\begin{enumerate}[(i)]
		\item if $\chi \neq \chi_\pi$, then $\pi$ does not appear in the expansion of $\pi_{\psi, \chi}$ into irreducible characters;
		\item if $\chi = \chi_\pi$, then $\pi$ intervenes in the expansion of $\pi_{\psi, \chi}$ with coefficient $1$.
	\end{enumerate}
\end{proposition}
\begin{proof}
	Let $\mathfrak{a}^*_0 := X^*(T) \otimes \R$ where $T \subset G$ is the standard maximal torus, and let $\mathfrak{a}^{*, +}_0$ be the chamber in $\mathfrak{a}^*_0$ defined by $\check{\alpha} > 0$, where $\check{\alpha}$ ranges over $B$-simple coroots; its closure $\overline{\mathfrak{a}^{*, +}_0}$ is defined by $\check{\alpha} \geq 0$. Identify $\mathfrak{a}^*_0$ and $\R^n$ and use the standard inner product to define the Weyl-invariant norm $\|\cdot\|$ on $\mathfrak{a}^*_0$.
	
	Following the notation of \cite[\S 5]{Ar89-IOR1}, to $\pi \in \Pi_-(\tilde{G})$ is attached an element $\Lambda_\pi \in \overline{\mathfrak{a}^{*, +}_0}$ which measures the deviation from temperedness: it is the twisting parameter in the Langlands classification for $\pi$.
	
	On the other hand, for each $\phi \in \Phi(\tilde{G})$, define $\Lambda_\phi$ to be the $\lambda$ in the triplet $(M, \phi_{M, \mathrm{bdd}}, \lambda)$ associated with $\phi$, see \S\ref{sec:L-general}. For $\psi \in \Psi(\tilde{G})$ we put $\Lambda_\psi := \Lambda_{\phi_\psi}$.

	The same definitions apply to $G^!$ for all $\mathbf{G}^! \in \Endo_{\elli}(\tilde{G})$; we still equip the avatar of $\mathfrak{a}^*_0$ for $G^!$ with the standard inner product. They are naturally identified with $\mathfrak{a}^*_0$ as inner product spaces, so that $\mathfrak{a}^{*, +}_0$ is contained in the avatar for $G^!$.
	
	Here we are following the strategy in \cite[\S 7.4]{Ar13}. In \cite{Ar89-IOR1} and \cite[\S 2.2]{Ar13} one uses a partial order on $\mathfrak{a}^*_0$ instead of the norm, but it is less convenient for endoscopy. Also note that, in comparison with standard endoscopy treated in \textit{loc.\ cit.}, the roots/coroots of $G^!$ and $G$ differ by a dilation, but this is irrelevant here.
	
	Suppose that $\psi^! \in \Psi(G^!)$ (resp.\ $\phi^! \in \Phi(G^!)$) has image $\psi \in \Psi(\tilde{G})$ (resp.\ $\phi \in \Phi(\tilde{G})$), then
	\[ \| \Lambda_{\psi^!} \| = \| \Lambda_\psi \| \quad (\text{resp.}\; \| \Lambda_{\phi^!} \| = \| \Lambda_\phi \| ) \]
	under the identification above. This is readily seen by comparing the inducing data for $\psi^!$ and $\psi$ (resp.\ $\phi$ and $\phi^!$): the twisting parameters coincide up to Weyl group action, although the inducing Levi subgroup of $\tilde{G}$ may have semisimple rank larger than that of $G^!$. For example, suppose $n=2$ and $\mathbf{G}^!$ corresponds to $(n', n'') = (1, 1)$, then $\phi^! = (\phi', \phi'')$ with
	\[ \phi' = \phi'' = |\cdot|_F^{s} \oplus |\cdot|_F^{-s}, \quad s \in \R_{> 0} \]
	is induced from the minimal Levi of $G^!$ in its Langlands classification, but its image $\phi = 2|\cdot|_F^s \oplus 2|\cdot|_F^{-s}$ is induced from the Siegel Levi of $\tilde{G}$.
	
	Given $s \in S_{\psi, 2}/\text{conj}$, let $(\mathbf{G}^!, \psi^!) \mapsto (\psi, s)$ under the basic bijection (Proposition \ref{prop:basic-bijection} (ii)). Identify $s$ with a representative in $S_\psi$. We have
	\begin{equation}\label{eqn:L-within-aux0}
		\trans_{\mathbf{G}^!, \tilde{G}} \left( S\Theta^{G^!}_{\psi^!} \right) = \sum_{\chi \in \EuScript{S}_\psi^\vee} \chi(s_\psi s) \epsilon\left(\psi^{s = -1} \right) \pi_{\psi, \chi}.
	\end{equation}
	
	Use \cite[(2.2.12)]{Ar13} to express $S\Theta^{G^!}_{\psi^!}$ as $\sum_{\phi^! \in \Phi(G^!)} n(\psi^!, \phi^!) S\Theta^{G^!}_{\phi^!}$, with $S\Theta^{G^!}_{\phi^!}$ defined in \eqref{eqn:STheta-gen}. The left hand side of \eqref{eqn:L-within-aux0} becomes
	\[ \sum_{\phi^! \in \Phi(G^!)} n(\psi^!, \phi^!) \trans_{\mathbf{G}^!, \tilde{G}} \left( S\Theta^{G^!}_{\phi^!}\right). \]
	According to the proof of \cite[Proposition 7.4.1]{Ar13}:
	\begin{itemize}
		\item $n(\psi^!, \phi_{\psi^!}) = 1$;
		\item $n(\psi^!, \phi^!) \neq 0 \implies \|\Lambda_{\phi^!} \| \leq \|\Lambda_{\psi^!}\| \stackrel{\text{known}}{=} \|\Lambda_\psi \|$;
		\item assume $n(\psi^!, \phi^!) \neq 0$, then $\|\Lambda_{\phi^!}\| = \|\Lambda_{\psi^!}\| \iff \phi^! = \phi_{\psi^!}$.
	\end{itemize}
	
	For every $\phi^! \in \Phi(G^!)$, apply the basic bijection to obtain $(\mathbf{G}^!, \phi^!) \mapsto (\phi, s(\phi^!))$ where $\phi$ is the image of $\phi^!$. Proposition \ref{prop:endo-char-gen} yields
	\[ \trans_{\mathbf{G}^!, \tilde{G}} \left( S\Theta^{G^!}_{\phi^!}\right) = \sum_\rho \chi_\rho(s(\phi^!)) \epsilon\left( \phi^{s(\phi^!) = -1} \right) \Theta_\rho , \]
	where $\rho$ ranges over the standard modules attached to the members of $\Pi_\phi$, say $\rho$ is indexed by $\chi_\rho \in \EuScript{S}_\phi^\vee$.

	For every genuine standard module $\rho$ for $\tilde{G}$, denote its Langlands quotient by $\pi_\rho$ and put $\Lambda_\rho := \Lambda_{\pi_\rho}$. Expand the standard character $\Theta_\rho$ into $\sum_{\pi \in \Pi_-(\tilde{G})} m(\rho, \pi) \Theta_\pi$ with
	\begin{itemize}
		\item $m(\rho, \pi_\rho) = 1$;
		\item $m(\rho, \pi) \neq 0 \implies \|\Lambda_\pi\| \leq \|\Lambda_\rho \|$;
		\item assume $m(\rho, \pi) \neq 0$, then $\| \Lambda_\pi \| = \| \Lambda_\rho \| \iff \pi = \pi_\rho$.
	\end{itemize}
	In other words $\pi_\rho$ is ``the most non-tempered'' element in the composition series of $\rho$. These are basic facts about standard modules, see, e.g., \cite[\S 7.4]{Ar13} or \cite[\S 5]{Ar89-IOR1}; the case for coverings is identical.

	The left hand side of \eqref{eqn:L-within-aux0} is thus equal to
	\begin{multline*}
		\sum_{\pi \in \Pi_-(\tilde{G})} \; \sum_{\substack{\phi^! \in \Phi(G^!) \\ \phi := \text{its image}}} \; \sum_{\substack{\rho: \text{standard modules} \\ \text{attached to elements of $\Pi_\phi$}}} \\
		n(\psi^!, \phi^!) \chi_\rho(s(\phi^!)) m(\rho, \pi) \epsilon\left( \phi^{s(\phi^!) = -1} \right) \Theta_\pi .
	\end{multline*}
	
	Given $\pi \in \Pi_-(\tilde{G})$, in order for $\Theta_\pi$ to appear in the sum above, there must exist $\phi^!$ and $\rho$ such that
	\[ \| \Lambda_\psi \| \stackrel{\text{known}}{=} \| \Lambda_{\psi^!} \| \geq \| \Lambda_{\phi^!} \| \stackrel{\text{known}}{=} \| \Lambda_\phi \| \stackrel{\text{easy}}{=} \| \Lambda_\rho \| \geq \| \Lambda_\pi \| . \]
	
	Assume henceforth that $\pi \in \Pi_{\phi_\psi}$. Then $\Lambda_\pi = \Lambda_{\phi_\psi} =: \Lambda_\psi$ and equality holds everywhere in the formula above. Hence the only contribution to $\pi$ is the pair $(\phi^!, \rho)$ satisfying
	\[ \phi^! = \phi_{\psi^!}, \quad \pi = \pi_\rho. \]
	Moreover, $(\mathbf{G}^!, \psi^!) \mapsto (\psi, s)$ implies $(\mathbf{G}^!, \phi_{\psi^!}) \mapsto (\phi_\psi, s)$, so in this case $\phi = \phi_\psi$, and $s(\phi^!) = s$ up to $S_{\phi_\psi}$-conjugacy (recall: $S_\psi \subset S_{\phi_\psi} = S_\phi$).
	
	The coefficient of $\Theta_\pi$ in the left hand side of \eqref{eqn:L-within-aux0} is therefore
	\[ \chi_\rho(s(\phi^!)) \epsilon\left( \phi^{s(\phi^!) = -1} \right) = \chi_\pi(s) \epsilon\left( \phi^{s = -1} \right). \]
	
	A comparison with the right hand side of \eqref{eqn:L-within-aux0} gives
	\begin{equation*}
		\chi_\pi(s) \epsilon\left( \phi_\psi^{s = -1} \right) = \sum_{\chi \in \EuScript{S}_\psi^\vee} \chi(s_\psi s) \epsilon\left(\psi^{s = -1}\right) \left( \text{the coefficient of $\Theta_\pi$ in $\pi_{\psi, \chi}$} \right).
	\end{equation*}
	
	Lemma \ref{prop:s-phi-psi} implies $\chi_\pi(s) = \chi_\pi(s_\psi s)$. Combined with Lemma \ref{prop:epsilon-phi-psi}, we conclude
	\begin{equation}\label{eqn:L-within-aux1}
		\chi_\pi(s_\psi s) = \sum_{\chi \in \EuScript{S}_\psi^\vee} \chi(s_\psi s) \left( \text{the coefficient of $\Theta_\pi$ in $\pi_{\psi, \chi}$} \right).
	\end{equation}
	
	Now one can vary $s$. Putting $s = s_\psi^{-1}$, the sum of coefficients of $\Theta_\pi$ in various $\pi_{\psi, \chi}$ is seen to be $1$. Under the assumption (a), there exists a unique $\chi \in \EuScript{S}_\psi^\vee$ for which the coefficient of $\Theta_\pi$ is nonzero, and in that case the coefficient is $1$; rewriting \eqref{eqn:L-within-aux1} as $\chi_\pi(s_\psi s) = \chi(s_\psi s)$, we obtain $\chi = \chi_\pi$, whence (i) and (ii).
	
	Under (b), there exists a unique $\chi$ such that $\Theta_\pi$ appears in $\pi_{\psi, \chi}$, necessarily with coefficient $c_\chi$. Hence \eqref{eqn:L-within-aux1} reduces to
	\[ \chi_\pi(s_\psi s) = c_\chi \chi(s_\psi s). \]
	Put $s = s_\psi^{-1}$ to see $c_\chi = 1$, and then $\chi = \chi_\pi$. The proof of (i) and (ii) is complete.
\end{proof}

Below is the analogue of \cite[Proposition 7.4.1]{Ar13}.

\begin{proposition}\label{prop:L-within}
	Grant the validity of Theorem \ref{prop:local-desiderata} for $\psi \in \Psi(\tilde{G})$. With the formalism of Arthur packets in \S\ref{sec:A-packets}, there is a unique injection from $\Pi_{\phi_\psi}$ into $\Pi_\psi^{\mathrm{mf}} \subset \Pi_\psi$ (Definition \ref{def:multiplicity-free}) making the diagram
	\[\begin{tikzcd}
		& \Pi_-(\tilde{G}) \\
		\Pi_{\phi_\psi} \arrow[hookrightarrow, r] \arrow[d, "{1:1}"'] \arrow[ru, "\text{incl.}"] & \Pi_\psi^{\mathrm{mf}} \arrow[d] \arrow[u] \\
		\EuScript{S}_{\phi_\psi}^\vee \arrow[hookrightarrow, r] & \EuScript{S}_\psi^\vee
	\end{tikzcd}\]
	commute. In particular, $\Pi_{\phi_\psi} \subset \Pi_{\mathrm{unit}, -}(\tilde{G})$.
\end{proposition}
\begin{proof}
	The assumption (a) in Proposition \ref{prop:L-within-0} holds for $\psi$ as implied by Theorem \ref{prop:local-desiderata}. For all $\pi \in \Pi_{\phi_\psi}$, denote the corresponding element of $\EuScript{S}_{\phi_\psi}^\vee$ by $\chi_\pi$. In the notation of Definition \ref{def:A-packet}, the desired map is simply $\pi \mapsto (\chi_\pi, \pi, 1)$. All the required properties are immediate from Proposition \ref{prop:L-within-0}.
\end{proof}

\section{Anti-tempered parameters}\label{sec:anti-tempered}
In this section, $F$ is a non-Archimedean local field of characteristic zero. We fix an additive character $\bpsi$ of $F$.

\subsection{Two involutions}
Let $R$ be any connected reductive $F$-group. Let $R(F)\dcate{Mod}$ be the Abelian category of smooth representations of $R(F)$. The \emph{Aubert--Zelevinsky involution} is an involutive endo-functor $\pi \mapsto \hat{\pi}$ (see \cite[\S A.2]{KMSW}).
\index{Aubert--Zelevinsky involution}
\index{pi-hat@$\hat{\pi}$}

Moreover, $\pi \mapsto \hat{\pi}$ leaves each Bernstein block of $R(F)\dcate{Mod}$ invariant. We refer to \cite[\S A.2]{KMSW} for an overview of this functor.

Let $R(F)\dcate{Mod}_{\mathrm{f}}$ be the Abelian subcategory of $R(F)\dcate{Mod}$ consisting of objects of finite length. Its Grothendieck group is denoted by $\mathrm{Groth}(R)$. Every $\pi$ in $R(F)\dcate{Mod}_{\mathrm{f}}$ has image $[\pi]$ in $\mathrm{Groth}(R)$. Fix a minimal parabolic subgroup $P_0$ of $R$.

\begin{definition}
	\index{DR@$\mathbf{D}^R$}
	\index{r-P@$r_P$}
	For every standard parabolic subgroup $P$ of $R$, take a Levi decomposition $P = MU$ and let $A_P = A_M$ be the maximal central split torus in $M$; denote the parabolic induction (resp.\ Jacquet functor) along $P$ on the level of $\mathrm{Groth}(\cdot)$ by $i_P$ (resp.\ $r_P$). Following \cite{Aub95}, define the endomorphism of $\mathrm{Groth}(R)$
	\[ \mathbf{D}^R := \sum_{P \supset P_0} (-1)^{\dim A_{P_0} / A_P} i_P r_P. \]
\end{definition}

For every irreducible smooth representation $\pi$ of $R(F)$, denote its cuspidal support by $[M_\pi, \pi']$ where $M_\pi$ is a standard Levi subgroup of $R$ and $\pi'$ is a supercuspidal representation of $M_\pi(F)$. Put
\begin{equation*}
	\index{beta-pi@$\beta(\pi)$}
	\beta(\pi) = \beta^R(\pi) := (-1)^{\dim A_{P_0}/A_{M_\pi}}.
\end{equation*}
It is well-known (see \cite{Aub95}) that
\begin{equation}\label{eqn:D-sign}
	\beta(\pi) \mathbf{D}^R [\pi] = [\hat{\pi}], \quad (\mathbf{D}^R)^2 = \identity .
\end{equation}

We shall also view $\mathbf{D}^R$ as an endomorphism of the space of virtual characters. As such, it commutes with parabolic induction; see, e.g., \cite[(A.2.3)]{KMSW}. It also fixes every irreducible supercuspidal representation.

All these carry over verbatim to genuine representations of coverings, such as the metaplectic group $\tilde{G} = \Mp(W)$, or groups of metaplectic type in general, and the notation $\mathbf{D}^{\tilde{G}}$ will be used. Indeed, the parabolic induction, Jacquet functor and Bernstein decomposition are available in this setting, thus the Aubert--Zelevinsky involution and its functorial version are defined by exactly the same recipe, with the same basic properties.

Secondly, we introduce an involution on Arthur parameters.

\begin{definition}\label{def:anti-tempered-parameter}
	\index{psi-hat@$\hat{\psi}$}
	\index{Arthur parameter!anti-tempered}
	For every Arthur parameter $\psi$ of $R$, viewed as a homomorphism from $\Weil{F} \times \SL(2, \CC) \times \SL(2, \CC)$ to the L-group of $R$, let $\hat{\psi}$ be the Arthur parameter obtained from $\psi$ by switching two $\SL(2, \CC)$ factors; we say $\hat{\psi}$ is dual to $\psi$. This yields an involution $\psi \mapsto \hat{\psi}$ on the set of Arthur parameters of $R$.
	
	An Arthur parameter $\psi \in \Psi(R)$ is said to be \emph{anti-tempered} if $\hat{\psi} \in \Phi_{\mathrm{bdd}}(R)$.
	
	The same definitions pertain to Arthur parameters for $\tilde{G} = \Mp(W)$, or groups of metaplectic type in general.
\end{definition}

It is evident that
\begin{equation}\label{eqn:S-duality}
	S_\psi = S_{\hat{\psi}}, \quad \EuScript{S}_\psi = \EuScript{S}_{\hat{\psi}}.
\end{equation}
Also note that $\psi$ is of good (resp.\ quasi-good) parity if and only if $\hat{\psi}$ is. Indeed, by writing every simple summand in $\psi$ as $\rho \boxtimes r(a) \boxtimes r(b)$, swapping $a, b$ does not alter its type, and the value of $\det(\rho \boxtimes r(a))(-1)^{b} = \det(\rho)(-1)^{ab}$ is unaltered.

We are interested in the involutions for $\Pi_-(\tilde{G})$ and $\Psi(\tilde{G})$ so obtained, but some corresponding results for odd orthogonal groups are needed.

\subsection{Results for odd orthogonal groups}\label{sec:anti-tempered-SO}
Take $(n', n'') \in (\Z_{\geq 0})^2$ and let
\[ G' := \SO(2n' + 1), \quad G'' := \SO(2n'' + 1), \quad G^! := G' \times G'' . \]

Let $\psi^! = (\psi', \psi'') \in \Psi(G^!)$ be anti-tempered (Definition \ref{def:anti-tempered-parameter}). We need the following facts:
\begin{itemize}
	\item $\mathbf{D}^{G^!} S\Theta_{\psi^!}^{G^!}$ is still a stable distribution on $G^!(F)$;
	\item there exists a unique $\beta^{G^!}(\psi^!) \in \{\pm 1\}$, characterized by
	\begin{equation}\label{eqn:beta-characterization}
		\mathbf{D}^{G^!} S\Theta_{\psi^!}^{G^!} = \beta^{G^!}(\psi^!) S\Theta^{G^!}_{\widehat{\psi^!}};
	\end{equation}
	\item one has $\beta^{G^!}(\psi^!) = \beta^{G^!}(\widehat{\psi^!})$ and $\beta^{G^!}(\psi^!) = \beta^{G'}(\psi') \times \beta^{G''}(\psi'')$.
\end{itemize}
\index{beta-psi@$\beta(\psi)$}

Our notation $\beta^{G^!}(\psi^!)$ follows Arthur \cite[Chapter 7]{Ar13}. An overview of these properties can also be found in \cite[\S 1]{LLS24}. In the former reference, a simple formula for $\beta^{G^!}(\psi^!)$ that depends only on $\phi_{\psi^!}$ is given, but we make no direct use of it.

Now consider the special case $H = \SO(2n + 1)$. The equivalence classes of elliptic endoscopic data $\mathbf{H}^\dagger$ of $H$ are in bijection with conjugacy classes of elements $s \in \check{H}$ satisfying $s^2 = 1$, taken up to translation by $Z_{\check{H}}$. They are parameterized by unordered pairs $[n', n'']$ of non-negative integers such that $n = n ' + n''$, namely $2n'$ (resp.\ $2n''$) is the multiplicity of the eigenvalue $+1$ (resp.\ $-1$) of $s$. The underlying endoscopic group is $H^\dagger = \SO(2n' + 1) \times \SO(2n'' + 1)$.

Let $\psi^\dagger \in \Psi(H^\dagger)$. The pair $(\mathbf{H}^\dagger, \psi^\dagger)$ gives rise to $(\psi, s)$ as in \S\ref{sec:basic-bijection}, namely $\psi \in \Psi(H)$ is the image of $\psi^\dagger$, and $s$ is the conjugacy class in $\check{H}$ up to $Z_{\check{H}}$ described above. Observe that $\psi$ is anti-tempered if and only if $\psi^\dagger$ is.

Assume $\psi \in \Psi(H)$ is anti-tempered, say $\psi = \hat{\phi}$ where $\phi \in \Phi_{\mathrm{bdd}}(H)$. To $\psi$ is attached $\overline{\EuScript{S}}_\psi = \overline{\EuScript{S}}_\phi$, the quotient of the component group $\EuScript{S}_\psi = \EuScript{S}_\phi$ of $S_\psi = S_\phi$ by the image of $Z_{\check{H}}$.
\index{S-curly-bar-psi@$\overline{\EuScript{S}}_\psi$}

\begin{proposition}[Liu--Lo--Shahidi]\label{prop:beta-variance}
	\index{mu-psi@$\mu_\psi$}
	Let $\psi = \hat{\phi} \in \Psi(H)$ be anti-tempered. There exists a unique $\mu_\psi \in \overline{\EuScript{S}}_\psi^\vee$ characterized by
	\[ \beta^{H^\dagger}(\psi^\dagger) = \beta^H(\psi) \mu_\psi(s) \]
	for all $(\mathbf{H}^\dagger, \psi^\dagger)$ that maps to $(\psi, s)$. Furthermore, $\mu_\psi(s_\psi) = 1$.
\end{proposition}
\begin{proof}
	See \cite[Lemma 5.7]{LLS24} and the discussions that follow it. For the uniqueness of $\mu_\psi$, note that the classes $s$ so obtained exhaust $\overline{\EuScript{S}}_\psi$.
\end{proof}

\index{epsilon-M-MW-psi@$\epsilon^{\mathrm{M}/\mathrm{MW}}_\psi$}
In \textit{loc.\ cit.}, the character $\mu_\psi$ is shown to be Xu's character $\epsilon^{\mathrm{M}/\mathrm{MW}}_\psi$ (relative to suitably chosen auxiliary data) that quantifies the difference between Arthur's parametrization of the packet $\Pi^H_\psi$ and Moeglin's; see \cite{Xu17} or \cite[\S 1]{Xu21}. Xu also gave explicit algorithms to determine $\epsilon^{\mathrm{M}/\mathrm{MW}}_\psi$; the necessary details will be recalled near the end of \S\ref{sec:proof-Mp4}.

Note that the Whittaker datum for $H$ is unique up to conjugacy. The local Langlands correspondence for $H$ is known from Arthur's work \cite[Theorem 1.5.1]{Ar13}: it attaches to $\phi$ the L-packet $\Pi_\phi^H$, together with the Whittaker-normalized bijection:
\begin{align*}
	\index{sigma-phi-chi@$\sigma_{\phi, \chi}, \sigma_{\psi, \chi}$}
	\index{Pi-phi-H@$\Pi_\phi^H$}
	\overline{\EuScript{S}}_\phi^\vee & \xleftrightarrow{1:1} \Pi_\phi^H \\
	\chi & \mapsto \sigma_{\phi, \chi}.
\end{align*}

The Arthur packet $\Pi_\psi^H$ attached to $\psi$, or the characters $\sigma_{\psi, \chi}$ indexed by $\chi \in \overline{\EuScript{S}}_\psi^\vee = \overline{\EuScript{S}}_\phi^\vee$ is available, mainly from \textit{loc.\ cit.} To clarify the role of $\mu_\psi$, we record the following result which completely describes the structure of $\Pi_\psi^H$ in terms of $\Pi_\phi^H$ and Aubert--Zelevinsky involution, although it will not be used in the sequel.

\begin{theorem}[Liu--Lo--Shahidi]\label{prop:LLS}
	Let $\psi = \hat{\phi} \in \Psi(H)$ be anti-tempered. The $\mu_\psi \in \overline{\EuScript{S}}_\psi^\vee$ in Proposition \ref{prop:beta-variance} is characterized by the property that
	\[ \left( \sigma_{\psi, \chi} \right)^\wedge = \sigma_{\phi, \chi \mu_\psi} \]
	for all $\chi \in \overline{\EuScript{S}}_\psi^\vee = \overline{\EuScript{S}}_\phi^\vee$. On the left hand side, we identify $\sigma_{\psi, \chi}$ with a direct sum of irreducible representations and apply $(\cdots)^\wedge$.
\end{theorem}
\begin{proof}
	See \cite[Theorem 5.9]{LLS24} and the item (2) after \cite[Theorem 1.3]{LLS24}. Proposition \ref{prop:beta-variance} played a key role in their arguments.
\end{proof}

The existence of such a character $\mu_\psi$ satisfying $\left( \sigma_{\psi, \chi} \right)^\wedge = \sigma_{\phi, \chi \mu_\psi}$ is postulated in \cite[Conjecture 1.4]{Hi04}, for all quasisplit $H$ and all $\psi \in \Psi(H)$. It was asserted in \cite[\S 7.1]{Ar13} that $\mu_\psi = \mathbf{1}$, which is not quite correct: it is already false for $\SL(2)$, and there is also a counterexample for $\SO(5)$ with $\psi = 2(\zeta \boxtimes r(1) \boxtimes r(2))$, where $\zeta$ is a quadratic character. For a different proof of the existence of $\mu_\psi$, see \cite[\S 8.4.1]{Rel18} by C.\ M{\oe}glin and D.\ Renard.

A more general discussion on $\sigma_{\psi, \chi}^\wedge$ can be found in \cite[Section 6]{At22}. We now proceed to state an analogue of Theorem \ref{prop:LLS} for metaplectic groups.

\subsection{Anti-tempered packets: desiderata}
Consider the metaplectic group $\tilde{G} = \Mp(W)$. We need the following fact that transfer commutes with Aubert--Zelevinsky involution, à la Hiraga \cite{Hi04}.

\begin{theorem}[F.\ Chen \cite{Chen26}]\label{prop:trans-D}
	Let $\mathbf{G}^! \in \Endo_{\elli}(\tilde{G})$, then $\mathbf{D}^{\tilde{G}} \trans_{\mathbf{G}^!, \tilde{G}} = \trans_{\mathbf{G}^!, \tilde{G}} \mathbf{D}^{G^!}$ as linear maps from $SD_{\mathrm{spec}}(G^!) \otimes \mes(G^!)^\vee$ to $D_{\mathrm{spec}, -}(\tilde{G}) \otimes \mes(G)^\vee$.
\end{theorem}

\begin{definition}\label{def:mu-tilde}
	\index{mu-tilde@$\tilde{\mu}_{\psi}$}
	Let $\psi = \hat{\phi} \in \Psi(\tilde{G})$ be anti-tempered. Set
	\begin{equation*}
		\tilde{\mu}_\psi := \mu_\psi \nu_\psi \nu_\phi \; \in \EuScript{S}_\psi^\vee = \EuScript{S}_\phi^\vee .
	\end{equation*}
\end{definition}

Let $H := \SO(2n + 1)$ and denote $\beta(\psi) := \beta^H(\psi)$, for all anti-tempered $\psi \in \Psi(\tilde{G}) = \Psi(H)$.
\index{beta-psi}

\begin{lemma}\label{prop:anti-tempered-bp}
	Let $\psi = \hat{\phi} \in \Psi(\tilde{G})$ be anti-tempered. Assume $\psi$ is of quasi-good parity; equivalently, $\phi$ is of quasi-good parity. Then
	\[ \pi_{\psi, \chi} = \beta(\psi) \beta^{\tilde{G}}(\pi_{\phi, \chi\tilde{\mu}_\psi}) \chi(s_\psi) \widehat{\pi_{\phi, \chi\tilde{\mu}_\psi}}, \quad \chi \in \EuScript{S}_\phi^\vee. \]
	Moreover, the statement of Theorem \ref{prop:local-desiderata} for $\psi$ (except unitarity) is equivalent to
	\begin{equation}\label{eqn:anti-tempered-bp}
		\beta(\psi) \beta^{\tilde{G}}(\pi_{\phi, \chi\tilde{\mu}_\psi}) = \chi(s_\psi), \quad \chi \in \EuScript{S}_\phi^\vee,
	\end{equation}
	which is in turn equivalent to $\pi_{\psi, \chi} = \widehat{\pi_{\phi, \chi\tilde{\mu}_\psi}}$ for all $\chi$.
\end{lemma}
\begin{proof}
	The discussion after \eqref{eqn:S-duality} justifies the equivalence between quasi-good parity conditions.
	We shall write $\tilde{\mu} := \tilde{\mu}_\psi$ for typographic concerns. Set $c_\psi(\chi) := \beta(\psi) \beta^{\tilde{G}}(\pi_{\phi, \chi\tilde{\mu}}) \chi(s_\psi) \in \{ \pm 1\}$. Given a conjugacy class in $S_{\psi, 2} = S_{\phi, 2}$, say represented by $s$, take the $(\mathbf{G}^!, \psi^!)$ that maps to $(\psi, s)$ (recall Proposition \ref{prop:basic-bijection} (ii)). Equivalently, $(\mathbf{G}^!, \phi^!) \mapsto (\phi, s)$ where $\phi^! = (\psi^!)^\wedge$. In view of Proposition \ref{prop:nu-good-parity} and the quasi-good parity condition,
	\[ T_{\psi, s} = \nu_\psi(s) \trans_{\mathbf{G}^!, \tilde{G}} \left( S\Theta^{G^!}_{\psi^!} \right) \]
	since $\epsilon(\psi^{s = -1}) = \nu_\psi(s)$. In turn, \eqref{eqn:beta-characterization} and Theorem \ref{prop:trans-D} imply that
	\begin{align*}
		T_{\psi, s} & = \nu_\psi(s) \beta^{G^!}(\psi^!) \trans_{\mathbf{G}^!, \tilde{G}} \mathbf{D}^{G^!} \left( S\Theta^{G^!}_{\phi^!} \right) \\
		& = \nu_\psi(s) \beta^{G^!}(\psi^!) \mathbf{D}^{\tilde{G}} \trans_{\mathbf{G}^!, \tilde{G}} \left( S\Theta^{G^!}_{\phi^!} \right).
	\end{align*}
	By \eqref{eqn:D-sign}, Theorem \ref{prop:Luo} (the tempered character relation) and the quasi-good parity condition, the right hand side equals
	\[ \nu_\psi(s) \beta^{G^!}(\psi^!) \sum_\chi \chi(s) \nu_\phi(s) \beta^{\tilde{G}}(\pi_{\phi, \chi}) \widehat{\pi_{\phi, \chi}}. \]
	Then, by Proposition \ref{prop:beta-variance} applied to $H^\dagger = G^!$, the above becomes
	\[ \beta(\psi) \sum_\chi \underbracket{\nu_\psi(s) \nu_\phi(s) \mu_\psi(s)}_{= \tilde{\mu}(s)} \chi(s) \beta^{\tilde{G}}(\pi_{\phi, \chi}) \widehat{\pi_{\phi, \chi}}. \]
	Substituting $\chi$ with $\chi\tilde{\mu}$ in the sum and noting that $\chi(s_\psi)^2 = 1$, we arrive at
	\begin{align*}
		T_{\psi, s} & = \sum_\chi \chi(s_\psi s) \chi(s_\psi) \beta(\psi) \beta^{\tilde{G}}(\pi_{\phi, \chi\tilde{\mu}}) \widehat{\pi_{\phi, \chi\tilde{\mu}}} \\
		& = \sum_\chi \chi(s_\psi s) c_\psi(\chi) \widehat{\pi_{\phi, \chi\tilde{\mu}}}.
	\end{align*}
	
	The characterization of $\pi_{\psi, \chi}$ in Definition \ref{def:pi-psi-chi} implies
	\[ \pi_{\psi, \chi} = c_\psi(\chi) \widehat{\pi_{\phi, \chi\tilde{\mu}}}, \]
	for all $\chi$, and the desired equivalence follows at once, since $\widehat{\pi_{\phi, \chi \tilde{\mu}}}$ is again an irreducible genuine representation and $\chi(s_\psi)^2 = 1$.
\end{proof}

We now deal with anti-tempered parameters in general.

\begin{lemma}\label{prop:mu-tilde-induction}
	Let $\psi = \hat{\phi} \in \Psi(\tilde{G})$ be anti-tempered. Proposition \ref{prop:pi-psi-gp} expresses $\psi$ as the image of $\psi_M = (\psi_0, \psi_{\GL}) \in \Psi(\tilde{M})$ with $\psi_0 \in \Psi_{\mathrm{gp}}(\Mp(W^\flat))$. Then $\tilde{\mu}_\psi$ matches $\tilde{\mu}_{\psi_0}$ under the isomorphism $\EuScript{S}_{\psi_0} \simeq \EuScript{S}_\psi$. 
\end{lemma}
\begin{proof}
	The same reduction to good parity applies to $\phi$, and $\widehat{\psi_0} = \phi_0$. It is clear that $\nu_\psi$ and $\nu_\phi$ depend only on $\psi_0$ and $\phi_0$; they match $\nu_{\psi_0}$ and $\nu_{\phi_0}$, respectively.
	
	It remains to show that $\mu_\psi$ matches $\mu_{\psi_0}$. This can be seen by the explicit description of $\mu_\psi = \epsilon_\psi^{\mathrm{M}/\mathrm{MW}}$, e.g.\ in \cite[\S 1]{Xu21} or \cite[\S 6]{LLS24}.
\end{proof}

\begin{corollary}\label{prop:anti-tempered-chi}
	Granting the validity of Theorem \ref{prop:local-desiderata} for anti-tempered parameters except unitarity, consider an anti-tempered $\psi = \hat{\phi} \in \Psi(\tilde{G})$. Then $\pi_{\psi, \chi} = \widehat{\pi_{\phi, \chi\tilde{\mu}_\psi}}$ in $\mathrm{Groth}(\tilde{G})$ for all $\chi \in \EuScript{S}_\psi^\vee$.
\end{corollary}
\begin{proof}
	Aubert--Zelevinsky involution commutes with parabolic induction on the level of virtual characters or $\mathrm{Groth}(\tilde{G})$: see, e.g., \cite[(A.2.3)]{KMSW}. Now we can express $\pi_{\psi, \chi}$ (resp.\ $\pi_{\phi, \chi}$) as parabolic induction from $(\psi_0, \psi_{\GL})$ (resp.\ $(\phi_0, \phi_{\GL})$) via Proposition \ref{prop:pi-psi-gp}, then use Lemmas \ref{prop:anti-tempered-bp} and \ref{prop:mu-tilde-induction}.
\end{proof}

Before proving Theorem \ref{prop:local-desiderata}, we will have to establish \eqref{eqn:anti-tempered-bp} for certain $\psi$. This is the topic of \S\ref{sec:anti-tempered-special}.

\subsection{Proof for a special case}\label{sec:anti-tempered-special}
Perhaps \eqref{eqn:anti-tempered-bp} is accessible by the method of Jacquet modules, cf.\ \cite[\S 8.3.4]{Rel18}. Instead, we will take a shortcut through $\Theta$-correspondence to prove a special case thereof (Proposition \ref{prop:anti-tempered-star}).

Given $\psi \in \Psi(\tilde{G})$, decompose it into
\[ \psi = \bigoplus_{i \in I} m_i \left( \rho_i \boxtimes r(a_i) \boxtimes r(b_i) \right) \]
where $m_i \in \Z_{\geq 1}$ and $(\rho_i, a_i, b_i)$ are distinct. View each $\rho_i$ as an irreducible smooth representation of $\GL(d_i, F)$ through local Langlands correspondence, where $d_i := \dim \rho_i$. We say $\psi$ contains no unramified characters if none of these $\rho_i$'s are unramified characters.

\begin{definition}\label{def:Psi-star}
	\index{Psi-G-tilde-star@$\Psi(\tilde{G})^{\star}$}
	Let $\Psi(\tilde{G})^{\star}$ be the subset of $\Psi_{\mathrm{gp}}(\tilde{G})$ consisting of anti-tempered $\psi$ that contains no unramified characters.
\end{definition}

Recall that $\tilde{\mu}_\psi$ (resp.\ $\mu_\psi$) is defined relative to $\tilde{G}$ (resp.\ $\SO(2n+1)$) in Definition \ref{def:mu-tilde} (resp.\ Proposition \ref{prop:beta-variance}).

\begin{lemma}\label{prop:tilde-mu-star}
	For all $\psi = \hat{\phi} \in \Psi(\tilde{G})^{\star}$, we have $\tilde{\mu}_\psi = \mu_\psi$.
\end{lemma}
\begin{proof}
	We have to show that $\nu_\psi = \nu_\phi$ as characters of $\EuScript{S}_\psi = \EuScript{S}_\phi$. Decompose $\psi$ into $\bigoplus_i m_i (\rho_i \boxtimes r(1) \boxtimes r(b_i))$. It suffices to show that for each $i$,
	\[ \epsilon\left( \rho_i \boxtimes r(b_i), \bpsi \right) = \epsilon\left( \rho_i \boxtimes r(1), \bpsi \right)^{b_i}. \]
	The above follows from \eqref{eqn:epsilon-SL} since $\rho_i$ cannot be an unramified character.
\end{proof}

Let $\psi = \hat{\phi} \in \Psi(\tilde{G})^{\star}$. The local Langlands correspondence for $\tilde{G}$ of Gan--Savin \cite{GS1} gives a bijection between $\Pi_\phi$ and the Vogan L-packet $\Pi^H_\phi$ for $H := \SO(2n+1)$, mapping $\pi_{\phi, \chi}$ to $\sigma_{\phi, \chi}$ for each $\chi \in \EuScript{S}_\phi^\vee$; the group $\sigma_{\phi, \chi}$ lives on will be specified below. The bijection is based on $\Theta$-correspondence for dual pairs $(\Or(V^{\pm}), \Sp(W))$ where $V^{\pm}$ is the quadratic $F$-vector space of dimension $2n+1$, discriminant $1$ and Hasse invariant $\pm 1$. In particular $H \simeq \SO(V^+)$.

Let $z$ denote the image of $-1 \in Z_{\tilde{G}^\vee}$ in $\EuScript{S}_\phi$. The general properties of Vogan L-packets imply that $\sigma_{\phi, \chi} \in \Pi^H_\phi$ lives on $\SO(V^{\pm})$ if and only if $\chi(z) = \pm 1$.

\begin{lemma}\label{prop:Theta-cuspidal-support}
	Suppose that $\psi = \hat{\phi} \in \Psi(\tilde{G})^{\star}$ and $\pi \in \Pi_\phi$ matches $\sigma \in \Pi^H_\phi$ via Gan--Savin. Express the cuspidal supports of $\pi$ and $\sigma$ as
	\[ [\rho_1, \ldots, \rho_r; \pi_0], \quad [\xi_1, \ldots, \xi_s; \sigma_0] \]
	respectively, where $\rho_i$ and $\xi_j$ are irreducible supercuspidal representations of certain general linear groups, and $\pi_0$ (resp.\ $\sigma_0$) is an irreducible supercuspidal representation of a smaller metaplectic group (resp.\ odd $\SO$ of the same Hasse invariant). Then $r = s$ and the sequences $(\rho_1, \ldots)$, $(\xi_1, \ldots)$ coincide up to Weyl group actions.
\end{lemma}
\begin{proof}
	The assertion will follow from Kudla's result \cite[Theorem 2.5]{Ku86} on the behavior of cuspidal supports under $\Theta$-correspondence, provided that none of $\rho_i$, $\xi_j$ belong to $|\det|_F^{\Z/2}$. Note that since we are working with eightfold coverings, the quadratic characters $\chi_V$ in \textit{loc.\ cit.}\ disappear: since $V^{\pm}$ have discriminant $1$, they are absorbed into the splitting of the eightfold covering over the $\GL$ factors (see \cite[Proposition 2.14 and \S 5.4]{Li11}).
	
	Given $(n', n'') \in \Z_{\geq 0}^2$ with $n' + n'' = n$, we obtain elliptic endoscopic data $\mathbf{G}^!$ and $\mathbf{H}^\dagger$ of $\tilde{G}$ and $H$, respectively. They have the same endoscopic group $\SO(2n' + 1) \times \SO(2n''+1)$. Consider Jacquet modules on the level of characters of $G^! = H^\dagger$, say along a Levi subgroup $L$, denoted by $r_L$.
	
	We claim that for every $(n', n'')$ and $(\phi', \phi'') \in \Phi_{\mathrm{bdd}}(G^!) = \Phi_{\mathrm{bdd}}(H^\dagger)$ that maps to $\phi$,
	\begin{gather}\label{eqn:Theta-cuspidal-support-aux0}
		\begin{minipage}[c]{0.8\textwidth}
			the expansion of $r_L \left( S\Theta^{H'}_{\phi'} \boxtimes S\Theta^{H''}_{\phi''}\right)$ into irreducibles contains no $|\det|_F^{\Z/2}$ on $\GL$ factors, where $H'$ and $H''$ are the two factors of $G^!$.
		\end{minipage}
	\end{gather}
	
	It will follow from \eqref{eqn:Theta-cuspidal-support-aux0} that the Jacquet modules of $\pi$ and $\sigma$ contain no $|\det|_F^{\Z/2}$ on $\GL$ factors, as desired. Indeed:
	\begin{itemize}
		\item $\Theta_\pi$ and $\Theta_\sigma$ are linear combinations of transfers of various $S\Theta^{H'}_{\phi'} \boxtimes S\Theta^{H''}_{\phi''}$ (as in \eqref{eqn:Theta-cuspidal-support-aux0}), and
		\item there is a formula of Bernstein--Zelevinsky type for $r_{M^H} \trans_{\mathbf{H}^\dagger, H}$ (resp.\ $r_{\tilde{M}} \trans_{\mathbf{G}^!, \tilde{G}}$) where $M^H$ (resp.\ $\tilde{M}$) is a Levi subgroup of $H$ (resp.\ $\tilde{G}$); see \cite[Theorem 5.6]{Hi04} for the case of $H$, and the case of $\tilde{G}$ only involves some extra signs, see \cite{Chen26}.
	\end{itemize}

	To justify \eqref{eqn:Theta-cuspidal-support-aux0}, we may assume $(n', n'') = (n, 0)$ by treating the factors $H'$ and $H''$ separately, so $\phi = \phi^!$. Let us show that, assuming $\phi = \bigoplus_{i \in I} \rho_i \boxtimes r(a_i) \in \Phi_{2, \mathrm{bdd}}(H)$,
	\begin{gather}\label{eqn:Theta-cuspidal-support-aux1}
		\begin{minipage}[c]{0.8\textwidth}
			for every irreducible constituent $\Pi$ of $r_L \left( S\Theta^H_\phi \right)$, the $\GL$ part of $\Pi$ has cuspidal support built from $\rho_i |\det|_F^{(a_i - 1)/2}$, where $i$ ranges over certain elements of $I$.
		\end{minipage}
	\end{gather}
	Indeed, this property is familiar from Moeglin's works, and a precise statement can be found in \cite[Lemma 7.2]{Xu17c}; it is proved by transfer to twisted $\GL(2n)$.
	
	For general $\phi \in \Phi_{\mathrm{bdd}}(H)$, by expressing $S\Theta^H_\phi$ as induction from a discrete series L-packet from a Levi, \eqref{eqn:Theta-cuspidal-support-aux0} follows from \eqref{eqn:Theta-cuspidal-support-aux1} and the geometric lemma.	
\end{proof}

\begin{lemma}\label{prop:beta-LLC}
	Suppose that $\psi = \hat{\phi} \in \Psi(\tilde{G})^{\star}$ and $\chi \in \EuScript{S}_\phi^\vee$. Put $e := \chi(z)$, where $z$ is the image of $-1 \in Z_{\tilde{G}^\vee}$ in $\EuScript{S}_\phi$. Then
	\[ \beta^{\tilde{G}}(\pi_{\phi, \chi}) = e \beta^{\SO(V^e)}(\sigma_{\phi, \chi}). \]
\end{lemma}
\begin{proof}
	Recall that $B$ is the standard Borel subgroup of $G$, and let $P_0$ be a minimal parabolic subgroup of $\SO(V^e)$. Clearly, $\dim A_B$ equals $\dim A_{P_0}$ (resp.\ $\dim A_{P_0} + 1$) when $e = 1$ (resp.\ $e = -1$). The assertion is then immediate from Lemma \ref{prop:Theta-cuspidal-support}.
\end{proof}

\begin{proposition}\label{prop:anti-tempered-star}
	For $\psi = \hat{\phi} \in \Psi(\tilde{G})^{\star}$, by writing $\tilde{\mu} := \tilde{\mu}_\psi$, we have $\beta(\psi) \beta^{\tilde{G}}(\pi_{\phi, \chi\tilde{\mu}}) = \chi(s_\psi)$ for all $\chi \in \EuScript{S}_\psi^\vee$. Consequently, the statement of Theorem \ref{prop:local-desiderata} except unitarity holds for $\psi$.
\end{proposition}
\begin{proof}
	Recall that $\pi_{\phi, \chi} \in \Pi_\phi$ matches $\sigma_{\phi, \chi} \in \Pi^H_\phi$ via Gan--Savin, and the sign $\beta(\psi)$ is defined relative to $H = \SO(V^+)$. Set $e := \chi(z) \in \{\pm 1\}$. Since $\tilde{\mu} = \mu_\psi$ (Lemma \ref{prop:tilde-mu-star}) and $\mu_\psi(s_\psi) = 1$ (Proposition \ref{prop:beta-variance}), we have
	\[ \chi(s_\psi) = (\chi\tilde{\mu})(s_\psi) = e \beta(\psi) \beta^{\SO(V^e)}(\sigma_{\phi, \chi\tilde{\mu}}) \]
	by \cite[Proposition 5.8]{LLS24} applied to $\SO(V^e)$, noting that $e$ is absorbed into the factor $\beta(\psi)$ in \cite[\S 5.1]{LLS24}. Lemma \ref{prop:beta-LLC} then implies $\chi(s_\psi) = \beta(\psi) \beta^{\tilde{G}}(\pi_{\phi, \chi\tilde{\mu}})$.
	
	Thus \eqref{eqn:anti-tempered-bp} holds for $\psi$. The final assertion follows from Lemma \ref{prop:anti-tempered-bp}.
\end{proof}

\begin{remark}
	In the proof of Proposition \ref{prop:anti-tempered-star}, we applied \cite[Proposition 5.8]{LLS24} to $\SO(V^e)$, which is non-quasisplit when $e = -1$. This does not rely on the endoscopic classification of representations of $\SO(V^e)$ beyond the tempered case. Indeed, the proof given in \cite[\S 7]{LLS24} is an induction on rank using only tempered L-packets, with (a) the tempered local Langlands correspondence for $\SO(V^e)$ and its endoscopic groups, supplied by \cite{Is24}, and (b) Moeglin's results on the structure of cuspidal supports and Jacquet modules of tempered representations of $\SO(V^e)$. None of these involve non-tempered Arthur packets for $\SO(V^e)$.
\end{remark}

\section{Proof of \texorpdfstring{Theorem \ref{prop:local-desiderata}}{Theorem 4.5.2}}\label{sec:proof-local}
\subsection{Abundance of self-dual parameters}
Theorem \ref{prop:local-desiderata} will be proved by a global method. To achieve this, we need a sufficient supply of local parameters at auxiliary non-Archimedean places.

For every non-Archimedean local field $F$, denote its residual characteristic (resp.\ cardinality) by $p$ (resp.\ $q$), and let $\F$ be an algebraic closure of its residual field. Representations of $\Weil{F}$ are supposed to be finite-dimensional, continuous and semisimple, taken up to isomorphisms.

\begin{lemma}\label{prop:abundance}
	Fix $n \in \Z_{\geq 1}$.
	\begin{enumerate}[(i)]
		\item For every even integer $d \geq 2$ and every non-Archimedean local field $F$, there exists a simple irreducible representation $\rho^{\mathrm{symp}}_d$ (resp.\ $\rho^{\mathrm{orth}}_d$) of $\Weil{F}$ of dimension $d$ that is self-dual of symplectic (resp.\ orthogonal) type.
		\item There is a prime number $p_n$, depending only on $n$, such that if $F$ is a non-Archimedean local field with $p \geq p_n$, then there exist two distinct irreducible symplectic (resp.\ orthogonal) representations of $\Weil{F}$ of dimension $d$ for all even integer $d$ satisfying $1 \leq \frac{d}{2} \leq n$.
	\end{enumerate}
\end{lemma}
\begin{proof}
	Let $d \in \Z_{\geq 2}$ be even. Fix a representative of $\Frob$ in $\Weil{F}$ and consider the $d$-dimensional irreducible representations $\sigma$ of $\Weil{F}$ that factor through
	\[ \Weil{F}/I_F^{\mathrm{wild}} \simeq \Z \ltimes \varprojlim_e \zeta_e(\F), \]
	where $e$ is taken over positive integers coprime to $p$, $\zeta_e(\F)$ denotes the group of $e$-th roots of unity in $\F$, and $1 \in \Z$ (corresponding to $\Frob$) operates by $q$-th power on each $\zeta_e(\F)$.
	
	To describe all the possible $\sigma$, observe that $\sigma$ restricted to $\varprojlim_e \zeta_e(\F)$ is
	\[ \chi \oplus \chi^q \oplus \chi^{q^2} \oplus \cdots \quad \text{with finite length $d$, until the first repetition}, \]
	where $\chi: \varprojlim_e \zeta_e(\F) \to \CC^{\times}$ is continuous, $\chi \neq \mathbf{1}$, and $\Frob$ permutes the summands cyclically by
	\[ \chi \xrightarrow{\identity} \chi^q \xrightarrow{\identity} \cdots \xrightarrow{t} \left( \chi^{q^d} = \chi \right) \]
	where $t \in \CC^{\times}$. Note that $\{\chi, \chi^q, \ldots \}$ is uniquely determined by $\sigma$, and $\sigma(\Frob^d) = t \cdot \identity$.
	
	Note that $p \nmid \mathrm{ord}(\chi)$. We have $\chi^2 \neq \mathbf{1}$. Indeed, $\chi^2 = \mathbf{1}$ is ruled out when $p = 2$. If $p > 2$, then $\chi^2 = \mathbf{1}$ would imply $\chi^q = \chi$, contradicting $d \geq 2$.
	
	Since $\sigma^\vee(\Frob^d) = t^{-1} \cdot \identity$, if $\sigma$ is self-dual, then $t^2 = 1$ and there exists $k \notin d\Z$ such that $\chi^{q^k} = \chi^{-1}$, unique by demanding $1 \leq k < d$. In turn, this implies $\chi^{q^{2k}} = \chi$, hence $2k = d$ (since $2k < 2d$ and $2k$ is divisible by $d$) and $\chi$ factors through $\mu_{q^{d/2} + 1}(\F)$; also note that $\Frob^d$ acts trivially on $\mu_{q^{d/2} + 1}(\F)$. Conclusion: self-dual $\sigma$ satisfies $t \in \{\pm 1\}$ and factors through $(\Z/2d \Z) \ltimes \mu_{q^{d/2} + 1}(\F)$, a metacyclic group, and it factors through $(\Z/d \Z) \ltimes \mu_{q^{d/2} + 1}(\F)$ if and only if $t = 1$.
	
	Conversely, from every injective homomorphism $\chi: \mu_{q^{d/2} + 1}(\F) \to \CC^{\times}$ and $t \in \{\pm 1\}$, one obtains an irreducible representation $\sigma$ by the recipe above. The circular pattern $\chi, \chi^q, \chi^{q^2}, \ldots$ of $d$ terms is determined by the isomorphism class of $\sigma$, up to rotation.
	
	Define a bilinear pairing $B$ on the underlying space of $\sigma$ as follows:
	\begin{itemize}
		\item $B$ pairs the summands $\chi$ with $\chi^{q^{d/2}} = \chi^{-1}$, and $\chi^q$ with $\chi^{q^{\frac{d}{2} + 1}} = \chi^{q^{d/2} q} = \chi^{-q}$, etc.\ non-trivially, whilst the other pairs of summands are mutually orthogonal;
		\item the non-trivial pairings above are $(x, y) \mapsto xy$ on $\CC \times \CC$.
	\end{itemize}
	One readily verifies that $B$ is $\Weil{F}$-invariant, and $B$ is symmetric (resp.\ alternating) when $t = 1$ (resp.\ $t = -1$).
	
	Hence there exist $\varphi(q^{d/2} + 1)/d$ distinct symplectic (resp.\ orthogonal) tamely ramified $d$-dimensional irreducible representations of $\Weil{F}$, where $\varphi$ is Euler's totient function.
	
	Finally, (i) follows from $\varphi(q^{d/2} + 1) \geq 1$. We also have
	\begin{equation*}
		\frac{\varphi(q^{d/2} + 1)}{d} \geq \frac{\varphi(p^{d/2} + 1)}{d} \;\stackrel{\text{if}\; d \leq 2n}{\geq}\; \frac{\varphi(p^{d/2} + 1)}{2n}.
	\end{equation*}
	The first inequality above follows from the elementary bound $\varphi(N) \geq \sqrt{N}$ for $N \notin \{2, 6\}$, by writing $q = p^r$ and assuming $r \ge 2$, in which case $(p^{d/2})^r + 1 \notin \{2, 6\}$ and then $\varphi(q^{d/2} + 1) \geq \sqrt{(p^{d/2})^r + 1} > p^{d/2} \geq \varphi(p^{d/2}+1)$. That elementary bound also implies $\lim_{N \to +\infty} \varphi(N) = +\infty$, which proves (ii).
\end{proof}

\begin{remark}\label{rem:abundance-det}
	In the orthogonal case in Lemma \ref{prop:abundance}, the determinant of the irreducible representation we considered is necessarily the non-trivial unramified quadratic character $\eta_{\mathrm{nr}}$ of $\Weil{F}$. Indeed, the action of $\varprojlim_e \zeta_e(\F)$ is
	\[ \chi^{1 + q + \cdots + q^{d-1}} = \left( \chi^{q^{d/2} + 1} \right)^{\frac{q^{d/2} - 1}{q - 1}} = \mathbf{1}, \]
	whilst the Frobenius action is $(-1)^{d-1} t = -1$.
\end{remark}

\subsection{Globalization}\label{sec:globalization-lemma}
Let $E$ be a local or global field of characteristic zero, and $N \in \Z_{\geq 1}$. By \cite[p.12]{Ar13}, elliptic endoscopic data $\mathbf{R}$ of the twisted $\GL(N)$ over $E$ are classified by triplets $(N_{\mathrm{orth}}, N_{\mathrm{symp}}, \eta)$, where $N_{\mathrm{orth}} \in \Z_{\geq 0}$, $N_{\mathrm{symp}} \in 2\Z_{\geq 0}$ satisfy $N_{\mathrm{orth}} + N_{\mathrm{symp}} = N$, and $\eta = \eta_{\mathbf{R}}$ is a quadratic character of $\Weil{E}$ such that
\[ N_{\mathrm{orth}} = 0 \implies \eta = \mathbf{1}, \quad N_{\mathrm{orth}} = 2 \implies \eta \neq \mathbf{1}. \]
The endoscopic group $R$ is a product of two quasisplit classical groups; it satisfies
\[ R^\vee = \SO(N_{\mathrm{orth}}, \CC) \times \Sp(N_{\mathrm{symp}}, \CC) \]
and $\eta_{\mathbf{R}}$ determines the first factor of $R$; note that $R$ contains no split $\SO(2)$-factors. We say $\mathbf{R}$ is simple if $N_{\mathrm{orth}} N_{\mathrm{symp}} = 0$.

\index{Phi-tilde-ell@$\tilde{\Phi}_{\elli}(N), {}^{\pm} \tilde{\Phi}_{\elli}(N)$}
Suppose $E$ is local. Let $\tilde{\Phi}_{\elli}(N)$ be the set of all $\phi = \bigoplus_{i=1}^k \phi_i \in \Phi_{\mathrm{bdd}}(N)$ where $\phi_1, \ldots, \phi_k$ are simple, self-dual and distinct; elements of $\tilde{\Phi}_{\elli}(N)$ will be called elliptic parameters. Let ${}^+ \tilde{\Phi}_{\elli}(N)$ (resp.\ ${}^- \tilde{\Phi}_{\elli}(N)$) be the subset of $\tilde{\Phi}_{\elli}(N)$ defined by requiring all $\phi_i$ to be orthogonal (resp.\ symplectic). As explained in \cite[pp.33--34]{Ar13}, for each $\phi \in \tilde{\Phi}_{\elli}(N)$, there is a unique elliptic endoscopic datum $\mathbf{R}$ such that $\phi$ factors through $\Phi_2(R)$; if $\phi \in {}^{\pm} \tilde{\Phi}_{\elli}(N)$ then $\mathbf{R}$ is simple.

Ditto for global $E$, with only notational differences with the following exception. When we say a formal sum $\dot{\phi} = \bigoplus_i \dot{\phi}_i$ of distinct self-dual cuspidal automorphic representations factors through $\dot{\Phi}_2(R)$, this means that each $\dot{\phi}_i$ is of symplectic or orthogonal type, in accordance with the two factors of $R^\vee$. The datum $\mathbf{R}$ is pinned down as follows: the type condition determines the pair $(N_{\mathrm{orth}}, N_{\mathrm{symp}})$ and, in the simple case, the sign $\zeta$ for which $\dot{\phi} \in {}^{\zeta}\dot{\tilde{\Phi}}_{\elli}(N)$ (denoted $\zeta_\alpha$ below); the quadratic character $\dot{\eta}_{\mathbf{R}}$ in the orthogonal case is the one whose component at each $v$ is the $\eta_v$ attached to $\dot{\phi}_v$ by the local recipe, which is a global object because $\dot{\phi}$ is automorphic.

Assume that we are given the following data.
\begin{itemize}
	\item A totally real number field $\dot{F}$.
	\item A finite set $V$ of places of $\dot{F}$ containing all Archimedean ones.
	\item A place $u \in V$, and we set $F := \dot{F}_u$.
	\item Families $(\zeta_\alpha)_{\alpha \in \mathcal{A}} \in \{\pm 1\}^{\mathcal{A}}$ and $(n_\alpha)_{\alpha \in \mathcal{A}} \in (\Z_{\geq 1})^{\mathcal{A}}$ where $\mathcal{A}$ is a finite set.
	\item For each $\alpha \in \mathcal{A}$, we are given a family $(\phi^\alpha_v)_{v \in V}$ where $\phi^\alpha_v \in {}^{\zeta_\alpha} \tilde{\Phi}_{\elli}(n_\alpha)$, defined over $\dot{F}_v$.
	\item For all $\alpha \in \mathcal{A}$ and $v \in V$, let $\mathbf{G}^\alpha_v$ be the simple elliptic endoscopic datum of the twisted $\GL(n_\alpha)$ such that $\phi_v^\alpha$ factors through $\Phi_2(G^\alpha_v)$. Denote the quadratic character associated with $\mathbf{G}^\alpha_v$ by $\eta^\alpha_v$.
	
	Note that $(N_{\mathrm{orth}}, N_{\mathrm{symp}})$ in $\mathbf{G}^\alpha_v$ equals $(n_\alpha, 0)$ (resp.\ $(0, n_\alpha)$) when $\zeta_\alpha = +1$ (resp.\ $\zeta_\alpha = -1$). Therefore, $\eta^\alpha_v$ is non-trivial only for $\zeta_\alpha = +1$.
\end{itemize}

In the application of \S\ref{sec:aux-parameters}, the collection above is the family of constituents of a single Arthur parameter $\psi$; these must be globalized simultaneously, over the same number field $\dot{F}$ and with the same sets of places, so that the resulting $\dot{\psi}$ belongs to $\dot{\Psi}_2(\dot{\tilde{G}})$.

For the globalization that is yet to be performed, we will distinguish two different settings. In Case 1 (the ``real case in general position''), one takes $u$ to be the only Archimedean place of $\dot{F} = \Q$; moreover, $\phi^\alpha_u$ is assumed to be in \emph{general position} (i.e.\ highly regular) for all $\alpha$, to be specified below. This case will play a technical role for bootstrapping the global arguments. In Case 2 (the ``general case'' with $F \neq \CC$), which will serve our general application, $\dot{F}$ possesses an Archimedean place $v \neq u$, and the condition for general position is imposed on $\phi^\alpha_v$ for all $\alpha$ and all Archimedean places $v \neq u$. This distinction will also be reflected in the key Lemma \ref{prop:globalization}.

For the precise notion of L-parameters in general position, we work over $\R$ and fix $m > 0$; the $2$-dimensional simple summands in every $\phi_{\R} \in \tilde{\Phi}_{\elli}(m)$ can be described by a sequence $\mu_1 > \cdots > \mu_q$ in $\frac{1}{2}\Z_{\geq 1}$ with $q = \lfloor \frac{m}{2} \rfloor$ or $\lfloor \frac{m}{2} \rfloor - 1$, as in the discussions surrounding \cite[(6.1.4)]{Ar13}: by writing $\Weil{\R} = \CC^{\times} \cup j\CC^{\times}$ where $j^2 = -1$, we have $z \in \CC^{\times}$ acting through $\phi$ by $\mathrm{diag}\left((z\bar{z}^{-1})^{\mu_i}, (z\bar{z}^{-1})^{-\mu_i} \right)$ and $j$ acting by $\bigl(\begin{smallmatrix} 0 & 1 \\ (-1)^{2\mu_i} & 0 \end{smallmatrix}\bigr)$, for all $i$.

\begin{definition}\label{def:general-position}
	For $\phi_{\R} \in \tilde{\Phi}_{\elli}(m)$ and $\mu_1 > \cdots > \mu_q$ as above, we set
	\begin{equation*}
		\mathsf{R}(\phi_{\R}) := \min \left( \{\mu_i : 1 \leq i \leq q \} \cup \{\mu_i - \mu_j : 1 \leq i < j \leq q \} \right)
	\end{equation*}
	with the convention that $\min\emptyset = +\infty$. Given a real number $\mathsf{R}$, we say $\phi_{\R}$ is in $\mathsf{R}$-general position if $\mathsf{R}(\phi_{\R}) > \mathsf{R}$.
\end{definition}
\index{general position}
\index{R-phi@$\mathsf{R}(\cdot)$}

\begin{lemma}\label{prop:globalization}
	Given $\mathsf{R}' \in \R_{\geq 0}$, there exist a finite set $V' \supset V$ of places of $\dot{F}$ together with a real number $\mathsf{R} \geq \mathsf{R}'$ depending on
	\begin{itemize}
		\item the data $(\dot{F}, V, u, (\zeta_\alpha)_{\alpha \in \mathcal{A}}, (n_\alpha)_{\alpha \in \mathcal{A}})$,
		\item those $\phi^\alpha_v$ with $v \nmid \infty$, and $\eta^\alpha_v$ with $v \mid \infty$,
		\item $\phi^\alpha_u$ if $F = \R$ and $\dot{F}$ has more than one Archimedean place,
	\end{itemize}
	such that the following holds. Suppose that
	\begin{itemize}
		\item either $u$ is the only Archimedean place of $\dot{F}$, and $\mathsf{R}(\phi^\alpha_u) > \mathsf{R}$ for all $\alpha \in \mathcal{A}$;
		\item or there exists some Archimedean place $v \neq u$, and we have $\mathsf{R}(\phi^\alpha_v) > \mathsf{R}$ for all those $v$ and all $\alpha \in \mathcal{A}$; moreover, for each $\alpha$, the product of the central characters of all prescribed local data is trivial on the central subgroup $Z^{\infty, u}$ introduced in the proof of \cite[Lemma 6.2.2]{Ar13} and in Remark 2 after Corollary 6.2.4 of \textit{loc.\ cit.}
	\end{itemize}
	Then there exist
	\[ \dot{\phi}^\alpha \in {}^{\zeta_\alpha} \dot{\tilde{\Phi}}_{\elli}(n_\alpha), \quad \alpha \in \mathcal{A} \]
	such that for every place $v$ of $\dot{F}$, we have
	\begin{itemize}
		\item $\dot{\phi}^\alpha_v = \phi^\alpha_v$ if $v \in V$;
		\item $\dot{\phi}^\alpha_v \in \Phi_{\mathrm{bdd}}(n_\alpha)$ if $v \in V' \smallsetminus V$, which is trivial on $\SL(2, \CC) \subset \mathcal{L}_{\dot{F}_v}$, and it either belongs to ${}^{\zeta_\alpha} \tilde{\Phi}_{\elli}(n_\alpha)$ or is of the form $\mu \oplus \mu^{-1}$ where $\mu$ is a unitary character (the latter occurs only when $(n_\alpha, \zeta_\alpha) = (2, 1)$);
		\item $\dot{\phi}^\alpha_v$ is unramified if $v \notin V'$.
	\end{itemize}
	
	Moreover, $V'$ can be taken arbitrarily large, but $\mathsf{R}$ will depend on it.
\end{lemma}
\begin{proof}
	First off, we globalize $(\eta^\alpha_v)_{v \in V}$ to a quadratic character $\dot{\eta}^\alpha$ of $\Weil{\dot{F}}$ for each $\alpha \in \mathcal{A}$. This matters only when $\zeta_\alpha = +1$. Once this is done, the endoscopic data $(\mathbf{G}^\alpha_v)_{v \in V}$ also globalize to $\dot{\mathbf{G}}^\alpha$; the associated quadratic character is $\dot{\eta}^\alpha$. The local components of $\dot{\eta}^\alpha$ might be trivial in the orthogonal case, but we may and do require that $G^\alpha_v$ is not the split $\SO(2)$ at all $v \mid \infty$.
	
	Indeed, this amounts to globalizing square classes. By the usual argument of approximation, we obtain $\dot{\eta}^\alpha$ with $\dot{\eta}^\alpha_v = \eta^\alpha_v$ for each $v \in V$. Take $V'_\alpha \supset V$ so large that $\dot{\eta}^\alpha_v$ is unramified off $V'_\alpha$. Consequently, $\dot{\mathbf{G}}^\alpha$ is unramified off $V'_\alpha$.
	
	Denote by $\dot{G}^\alpha_v$ the base change of $\dot{G}^\alpha$ to $\dot{F}_v$, and observe that $\dot{G}^\alpha_v$ has square-integrable representations modulo center for each $v \mid \infty$, since the split $\SO(2)$ is excluded.
	
	Now put $V' := \bigcup_\alpha V'_\alpha \supset V$. Outside $V'$, everything is unramified. For each $\alpha \in \mathcal{A}$, we consider two cases separately.

	When parameters at Archimedean places $v \neq u$ are prescribed, we apply Remark 2 after \cite[Corollary 6.2.4]{Ar13}, together with Remark 3 in the case of non-Archimedean prescriptions. The bounds $\mathsf{R}_\alpha$ may be chosen uniformly over the finitely many choices of central characters.
	
	First, assuming $(n_\alpha, \zeta_\alpha) \neq (2, 1)$, apply Remark 3 after \cite[Corollary 6.2.4]{Ar13} (based on Lemma 6.2.2 and Corollary 6.2.3 of \textit{loc.\ cit.}\ and their proofs) to $\dot{\mathbf{G}}^\alpha$ and
	\begin{itemize}
		\item the set of places $V'$;
		\item chosen representations of $G^\alpha_v$ parameterized by $\phi^\alpha_v$, for all $v \in V$;
		\item chosen representations in $\Pi_2(\dot{G}^\alpha_v)$ whose L-parameter (still denoted by $\phi^\alpha_v$) is trivial on $\SL(2, \CC)$, for all $v \in V' \smallsetminus V$.
	\end{itemize}
	Note that Arthur's classification is used when choosing representations.

	The last item about $v \in V' \smallsetminus V$ requires explanation. By Lemma \ref{prop:abundance} (i), for every even $d \geq 2$ there exists a $d$-dimensional irreducible representation $\rho^{\mathrm{orth}}_d$ (resp.\ $\rho^{\mathrm{symp}}_d$) of $\Weil{\dot{F}_v}$ of orthogonal (resp.\ symplectic) type, which we fix. Choose the L-parameter $\phi^\alpha_v \in \Phi_{2, \mathrm{bdd}}(\dot{G}^\alpha_v)$ as follows.
	\begin{itemize}
		\item If $\zeta_\alpha = -1$, then $n_\alpha$ is even and we take $\rho^{\mathrm{symp}}_{n_\alpha}$ as the parameter.
		\item We shall use Remark \ref{rem:abundance-det} and its notation hereafter. If $\zeta_\alpha = 1$ and $n_\alpha$ is odd, we take $\rho^{\mathrm{orth}}_{n_\alpha - 1} \oplus \dot{\eta}^\alpha_v \eta_{\mathrm{nr}}$ (resp.\ $\dot{\eta}^\alpha_v$) when $n_\alpha > 1$ (resp.\ $n_\alpha = 1$).
		\item Suppose $\zeta_\alpha = 1$ and $n_\alpha > 2$ is even. If $\dot{\eta}^\alpha_v = \eta_{\mathrm{nr}}$ then we take $\rho^{\mathrm{orth}}_{n_\alpha}$; otherwise we take $\rho^{\mathrm{orth}}_{n_\alpha - 2} \oplus \mathbf{1} \oplus \eta_{\mathrm{nr}} \dot{\eta}^\alpha_v$.
	\end{itemize}
	
	By the result cited above and its proof, for all $\alpha \in \mathcal{A}$ with $(n_\alpha, \zeta_\alpha) \neq (2, 1)$, we can find $\mathsf{R}_\alpha \in \R_{\geq 0}$ depending on the data in the assertion together with those $\phi^\alpha_v$ with $v \in V' \smallsetminus V$ chosen above, such that the following properties hold. Whenever the L-parameters $\phi^\alpha_v$ at $v \mid \infty$ satisfy $\mathsf{R}(\phi^\alpha_v) > \mathsf{R}_\alpha$, excluding $v = u$ when $F = \R$ and $\dot{F}$ has more than one Archimedean place, there exists an irreducible representation $\dot{\pi}^\alpha$ in the discrete $L^2$-automorphic spectrum of $\dot{G}^\alpha(\A)$ parameterized by some $\dot{\phi}^\alpha \in \dot{\tilde{\Phi}}(n_\alpha)$. They localize to the representations and parameters chosen above, at each $v \in V'$. We have $\dot{\phi}^\alpha \in {}^{\zeta_\alpha} \dot{\tilde{\Phi}}_{\elli}(n_\alpha)$ since $\phi^\alpha_v$ does at $v \in V$.

	Secondly, assume $(n_\alpha, \zeta_\alpha) = (2, 1)$. Then $\dot{G}^\alpha$ is the norm-one torus of the quadratic extension of $\dot{F}$ attached to $\dot{\eta}^\alpha$. We can still apply Arthur's Remark 3 after \cite[Corollary 6.2.4]{Ar13}, but the issue is that $\dot{G}^\alpha_v$ might be a split $\SO(2)$ (equivalently, $\dot{\eta}^\alpha_v = \mathbf{1}$) at some $v \in V' \smallsetminus V$. Therefore, Arthur's globalization must be applied to the smaller set of places
	\[ \underline{V}' := \left\{ v \in V' : \dot{\eta}^\alpha_v \neq \mathbf{1} \right\} \supset V, \]
	off which $\dot{G}^\alpha$ is unramified, by putting $\mathbf{1} \oplus \dot{\eta}^\alpha_v \in \Phi_{2, \mathrm{bdd}}(\dot{G}^\alpha_v)$ at all $v \in \underline{V}' \smallsetminus V$. This produces $\mathsf{R}_\alpha$ and $\dot{\phi}^\alpha$ as before, but for $v \in V' \smallsetminus \underline{V}'$ we have $\dot{G}^\alpha_v(\dot{F}_v) \simeq \dot{F}_v^{\times}$ and $\dot{\phi}^\alpha_v = \mu \oplus \mu^{-1}$ where $\mu$ is a unitary character (in fact unramified, as guaranteed by Arthur's result); $\dot{\phi}^\alpha_v$ is still trivial on $\SL(2, \CC)$ at such places. Note the globalization of characters or L-parameters of $\dot{G}^\alpha$ can also be dealt with via automorphic induction.

	Finally, set $\mathsf{R} := \max\left( \{\mathsf{R}_\alpha: \alpha \in \mathcal{A} \} \cup \{\mathsf{R}'\} \right)$. Since each $V'_\alpha$ can be taken arbitrarily large, so can $V'$.
\end{proof}

Several remarks are in order.
\begin{enumerate}
	\item In the first step of the proof, we cannot assume $V'_\alpha = V$. For an overview of the obstruction to globalization of quadratic characters with specified ramification locus, see \cite{FCR24}.

	\item In \cite[Chapter 6, p.304]{Ar13}, Arthur allows non-hyperspecial maximal compact subgroups when talking about spherical representations of ramified groups. This is not a problem here, since everything is unramified off $V'$ or $\underline{V}'$.
	
	\item In \textit{loc.\ cit.}, one assumes $\dot{F}$ has many real places. This is not really necessary: see the variant in \cite[\S 7.2]{Ar13}. Nonetheless, the condition of general position on parameters at Archimedean places (possibly excluding $u$) must be imposed.
\end{enumerate}

\subsection{Choice of auxiliary parameters}\label{sec:aux-parameters}
Until \S\ref{sec:proof-1}, we let $F$ be a local field of characteristic zero. As usual, $\tilde{G} = \Mp(W)$ with $\dim W = 2n$, relative to $\bpsi$.

Let $\psi = \bigoplus_{i \in I = I^+} m_i \left( \phi_i \boxtimes r(b_i)\right) \in \Psi_{\mathrm{gp}}(\tilde{G})$, where $m_i, b_i \in \Z_{\geq 1}$, each $\phi_i$ is simple, and $(\phi_i, b_i)$ are distinct. They satisfy $\sum_i m_i b_i \dim \phi_i = 2n$.

To globalize, we use the terminology from \S\ref{sec:globalization-lemma}, and distinguish two cases as discussed there. Let us choose $\dot{F}$ and the place $u$ first.
\begin{itemize}
	\item Case 1 (``the real case in general position''): Let $\dot{F} = \Q$ and let $u$ be the Archimedean place.
	\item Case 2 (``the general case''). Let $\dot{F}$ be a totally real number field, required to have at least two Archimedean places if $F = \R$, together with a place $u$ and a chosen isomorphism $\dot{F}_u \simeq F$. For the existence, see, e.g., \cite[Lemma 6.2.1]{Ar13}.
\end{itemize}

We also globalize $(W, \lrangle{\cdot|\cdot})$ and $\bpsi$ to $\dot{F}$, and obtain the adélic metaplectic covering
\[ 1 \to \bmu_8 \to \dot{\tilde{G}} \to \dot{G}(\A) \to 1, \]
where $\dot{G} = \Sp(\dot{W})$. The construction involves the choice of an $\mathfrak{o}_{\dot{F}}$-lattice $L \subset \dot{W}$, which gives rise to the set of places $V_{\mathrm{ram}}$ in \eqref{eqn:V-ram} containing all dyadic and Archimedean places.

In both cases we put
\[ J := \left\{ (i, j) : i \in I, \; j \in \Z, \; 1 \leq j \leq m_i \right\}. \]
Take a set $V_0$ of non-Archimedean places of $\dot{F}$ such that
\begin{itemize}
	\item $u \notin V_0$;
	\item for every $v \in V_0$, denote its residual characteristic by $p(v)$, then $p(v) \geq \max\{3, p_n\}$ (see Lemma \ref{prop:abundance});
	\item there is a bijection $\iota: J \xrightarrow{1:1} V_0$, which we fix.
\end{itemize}

Define the finite set of places
\begin{equation*}
	V := \{u\} \cup V_0 \cup V_{\mathrm{ram}}.
\end{equation*}

Let $d_i = \dim \phi_i$. We set out to choose parameters $\phi_{i,j,v} \in \tilde{\Phi}_{\elli}(d_i)$ for all $(i, j) \in J$ and $v \in V$, such that $\phi_{i,j,v}$ has the same type (symplectic or orthogonal) as $\phi_i$.

\paragraph*{The case $v = u$}
For all $(i, j) \in J$, put $\phi_{i, j, u} := \phi_i$.

\paragraph*{The case $v \in V_0$}
Let $(i', j') := \iota^{-1}(v) \in J$. For every even $d$ such that $1 \leq \frac{d}{2} \leq n$, choose by Lemma \ref{prop:abundance} (ii)
\[\begin{array}{ll}
	\rho^{\mathrm{symp}}_{1, d} \not\simeq \rho^{\mathrm{symp}}_{2, d}: & \text{symplectic irreducible representations of}\; \Weil{\dot{F}_v} \;\text{with}\; \dim = d, \\
	\rho^{\mathrm{orth}}_{1, d} \not\simeq \rho^{\mathrm{orth}}_{2, d}: & \text{orthogonal irreducible representations of}\; \Weil{\dot{F}_v} \;\text{with}\; \dim = d.
\end{array}\]
Also fix ramified quadratic characters $\eta_1 \neq \eta_2$ of $\Weil{\dot{F}_v}$, which exist.
	
Consider now $(i, j) \in J$.
\begin{enumerate}
	\item If $b_i$ is odd then $\phi_i$ is symplectic, put
	\[\phi_{i,j,v} := \begin{cases}
		\rho^{\mathrm{symp}}_{1, d_i}, & \text{if}\; (i, j) = (i', j'), \\
		\rho^{\mathrm{symp}}_{2, d_i}, & \text{if}\; (i, j) \neq (i', j').
	\end{cases}\]
	\item If $b_i$ is even then $\phi_i$ is orthogonal;
	\begin{enumerate}
		\item when $d_i$ is even, put
		\[\phi_{i,j,v} := \begin{cases}
			\rho^{\mathrm{orth}}_{1, d_i}, & \text{if}\; (i, j) = (i', j'), \\
			\rho^{\mathrm{orth}}_{2, d_i}, & \text{if}\; (i, j) \neq (i', j'),
		\end{cases}\]
		\item when $d_i$ is odd, put
		\[\phi_{i,j,v} := \begin{cases}
			\rho^{\mathrm{orth}}_{1, d_i - 1} \oplus \eta_1, & \text{if}\; (i, j) = (i', j'), \\
			\rho^{\mathrm{orth}}_{2, d_i - 1} \oplus \eta_2, & \text{if}\; (i, j) \neq (i', j').
		\end{cases}\]
	\end{enumerate}
	Here we set $\rho^{\mathrm{orth}}_{*, 0} = 0$.
\end{enumerate}
These parameters are all trivial on $\SL(2, \CC) \subset \mathcal{L}_{\dot{F}_v}$.

\paragraph*{The case $v \notin \{u\} \cup V_0$, $v \nmid \infty$ and $v \in V_{\mathrm{ram}}$}
For each even $d \geq 2$, we can take a symplectic (resp.\ orthogonal) irreducible representation $\rho^{\mathrm{symp}}_d$ (resp.\ $\rho^{\mathrm{orth}}_d$) of $\Weil{\dot{F}_v}$ of dimension $d$, by Lemma \ref{prop:abundance} (i).

Also fix a ramified quadratic character $\eta$ of $\Weil{\dot{F}_v}$. For each $(i, j) \in J$, put
\[ \phi_{i, j, v} := \begin{cases}
	\rho^{\mathrm{symp}}_{d_i}, & \text{if $b_i$ is odd,} \\
	\rho^{\mathrm{orth}}_{d_i}, & \text{if $b_i, d_i$ are both even,} \\
	\rho^{\mathrm{orth}}_{d_i - 1} \oplus \eta, & \text{if $b_i$ is even and $d_i$ is odd.}
\end{cases}\]
Here we set $\rho^{\mathrm{orth}}_0 = 0$. Again, these parameters are of the required type, and trivial on $\SL(2, \CC) \subset \mathcal{L}_{\dot{F}_v}$.

\paragraph*{The constant $\mathsf{R}$}
For each Archimedean place $v \neq u$ and each $(i, j) \in J$, choose a quadratic character $\eta^{(i,j)}_v$ that is associated with a simple elliptic endoscopic datum $\mathbf{G}^{(i, j)}_v$ of the twisted $\GL(d_i)$ over $\R$, of the same type (orthogonal/symplectic) as that attached to $\phi_i$ (see \S\ref{sec:globalization-lemma}), such that $\Phi_2(G^{(i, j)}_v) \neq \emptyset$. Together with the data chosen above, we may apply Lemma \ref{prop:globalization} to $\mathcal{A} := J$, the set of places $V$ and $\zeta_{i, j} := 1$ (resp. $-1$) if $\phi_i$ is of orthogonal (resp.\ symplectic) type; this produces $V' \supset V$ and a real number $\mathsf{R}$, as large as we desire, independent of the data $\phi_{i, j, v}$ to be chosen at the places $v \mid \infty$ with $v \neq u$ in Case 2.

\paragraph*{The case $v \mid \infty$ (so $v \in V_{\mathrm{ram}}$) but $v \neq u$}
Now we choose $\phi_{i, j, v} \in \Phi_2(G^{(i, j)}_v) \subset \tilde{\Phi}_{\elli}(d_i)$, such that $\mathsf{R}(\phi_{i, j, v}) > \mathsf{R}$, where $\mathsf{R}$ and $\mathbf{G}^{(i, j)}_v$ are obtained in the previous step.

For each $(i, j)$, we also arrange the compatibility of central characters required by Lemma \ref{prop:globalization}, as performed in the proof of \cite[Lemma 6.2.2]{Ar13} (cf.\ Remark 2 after Corollary 6.2.4).

\paragraph*{Globalization}
Assume that, when $F = \R$ and $u$ is the unique Archimedean place of $\dot{F}$ (i.e.\ Case 1), we have $\mathsf{R}(\phi_i) > \mathsf{R}$ for all $i \in I$, and impose no other assumptions in Case 2. Then these data can be fed into Lemma \ref{prop:globalization} to yield $\dot{\phi}_{i, j}$ for all $(i, j) \in J$, so that everything is unramified outside $V'$. Define
\begin{equation}\label{eqn:globalized-psi}
	\dot{\psi} := \bigoplus_{i \in I} \bigoplus_{j=1}^{m_i} \dot{\phi}_{i,j} \boxtimes r(b_i).
\end{equation}

\begin{lemma}\label{prop:aux-parameters}
	With the choices and assumptions above, the following properties hold for $\dot{\psi}$:
	\begin{enumerate}[(i)]
		\item $\dot{\psi} \in \dot{\Psi}_2(\dot{\tilde{G}})$;
		
		\item $\dot{\psi}_u = \psi$;

		\item if $v \mid \infty$ and $v \neq u$, then each $\dot{\phi}_{i, j, v}$ is elliptic satisfying $\mathsf{R}(\dot{\phi}_{i, j, v}) > \mathsf{R}$;
		
		\item $\dot{\psi}_v$ belongs to $\Psi(\dot{\tilde{G}}_v)^\star$ (Definition \ref{def:Psi-star}) for all $v \in V \smallsetminus \{u\}$ such that $v \nmid \infty$;
		
		\item $\dot{\psi}_v$ is anti-tempered (Definition \ref{def:anti-tempered-parameter}) of quasi-good parity for all $v \in V' \smallsetminus \{u\}$ such that $v \nmid \infty$;

		\item the localization homomorphisms \eqref{eqn:S-localization} induce an injection $\EuScript{S}_{\dot{\psi}} \hookrightarrow \prod_{v \in V_0} \EuScript{S}_{\dot{\psi}_v}$ and a surjection $\EuScript{S}_{\dot{\psi}} \twoheadrightarrow \EuScript{S}_{\dot{\psi}_u} = \EuScript{S}_\psi$;
		\item $\dot{\psi}_v \in \Psi(\dot{\tilde{G}}_v)$ (not just in $\Psi^+(\dot{\tilde{G}}_v)$) for all $v \in V'$.
	\end{enumerate}
\end{lemma}
\begin{proof}
	For (i), we have $\dot{\psi} \in \dot{\Psi}(\dot{\tilde{G}})$ since $\dot{\phi}_{i,j}$ is of the same type as $\phi_i$; indeed, by construction, $\dot{\phi}_{i, j}$ arises from the simple twisted endoscopic datum $\dot{\mathbf{G}}^{(i,j)}$ whose sign $\zeta_{i, j}$ matches the type of $\phi_i$. The property $\dot{\psi} \in \dot{\Psi}_2(\dot{\tilde{G}})$ follows from
	\[ (i, j) \neq (i', j') \implies \phi_{i,j,v} \neq \phi_{i',j',v} \;\text{where}\; v := \iota(i', j'). \]
	
	(ii) and (iii) are immediate from construction.
	
	For (iv), observe that by construction, the L-parameters at these places $v$ are trivial on $\SL(2, \CC)$, and contain no unramified characters.

	(v) is almost contained in Lemma \ref{prop:globalization}, in view of (iv). The only issue is when $(\dim \phi_{i, j}, \zeta_{i, j}) = (2, 1)$ and $v \in V' \smallsetminus V$, in which case we may have $\dot{\phi}_{i,j,v} = \mu \oplus \mu^{-1}$ for some non-quadratic unitary character $\mu$, and the summands $\mu^{\pm 1} \boxtimes r(b_i)$ of $\dot{\psi}_v$ so obtained are exactly the obstructions to its good parity. Nonetheless, $b_i$ is even in this case since $\psi \in \Psi_{\mathrm{gp}}(\tilde{G})$. Hence $\mu(-1)^{b_i} = 1$ and Definition \ref{def:good-parity} ensures that $\dot{\psi}_v$ is of quasi-good parity.
	
	Consider the first map in (vi). By (i) we have $\EuScript{S}_{\dot{\psi}} \simeq \bmu_2^J$. For each $(i, j) \in J$, let $v := \iota(i, j) \in V_0$, then
	\begin{align*}
		\dot{\psi}_v & = \left( \phi_{i, j, v} \boxtimes r(b_i) \right) \oplus \bigoplus_{(i', j') \neq (i, j)} \phi_{i', j', v} \boxtimes r(b_{i'}) \\
		& =: \mathcal{P} \oplus \mathcal{P}'.
	\end{align*}
	
	By construction, the irreducible constituents in $\mathcal{P}$ and $\mathcal{P}'$ are disjoint. It follows readily that for all $x \in \EuScript{S}_{\dot{\psi}}$, its image in $\EuScript{S}_{\dot{\psi}_v}$ determines the $(i, j)$-component of $x$. As $(i, j)$ is arbitrary, this implies injectivity.
	
	The second map in (vi) can be identified with the homomorphism $\bmu_2^J \to \bmu_2^I$ induced by the projection $\mathrm{pr}: J \to I$ with $\mathrm{pr}(i, j) = i$, namely by taking the product of coordinates along each fiber of $\mathrm{pr}$. Hence the map is surjective.
	
	Finally, (vii) follows by combining (ii)--(v).
\end{proof}

\subsection{Proof for real parameters in general position}\label{sec:proof-0}

Continue the discussions in \S\ref{sec:aux-parameters}. In particular, we have the globalization $\dot{\psi}$ of $\psi$ in \eqref{eqn:globalized-psi}. Theorem \ref{prop:global-multiplicity} gives
\begin{equation}\label{eqn:multiplicity-formula-V'}
	\Tr L^2_{\dot{\psi}}(\dot{f}) = \sum_{\chi_{V'} \in \mathcal{X}(\dot{\psi}, V')} \prod_{v \in V'} \pi_{\dot{\psi}_v, \chi_v}(f_v)
\end{equation}
for all factorizable $\dot{f} = \prod_v f_v \in \mathcal{H}_{\asp}$ where $f_v$ equals $f_{K_v}$ tensored with the unramified Haar measure for all $v \notin V'$.

We now prove Theorem \ref{prop:local-desiderata} for Case 1 (``real case in general position'') posited in \S\ref{sec:globalization-lemma}. Recall that $u$ is the only Archimedean place of $\dot{F} = \Q$, so that $\dot{\psi}_u = \psi$ and the assumption $\mathsf{R}(\phi_i) > \mathsf{R}$ for all $i \in I$ (Definition \ref{def:general-position}, Lemma \ref{prop:globalization}) is in force. Given $\chi \in \EuScript{S}_\psi^\vee$, we choose $\chi_v \in \EuScript{S}_{\dot{\psi}_v}^\vee$ for all $v \in V'$ as follows.

\begin{description}
	\item[The case $v = u$]	Take $\chi_u = \chi$.
	
	\item[The case $v \in V' \smallsetminus \left( \{u\} \cup V_0 \right)$] In this case $\dot{\psi}_v$ is anti-tempered by Lemma \ref{prop:aux-parameters} (v). Take any character $\chi_v$ of the component group for $\phi_{\dot{\psi}_v}$, which is a quotient of the component group of $\dot{\psi}_v$ (see \S\ref{sec:L-within}).
	
	\item[The case $v \in V_0$] Take $\chi_v$ for these places $v$ so that $\chi_{V'} = \prod_{v \in V'} \chi_v$ belongs to $\mathcal{X}(\dot{\psi}, V')$. This is possible by dualizing the injection in Lemma \ref{prop:aux-parameters} (vi).
\end{description}

Next, we choose $f_v$ for all $v \in V'$ as follows.
\begin{description}
	\item[The case $v = u$]
	No conditions on $f_u$.
	
	\item[The case $v \in V_0$]
	In this case $\dot{\psi}_v \in \Psi(\dot{\tilde{G}}_v)^\star$ by Lemma \ref{prop:aux-parameters} (iv), hence the statement of Theorem \ref{prop:local-desiderata} (except unitarity) holds by Proposition \ref{prop:anti-tempered-star}, and $\pi_{\dot{\psi}_v, \chi'_v}$ are distinct irreducible characters as $\chi'_v \in \EuScript{S}_{\dot{\psi}_v}^\vee$ varies (Lemma \ref{prop:anti-tempered-bp} and Theorem \ref{prop:Luo} (i), (ii)). Therefore we can take $f_v$ such that
	\[ \pi_{\dot{\psi}_v, \chi'_v}(f_v) = \begin{cases}
		1, & \text{if}\; \chi'_v = \chi_v, \\
		0, & \text{if}\; \chi'_v \neq \chi_v,
	\end{cases}\]
	holds for all $\chi'_v$, by the linear independence of irreducible characters. Denote by $\pi_v$ the irreducible representation with character $\pi_{\dot{\psi}_v, \chi_v}$.
	
	\item[The case $v \in V' \smallsetminus (\{u\} \cup V_0)$]
	In this case $\dot{\psi}_v$ is anti-tempered of quasi-good parity by Lemma \ref{prop:aux-parameters} (v). The assumption (b) in Proposition \ref{prop:L-within-0} holds by Lemma \ref{prop:anti-tempered-bp}. Since $\chi_v$ parametrizes an irreducible representation $\pi_v$ from the L-packet attached to $\phi_{\dot{\psi}_v}$, Proposition \ref{prop:L-within-0} implies that one can take $f_v$ such that
	\[ \pi_{\dot{\psi}_v, \chi'_v}(f_v) = \begin{cases}
		1, & \text{if}\; \chi'_v = \chi_v, \\
		0, & \text{if}\; \chi'_v \neq \chi_v,
	\end{cases}\]
	for all $\chi'_v \in \EuScript{S}_{\dot{\psi}_v}^\vee$.
\end{description}

With these choices of $\chi_{V'}$, $(f_v)_{v \in V'}$ and the resulting $(\pi_v)_{v \in V' \smallsetminus \{u\}}$, the terms on the right hand side of \eqref{eqn:multiplicity-formula-V'} vanish unless $\chi'_v = \chi_v$ for all $v \in V' \smallsetminus \{u\}$, by the separation properties of the $f_v$. As to the left hand side, express $\Tr L^2_{\dot{\psi}}(\dot{f})$ as a sum of irreducible characters $\Theta_{\dot{\pi}}(\dot{f})$ with positive integral coefficients, where $\dot{\pi}_v$ is spherical for $v \notin V'$. For each $v \in V' \smallsetminus \{u\}$, the local component $\dot{\pi}_v$ must appear in the expansion of $\pi_{\dot{\psi}_v, \chi'_v}$ for some $\chi'_v$, by leveraging \eqref{eqn:multiplicity-formula-V'} before specialization of $\dot{f}$ and the linear independence of irreducible characters. Now the separation properties of $f_v$ give $\prod_v \Theta_{\dot{\pi}_v}(f_v) = 0$ unless $\dot{\pi}_v \simeq \pi_v$ for all $v \in V' \smallsetminus \{u\}$.

Consequently, the multiplicity formula \eqref{eqn:multiplicity-formula-V'} becomes
\begin{equation*}
	\sum_{\dot{\pi} \in A} \Theta_{\dot{\pi}_u}(f_u) = \sum_{\chi'_u \in B} \pi_{\psi, \chi'_u}(f_u),
\end{equation*}
where
\begin{itemize}
	\item $A$ is the set of irreducible summands in $L^2_{\dot{\psi}}$, counting multiplicities, such that $\dot{\pi}_v \simeq \pi_v$ for all $v \in V' \smallsetminus \{u\}$, and $\dot{\pi}_v$ is unramified when $v \notin V'$;
	\item $B$ is the set of $\chi'_u \in \EuScript{S}_{\dot{\psi}_u}^\vee = \EuScript{S}_\psi^\vee$ such that the tuple formed by $\chi'_u$ and $(\chi_v)_{v \in V' \smallsetminus \{u\}}$ belongs to $\mathcal{X}(\dot{\psi}, V')$.
\end{itemize}

Note that $\chi \in B$ by construction. If $\chi'_u \in B$, then the definition of $B$ implies that $\chi'_u \chi^{-1}$ has trivial pullback to $\EuScript{S}_{\dot{\psi}}$, which gives $\chi'_u = \chi$ by Lemma \ref{prop:aux-parameters} (vi). Hence $B = \{\chi\}$. Since $f_u$ is arbitrary, we infer that
\begin{equation}\label{eqn:proof-1-aux}
	\sum_{\dot{\pi} \in A} \Theta_{\dot{\pi}_u} = \pi_{\psi, \chi}.
\end{equation}
Elements of $A$ are unitary representations. This proves Theorem \ref{prop:local-desiderata} in the case under consideration, including the unitarity of $\pi_{\psi, \chi}$.

\subsection{Proof for the real and non-Archimedean cases}\label{sec:proof-1}

We now prove the general non-complex case of Theorem \ref{prop:local-desiderata}.

\begin{proof}[Proof of Theorem \ref{prop:local-desiderata} for $F \neq \CC$]
	Given any $\psi \in \Psi^+(\tilde{G})$, Proposition \ref{prop:pi-psi-gp} reduces the problem to the case of good parity, by expressing $\psi$ as the image of some $\psi_M \in \Psi^+(\tilde{M})$ whose metaplectic component is of good parity; $\psi_M \in \Psi(\tilde{M})$ whenever $\psi \in \Psi(\tilde{G})$.
	
	Thus it suffices to prove Theorem \ref{prop:local-desiderata} in Case 2 posited in \S\ref{sec:globalization-lemma} and \S\ref{sec:aux-parameters}, where $\psi \in \Psi_{\mathrm{gp}}(\tilde{G})$. To achieve this, we continue to use the global data obtained in \S\ref{sec:aux-parameters}, and repeat the argument for Case 1 in \S\ref{sec:proof-0}, based on \eqref{eqn:multiplicity-formula-V'}.

	The only difference is that there exist Archimedean places $v \neq u$. At such a place $v$, the parameters $\dot{\phi}_{i, j, v} \in \tilde{\Phi}_{\elli}(d_i)$ satisfy $\mathsf{R}(\dot{\phi}_{i, j, v}) > \mathsf{R}$ by Lemma \ref{prop:aux-parameters} (iii); we fix in advance the auxiliary data in Case 1 for each real $v \neq u$, each possible merging of the simple constituents of various $\phi_{i,j,v} \boxtimes r(b_i)$, and each possible choice of the associated real quadratic characters. By Lemma \ref{prop:globalization} (Case 1), the resulting thresholds for $\mathsf{R}(\cdots)$ depend only on the choices above and the non-Archimedean auxiliary parameters, not on the Archimedean exponents; the initial $\mathsf{R}'$ is then taken to exceed all of them. Then Theorem \ref{prop:local-desiderata} is available for $\dot{\psi}_v$ and the assumption (a) in Proposition \ref{prop:L-within-0} holds. For such $v$ we choose:
	\begin{itemize}
		\item $\chi_v$ to be a character of the component group of $\phi_{\dot{\psi}_v}$, giving rise to an irreducible representation $\pi_v$ in the L-packet attached to $\phi_{\dot{\psi}_v}$;
		\item $f_v$ satisfying
		\[ \pi_{\dot{\psi}_v, \chi'_v}(f_v) = \begin{cases}
			1, & \text{if}\; \chi'_v = \chi_v, \\
			0, & \text{if}\; \chi'_v \neq \chi_v,
		\end{cases}\]
		for all $\chi'_v \in \EuScript{S}_{\dot{\psi}_v}^\vee$, by Proposition \ref{prop:L-within-0}.
	\end{itemize}
	The other choices are the same; in particular, $\chi_v$ for $v \in V_0$ must be chosen at the last step. The same arguments lead to $\sum_{\dot{\pi} \in A} \Theta_{\dot{\pi}_u} = \pi_{\psi, \chi}$, proving Theorem \ref{prop:local-desiderata} for $\psi$.
\end{proof}

\subsection{Proof for the complex case}\label{sec:proof-2}
Assume $F = \CC$ hereafter. Parameters $\psi \in \Psi_{\mathrm{gp}}(\tilde{G})$ take the following form:
\[ \psi = \bigoplus_{i \in I} m_i \left( \mathbf{1} \boxtimes r(b_i)\right), \]
where $m_i \in \Z_{\geq 1}$ and $b_i \in 2\Z_{\geq 1}$ for all $i \in I$, and $i \neq j \implies b_i \neq b_j$. They satisfy $\sum_i m_i b_i = 2n$. Indeed, the trivial character $\mathbf{1}$ is the unique self-dual irreducible L-parameter in this setting.

We globalize $\psi$ as follows. Take
\begin{itemize}
	\item $\dot{F}$: an imaginary quadratic field;
	\item $u$: the unique Archimedean place, so that there is an isomorphism $\dot{F}_u \simeq \CC$ which we fix;
	\item $V_0$: a finite set of non-Archimedean places, together with a bijection $\iota: J \xrightarrow{1:1} V_0$ where
	\[ J := \left\{ (i, j): i \in I, \; j \in \Z, \; 1 \leq j \leq m_i \right\}. \]
\end{itemize}

We globalize $(W, \lrangle{\cdot|\cdot})$ (resp.\ $\bpsi$) to $(\dot{W}, \lrangle{\cdot|\cdot})$ (resp.\ $\dot{\bpsi}$) over $\dot{F}$. In this way we obtain the covering $\dot{\tilde{G}} \twoheadrightarrow \dot{G}(\A)$. Define
\begin{equation*}
	V := \{u\} \cup V_0.
\end{equation*}

Next, we define quadratic characters $\phi_{i, j, v}$ of $\Weil{\dot{F}_v}$ for all $v \in V$ and $(i, j) \in J$.
\begin{description}
	\item[The case $v = u$] Take $\phi_{i, j, v} = \mathbf{1}$.
	
	\item[The case $v \in V_0$] Let $(i', j') := \iota^{-1}(v) \in J$. Choose quadratic characters $\eta_1 \neq \eta_2$ of $\Weil{\dot{F}_v}$, which certainly exist. For all $(i, j) \in J$, put
	\[ \phi_{i, j, v} := \begin{cases}
		\eta_1, & \text{if}\; (i, j) = (i', j'), \\
		\eta_2, & \text{if}\; (i, j) \neq (i', j').
	\end{cases}\] 
\end{description}

For every $(i, j) \in J$, we can globalize $(\phi_{i, j, v})_{v \in V}$ to $\dot{\phi}_{i, j}$, which is unramified outside some finite set $V'_{i, j} \supset V$ containing $V_{\mathrm{ram}}$. Indeed, this amounts to globalizing square classes by Kummer theory, and can be done by approximation as in \S\ref{sec:globalization-lemma}. Put
\begin{equation*}
	\dot{\psi} := \bigoplus_{i \in I} \bigoplus_{j=1}^{m_i} \dot{\phi}_{i, j} \boxtimes r(b_i).
\end{equation*}
It is unramified outside $V' := \bigcup_{i, j} V'_{i, j} \supset V$. Below is a variant of Lemma \ref{prop:aux-parameters}.

\begin{lemma}\label{prop:aux-parameters-cplx}
	The following properties hold for $\dot{\psi}$:
	\begin{enumerate}[(i)]
		\item $\dot{\psi} \in \dot{\Psi}_2(\dot{\tilde{G}})$;
		
		\item $\dot{\psi}_u = \psi$;
		
		\item the localization homomorphisms \eqref{eqn:S-localization} induce an injection $\EuScript{S}_{\dot{\psi}} \hookrightarrow \prod_{v \in V_0} \EuScript{S}_{\dot{\psi}_v}$ and a surjection $\EuScript{S}_{\dot{\psi}} \twoheadrightarrow \EuScript{S}_{\dot{\psi}_u} = \EuScript{S}_\psi$;
		\item $\dot{\psi}_v \in \Psi(\dot{\tilde{G}}_v)$ (not just in $\Psi^+(\dot{\tilde{G}}_v)$) and is anti-tempered, for all non-Archimedean places $v$ of $\dot{F}$.
	\end{enumerate}
\end{lemma}
\begin{proof}
	For (i), we clearly have $\dot{\psi} \in \dot{\Psi}(\dot{\tilde{G}})$. The property $\dot{\psi} \in \dot{\Psi}_2(\dot{\tilde{G}})$ follows from
	\[ (i, j) \neq (i', j') \implies \phi_{i,j,v} = \eta_2 \neq \eta_1 = \phi_{i',j',v} \quad \text{where}\; v := \iota(i', j'). \]
	
	(ii) is evident from construction.
	
	Consider the first map in (iii). By (i) we have $\EuScript{S}_{\dot{\psi}} \simeq \bmu_2^J$. For each $(i, j) \in J$, let $v := \iota(i, j) \in V_0$. Then
	\[ \dot{\psi}_v = \left( \phi_{i, j, v} \boxtimes r(b_i) \right) \oplus \bigoplus_{(i', j') \neq (i, j)} \phi_{i', j', v} \boxtimes r(b_{i'}). \]
	
	As $\phi_{i, j, v} \boxtimes r(b_i) \neq \phi_{i', j', v} \boxtimes r(b_{i'})$ whenever $(i, j) \neq (i', j')$, it follows readily that for all $x \in \EuScript{S}_{\dot{\psi}}$, its image in $\EuScript{S}_{\dot{\psi}_v}$ determines the $(i, j)$-component of $x$. As $(i, j)$ is arbitrary, this implies injectivity.
	
	The second map in (iii) can be identified with the homomorphism $\bmu_2^J \to \bmu_2^I$ induced by the projection $J \to I$, hence surjective.
	
	Finally, (iv) follows from the simple fact that the local components of $\dot{\phi}_{i, j}$ are still quadratic characters; in particular they are bounded.
\end{proof}

The multiplicity formula \eqref{eqn:multiplicity-formula-V'} still holds in this context.

\begin{proof}[Proof of Theorem \ref{prop:local-desiderata} for $F = \CC$]
	As in \S\ref{sec:proof-1}, it suffices to deal with $\psi \in \Psi_{\mathrm{gp}}(\tilde{G})$. Consider the globalized data above. Given $\chi \in \EuScript{S}_\psi^\vee$, choose $\chi_v \in \EuScript{S}_{\dot{\psi}_v}^\vee$ for all $v \in V'$ as follows.
	
	\begin{description}
		\item[The case $v = u$]	Take $\chi_u = \chi$.

		\item[The case $v \in V'$ but $v \notin \{u\} \cup V_0$] Take an arbitrary $\chi_v$.
		
		\item[The case $v \in V_0$] Take $\chi_v$ for these places $v$ so that $\chi_{V'} = \prod_{v \in V'} \chi_v$ belongs to $\mathcal{X}(\dot{\psi}, V')$. This is possible by dualizing the injection in Lemma \ref{prop:aux-parameters-cplx} (iii).
	\end{description}
	
	Next, we choose $f_v$ for all $v \in V'$ as follows.
	\begin{description}
		\item[The case $v = u$]
		No conditions on $f_u$.
		
		\item[The case $v \in V'$ but $v \neq u$]
		Since $v$ is non-Archimedean, the statement of Theorem \ref{prop:local-desiderata} is known to hold by \S\ref{sec:proof-1}. Since $\dot{\psi}_v$ is anti-tempered by Lemma \ref{prop:aux-parameters-cplx} (iv), $\pi_{\dot{\psi}_v, \chi'_v}$ are distinct irreducible characters as $\chi'_v \in \EuScript{S}_{\dot{\psi}_v}^\vee$ varies. We can take $f_v$ such that
		\[ \pi_{\dot{\psi}_v, \chi'_v}(f_v) = \begin{cases}
			1, & \text{if}\; \chi'_v = \chi_v, \\
			0, & \text{if}\; \chi'_v \neq \chi_v,
		\end{cases}\]
		holds for all $\chi'_v$. Denote by $\pi_v$ the irreducible representation with character $\pi_{\dot{\psi}_v, \chi_v}$.
	\end{description}
	
	With these choices, \eqref{eqn:multiplicity-formula-V'} becomes
	\begin{equation*}
		\sum_{\dot{\pi} \in A} \Theta_{\dot{\pi}_u}(f_u) = \sum_{\chi'_u \in B} \pi_{\psi, \chi'_u}(f_u),
	\end{equation*}
	where
	\begin{itemize}
		\item $A$ is the set of irreducible summands in $L^2_{\dot{\psi}}$, counting multiplicities, such that $\dot{\pi}_v \simeq \pi_v$ for all $v \in V' \smallsetminus \{u\}$, and $\dot{\pi}_v$ is unramified when $v \notin V'$;
		\item $B$ is the set of $\chi'_u \in \EuScript{S}_{\dot{\psi}_u}^\vee = \EuScript{S}_\psi^\vee$ such that the tuple formed by $\chi'_u$ and $(\chi_v)_{v \in V' \smallsetminus \{u\}}$ belongs to $\mathcal{X}(\dot{\psi}, V')$.
	\end{itemize}
	
	Now we can repeat the arguments in \S\ref{sec:proof-1}: by dualizing the surjection in Lemma \ref{prop:aux-parameters-cplx} (iii), we see $B = \{\chi\}$ and
	\[ \sum_{\dot{\pi} \in A} \Theta_{\dot{\pi}_u} = \pi_{\psi, \chi}. \]
	Therefore $\pi_{\psi, \chi}$ is a linear combination (possibly empty) of unitary irreducible characters with coefficients in $\Z_{\geq 1}$, as desired.
\end{proof}

As a result, the assumption (a) in Proposition \ref{prop:L-within-0} holds for all local fields $F$ of characteristic zero. The main local results are thus completely proven.

\section{Further properties of \texorpdfstring{$\pi_{\psi, \chi}$}{pipsichi}}\label{sec:further-properties}
Throughout this section, we work over a local field $F$ of characteristic zero. Let $\tilde{G} = \Mp(W)$ where $\dim W =2n$.

\subsection{Variation of additive characters: transfer}\label{sec:variation-1}
Set $H := \SO(2n+1)$. Let $\zeta$ be a quadratic character of $\Weil{F}$. To simplify notation, the quadratic character of $F^{\times}$ attached to $\zeta$ by local class field theory will also be denoted by $\zeta$.

View $\zeta$ as a $1$-cocycle $\Weil{F} \to Z_{H^\vee}$. The involution $\phi \mapsto \phi \zeta$ on $\Phi(H)$ preserves $\Phi_{\mathrm{bdd}}(H)$, and induces an involution
\begin{align*}
	\index{Upsilon-zeta@$\Upsilon_\zeta$}
	\Upsilon_\zeta: SD_{\mathrm{spec}}(H) \otimes \mes(H)^\vee & \rightiso SD_{\mathrm{spec}}(H) \otimes \mes(H)^\vee \\
	S\Theta^H_\phi & \mapsto S\Theta^H_{\phi\zeta}.
\end{align*}

On the other hand, $\psi \mapsto \psi\zeta$ also defines an involution on $\Psi(H)$, leaving $\psi|_{\{1\} \times \SL(2, \CC)}$ intact.

\begin{proposition}\label{prop:variation-Upsilon-psi}
	We have $\Upsilon_\zeta\left( S\Theta^H_\psi \right) = S\Theta^H_{\psi\zeta}$ for every $\psi \in \Psi(H)$.
\end{proposition}
\begin{proof}
	Transfer $S\Theta^H_\psi$ to the invariant distribution $\widetilde{\Theta}^{\GL}_\psi$ on the twisted $\GL(2n, F)$. This is the character of the Whittaker-normalized representation of twisted $\GL(2n, F)$ indexed by $\psi \in \tilde{\Psi}(2n)$, in the sense of \cite[\S 2.2]{Ar13}. For $\phi \in \Phi(H)$, we obtain $\widetilde{\Theta}^{\GL}_\phi$ by taking Whittaker-normalized twisted standard modules; see the discussions prior to \cite[(2.2.9)]{Ar13}.
	
	Expand $S\Theta^H_\psi$ as $\sum_{\phi \in \Phi(H)} n(\psi, \phi) S\Theta^H_\phi$, thus $\Upsilon_\zeta\left( S\Theta^H_\psi \right) = \sum_\phi n(\psi, \phi) S\Theta^H_{\phi\zeta}$; we also have
	\[ \widetilde{\Theta}^{\GL}_\psi = \sum_\phi n(\psi, \phi) \widetilde{\Theta}^{\GL}_\phi. \]
	Multiply both sides above by the $\tilde{\theta}$-invariant character $\zeta \circ \det$ on $\GL(2n, F)$. Since
	\begin{itemize}
		\item the local Langlands correspondence for $\GL(2n)$ is compatible with twist by $\zeta \circ \det$,
		\item multiplication by $\zeta \circ \det$ does not affect Whittaker-normalization for representations of the twisted $\GL(2n, F)$ (it is trivial on unipotent radicals),
	\end{itemize}
	we obtain
	\[ \widetilde{\Theta}^{\GL}_{\psi\zeta} = (\zeta \circ \det) \cdot \widetilde{\Theta}^{\GL}_\psi = \sum_\phi n(\psi, \phi) \widetilde{\Theta}^{\GL}_{\phi\zeta}. \]
	Therefore $S\Theta^H_{\psi\zeta}$ and $\Upsilon_\zeta\left( S\Theta^H_\psi \right)$ have the same transfer to twisted $\GL(2n, F)$. These distributions are thus equal; in fact, the transfer of test functions from twisted $\GL(2n, F)$ is surjective by \cite[Proposition 2.1.1]{Ar13}, since there are no outer automorphisms in our case. This completes the proof.
\end{proof}

The definition of $\Upsilon_\zeta$ extends to groups of the form $\SO(2n' + 1) \times \SO(2n'' +1)$, by twisting both factors by the same $\zeta$.

Consider next $\tilde{G}$. Use the recipe of \S\ref{sec:MMp} to define the twofold covering
\begin{equation*}
	\index{G-tilde-2@$\tilde{G}^{(2)}$}
	\tilde{G} \supset \tilde{G}^{(2)} := \MMp(W) \twoheadrightarrow G(F).
\end{equation*}

The covering $\tilde{G}^{(2)}$ is independent of $\bpsi$ in the sense of Proposition \ref{prop:MMp-uniqueness}. The definition of  $\Endo_{\elli}(\tilde{G})$ is also insensitive to $\bpsi$, and that set can be denoted by $\Endo_{\elli}(\tilde{G}^{(2)})$. There is a natural isomorphism $D_-(\tilde{G}) \simeq D_-(\tilde{G}^{(2)})$ (Proposition \ref{prop:MMp-dist}). The spectral transfer from $\mathbf{G}^! \in \Endo_{\elli}(\tilde{G}^{(2)})$ can be viewed as a linear map
\begin{equation*}
	\index{trans-bpsi@$\trans^{\bpsi}_{\mathbf{G}^{"!}, \tilde{G}^{(2)}}$}
	\trans^{\bpsi}_{\mathbf{G}^!, \tilde{G}^{(2)}}: SD_{\mathrm{spec}}(G^!) \otimes \mes(G^!)^\vee \to D_{\mathrm{spec}, -}(\tilde{G}^{(2)}) \otimes \mes(G)^\vee ;
\end{equation*}
the superscript indicates that the transfer depends on $\bpsi$.

Similarly, denote the genuine characters in Definition \ref{def:pi-psi-chi} by $\pi^{\bpsi}_{\psi, \chi}$; they now live on $\tilde{G}^{(2)}$. In fact, all these objects depend only on $\bpsi$ modulo dilation by $F^{\times 2}$.
\index{pi-psi-chi-bpsi@$\pi^{\bpsi}_{\psi, \chi}$}

On the other hand, the $\nu_\psi$ in Definition \ref{def:nu-psi} is independent of $\bpsi$.

\begin{proposition}\label{prop:variation-trans}
	Let $\zeta_c$ be a quadratic character of $\Weil{F}$ corresponding to a class in $F^{\times} / F^{\times 2}$ with representative $c$. Suppose $\mathbf{G}^! \in \Endo_{\elli}(\tilde{G}^{(2)})$ corresponds to $(n', n'') \in \Z_{\geq 0}^2$. Then
	\[ \trans^{\bpsi_c}_{\mathbf{G}^!, \tilde{G}^{(2)}} \circ \Upsilon_{\zeta_c} = \zeta_c(-1)^{n''} \trans^{\bpsi}_{\mathbf{G}^!, \tilde{G}^{(2)}}. \]
\end{proposition}
\begin{proof}
	It suffices to compare the images of $S\Theta^{G^!}_{\phi^!}$ under these maps, where $\phi^! \in \Phi(G^!)$. Let us abbreviate $\zeta_c$ by $\zeta$.
	
	Let us show that both sides behave alike under parabolic induction. Denote by $\tilde{G}$ (resp.\ $\tilde{G}_c$) the eightfold covering associated with $\bpsi$ (resp.\ $\bpsi_c$), and similarly for Levi subgroups. Suppose that $\phi^!$ comes from $\phi_{M^!} = ((\phi_i)_{i \in I' \sqcup I''}, \underline{\phi}', \underline{\phi}'')$ where $M^!$ is the Levi subgroup
	\begin{align*}
		M^! & = \left(\, \prod_{i \in I'} \GL(n_i) \times \SO(2\underline{n}' + 1) \right) \times \left(\, \prod_{i \in I''} \GL(n_i) \times \SO(2\underline{n}'' + 1) \right) \\
		& \subset \SO(2n'+1) \times \SO(2n''+1) = G^! .
	\end{align*}
	
	Take $M \subset G$ such that $\mathbf{M}^! \in \Endo_{\elli}(\tilde{M}^{(2)})$. Since the genuine invariant distributions on twofold and eightfold coverings can be identified via Proposition \ref{prop:MMp-dist}, by using \cite[Proposition 3.8.4]{Li26} to commute parabolic induction and transfer, one readily verifies that the right (resp.\ left) hand side of the asserted equality equals $\Ind^{\tilde{G}}_{\tilde{M}}$ (resp.\ $\Ind^{\tilde{G}_c}_{\tilde{M}_c}$) of the product of
	\begin{itemize}
		\item the analogous expression for the twofold metaplectic factor in $\tilde{M}$ (resp.\ $\tilde{M}_c$), the endoscopic datum $(\underline{n}', \underline{n}'')$ and the parameter $(\underline{\phi}', \underline{\phi}'')$,
		\item the distribution $\bigotimes_{i \in I' \cup I''} \Theta_{\phi_i}$ (resp.\ $\bigotimes_{i \in I' \cup I''} \Theta_{\phi_i \zeta}$) on $\GL$ factors,
		\item the product of all $\zeta(-1)^{\dim \phi_i} \det \phi_i(-1)$ (resp.\ $\det(\phi_i \zeta)(-1)$) over $i \in I''$; the remaining $\zeta(-1)^{\underline{n}''}$ factor on the right hand is absorbed into the ``analogous expression'' in the first item.
	\end{itemize}
	However $\zeta(-1)^{\dim \phi_i} \det \phi_i(-1) = \det(\phi_i \zeta)(-1)$. Moreover, the splitting over $\GL$ factors of Levi subgroups involves the additive character, and the $\zeta$-twist in the second item above is exactly what measures the dependence of parabolic induction on additive characters; for example, see \S 2.4 and the proof of Proposition 12.2 in \cite{GS1}. In this way, we are reduced to the case $\phi^! \in \Phi_{\mathrm{gp, bdd}}(G^!)$.
	
	Suppose that $(\mathbf{G}^!, \phi^!) \mapsto (\phi, s)$ via Proposition \ref{prop:basic-bijection}, so that $2n'' = \dim \phi^{s = -1}$. Now $\epsilon(\phi^{s = -1}) = \nu_\phi(s)$ by Proposition \ref{prop:nu-good-parity}; ditto for $\phi\zeta$.
	
	In view of Theorem \ref{prop:Luo}, which gives $\trans^{\bpsi}_{\mathbf{G}^!, \tilde{G}^{(2)}} \left( S\Theta^{G^!}_{\phi^!} \right) = \sum_\chi (\nu_\phi \chi)(s) \pi^{\bpsi}_{\phi, \chi}$ using Proposition \ref{prop:nu-good-parity}, and similarly for $(\phi\zeta, \bpsi_c)$ replacing $(\phi, \bpsi)$, the problem reduces to showing that
	\begin{equation}\label{eqn:variation-psi-trans-aux}
		\sum_\chi (\nu_{\phi\zeta} \chi)(s) \pi^{\bpsi_c}_{\phi\zeta, \chi} = \zeta(-1)^{\frac{1}{2} \dim \phi^{s = -1}} \sum_\chi (\nu_\phi \chi)(s) \pi^{\bpsi}_{\phi, \chi}
	\end{equation}
	where $\chi$ ranges over $\EuScript{S}_\phi^\vee = \EuScript{S}_{\phi\zeta}^\vee$, for all $\phi \in \Phi_{\mathrm{gp, bdd}}(\tilde{G})$ and $s \in S_{\phi, 2}$.
	
	Write $\phi = \bigoplus_k m_k \phi_k$ (sorry for overlaps with the previous notation). The conjugacy class of $s$ corresponds to pairs $(m'_k, m''_k)_k$ satisfying $m_k = m'_k + m''_k$; here $m'_k$ (resp.\ $m''_k$) is the number of summands in $m_k \phi_k$ on which $s$ acts by $+1$ (resp.\ $-1$). The image of $s$ in $\EuScript{S}_\phi$ is determined by $(m''_k \bmod 2)_k$. Each $\phi_k$ is symplectic, hence
	\begin{align*}
		\frac{1}{2} \dim \phi^{s = -1} & = \frac{1}{2} \sum_k m''_k \dim \phi_k \\
		& \equiv \frac{1}{2}\sum_{k: m''_k \;\text{odd}} \dim \phi_k \pmod{2}.
	\end{align*}
	
	On the other hand, \cite[Theorem 12.1]{GS1} gives an explicit bijection $\chi \mapsto \chi_c$ such that $\pi^{\bpsi_c}_{\phi\zeta, \chi_c} \simeq \pi^{\bpsi}_{\phi, \chi}$ and
	\[ \frac{\chi_c(s)}{\chi(s)} = \frac{\nu_\phi(s)}{\nu_{\phi\zeta}(s)} \prod_{k: m''_k \;\text{odd}} \zeta(-1)^{\frac{\dim \phi_k}{2}}. \]
	
	By using the previous congruence and re-indexing the left hand side of \eqref{eqn:variation-psi-trans-aux} by $\chi_c$, the desired equality follows at once.
\end{proof}

The difference $\delta_c := \chi_c / \chi$ does not depend on $\chi$. A more general formula will be given in the sequel.

\begin{remark}\label{rem:variation-unconditional}
	The proof above for Proposition \ref{prop:variation-trans} is based on Luo's endoscopic character relations and \cite[Theorem 12.1]{GS1}; the latter relies on deep results such as the local Gross--Prasad conjecture for special orthogonal groups, and the hypothesis LLC there is still unavailable for non-quasisplit even orthogonal groups. Furthermore, \cite[Theorem 12.1]{GS1} is only stated for non-Archimedean $F$, although the Archimedean case can be extracted from \cite{AB98} after some effort, or follows from similar arguments. A direct proof of both Proposition \ref{prop:variation-trans} and \cite[Theorem 12.1]{GS1} can be given using properties of metaplectic transfer factors, which works uniformly for all $F$ and does not presume the LLC for even orthogonal groups; details are given in \cite[Theorems 3.3.2 and 5.3.2]{Li24b}.
\end{remark}

\subsection{Variation of additive characters: \texorpdfstring{$\pi_{\psi, \chi}$}{pipsichi}}
Let $\psi \in \Psi^+(\tilde{G})$, decomposed into the form \eqref{eqn:psi-decomp-2}. Describe $S_\psi$ and $\EuScript{S}_\psi \simeq \bmu_2^{I^+}$ as in \eqref{eqn:S-psi}.

Consider a quadratic character $\zeta = \zeta_c$ of $\Weil{F}$, corresponding to $c F^{\times 2} \in F^{\times}/F^{\times 2}$. We have
\[ \pi^{\bpsi_c}_{\psi\zeta, \chi}, \; \pi^{\bpsi}_{\psi, \chi} \in D_{\mathrm{spec}, -}(\tilde{G}^{(2)}) \otimes \mes(G)^\vee \]
for each $\chi \in \EuScript{S}_\psi^\vee = \EuScript{S}_{\psi\zeta}^\vee$.

\begin{definition}\label{def:delta-c}
	\index{delta-c@$\delta_c$}
	Identify $\EuScript{S}_\psi^\vee = \EuScript{S}_{\psi\zeta_c}^\vee$ with $\bmu_2^{I^+}$. Define $\delta_c \in \EuScript{S}_\psi^\vee$ by
	\[ \delta_{c, i} = \begin{cases}
		\zeta_c(-1)^{\frac{1}{2} \dim \phi_i} \dfrac{\epsilon(\phi_i, \bpsi)}{\epsilon(\phi_i \zeta_c, \bpsi)}, & \text{if $b_i$ is odd,} \\
		1, & \text{otherwise,}
	\end{cases} \]
	for all $i \in I^+$.
\end{definition}

\begin{lemma}
	We have $\delta_c(s_\psi) = 1$.
\end{lemma}
\begin{proof}
	Let $x_\psi \in \EuScript{S}_\psi$ be the image of $s_\psi$. Let $i \in I^+$. Then the $i$-th component of $x_\psi$ is non-trivial only when $b_i$ is even.
\end{proof}

The following result generalizes \cite[Theorem 12.1]{GS1}. The main inputs in its proof are Propositions \ref{prop:variation-Upsilon-psi} and \ref{prop:variation-trans}.

\begin{proposition}\label{prop:variation-pi}
	For all $\chi \in \EuScript{S}_\psi^\vee$, we have
	\[ \pi^{\bpsi_c}_{\psi\zeta_c, \chi \delta_c} = \pi^{\bpsi}_{\psi, \chi}. \]
\end{proposition}
\begin{proof}
	Given any $x = (x_i)_i \in \bmu_2^{I^+} \simeq \EuScript{S}_\psi$, we take a representative $s$ in $S_\psi$ such that for all $i \in I^+$, if $x_i = \pm1$ then its projection $s_i$ to $\Or(m_i, \CC)$ is
	\[ \begin{pmatrix} \pm 1 & & & \\ & 1 & & \\ & & \ddots & \\ & & & 1 \end{pmatrix}; \]
	the other components of $s$ in \eqref{eqn:S-psi} are trivial. If $(n', n'')$ corresponds to $\mathbf{G}^!$, then $n'' = \frac{1}{2} \sum_{\substack{i \in I^+ \\ x_i = -1}} b_i \dim \phi_i$.
	
	Abbreviate $\zeta_c$ by $\zeta$. Let $\chi, \chi' \in \EuScript{S}_\psi^\vee = \EuScript{S}_{\psi\zeta}^\vee$. We have
	\begin{align*}
		\pi^{\bpsi}_{\psi, \chi} & = 2^{-|I^+|} \sum_x \chi(x_\psi x) \epsilon(\psi^{s = -1}) \trans^{\bpsi}_{\mathbf{G}^!, \tilde{G}^{(2)}} \left( S\Theta^{G^!}_{\psi^!} \right) \\
		& = 2^{-|I^+|} \sum_x \chi(x_\psi x) \prod_{\substack{i \in I^+ \\ x_i = -1}} \epsilon(\phi_i, \bpsi)^{b_i} \cdot \trans^{\bpsi}_{\mathbf{G}^!, \tilde{G}^{(2)}} \left( S\Theta^{G^!}_{\psi^!} \right),
	\end{align*}
	where $(\mathbf{G}^!, \psi^!) \mapsto (\psi, s)$.
	
	Combine Propositions \ref{prop:variation-Upsilon-psi} and \ref{prop:variation-trans} to obtain
	\begin{align*}
		\pi^{\bpsi_c}_{\psi\zeta, \chi'} & = 2^{-|I^+|} \sum_x \chi'(x_\psi x) \prod_{\substack{i \in I^+ \\ x_i = -1}} \epsilon(\phi_i \zeta, \bpsi_c)^{b_i} \zeta(-1)^{\frac{1}{2} b_i \dim \phi_i} \cdot \trans^{\bpsi}_{\mathbf{G}^!, \tilde{G}^{(2)}} \left( S\Theta^{G^!}_{\psi^!} \right) \\
		& = 2^{-|I^+|} \sum_x \chi'(x_\psi x) \epsilon(\psi^{s = -1}) \prod_{\substack{i \in I^+ \\ x_i = -1}} \frac{\epsilon(\phi_i \zeta, \bpsi)^{b_i}}{\epsilon(\phi_i, \bpsi)^{b_i}} \zeta(-1)^{\frac{1}{2} b_i \dim \phi_i} \cdot \trans^{\bpsi}_{\mathbf{G}^!, \tilde{G}^{(2)}} 
		\left( S\Theta^{G^!}_{\psi^!} \right).
	\end{align*}
	Consider each term in the product over $i$.
	\begin{itemize}
		\item If $b_i$ is odd, then $\phi_i$ is symplectic, the $\epsilon$-factors are $\pm 1$ and we get
		\[ \frac{\epsilon(\phi_i, \bpsi)}{\epsilon(\phi_i \zeta, \bpsi)} \zeta(-1)^{\frac{1}{2} \dim\phi_i}. \]
		\item If $b_i$ is even, then
		\[ \left( \frac{\epsilon(\phi_i, \bpsi)^2}{\epsilon(\phi_i \zeta, \bpsi)^2}\right)^{b_i/2} = \left(\frac{(\det \phi_i)(-1)}{(\det \phi_i\zeta)(-1)}\right)^{b_i/2} = \zeta(-1)^{\frac{b_i}{2} \dim\phi_i} \]
		since $\phi_i$ is self-dual.
	\end{itemize}
	Therefore
	\[ \prod_{\substack{i \in I^+ \\ x_i = -1}} \frac{\epsilon(\phi_i \zeta, \bpsi)^{b_i}}{\epsilon(\phi_i, \bpsi)^{b_i}} \zeta(-1)^{\frac{1}{2} b_i \dim \phi_i} = \prod_{\substack{i \in I^+: x_i = -1 \\ b_i\;\text{odd}}} \frac{\epsilon(\phi_i, \bpsi)}{\epsilon(\phi_i\zeta, \bpsi)} \zeta(-1)^{\frac{1}{2} \dim\phi_i} = \delta_c(x). \]
	
	The desired assertion follows immediately.
\end{proof}

\subsection{Normalized intertwining operators via Arthur parameters}\label{sec:normalizing}
\index{PM@$\mathcal{P}(M)$}
For every Levi subgroup $M \subset G$, let $\mathcal{P}(M) = \mathcal{P}^G(M)$ be the set of parabolic subgroups with Levi factor $M$. For all $P, Q \in \mathcal{P}(M)$ and a genuine irreducible representation $\pi \in \Pi_-(\tilde{M})$, Harish-Chandra's theory gives the standard intertwining operator $J_{\tilde{Q}|\tilde{P}}(\pi_\lambda): I_{\tilde{P}}(\pi_\lambda) \to I_{\tilde{Q}}(\pi_\lambda)$; here
\begin{itemize}
	\item $\mathfrak{a}^*_M := X^*(M) \dotimes{\Z} \R$, and $\lambda \in \mathfrak{a}^*_{M, \CC} := \mathfrak{a}^*_M \dotimes{\R} \CC$ acts on $\Pi_-(\tilde{M})$ in the evident way, written as $\pi \mapsto \pi_\lambda$;
	\item $I_{\tilde{P}}$ and $I_{\tilde{Q}}$ are parabolic inductions;
	\item $J_{\tilde{Q}|\tilde{P}}(\pi_\lambda)$ is understood as a meromorphic family of intertwining operators with parameter $\lambda \in \mathfrak{a}^*_{M, \CC}$.
\end{itemize}
It depends on the choice of invariant measures on $V := R_{\mathrm{u}}(P) \cap R_{\mathrm{u}}(Q) \backslash R_{\mathrm{u}}(Q)$, or equivalently on $R_{\mathrm{u}}(Q) \cap R_{\mathrm{u}}(P^-)$ under the usual identification, where $R_{\mathrm{u}}(\cdot)$ means the unipotent radical and $P^-$ is the opposite of $P$. For non-Archimedean $F$, we fix a special maximal compact open subgroup $K \subset G(F)$ and normalize the Haar measures by demanding that $\mathrm{vol}(R(F) \cap K) = 1$ where $R$ stands for $R_{\mathrm{u}}(Q)$ or $R_{\mathrm{u}}(P) \cap R_{\mathrm{u}}(Q)$. For Archimedean $F$, we adopt Arthur's normalization in \cite[\S 3]{Ar89-IOR1}.

Assume that a family of normalizing factors $r_{\tilde{Q}|\tilde{P}}(\pi_\lambda)$ is chosen for $\tilde{G}$, where we consider $P, Q \in \mathcal{P}(M)$, $\pi \in \Pi_-(\tilde{M})$ and twisting parameters $\lambda \in \mathfrak{a}^*_{M, \CC}$ as above, so that $\lambda \mapsto r_{\tilde{Q}|\tilde{P}}(\pi_\lambda)$ is a scalar-valued meromorphic function. They are subject to various conditions, e.g.\ conjugation-invariance, induction in stages and $r_{\tilde{Q}|\tilde{P}}((\pi_\lambda)_\mu) = r_{\tilde{Q}|\tilde{P}}(\pi_{\lambda + \mu})$, etc., and they serve to define the normalized intertwining operators as meromorphic families
\[ R_{\tilde{Q}|\tilde{P}}(\pi_\lambda) := r_{\tilde{Q}|\tilde{P}}(\pi_\lambda)^{-1} J_{\tilde{Q}|\tilde{P}}(\pi_\lambda), \]
which turn out to give unitary operators when $\pi \in \Pi_{\mathrm{unit}, -}(\tilde{M})$ and $\lambda \in \sqrt{-1}\mathfrak{a}_M^*$.
\index{R-QP@$R_{\tilde{Q}"|\tilde{P}}(\pi_\lambda)$}

The same is true for groups of metaplectic type. We refer to \cite[\S 8.1]{Li26} for further explanations and references. For what follows, we recall from \textit{loc.\ cit.}\ that their construction reduces to the case where $\pi$ is essentially square-integrable.

Specifically, $r_{\tilde{Q}|\tilde{P}}(\pi_\lambda)$ for tempered $\pi$ is obtained from the square-integrable case by parabolic induction, and the case of general $\pi$ is obtained through Langlands quotient; we refer to \cite{Ar89-IOR1} for details.

\begin{proposition}\label{prop:normalizing-factor}
	A family of normalizing factors for $\tilde{G}$ can be chosen so that for every Levi subgroup $M$ and $\phi \in \Phi_{2, \mathrm{bdd}}(\tilde{M})$, we have
	\begin{equation*}
		\pi, \pi' \in \Pi^{\tilde{M}}_\phi \implies r_{\tilde{Q}|\tilde{P}}(\pi_\lambda) = r_{\tilde{Q}|\tilde{P}}(\pi'_\lambda), \quad P, Q \in \mathcal{P}(M).
	\end{equation*}
	More precisely, the following properties hold.
	\begin{enumerate}[(i)]
		\item In the unramified situation, use $K := G(\mathfrak{o}_F)$ to normalize Haar measures and denote the residual cardinality by $q_F$. For $K$-spherical $\pi$,
		\begin{itemize}
			\item as predicted by Langlands' recipe (see \cite[p.81]{Ar13}),
			\[ r_{\tilde{Q}|\tilde{P}}(\pi_\lambda) = \frac{L(0, \rho^\vee_{\tilde{Q}|\tilde{P}} \circ \phi_\lambda)}{L(1, \rho^\vee_{\tilde{Q}|\tilde{P}} \circ \phi_\lambda)} \cdot q_F^{-\frac{1}{2} \dim \rho^\vee_{\tilde{Q}|\tilde{P}}} , \]
			\item $R_{\tilde{Q}|\tilde{P}}(\pi_\lambda)$ preserves spherical vectors in spherical principal series.
		\end{itemize}
		\item In general, writing $M = \prod_k \GL(n_k) \times \Sp(W^\flat)$ with $\dim W^\flat = 2m$, if $\pi$ has L-parameter $\phi$ then $r_{\tilde{Q}|\tilde{P}}(\pi_\lambda)$ can be taken to be
		\[ c \cdot \frac{L(0, \rho^\vee_{\tilde{Q}|\tilde{P}} \circ \phi_\lambda)}{L(1, \rho^\vee_{\tilde{Q}|\tilde{P}} \circ \phi_\lambda)} \cdot \frac{\epsilon(\frac{1}{2}, \rho^\vee_{\tilde{Q}|\tilde{P}} \circ \phi_\lambda, \bpsi)}{\epsilon(0, \rho^\vee_{\tilde{Q}|\tilde{P}} \circ \phi_\lambda, \bpsi)} \]
		where $c \in \R_{> 0}$ is a constant depending only on $M$, $P$, $Q$, and the remaining terms are as in Langlands' recipe for
		\[ M^H := \prod_k \GL(n_k) \times \SO(2m+1) \subset \SO(2n+1) =: H, \]
		by noting that $\mathcal{P}^G(M)$ and $\mathcal{P}^H(M^H)$ are in natural bijection and $\tilde{M}^\vee = (M^H)^\vee$, $\tilde{G}^\vee = H^\vee$; the representation $\rho^{\vee}_{\tilde{Q}|\tilde{P}}$ of $\tilde{M}^\vee$ is then defined as in \cite[\S 2.3]{Ar13}.
	\end{enumerate}
\end{proposition}
\begin{proof}
	We treat (ii) first. The general recipe in \cite[\S 2]{Ar89-IOR1}, which also applies to covering groups, reduces the prescription of normalizing factors to the case where $M$ is a maximal proper Levi subgroup, $Q$ is opposite to $P$, whereas $\pi$ (resp.\ $\phi$) is kept in $\Pi_{2, -}(\tilde{M})$ (resp.\ $\Phi_{2, \mathrm{bdd}}(\tilde{M})$).
	
	Consider the special case above. When $F = \R$ (resp.\ $F = \CC$), the formula for $r_{\tilde{Q}|\tilde{P}}(\pi_\lambda)$ is given in \cite[Theorem 8.1.6]{Li26} (resp.\ \cite[Theorem 8.1.7]{Li26}) by a comparison with the known formula on the $H$-side. One has $c = \sqrt{2}^e$ (resp.\ $c = 2^e$) where $e$ is the number of long roots in $R_{\mathrm{u}}(Q^-) \cap R_{\mathrm{u}}(P)$ provided that Arthur's normalization of Haar measures is adopted; we refer to \textit{loc.\ cit.} for details.
	
	When $F$ is non-Archimedean, the desired $r_{\tilde{Q}|\tilde{P}}(\pi_\lambda)$ is a (complex-conjugated) square-root of Harish-Chandra's $\mu$-function by \cite[(4.1)]{Ar89-IOR1}. Since the LLC for $\tilde{G}$ is realized by $\Theta$-correspondence between $\tilde{G}$ and pure inner forms of $H$ (or the corresponding full orthogonal groups), for which the LLC and the formula for $\mu$-functions are available from \cite{Ar13, Is24}, we are reduced to the comparison of $\mu$-functions under such $\Theta$-correspondences. The required equality now follows from \cite[Theorem 12.1]{GI14} or \cite[Proposition 10.1]{GS1}; in fact, one has $c=1$ if the Haar measures are chosen as in \cite[Appendix B]{GI14}. See also \cite[\S 7.3]{Is20}.
	
	For (i), one has $M=T$ and the Gindikin--Karpelevich formula \cite[Theorem 6.4]{Mc11} for $\tilde{G}$ yields the desired properties.
\end{proof}

\begin{remark}\label{rem:normalizing-factor}
	When $M$ is fixed, the constant $c$ in Proposition \ref{prop:normalizing-factor} (ii) is independent of $\bpsi$, and ``multiplicative'' in $P, Q \in \mathcal{P}(M)$, in the same manner as $J_{\tilde{Q}|\tilde{P}}(\pi_\lambda)$. These assertions are confirmed by the formulas for $c$ in the proof above; in particular, $c=1$ in the unramified setting. However, these formulas are based on different choices of Haar measures in different settings. A uniform formula of $c$ for all $F$ for carefully chosen Haar measures is surely desirable, but lies outside the scope of the present work. We hope to revisit this issue in future works on the local intertwining relations (cf. \cite[\S 2.4]{Ar13}).
\end{remark}

Note that $\mathfrak{a}^*_{M, \CC}$ acts on $\Phi(\tilde{M})$ compatibly with its action on $\Pi_-(\tilde{M})$ via the version for $\tilde{M}$ of \eqref{eqn:LLC-general}. In view of the local Langlands correspondence \eqref{eqn:LLC-general} for $\tilde{M}$ and the procedure above for prescribing normalizing factors, we may write
\[ r_{\tilde{Q}|\tilde{P}}(\phi_\lambda) := r_{\tilde{Q}|\tilde{P}}(\pi_\lambda), \quad \text{if}\; \pi \in \Pi^{\tilde{M}}_\phi. \]

Our goal is to normalize the intertwining operators through Arthur parameters instead of L-parameters. The recipe for quasisplit classical groups is given in \cite[\S 2.3]{Ar13}. For the reader's convenience, we supply some details for $\tilde{G}$ below.

\begin{definition}\label{def:approx}
	Given $P, P' \in \mathcal{P}(M)$ and $\phi_1, \phi_2 \in \Phi(\tilde{M})$, define
	\[ r_{\tilde{P}'|\tilde{P}}(\phi_1, \phi_2) := r_{\tilde{P}'|\tilde{P}}(\phi_1)^{-1} r_{\tilde{P}'|\tilde{P}}(\phi_2), \]
	which should be viewed as a meromorphic family.
	
	We write $\phi_1 \approx \phi_2$ if
	\[ r_{\tilde{P}''|\tilde{P}}(\phi_1, \phi_2) = r_{\tilde{P}''|\tilde{P}'}(\phi_1, \phi_2) r_{\tilde{P}'|\tilde{P}}(\phi_1, \phi_2) \]
	holds for all $P, P', P'' \in \mathcal{P}(M)$. This is an equivalence relation. Define $\phi_1 \approx_{M^H} \phi_2$ similarly for $M^H \subset H$.
\end{definition}

For $i = 1, 2$, the quotient $r_{\tilde{P}''|\tilde{P}}(\phi_i)^{-1} r_{\tilde{P}''|\tilde{P}'}(\phi_i) r_{\tilde{P}'|\tilde{P}}(\phi_i)$ as a meromorphic family equals its avatar for $M^H \subset H$ times a positive constant $c_{P''|P'|P}$ coming from the $c$ in Proposition \ref{prop:normalizing-factor}. The factor $c_{P''|P'|P}$ disappears in the obstruction $r_{\tilde{P}''|\tilde{P}}(\phi_1, \phi_2)^{-1} r_{\tilde{P}''|\tilde{P}'}(\phi_1, \phi_2) r_{\tilde{P}'|\tilde{P}}(\phi_1, \phi_2)$ to $\approx$. Therefore $\phi_1 \approx \phi_2$ if and only if $\phi_1 \approx_{M^H} \phi_2$.

Using explicit formulas in Proposition \ref{prop:normalizing-factor}, one can actually show that
\begin{equation}\label{eqn:approx-Gamma}
	r_{\tilde{P}''|\tilde{P}}(\phi_i)^{-1} r_{\tilde{P}''|\tilde{P}'}(\phi_i) r_{\tilde{P}'|\tilde{P}}(\phi_i) = c_{P''|P'|P} \prod_\alpha g_\alpha(\phi_i)
\end{equation}
where $\alpha$ ranges over reduced roots of the adjoint $A_M$-action up to signs, that separate both $P$, $P'$ and $P'$, $P''$, and $g_\alpha(\phi_i)$ is a product of Rankin--Selberg or symmetric square $\gamma$-factors for general linear groups attached to $\phi_i$.

Following \cite[pp.42--43]{Ar89-IOR1}, define the \emph{block equivalence} (in Vogan's sense) to be the finest equivalence relation on $\Pi_-(\tilde{M})$ satisfying
\[ m(\rho, \pi) \neq 0 \implies \pi \;\text{is block-equivalent to}\; \pi_\rho, \]
where $\rho$ is a genuine standard module with Langlands quotient $\pi_\rho$, and $m(\rho, \pi)$ is the coefficient of $\Theta_\pi$ in the expansion of $\Theta_\rho$.

\begin{lemma}\label{prop:approx-block}
	Let $M \subset G$ be a Levi subgroup. If there exist $\pi_i \in \Pi^{\tilde{M}}_{\phi_i}$ for $i = 1, 2$ such that $\pi_1$ is block-equivalent to $\pi_2$, then $\phi_1 \approx \phi_2$.
\end{lemma}
\begin{proof}
	Our normalizing factors in Proposition \ref{prop:normalizing-factor} satisfy the same properties posited by Arthur. Now we can apply \cite[Proposition 5.2]{Ar89-IOR1}.
\end{proof} 

Block equivalence is also available in the twisted setting by \textit{loc.\ cit.}, specifically for $\widetilde{\GL}(n, F)$, the twisted $\GL(n, F)$. The block equivalences on $\widetilde{\GL}$ restrict to $\GL$. Indeed, suppose the aforementioned representation $\rho$ (resp.\ $\pi$) above lives on $\GL(n, F)$ and underlies the twisted representation $\tilde{\rho}$ (resp.\ $\tilde{\pi}$) of $\widetilde{\GL}(n, F)$; in the second paragraph of the proof of \cite[Lemma 7.5.4]{Ar13}, the formulas for $m(\tilde{\rho}, \tilde{\pi})$ and $m(\rho, \pi)$ entail that $m(\tilde{\rho}, \tilde{\pi}) \neq 0$ implies $m(\rho, \pi) \neq 0$ since $m^{\pm}(\rho, \pi) \in \Z_{\geq 0}$.

Now consider $\psi \in \Psi(\tilde{M})$, the multi-set $\Pi^{\tilde{M}}_\psi$ and $\phi_\psi \in \Phi(\tilde{M})$ (Definition \ref{def:phi-psi} for $\tilde{M}$).

\begin{proposition}\label{prop:normalization-approx}
	Let $M \subset G$ be a Levi subgroup and $\psi \in \Psi(\tilde{M})$. If $\phi \in \Phi(\tilde{M})$ and there exists $\pi \in \Pi^{\tilde{M}}_\phi$ which has nonzero multiplicity in $\Pi^{\tilde{M}}_\psi$, then $\phi \approx \phi_\psi$.
\end{proposition}
\begin{proof}
	Write $M = \prod_k \GL(n_k) \times \Sp(W^\flat)$. Let $\underline{\psi}$ be the metaplectic part of $\psi$, so that $s_{\underline{\psi}}$ is trivial on the $\GL$-factors. Consider the $\mathbf{M}^! \in \Endo_{\elli}(\tilde{M})$ and $\psi^! \in \Psi(M^!)$ such that $(\mathbf{M}^!, \psi^!) \mapsto (\psi, s_{\underline{\psi}}^{-1})$, say
	\begin{gather*}
		M^! = \prod_k \GL(n_k) \times \SO(2m'+1) \times \SO(2m''+1), \\
		m := m' + m'' = \frac{1}{2} \dim W^\flat, \quad \psi^! = ((\psi_k)_k, \psi', \psi'').
	\end{gather*}
	
	Expand $S\Theta^{M^!}_{\psi^!}$ as $\sum_{i \in \mathbb{I}} \tilde{n}(\psi^!, \phi^!_i) S\Theta^{M^!}_{\phi^!_i}$ where $\mathbb{I}$ is a finite set, $S\Theta^{M^!}_{\phi^!_i}$ is the stabilized standard character \eqref{eqn:STheta-gen} attached to some $\phi^!_i \in \Phi(M^!)$, assumed to be distinct, and $\tilde{n}(\psi^!, \phi^!_i) \in \Z \smallsetminus \{0\}$ (uniquely determined) for all $i$. Let $(\mathbf{M}^!, \phi^!_i) \mapsto (\phi_i, s_i)$ under the basic bijection (Proposition \ref{prop:basic-bijection} (ii)), and let $\underline{\phi}_i$ be the metaplectic part of $\phi_i$. Proposition \ref{prop:endo-char-gen} yields
	\[ \trans_{\mathbf{M}^!, \tilde{M}} \left( S\Theta^{M^!}_{\phi^!_i} \right) = \epsilon\left( \underline{\phi}_i^{s_i = -1} \right)^{-1} \sum_\rho \chi_\rho(s_i) \Theta_{\rho}, \]
	where $\rho$ ranges over the standard modules attached to the members of $\Pi^{\tilde{M}}_{\phi_i}$, corresponding to $\chi_\rho \in \EuScript{S}_{\phi_i}^{\vee}$. In the formalism of \S\ref{sec:A-packets},
	\begin{gather*}
		\sum_{i \in \mathbb{I}} \tilde{n}(\psi^!, \phi^!_i) \trans_{\mathbf{M}^!, \tilde{M}} \left( S\Theta^{M^!}_{\phi^!_i} \right)
		= \epsilon\left( \underline{\psi}^{s_{\underline{\psi}}^{-1} = -1} \right)^{-1} \sum_\chi \chi(s_{\underline{\psi}} s_{\underline{\psi}}^{-1} ) \pi_{\psi, \chi} = \pm \sum_{\pi' \in \Pi^{\tilde{M}}_\psi} \Theta_{\pi'}.
	\end{gather*}
	The character $\Theta_\pi$ occurs on the right hand side with nonzero coefficient. By the linear independence of irreducible characters, there exists $i_0 \in \mathbb{I}$ such that $\trans_{\mathbf{M}^!, \tilde{M}} \left( S\Theta^{M^!}_{\phi^!_{i_0}} \right)$ contains $\Theta_\pi$ in its expansion into irreducibles. Thus there exists $\rho$ with $m(\rho, \pi) \neq 0$ (see the proof of Proposition \ref{prop:L-within-0}), so $\pi$ is block-equivalent to the Langlands quotient $\pi_\rho \in \Pi^{\tilde{M}}_{\phi_{i_0}}$. Lemma \ref{prop:approx-block} implies
	\begin{equation}\label{eqn:normalization-approx-1}
		\phi_0 := \phi_{i_0} \approx \phi.
	\end{equation}
	Write $\phi^!_{i_0} = ((\phi_{0,k})_k, \phi'_0, \phi''_0)$. It remains to show that $\phi_0 \approx \phi_\psi$. With the evident notation, \cite[(2.2.15)]{Ar13} gives
	\[ \tilde{n}(\psi^!, \phi^!_{i_0}) = \prod_k n(\psi_k, \phi_{0,k}) \cdot \tilde{n}(\psi', \phi'_0) \tilde{n}(\psi'', \phi''_0) \neq 0. \]
	
	For an L-parameter $\phi$ on some $\GL$, or a self-dual parameter on $\widetilde{\GL}$, let $\rho_\phi$ and $\pi_\phi$ denote the standard module on $\GL$ and its Langlands quotient; set $\pi_\psi := \pi_{\phi_\psi}$, and similarly for $\psi_k$, $\psi'$, $\psi''$ and $\underline{\psi}$. By \cite[p.42]{Ar89-IOR1}, the non-vanishing above gives block equivalence between the representations parameterized by $\psi_k$ and $\phi_{0,k}$ (resp.\ by $\psi'$ and $\phi'_0$, and by $\psi''$ and $\phi''_0$), upon restricting block equivalences from $\widetilde{\GL}$ to $\GL$.
	
	The normalized parabolic induction of $\pi_{\psi'} \boxtimes \pi_{\psi''}$ to $\GL(2m, F)$ is irreducible by unitarity, and equals $\pi_{\underline{\psi}}$ by the argument in the proof of \cite[Lemma 2.2.2]{Ar13} (from (2.2.14) to ``... is actually the Langlands quotient'', but the decomposition into symplectic and orthogonal parts is not needed here). Parabolic induction preserves block equivalence in $\GL$, so $\pi_{\underline{\psi}}$ is block-equivalent to every constituent of the induction of $\pi_{\phi'_0} \boxtimes \pi_{\phi''_0}$. By the arguments of \cite{Ar13} cited above, this induced representation is a quotient of $\rho_{\phi'_0 \oplus \phi''_0}$; its constituents are thus block-equivalent to $\pi_{\phi'_0 \oplus \phi''_0}$. Together with the $\GL(n_k)$ factors, this yields block equivalence between $\pi_\psi$ and $\pi_{\phi_0}$.
	
	With the block equivalence on $\GL$ so obtained as the input, one may apply the argument for \cite[Proposition 2.3.1]{Ar13} (the paragraph from ``As Langlands parameters'' to ``...holds as stated.'') to infer that $\phi_0 \approx_{M^H} \phi_\psi$; another way to see this is to combine block equivalence with the inductivity of those $\gamma$-factors in \eqref{eqn:approx-Gamma} in which $\phi_i$ is viewed as an L-parameter of some $\GL$.
	
	Hence $\phi_0 \approx \phi_\psi$ by the discussions after Definition \ref{def:approx}. Together with \eqref{eqn:normalization-approx-1}, this completes the proof.
\end{proof}

\begin{corollary}[Cf.\ {\cite[Proposition 2.3.1]{Ar13}}]
	Let $\psi$ be as in Proposition \ref{prop:normalization-approx}. Put $r_{\tilde{Q}|\tilde{P}}(\psi_\lambda) := r_{\tilde{Q}|\tilde{P}}((\phi_\psi)_\lambda)$ and define
	\[ R_{\tilde{Q}|\tilde{P}}(\pi_\lambda, \psi_\lambda) := r_{\tilde{Q}|\tilde{P}}(\psi_\lambda)^{-1} J_{\tilde{Q}|\tilde{P}}(\pi_\lambda), \quad P, Q \in \mathcal{P}(M), \; \pi \in \Pi^{\tilde{M}}_\psi \]
	as meromorphic families in $\lambda \in \mathfrak{a}_{M, \CC}^*$; note that $\pi$ is unitary by Proposition \ref{prop:local-desiderata}. They satisfy
	\begin{enumerate}[(i)]
		\item $R_{\tilde{P}''|\tilde{P}}(\pi_\lambda, \psi_\lambda) = R_{\tilde{P}''|\tilde{P}'}(\pi_\lambda, \psi_\lambda) R_{\tilde{P}'|\tilde{P}}(\pi_\lambda, \psi_\lambda)$ for all $P, P', P'' \in \mathcal{P}(M)$;
		\item $R_{\tilde{Q}|\tilde{P}}(\pi_\lambda, \psi_\lambda)^* = R_{\tilde{P}|\tilde{Q}}(\pi_{-\overline{\lambda}}, \psi_{-\overline{\lambda}})$.
	\end{enumerate}
	In particular, $R_{\tilde{Q}|\tilde{P}}(\pi_\lambda, \psi_\lambda)$ is a unitary operator when $\lambda \in \sqrt{-1} \mathfrak{a}^*_M$, and is analytic on $\sqrt{-1} \mathfrak{a}^*_M$.
\end{corollary}
\begin{proof}
	Given Proposition \ref{prop:normalization-approx}, the proof is the same as \cite[pp.87--89]{Ar13}: it allows us to transport (i) and (ii) to the known cases for $R_{\tilde{Q}|\tilde{P}}(\pi_\lambda)$ that are normalized using $\phi_\lambda$ instead of $(\phi_\psi)_\lambda$.
\end{proof}

\subsection{Remarks on quadratic unipotent packets}\label{sec:quadratic-unipotent}
The goal here is to establish Moeglin's results in \cite[\S 4.2]{Moe17} for quadratic unipotent parameters for $\tilde{G}$ over Archimedean $F$.

\begin{definition}\label{def:quadratic-unipotent}
	\index{Arthur parameter!quadratic unipotent}
	Let $\psi \in \Psi(\tilde{G})$. If $\psi = \bigoplus_{i \in I} \zeta_i \boxtimes r(b_i)$, where $\zeta_i$ are quadratic characters of $\Weil{F}$ inflated to $\mathcal{L}_F$, and $b_i \in \Z_{\geq 1}$ (not necessarily distinct), then $\psi$ is said to be \emph{quadratic unipotent}. Ditto for global Arthur parameters.
\end{definition}

\begin{remark}
	Suppose $F = \CC$. In the notation of \eqref{eqn:psi-decomp-2}, every $\psi$ with $J = \emptyset$ is quadratic unipotent; in fact $\zeta_i = \mathbf{1}$ for all $i$. In particular, every $\psi \in \Psi_{\mathrm{gp}}(\tilde{G})$ is quadratic unipotent.
\end{remark}

The main theorems below for $\tilde{G}$ are ``essentially proven'' in \cite{Moe17}. We state the desiderata first, and justify them at the end.

\begin{theorem}[Cf.\ {\cite[Corollaire 4.2.2]{Moe17}}]\label{prop:real-unip-packet}
	Suppose that $F = \R$ and let $\psi \in \Psi_{\mathrm{gp}}(\tilde{G})$ be quadratic unipotent. Then
	\begin{enumerate}[(i)]
		\item $\pi_{\psi, \chi}$ is a multiplicity-free sum of irreducible characters, possibly zero, for each $\chi \in \EuScript{S}_\psi^\vee$;
		\item if $\chi \neq \chi'$, then there is no common irreducible constituent in the expansions of $\pi_{\psi, \chi}$ and $\pi_{\psi, \chi'}$.
	\end{enumerate}
	In other words, $\Pi_\psi$ is multiplicity-free.
\end{theorem}

\begin{theorem}[Cf.\ {\cite[Théorème 4.2.3]{Moe17}}]\label{prop:cplx-packet}
	Suppose that $F = \CC$ and $\psi \in \Psi_{\mathrm{gp}}(\tilde{G})$. We have
	\begin{enumerate}[(i)]
		\item $\pi_{\psi, \chi}$ is either zero or irreducible, for each $\chi \in \EuScript{S}_\psi^\vee$;
		\item if $\chi \neq \chi'$ and $\pi_{\psi, \chi}$ and $\pi_{\psi, \chi'}$ are both nonzero, then $\pi_{\psi, \chi} \neq \pi_{\psi, \chi'}$.
	\end{enumerate}
	In particular, $\Pi_\psi$ is multiplicity-free.
\end{theorem}

We remark that the multiplicity-freeness for Arthur packets over non-Archimedean $F$ is recently established by J.\ Chen \cite{ChenJ26}.

In \cite{Moe17}, one begins by considering \emph{very regular quadratic unipotent parameters} in both the local and global cases: a local Arthur parameter $\psi$ written as in \eqref{eqn:psi-decomp-2} is said to be \emph{very regular} if $m_i = 1$ for all $i \in I$, and $i \neq j \implies \phi_i \neq \phi_j$; such parameters always belong to $\Psi_2(\tilde{G})$, thus of good parity. Ditto in the global setting.
\index{Arthur parameter!very regular}

Metaplectic groups are included in \textit{loc.\ cit.}\ and they fit well with the yoga of $\Theta$-lifting therein, but in order to define a genuine irreducible $L^2$-automorphic representation $\dot{\pi}$ on an adélic metaplectic group $\dot{\tilde{G}}$ associated with a very regular parameter $\dot{\psi}$, two extra conditions (1) and (2) are imposed in \S 3.3 of \textit{loc.\ cit.} We justify them below, with the terminology of \textit{loc.\ cit.}

Consider an adélic metaplectic covering $\dot{\tilde{G}} = \Mp(\dot{W}, \A)$ with $\dim \dot{W} = 2n$.

\begin{proposition}
	Let $\dot{\psi} = \bigoplus_{i=1}^t \dot{\phi}_i \boxtimes r(b_i) \in \dot{\Psi}_2(\dot{\tilde{G}})$ be very regular, $b_1 \geq \cdots \geq b_t$ and $d_i := \dim \dot{\phi}_i$, so that $\sum_{i=1}^t b_i d_i = 2n$. Let $\dot{\pi} \subset L^2_{\dot{\psi}}$ be irreducible. Then
	\begin{enumerate}[(i)]
		\item the exponents of $\dot{\pi}$ are all negative;
		\item there exist an integer $0 \leq \ell_i \leq \frac{b_i}{2}$ for all $1 \leq i \leq t$, and a genuine irreducible cuspidal automorphic representation $\pi_{\mathrm{cusp}}$ of $\Mp(\dot{W}', \A)$ where
		\[ \dot{W}' \subset \dot{W}, \quad \dim \dot{W}' = 2n - 2\sum_{i=1}^t d_i \ell_i, \]
		such that the constant terms of $\dot{\pi}$ have nonzero projection to the cuspidal support of the parabolically induced representation
		\[ \left(\, \bigtimes_{\substack{i = 1, \ldots, t \\ \ell_i \neq 0}}
		\mathrm{Speh}(\dot{\phi}_i, \ell_i)
		|\det|_{\A}^{-(b_i - \ell_i)/2} \right) \times \pi_{\mathrm{cusp}} \]
		from $\prod_i \GL(d_i \ell_i, \A) \times \Mp(\dot{W}', \A)$ to $\dot{\tilde{G}}$.
	\end{enumerate}
\end{proposition}
\begin{proof}
	For quasisplit classical groups, the corresponding properties are proved in \S 1 of \textit{loc.\ cit.}, specifically Théorème 1.2.1 (i)--(ii) together with \S 1.3. These proofs are largely based on subtle yet elementary operations on Eisenstein series and cuspidal supports; however, there are also non-trivial inputs from \cite[\S\S 1.1--1.3 and \S 3.5]{Moe08}. Endoscopic character relations and regularity play a critical role in these sources.

	For metaplectic groups, we now have local Arthur parameters, Arthur packets characterized by endoscopic character relations and the global multiplicity formula; see the properties established in \S\S\ref{sec:local-desiderata}--\ref{sec:A-packets} and \S\ref{sec:multiplicity-formula}, which are as strong as in \cite[\S\S 1.5--2.2]{Ar13}.

	For the normalization of intertwining operators in \cite[\S 3.5]{Moe08} in the unramified case, one may use the Gindikin--Karpelevich formula \cite[Theorem 6.4]{Mc11} for $\tilde{G}$, or consult Proposition \ref{prop:normalizing-factor} and Remark \ref{rem:normalizing-factor}.

	All these are compatible with isomorphisms of spherical Hecke algebras in the unramified setting. Compatibility between Hecke algebra isomorphisms and $\Theta$-lifting is also used in Moeglin's work, see \cite[Remarque 3.2.2]{Moe17} which includes the metaplectic case. These ingredients are what one needs to remove/justify the conditions (1) and (2) in \S 3.3 of \textit{loc.\ cit.}
\end{proof}

The following multiplicity-one result by Moeglin is therefore justified.

\begin{theorem}[Cf.\ {\cite[Théorème 3.6.1]{Moe17}}]
	Let $\dot{\pi}$ be an irreducible constituent of $L^2_{\dot{\psi}}$ where $\dot{\psi} \in \dot{\Psi}_2(\dot{\tilde{G}})$ is quadratic unipotent and very regular. Then $\dot{\pi}$ occurs with multiplicity one in $L^2_{\mathrm{disc}, -}$.
\end{theorem}

Next, in the local setting, we have to fulfill the hypotheses in \cite[\S 4.1]{Moe17} about globalization of Archimedean quadratic unipotent parameters of good parity.

\begin{lemma}\label{prop:globalization-very-regular}
	Let $F$ be an Archimedean local field. Given a quadratic unipotent local parameter $\psi \in \Psi_{\mathrm{gp}}(\tilde{G})$ and $\pi \in \Pi_-(\tilde{G})$ that appears in $\Pi_\psi$, we can
	\begin{enumerate}[(i)]
		\item globalize $\psi$ to a very regular quadratic unipotent $\dot{\psi} \in \dot{\Psi}_2(\dot{\tilde{G}})$ over some number field $\dot{F}$, with a place $u$ satisfying $\dot{F}_u \simeq F$;
		\item find an irreducible constituent $\dot{\pi}$ in $L^2_{\dot{\psi}}$ such that $\pi \simeq \dot{\pi}_u$.
	\end{enumerate}
\end{lemma}
\begin{proof}
	Re-use the proofs in \S\S\ref{sec:proof-0}--\ref{sec:proof-2}. For (i), an inspection of \S\ref{sec:aux-parameters} shows that if each $\phi_i$ is a quadratic character, then so is each summand $\dot{\phi}_{i, j}$ of $\dot{\psi}$. In fact, the whole argument for Lemma \ref{prop:globalization} reduces to globalization of quadratic characters in the present case.
	
	For the regularity of $\dot{\psi}$, recall the proof of Lemma \ref{prop:aux-parameters} (i).
	
	For (ii), recall the final steps of the proof of Theorem \ref{prop:local-desiderata}, in which $\pi_{\psi, \chi}$ is realized as a finite sum of irreducible constituents of $L^2_{\dot{\psi}}$, for each $\chi \in \EuScript{S}_\psi^\vee$.
\end{proof}

In this way, Theorems \ref{prop:real-unip-packet} and \ref{prop:cplx-packet} are now justified since the global results in \cite[\S 3.6]{Moe17} carry over to metaplectic groups, whilst the details about the globalization argument in \cite[\S 4.2]{Moe17} (choice of data at auxiliary places, etc.) are supplied by \S\S\ref{sec:proof-0}--\ref{sec:proof-2}. The remaining arguments are the same.

\section{Examples}\label{sec:examples}
\subsection{Local principal Arthur parameters}\label{sec:local-principal}
Fix a local field $F$ of characteristic zero. Define $\tilde{G} = \Mp(W)$ using $(W, \lrangle{\cdot|\cdot})$ and $\bpsi: F \to \CC^{\times}$. Set $n := \frac{1}{2} \dim W$.

In what follows, $\zeta = \zeta_c$ is always a quadratic character of $\Weil{F}$ associated with some $c F^{\times 2} \in F^{\times}/F^{\times 2}$, inflated to $\mathcal{L}_F$ in the non-Archimedean case. We will also use the twofold covering $\tilde{G}^{(2)}$ in \S\ref{sec:variation-1}.

\begin{definition}\label{def:principal}
	\index{Arthur parameter!principal}
	\index{Arthur packet!principal}
	Arthur parameters of the form $\zeta \boxtimes r(2n) \in \Psi_2(\tilde{G})$ are said to be \emph{principal}; the corresponding packets are said to be \emph{principal Arthur packets}.
\end{definition}

For principal $\psi$, we always have $\EuScript{S}_\psi = S_\psi = \Or(1, \CC) = \bmu_2$. The element $s_\psi$ is the non-trivial element of $Z_{\tilde{G}^\vee}$, and its image in $\EuScript{S}_\psi$ is also non-trivial.

Identify $\EuScript{S}_\psi^\vee$ with $\bmu_2$, so that $\pi_{\psi, \pm} = \pi^{\bpsi}_{\psi, \pm}$ make sense.

\begin{definition}
	\index{Theta-psi-pm@$\Theta_{\bpsi}^{\pm}, \Theta_{\bpsi}$}
	Let $\Theta_{\bpsi}^{\pm}$ (resp.\ $\Theta_{\bpsi}$) denote the character of $\omega_{\bpsi}^{\pm}$ (resp.\ $\omega_{\bpsi}$). They are represented by genuine $L^1_{\mathrm{loc}}$ functions on $\tilde{G}$ which are smooth on the regular semisimple locus.
\end{definition}

\begin{lemma}\label{prop:comparison-discriminant}
	Let $H := \SO(2n+1)$. Let $D^G$ (resp.\ $D^H$) denote the Weyl discriminants
	\[ D^G(\delta) = \det\left( \Ad(\delta) - \identity \middle| \mathfrak{g}/\mathfrak{g}_\delta \right), \quad D^H(\gamma) = \det\left( \Ad(\gamma) - \identity \middle| \mathfrak{h}/\mathfrak{h}_\gamma \right) \]
	where $\delta \in G(F)$ (resp.\ $\gamma \in H(F)$) is regular semisimple and $G_\delta$ (resp.\ $H_\gamma$) is its connected centralizer.
	\begin{enumerate}[(i)]
		\item If $\tilde{\delta} \in \tilde{G}$ has regular semisimple image $\delta \in G(F)$, whose stable conjugacy class corresponds to that of $\gamma \in H(F)$ relative to the endoscopic datum of $\tilde{G}$ indexed by $(n, 0)$ (see \cite[\S 5.1]{Li11}), then
		\[ \left| \Theta^+_{\bpsi}(\tilde{\delta}) - \Theta^-_{\bpsi}(\tilde{\delta}) \right| = \frac{\left| D^H(\gamma) \right|_F^{1/2} }{\left| D^G(\delta)\right|_F^{1/2}}. \]
		\item If the correspondence in (i) is taken relative to $(0, n)$, then
		\[ \left| \Theta_{\bpsi}(\tilde{\delta}) \right| = \left| \Theta^+_{\bpsi}(\tilde{\delta}) + \Theta^-_{\bpsi}(\tilde{\delta}) \right| = \frac{\left| D^H(\gamma) \right|_F^{1/2} }{\left| D^G(\delta)\right|_F^{1/2}}. \]
	\end{enumerate}
	Here we define $|z|_{\CC} := z\bar{z}$ when $z \in F = \CC$.
\end{lemma}
\begin{proof}
	Define $-\tilde{\delta} := (-1) \tilde{\delta}$ with $-1$ as in Definition \ref{def:minus-1}, and observe that $(\Theta^+_{\bpsi} - \Theta^-_{\bpsi})(\tilde{\delta}) = (\Theta^+_{\bpsi} + \Theta^-_{\bpsi})(-\tilde{\delta})$. By Maktouf's character formula (see \cite[Théorème 4.2]{Li11}),
	\[ \left| \Theta^+_{\bpsi}(\tilde{\delta}) - \Theta^-_{\bpsi}(\tilde{\delta}) \right| = \left| \Theta_{\bpsi}(-\tilde{\delta}) \right| = \left| \det(\delta + 1 | W) \right|_F^{-\frac{1}{2}}. \]
	
	In the case (i), the eigenvalues of $\delta$ match those of $\gamma$ (except $1$), counting multiplicities. The desired equality follows easily.
	
	In the case (ii), the eigenvalues of $-\delta$ match those of $\gamma$ (except $1$), and $D^G(\delta) = D^G(-\delta)$. The desired equality reduces to (i) at once.
\end{proof}

\begin{proposition}\label{prop:principal}
	Let $\psi = \zeta \boxtimes r(2n)$ be principal, where $\zeta = \zeta_c$ corresponds to $cF^{\times 2}$. We have $\pi^{\bpsi}_{\psi, \pm} = \Theta_{\bpsi_c}^{\pm}$ inside $D_{\mathrm{spec}, -}(\tilde{G}^{(2)}) \otimes \mes(G)^\vee$.
\end{proposition}
\begin{proof}
	To begin with, assume $c=1$ and suppress the exponent $\bpsi$. The pair $(\psi, s = 1)$ (resp.\ $(\psi, s = -1)$) maps to $(\mathbf{G}^!, \psi^!)$ where $\mathbf{G}^! \in \Endo_{\elli}(\tilde{G})$ corresponds to $(n, 0)$ (resp.\ $(0, n)$), and $\psi^!$ is the same parameter viewed in $\Psi(H)$; the spectral transfer will be abbreviated as $\trans_{(n, 0)}$ (resp.\ $\trans_{(0, n)}$).
	
	Denote the transfer factors in question by $\Delta_{(n, 0)}$ and $\Delta_{(0, n)}$. According to \cite[Définition 5.9]{Li11}, if $\tilde{\delta} \in \tilde{G}$ and $\gamma \in H(F)$ are regular semisimple and correspond to each other relative to $(n, 0)$ or $(0, n)$, then
	\begin{equation}\label{eqn:Delta-principal}
		\Delta_{(n, 0)}(\gamma, \tilde{\delta}) = \frac{\Theta^+_{\bpsi}(\tilde{\delta}) - \Theta^-_{\bpsi}(\tilde{\delta})}{\left| \Theta^+_{\bpsi}(\tilde{\delta}) - \Theta^-_{\bpsi}(\tilde{\delta}) \right|}, \quad \Delta_{(0, n)}(\gamma, \tilde{\delta}) = \frac{\Theta^+_{\bpsi}(\tilde{\delta}) + \Theta^-_{\bpsi}(\tilde{\delta})}{\left| \Theta^+_{\bpsi}(\tilde{\delta}) + \Theta^-_{\bpsi}(\tilde{\delta}) \right|}.
	\end{equation}
	
	Let us show that
	\begin{equation}\label{eqn:T-principal}\begin{aligned}
		T_{\psi, 1} := \trans_{(n, 0)}(\mathbf{1}) & = \Theta_{\bpsi}^+ - \Theta_{\bpsi}^- , \\
		T_{\psi, -1} := \underbracket{\epsilon(\mathbf{1}, \bpsi)^{2n}}_{= \mathbf{1}(-1)^n = 1} \trans_{(0, n)}(\mathbf{1}) & = \Theta_{\bpsi}^+ + \Theta_{\bpsi}^- .
	\end{aligned}\end{equation}
	Take the first equation for example. We may view $\mathbf{1}$ (resp.\ $\trans_{(n, 0)}(\mathbf{1})$) as a smooth function on the regular semisimple locus of $H(F)$ (resp.\ $\tilde{G}$) without loss of information; let $S$ (resp.\ $I$) be its product with $|D^H|_F^{1/2}$ (resp.\ $|D^G \circ \rev|_F^{1/2}$). Since every regular semisimple $\delta \in G(F)$ corresponds to a unique stable conjugacy class of $\gamma \in H(F)$ via endoscopy, Arthur's techniques imply $I(\tilde{\delta}) = \Delta_{(n, 0)}(\gamma, \tilde{\delta}) S(\gamma)$ whenever $\gamma$ and $\delta := \rev(\tilde{\delta})$ are in correspondence; cf.\ \cite[\S 8.3, pp.470--473]{Ar13}, which can be adapted to $\tilde{G}$ using \cite[\S\S 5--6]{Li19}. We then conclude by Lemma \ref{prop:comparison-discriminant} and \eqref{eqn:Delta-principal} that $\trans_{(n, 0)}(\mathbf{1}) = \Theta^+_{\bpsi} - \Theta^-_{\bpsi}$. The second equation in \eqref{eqn:T-principal} is settled in the same way.
	
	The earlier description of $s_\psi$ also gives
	\begin{align*}
		T_{\psi, 1} & = \pi_{\psi, +} - \pi_{\psi, -}, \\
		T_{\psi, -1} & = \pi_{\psi, +} + \pi_{\psi, -}.
	\end{align*}
	A comparison with \eqref{eqn:T-principal} implies $\pi_{\psi, \pm} = \Theta_{\bpsi}^{\pm}$.
	
	For the general case, we now have $\pi^{\bpsi_c}_{\mathbf{1} \boxtimes r(2n), \pm} = \Theta_{\bpsi_c}^{\pm}$. It remains to observe that $\delta_c = \mathbf{1}$ (Definition \ref{def:delta-c}) and apply Proposition \ref{prop:variation-pi}; note that by Remark \ref{rem:variation-unconditional}, Proposition \ref{prop:variation-pi} is unconditionally established.
\end{proof}

The recipe above in the $+$ case was first conceived by J.\ Adams in \cite{Ad98} for $F = \R$.

\begin{remark}\label{rem:principal-mf}
	Principal Arthur packets are multiplicity-free (Definition \ref{def:multiplicity-free}) since $\omega^+_{\bpsi} \not\simeq \omega^-_{\bpsi}$, as seen by comparing the actions of $-1 \in \tilde{G}$.
\end{remark}

\subsection{Global principal Arthur parameters}\label{sec:global-principal}
Fix a number field $\dot{F}$. Define $\dot{\tilde{G}} = \Mp(\dot{W}, \A)$ using $(\dot{W}, \lrangle{\cdot|\cdot})$ and $\dot{\bpsi} = \prod_v \bpsi_v : \dot{F} \backslash \A \to \CC^{\times}$. Set $n := \frac{1}{2} \dim \dot{W}$.

\begin{definition}\label{def:principal-global}
	\index{Arthur parameter!principal}
	Global Arthur parameters of the form $\dot{\zeta} \boxtimes r(2n) \in \dot{\Psi}_2(\dot{\tilde{G}})$, where $\dot{\zeta}$ is some quadratic character of $\dot{F}^{\times} \backslash \A^{\times}$, are said to be \emph{principal}.
\end{definition}

The descriptions of $S_{\dot{\psi}} = \EuScript{S}_{\dot{\psi}}$ and $s_{\dot{\psi}}$ are identical to the local setting.

\begin{lemma}\label{prop:principal-global-epsilon}
	If $\dot{\psi}$ is principal, then the character in Theorem \ref{prop:global-multiplicity} satisfies $\epsilon_{\dot{\psi}} = \mathbf{1}$.
\end{lemma}
\begin{proof}
	We have $\epsilon^{\mathrm{Art}}_{\dot{\psi}} = \mathbf{1}$ since $\overline{\EuScript{S}}_{\dot{\psi}} = \{1\}$. On the other hand, the value of $\nu_{\dot{\psi}}$ at the non-trivial element is $\epsilon(\dot{\zeta})^{2n}$. We have seen in \S\ref{sec:global-root-number} that $\epsilon(\dot{\zeta})^2 = 1$.
\end{proof}

Consider
\[ \epsilon = (\epsilon_v)_v \in \bigoplus_v \bmu_2, \quad c \in \dot{F}^{\times}, \]
where $v$ ranges over all places of $\dot{F}$. Define the irreducible genuine representation of $\dot{\tilde{G}}$
\begin{equation*}
	\omega^\epsilon_{\dot{\bpsi}_c, \A} := \bigotimes_v \omega^{\epsilon_v}_{\bpsi_{v, c}},
\end{equation*}
which decomposes the adélic Weil representation \eqref{eqn:adelic-Weil} as
\begin{equation*}
	\omega_{\bpsi_c, \A} = \bigoplus_\epsilon \omega^\epsilon_{\dot{\bpsi}_c, \A}.
\end{equation*}
To define it, we have to pass to twofold coverings. The representation depends only on the coset $c \dot{F}^{\times 2}$, up to canonical isomorphisms.

By class field theory, $c \dot{F}^{\times 2}$ corresponds to a quadratic character $\dot{\zeta} = \bigotimes'_v \zeta_v$ of $\dot{F}^{\times} \backslash \A^{\times}$. By Proposition \ref{prop:principal}, the $v$-component of $\omega^\epsilon_{\dot{\bpsi}_c, \A}$ lies in $\Pi_{\dot{\psi}_v}$ where $\dot{\psi} := \dot{\zeta} \boxtimes r(2n)$.

\begin{proposition}\label{prop:principal-global-mult}
	For each $c \in \dot{F}^{\times}$, take the quadratic character $\dot{\zeta} = \dot{\zeta}_c$ corresponding to $c \dot{F}^{\times 2}$, and the principal parameter $\dot{\psi} := \dot{\zeta} \boxtimes r(2n)$. Consider $\epsilon = (\epsilon_v)_v \in \bigoplus_v \bmu_2$.
	\begin{enumerate}[(i)]
		\item If $\prod_v \epsilon_v \neq 1$, then $\omega^\epsilon_{\dot{\bpsi}_c, \A}$ does not intervene in $L^2_{\mathrm{disc}, -}$.
		\item If $\prod_v \epsilon_v = 1$, then $\omega^\epsilon_{\dot{\bpsi}_c, \A}$ intervenes with multiplicity $1$; in fact, it appears in $L^2_{\dot{\psi}}$.
	\end{enumerate}
\end{proposition}
\begin{proof}
	Combine Proposition \ref{prop:principal}, Lemma \ref{prop:principal-global-epsilon}, Theorem \ref{prop:global-multiplicity} and its Corollary \ref{prop:global-mf}, by noting that principal Arthur packets are multiplicity-free over each $\dot{F}_v$.
\end{proof}

It remains to describe the automorphic realization of $\omega^\epsilon_{\dot{\bpsi}_c, \A}$ explicitly when $\prod_v \epsilon_v = 1$.

Consider the $\vartheta$-series map \eqref{eqn:theta-series}. Its restriction to $\omega^\epsilon_{\dot{\bpsi}, \A}$ is known to be nonzero if and only if $\prod_v \epsilon_v = 1$. This also holds if we use $\dot{\bpsi}_c$ instead.

\begin{proposition}\label{prop:principal-theta-series}
	Given $c \in \dot{F}^{\times}$, let $\dot{\zeta} = \dot{\zeta}_c$ be the quadratic character of $\dot{F}^{\times} \backslash \A^{\times}$ determined by $c \dot{F}^{\times 2}$ and set $\dot{\psi} := \dot{\zeta} \boxtimes r(2n)$. If $\epsilon \in \bigoplus_v \bmu_2$ satisfies $\prod_v \epsilon_v = 1$, then the corresponding irreducible constituent of $L^2_{\dot{\psi}}$ (see Proposition \ref{prop:principal-global-mult}) equals $\vartheta_{\bpsi_c}( \omega^\epsilon_{\bpsi_c, \A})$.
\end{proposition}
\begin{proof}
	In view of the multiplicity-one property in Proposition \ref{prop:principal-global-mult} (ii), it suffices to show that every member of $\vartheta_{\bpsi_c}( \omega^\epsilon_{\bpsi_c, \A})$ is $L^2$.
	
	This property is established in the first part of the proof of \cite[Theorem 2.5]{Ho81}. What one considers in \textit{loc.\ cit.}\ is the $\vartheta$-series on $\Mp(\dot{W}, \A)$ from a reductive dual pair $(\Or(\dot{V}), \Sp(\dot{W}))$ inside $\Sp(\dot{V} \otimes \dot{W})$ such that $k := \dim \dot{V}$ is even and $k \leq n$; our case corresponds to $k = 1$. Nonetheless, the estimates in \textit{loc.\ cit.}\ including (2.25) still work: in \textit{loc.\ cit.}, the (2.25) follows from (2.24), $|\gamma(d)| = 1$ and Schwartz--Bruhat decay, and (2.29) then gives $L^2$-integrability for $k \leq m$. This applies with $m = n$ and $k = 1$; odd $k$ merely places the $\vartheta$-series on the metaplectic cover.
\end{proof}

In short, the part of the genuine $L^2$-automorphic spectrum indexed by principal parameters is given by elementary $\vartheta$-series.

\subsection{The case \texorpdfstring{$n=1$}{n=1} (Waldspurger)}\label{sec:Waldspurger}
Let $n=1$. In the local setting, $\Psi(\tilde{G})$ is the union of $\Phi_{\mathrm{bdd}}(\tilde{G})$ and the subset of principal parameters (Definition \ref{def:principal}).
\begin{itemize}
	\item For $\phi \in \Phi_{\mathrm{bdd}}(\tilde{G})$ and $\chi \in \EuScript{S}_\phi^\vee$, we know that $\pi_{\phi, \chi}$ comes from $\SO(3) \simeq \PGL(2)$ or its non-split inner form via $\Theta$-lifting, according to whether $\chi$ is trivial or not, see Theorem \ref{prop:Luo}.
	\item For principal parameters $\psi = \zeta \boxtimes r(2)$, we obtain $\pi_{\psi, \pm} = \omega_{\bpsi_c}^{\pm}$ where $c F^{\times 2}$ corresponds to $\zeta$ (Proposition \ref{prop:principal}).
\end{itemize}

This recovers the local results of Waldspurger \cite{Wa80, Wa91}; see also \cite{GanP}.

In the global setting, $\dot{\Psi}_2(\dot{\tilde{G}})$ is the union of irreducible cuspidal automorphic representations of $\PGL(2, \A)$ and the subset of principal parameters $\dot{\zeta} \boxtimes r(2)$.

Recall that principal local Arthur packets are multiplicity-free (Remark \ref{rem:principal-mf}). Hence $L^2_{\mathrm{disc}, -}$ is multiplicity-free by Corollary \ref{prop:global-mf}.

Moreover, Arthur's sign character $\epsilon_{\dot{\psi}}^{\mathrm{Art}}$ is trivial as $\overline{\EuScript{S}}_{\dot{\psi}} = \{1\}$ when $n=1$. Therefore, the global multiplicity formula in Theorem \ref{prop:global-multiplicity} recovers global results of Waldspurger in \textit{loc.\ cit.}.

Suppose that $\epsilon \in \bigoplus_v \bmu_2$ satisfies $\prod_v \epsilon_v = 1$. For principal parameters $\dot{\zeta} \boxtimes r(2)$ as above, Proposition \ref{prop:principal-theta-series} implies that the irreducible constituent of $L^2_{\dot{\psi}}$ indexed by $\epsilon$ is concretely given by elementary $\vartheta$-series, namely $\vartheta_{\dot{\bpsi}_c}(\omega^\epsilon_{\dot{\bpsi}_c, \A})$. This is also contained in Waldspurger's works; cf.\ \cite[\S 3.4]{GanP}.

\subsection{The case \texorpdfstring{$n=2$}{n=2} (Gan--Ichino)}\label{sec:GI}
Here we work in the local setting, and assume $\dim W = 4$. In \cite{GI21}, Gan and Ichino attached to each pair $(\psi, \chi)$ a linear combination $\pi^{\mathrm{GI}}_{\psi, \chi}$ of genuine irreducible characters on $\tilde{G}$ with coefficients in $\Z_{\geq 0}$, where $\psi \in \Psi(\tilde{G})$ and $\chi \in \EuScript{S}_\psi^\vee$.
\index{pi-psi-chi-GI@$\pi^{\mathrm{GI}}_{\psi, \chi}$}

Moreover, they constructed $\pi^{\mathrm{GI}}_{\psi, \chi}$ for parameters of the form
\[ \psi = (\rho \boxtimes r(1)) \oplus (\zeta \boxtimes r(2)) \; \in \Psi^+(\tilde{G}) \]
where $\zeta$ is a quadratic character and $\rho$ is a $2$-dimensional symplectic representation of $\mathcal{L}_F$ that is \emph{almost tempered}, see \cite[\S 5]{GI18}. This accommodates the local components of global parameters for the discrete genuine $L^2$-automorphic spectrum; in \textit{loc.\ cit.}, Gan and Ichino also considered L-parameters whose $\mathcal{L}_F$-components are almost tempered and constructed the corresponding packets, but these are not needed here. In fact, this is the only place where the almost-tempered extension is needed in order to treat local components of non-tempered global Arthur parameters.
\index{almost tempered}

By collecting the irreducibles in various $\pi^{\mathrm{GI}}_{\psi, \chi}$ as in \S\ref{sec:A-packets}, we also obtain the Arthur packet $\Pi^{\mathrm{GI}}_\psi$ fibered over $\Pi_-(\tilde{G})$ and $\EuScript{S}_\psi^\vee$.
\index{Pi-psi-GI@$\Pi^{\mathrm{GI}}_\psi$}

\begin{theorem}\label{prop:Mp4}
	For all $(\psi, \chi)$ where $\psi$ is either in $\Psi(\tilde{G})$ or is of the form displayed above, we have $\pi^{\mathrm{GI}}_{\psi, \chi} = \pi_{\psi, \chi}$. Moreover, $\Pi^{\mathrm{GI}}_\psi = \Pi_\psi$ and $\Pi_\psi$ is multiplicity-free.
\end{theorem}

To prove this, several preparatory steps are in order.

\begin{lemma}\label{prop:GI-local}
	The formation of $\pi^{\mathrm{GI}}_{\psi, \chi}$ has the following properties.
	\begin{enumerate}[(i)]
		\item If $\phi \in \Phi_{\mathrm{bdd}}(\tilde{G})$ then $\pi^{\mathrm{GI}}_{\phi, \chi} = \pi_{\phi, \chi}$ for all $\chi \in \EuScript{S}_\phi^\vee$.
		
		\item It reduces to the case of $\psi \in \Psi_{\mathrm{gp}}(\tilde{G})$ by parabolic induction, in the same manner as Proposition \ref{prop:pi-psi-gp}.
		
		\item In the unramified situation, suppose that $\psi$ is also unramified. Then
		\[ \pi^{\mathrm{GI}}_{\psi, \chi}(f_K) = \begin{cases}
			1, & \text{if}\; \chi = \mathbf{1}, \\
			0, & \text{otherwise,}
		\end{cases}\]
		where $f_K$ is the unit of $\mathcal{H}_{\asp}(K \backslash \tilde{G} / K)$.
		
		\item The packets $\Pi^{\mathrm{GI}}_\psi$ are multiplicity-free.
		
		\item For $\psi \in \Psi(\tilde{G})$, the L-packets $\Pi_{\phi_\psi}$ embed into $\Pi^{\mathrm{GI}}_\psi$ so that the diagram
		\[\begin{tikzcd}
			& \Pi_-(\tilde{G}) \\
			\Pi_{\phi_\psi} \arrow[hookrightarrow, r] \arrow[d, "{1:1}"'] \arrow[ru, "\text{incl.}"] & \Pi^{\mathrm{GI}}_\psi \arrow[d] \arrow[u] \\
			\EuScript{S}_{\phi_\psi}^\vee \arrow[hookrightarrow, r] & \EuScript{S}_\psi^\vee
		\end{tikzcd}\]
		commutes.
	\end{enumerate}
\end{lemma}
\begin{proof}
	(i) follows by combining \cite[\S 2.2]{GI21} and Theorem \ref{prop:Luo} (iii).
	
	For (ii), it suffices to consider $\psi \notin \Phi_{\mathrm{bdd}}(\tilde{G})$. It almost suffices to inspect the table in \cite[Appendix C]{GI21}. Nonetheless, the case of Saito--Kurokawa parameters requires some care, since one has to include almost tempered summands as in \cite[\S 8.1]{GI21}, namely we shall consider
	\[ \psi = \zeta |\cdot|_F^s \boxtimes r(1) \boxtimes r(1) \oplus \zeta^{-1} |\cdot|_F^{-s} \boxtimes r(1) \boxtimes r(1) \oplus \chi_a \boxtimes r(1) \boxtimes r(2), \]
	where $\zeta$ is a unitary character, $\chi_a$ is the quadratic character determined by $a \in F^{\times}$ modulo $F^{\times 2}$, and $0 \leq s < \frac{1}{2}$. We have $\EuScript{S}_\psi^\vee \simeq \bmu_2$.

	Given such a parameter, let $\tilde{P}_1$ be the standard parabolic subgroup with Levi factor $\tilde{M}_1 := \GL(1, F) \times \Mp(W_1)$, where $W_1$ is a $2$-dimensional symplectic subspace of $W$, then $\psi$ is induced from $\Phi(1) \times \Psi_{\mathrm{gp}}(\Mp(W_1))$. It is shown at the end of \cite[\S 8.1]{GI21} (in the $+$ case by looking at $\phi_\psi$) and \cite[Lemmas C6 (iii) and C8 (ii)]{GI21} (in the $-$ case) that $\pi_{\psi, \pm}$ is an irreducible subquotient of $I_{\tilde{P}_1}\left( \zeta |\cdot|_F^s \boxtimes \omega^\pm_{W_1, \bpsi_a} \right)$. The case $F = \CC$ is omitted in \textit{loc.\ cit.}, but the proof is similar to that in Lemma A10 except that one employs the complex version of induction principle \cite[Corollary 3.21]{AB95} for $\Theta$-correspondence instead.
	
	In view of \S\ref{sec:Waldspurger}, it suffices to show $I_{\tilde{P}_1}\left( \zeta |\cdot|_F^s \boxtimes \omega^\pm_{W_1, \bpsi_a} \right)$ is irreducible. This is stated as a separate Lemma \ref{prop:Mp4-irred} below.
	
	For (iii), consider the table in \cite[C.1.4]{GI21}. After reduction to good parity by (ii), we are in the cases of lines 1, 3, 4, 5 in the first column of \textit{loc.\ cit.} It remains to
	\begin{itemize}
		\item show the spherical vector survives in Langlands quotients using the Gindikin--Karpelevich formula (see, e.g., \cite[Theorem 6.4]{Mc11}) as in the linear case,
		\item observe that $\omega^-_{W_1, \bpsi_a}$, $\pi_\varphi^{+, +}$, $\pi_\varphi^{-,-}$, $\mathrm{st}_{\chi_a, \bpsi}$ (notation of \textit{loc.\ cit.}) are not unramified, where $a \in \mathfrak{o}_F^{\times}$.
	\end{itemize}
	
	The assertions (iv) and (v) are in \cite[\S 2.2]{GI21}.
\end{proof}

\begin{lemma}\label{prop:Mp4-irred}
	Let $\tilde{P}_1 = \tilde{M}_1 U_1(F)$ be the standard parabolic subgroup of $\tilde{G}$ with $\tilde{M}_1 = \GL(1, F) \times \Mp(W_1)$. Let $\zeta$ be a unitary character of $F^{\times}$ and $\chi_a$ be the quadratic character determined by $a \in F^{\times}$ modulo $F^{\times 2}$. Then $I_{\tilde{P}_1}\left( \zeta |\cdot|_F^s \boxtimes \omega^\pm_{W_1, \bpsi_a}\right)$ is irreducible when $0 \leq s < \frac{1}{2}$.
\end{lemma}
\begin{proof}
	First off, note that $\omega_{W_1, \bpsi_a}^-$ is square-integrable. For non-Archimedean $F$, one can appeal to either \cite[Lemma C1]{GI21} or \cite{HM10} to conclude irreducibility.
	
	For $F = \CC$, one may employ \cite[Théorème 6.6 + Remarque 6.7]{MR17}, which applies to the complex-metaplectic setting as well.
	
	Assume $F = \R$ in what follows. The proof of \cite[Lemma A10]{GI21} contains the irreducibility in the $-$ case. We will apply \cite[Proposition 8.9]{PPSR10} to the parabolic subgroup $\tilde{P}_1$ to settle the $+$ case.
	
	To verify the conditions in the cited result, identify $\mathfrak{t}^*$ with $\R^2$ and let $\epsilon_1, \epsilon_2$ be the standard basis. Express $\zeta  \boxtimes \omega^+_{W_1, \bpsi_a}$ as a Langlands quotient induced from $\tilde{T}$. If we write
	\[ \tilde{T} \simeq \tilde{T}^1 \times (\R_{> 0}^{\times})^2, \]
	where $\tilde{T}^1$ is the maximal compact subgroup of $\tilde{T}$, then the inducing data alluded to above is expressed as $(\delta, \nu_T)$ where $\delta$ is a genuine character of $\tilde{T}^1$ and $\nu_T$ is the positive twisting parameter.
	
	The datum $\nu$ in \textit{loc.\ cit.}\ is $(s, \frac{1}{2}) \in \mathfrak{t}^*$. The cited condition for irreducibility is that $\lrangle{\nu, \beta} \notin \{1, 3, 5, \ldots\}$ for all $\beta \in \Delta(\mathfrak{u}_1)$ that are good for $\delta$, and $\lrangle{\nu, \beta} \notin \{2, 4, 6, \ldots\}$ for all $\beta \in \Delta(\mathfrak{u}_1)$ that are bad for $\delta$. Here 
	\begin{itemize}
		\item $\lrangle{\cdot, \cdot}$ is the standard inner product on $\R^2 \simeq \mathfrak{t}^*$,
		\item $\Delta(\mathfrak{u}_1) = \{ 2\epsilon_1$, $\epsilon_1 \pm \epsilon_2 \}$ is the set of positive roots appearing in $\mathfrak{u}_1$;
		\item good and bad roots for $\delta$ are defined in \S 3.1 and \S 3.4 of \textit{loc.\ cit.}\ through elements $m_\beta \in \tilde{M}_1$: we say $\beta$ is good for $\delta$ if $\delta(m_\beta) \neq -1$, otherwise bad.
	\end{itemize}
	
	In fact, $m_\beta$ is the element denoted by $\gamma$ in \cite[\S 8.1.2]{Li26}, where a precise description can be found. We have $\lrangle{\nu, \beta} \notin \Z_{\geq 1}$ since $0 \leq s < \frac{1}{2}$, so the goodness is irrelevant here.
\end{proof}

Global inputs are also needed. Let $\dot{F}$ be a number field and consider the adélic metaplectic covering $\dot{\tilde{G}} \to \dot{G}(\A)$ associated with $(\dot{W}, \lrangle{\cdot|\cdot})$ and $\dot{\bpsi} = \prod_v \bpsi_v$, where $\dim \dot{W} = 4$. In \cite[Theorem 2.1]{GI21} is given a multiplicity formula for $L^2_{\dot{\psi}} \subset L^2_{\mathrm{disc}, -}$, where $\dot{\psi} \in \dot{\Psi}_2(\dot{\tilde{G}})$. It takes the same form as Theorem \ref{prop:global-multiplicity}, except that
\begin{itemize}
	\item $\pi^{\mathrm{GI}}_{\dot{\psi}_v, \chi_v}$ replaces $\pi_{\dot{\psi}_v, \chi_v}$ at each place $v \in V$;
	\item a priori, another element $\tilde{\epsilon}_{\dot{\psi}}$ of $\EuScript{S}_{\dot{\psi}}^\vee$ replaces $\epsilon_{\dot{\psi}}$.
\end{itemize}

\begin{lemma}\label{prop:GI-epsilon}
	We have $\tilde{\epsilon}_{\dot{\psi}} = \epsilon_{\dot{\psi}}$ for all $\dot{\psi} \in \dot{\Psi}_2(\dot{\tilde{G}})$.
\end{lemma}
\begin{proof}
	In $\epsilon_{\dot{\psi}} = \epsilon^{\mathrm{Art}}_{\dot{\psi}} \nu_{\dot{\psi}}$ (see Theorem \ref{prop:global-multiplicity}), the latter factor has a simple description via global root numbers. From the formulas in \cite[\S 2.1]{GI21}, we see that $\nu_{\dot{\psi}} = \tilde{\epsilon}_{\dot{\psi}}$ except for Saito--Kurokawa parameters, in which case $\tilde{\epsilon}_{\dot{\psi}}$ includes an extra Rankin--Selberg $\epsilon$-factor.
	
	It remains to identify $\epsilon^{\mathrm{Art}}_{\dot{\psi}}$.	For this purpose, one can either compute $\epsilon^{\mathrm{Art}}_{\dot{\psi}}$ by unwinding definitions in \cite{Ar13} or consult \cite[pp.78--80]{Ar04}. In the latter reference, the group is $\GSp(4)$ whose dual is $\GSp(4, \CC)$, whilst $\dot{\tilde{G}}^\vee = \Sp(4, \CC)$. Nevertheless, $\epsilon^{\mathrm{Art}}_{\dot{\psi}}$ does not change after passage to $\GSp(4, \CC)$ since it is defined by adjoint actions on Lie algebras. Specifically, $\epsilon^{\mathrm{Art}}_{\dot{\psi}}$ is non-trivial only for Saito--Kurokawa parameters, in which case it is exactly the extra Rankin--Selberg $\epsilon$-factor in $\tilde{\epsilon}_{\dot{\psi}}$.
\end{proof}

Now return to the local setting.

\begin{lemma}\label{prop:Mp4-lemma}
	Let $F$ be non-Archimedean. For all $\psi \in \Psi(\tilde{G})^{\star}$ and $\chi \in \EuScript{S}_\psi^\vee$, we have $\pi^{\mathrm{GI}}_{\psi, \chi} = \pi_{\psi, \chi}$.
\end{lemma}

The proof of the Lemma is deferred to \S\ref{sec:proof-Mp4}.

\begin{proof}[Proof of Theorem \ref{prop:Mp4}]
	Let us inspect the proof of Theorem \ref{prop:local-desiderata} in \S\S\ref{sec:proof-0}--\ref{sec:proof-2}. The strategy there is to globalize the situation, and realize $\pi_{\psi, \chi}$ as a sum of characters of local components of certain irreducibles $\dot{\pi} \subset L^2_{\mathrm{disc}, -}$. Here, the only difference is that
	\begin{itemize}
		\item we shall consider $\pi^{\mathrm{GI}}_{\psi, \chi}$ instead of $\pi_{\psi, \chi}$;
		\item accordingly, we shall use the multiplicity formula from \cite{GI21}, of the same form as Theorem \ref{prop:global-multiplicity} by Lemma \ref{prop:GI-epsilon};
		\item the unramified places are taken care of by Lemma \ref{prop:GI-local} (iii);
	\end{itemize}
	
	Begin with the ``real case in general position'' (Case 1) in \S\ref{sec:globalization-lemma}. In the proof given in \S\ref{sec:proof-0}, we globalize $\tilde{G}$ to $\dot{\tilde{G}}$ over $\dot{F} = \Q$, and globalize the parameter $\psi$ to $\dot{\psi}$. Here we obtain the same multiplicity formula \eqref{eqn:multiplicity-formula-V'}, except that $\pi_{\dot{\psi}_v, \chi_v}$ is replaced by $\pi^{\mathrm{GI}}_{\dot{\psi}_v, \chi_v}$.
	
	The characters $\chi_v$ here are taken in the same way for $v \neq u$. We claim that the test functions $f_v$ for $v \neq u$ can also be chosen alike. Indeed:
	\begin{itemize}
		\item at the places $v \in V' \smallsetminus (\{u\} \cup V_0)$, we may choose $f_v$ and $\pi_v$ with the same separation property, in view of Lemma \ref{prop:GI-local} (i), (iv) and (v);
		\item ditto for $v \in V_0$, by using Lemma \ref{prop:Mp4-lemma}.
	\end{itemize}
	All in all, we conclude as in \S\ref{sec:proof-1} that
	\[ \sum_{\dot{\pi} \in A} \Theta_{\dot{\pi}_u} = \pi^{\mathrm{GI}}_{\psi, \chi}. \]
	A comparison with \eqref{eqn:proof-1-aux}, noting that the indexing set $A$ is the same, gives $\pi^{\mathrm{GI}}_{\psi, \chi} = \pi_{\psi, \chi}$.
	
	Next, consider the ``general case'' (Case 2) in \S\ref{sec:globalization-lemma}. In this case, $F \neq \CC$ and $\psi \in \Psi_{\mathrm{gp}}(\tilde{G})$. In the proof given in \S\ref{sec:proof-1}, one uses the same globalization argument. The only difference in this case is that we have some auxiliary places $v \mid \infty$. At such places $v$ with $v \neq u$, we may take the same $\chi_v$, $\pi_v$ and $f_v$ as in \S\ref{sec:proof-1} by the previous case. Again, $\sum_{\dot{\pi} \in A} \Theta_{\dot{\pi}_u} = \pi^{\mathrm{GI}}_{\psi, \chi}$ with the same $A$ as in \eqref{eqn:proof-1-aux}. Hence $\pi^{\mathrm{GI}}_{\psi, \chi} = \pi_{\psi, \chi}$.
	
	Assuming $F \neq \CC$, consider either a general $\psi \in \Psi(\tilde{G})$, or $\psi = (\rho \boxtimes r(1)) \boxplus (\zeta \boxtimes r(2)) \in \Psi^+(\tilde{G})$ that is almost-tempered. Lemma \ref{prop:GI-local} (ii) reduces the problem to the previous case by parabolic induction.

	Consider the final case $F = \CC$. Reduce to good parity as before, then repeat the proof in \S\ref{sec:proof-2} which proceeds by globalization to an imaginary quadratic field $\dot{F}$. The auxiliary places $v \neq u$ in the argument are either real or non-Archimedean, so $\pi^{\mathrm{GI}}_{\dot{\psi}_v, \chi_v} = \pi_{\dot{\psi}_v, \chi_v}$ is known. The same arguments then lead to
	\[ \sum_{\dot{\pi} \in A} \Theta_{\dot{\pi}_u} = \pi^{\mathrm{GI}}_{\psi, \chi}. \]
	Ditto for $\pi_{\psi, \chi}$. Hence $\pi^{\mathrm{GI}}_{\psi, \chi} = \pi_{\psi, \chi}$.
	
	It is now clear that $\Pi_\psi = \Pi^{\mathrm{GI}}_\psi$. Lemma \ref{prop:GI-local} (iv) implies $\Pi_\psi$ is multiplicity-free.
\end{proof}

\subsection{Proof of Lemma \texorpdfstring{\ref{prop:Mp4-lemma}}{10.4.5}}\label{sec:proof-Mp4}
Let $F$ be non-Archimedean and $\psi \in \Psi(\tilde{G})^\star$; in particular, $\psi = \hat{\phi}$ for some $\phi \in \Phi_{\mathrm{bdd}}(\tilde{G})$, and $\phi$ is also the composition of $\psi$ with the diagonal embedding of $\SL(2, \CC)$.

The tables in \cite[C.1.3 and C.1.4]{GI21} describe $\pi^{\mathrm{GI}}_{\psi, \chi}$ explicitly: it is either an irreducible character or zero. Therefore, we will treat every $\pi^{\mathrm{GI}}_{\psi, \chi} \neq 0$ as a genuine irreducible representation up to isomorphism.

We also identify $\EuScript{S}_\psi^\vee$ with a power of $\bmu_2$, and use the notation $\pi^{\mathrm{GI}}_{\psi, \pm}$, etc. We fix $\bpsi$ and work over $\tilde{G}$ to split the $\GL$ factors of Levi subgroups, but will implicitly pass to twofold coverings (see Proposition \ref{prop:MMp-dist}) when we need to import Weil representations attached to $\bpsi_a$ for various $a \in F^{\times}$.

We will consider each relevant case in the table of \cite[C.1.4]{GI21}. Let us begin with the principal parameters. Denote by $\chi_a$ the quadratic character of $\Weil{F}$ determined by the square class $a F^{\times 2}$, also viewed as a quadratic character of $F^{\times}$.

\begin{lemma}\label{prop:Mp4-proof-0}
	Let $\psi = \chi_a \boxtimes r(1) \boxtimes r(4) \in \Psi(\tilde{G})$, so that $\EuScript{S}_\psi^\vee \simeq \bmu_2$. Then $\pi^{\mathrm{GI}}_{\psi, \pm} = \pi_{\psi, \pm}$.
\end{lemma}
\begin{proof}
	It is easy to deduce from \textit{loc.\ cit.}\ that $\pi^{\mathrm{GI}}_{\psi, \pm} = \omega_{\bpsi_a}^{\pm}$ by expressing $\omega_{\bpsi_a}^{\pm}$ as Langlands quotients. We conclude by Proposition \ref{prop:principal}.
\end{proof}

To settle the remaining cases, we shall determine $\pi^{\mathrm{GI}, \wedge}_{\psi, \chi}$, where $(\cdot)^\wedge$ is the Aubert--Zelevinsky involution on the level of genuine representations. In what follows,
\begin{itemize}
	\item $(W_1, \lrangle{\cdot|\cdot})$ means a $2$-dimensional symplectic subspace in $W_2 := W$, and $\omega^{\pm}_{W_1, \bpsi}$ are the even and odd Weil representations on $\Mp(W_1)$ relative to $\bpsi$;
	\item $V_1^\pm$ is the quadratic $F$-vector space of dimension $3$, discriminant $1$ and Hasse invariant $\pm 1$;
	\item $\nu_c: \SO(V_1^\pm) \to \{1, -1\}$ is the spinor norm composed with $\chi_c$, for all $c \in F^{\times}$;
	\item for every representation $\sigma$ of $\SO(V_1^\pm)$ and $\epsilon \in \{1, -1\}$, let $\sigma^\epsilon$ be its extension to $\Or(V_1^\pm)$ such that $\sigma^\epsilon(-1) = \epsilon$;
	\item for every semi-standard parabolic subgroup $P$ of $G$, denote by $P^-$ its opposite;
	\item $I_{\tilde{P}}(\cdots) \twoheadrightarrow J_{\tilde{P}}(\cdots)$ means the Langlands quotient;
	\item $P_1$ (resp.\ $P_2$) is the standard parabolic subgroup of $G$ with Levi factor $\GL(1) \times \Sp(W_1)$ (resp.\ $\GL(2)$);
	\index{JP@$J_{\tilde{P}}$}
	\item $|\cdot| := |\cdot|_F$.
\end{itemize}
We refer to \textit{loc.\ cit.}\ for other unexplained notation for $\Theta$-lifts, Steinberg representations, etc.

\begin{lemma}\label{prop:Mp4-proof-1}
	Let $\psi = (\rho_0 \boxtimes r(1) \boxtimes r(1)) \oplus (\chi_a \boxtimes r(1) \boxtimes r(2)) \in \Psi(\tilde{G})^\star$, where $\rho_0$ is a $2$-dimensional symplectic irreducible representation of $\Weil{F}$, so that $\EuScript{S}_\psi^\vee \simeq \bmu_2^2$. Then $\pi^{\mathrm{GI}, \wedge}_{\psi, \epsilon, \eta} = \pi_{\phi, \epsilon, \eta}$ for all $\epsilon, \eta \in \{\pm\}$.
\end{lemma}
\begin{proof}
	The packet of $\Mp(W_1)$ (resp.\ $\SO(V_1^\pm)$) determined by $\rho_0$ has members $\pi_{0, \pm}$ (resp.\ $\sigma_{0, \pm}$). We have
	\begin{align*}
		\pi^{\mathrm{GI}}_{\psi, \pm, +} & = J_{\tilde{P}_1}\left( \chi_a |\cdot|^{1/2} \boxtimes \pi_{0, \pm} \right), & \pi_{\phi, \pm, +} & = \widetilde{\mathrm{St}}_{\bpsi}(\chi_a, \pi_{0, \pm}), \\
		\pi^{\mathrm{GI}}_{\psi, \pm, -} & = \pi_{\phi, \pm, -}, & \pi_{\phi, \pm, -} & = \Theta_{W_2, V_1^{\pm \epsilon_a}, \bpsi_a}\left( (\sigma_{0, \pm\epsilon_a} \otimes \nu_a)^{\mp\epsilon_0 \chi_a(-1)} \right),
	\end{align*}
	where $\epsilon_0 = \epsilon(\rho_0)$ and $\epsilon_a = \epsilon_0 \epsilon(\rho_0 \otimes \chi_a) \chi_a(-1)$.

	First, $\pi^{\mathrm{GI}, \wedge}_{\psi, \epsilon, +} = \pi_{\phi, \epsilon, +}$ will follow from the short exact sequence for Steinberg representation once we know $\pi_{0, \epsilon}$ is supercuspidal. To prove this, it suffices to show $\sigma_{0, \epsilon}$ is supercuspidal when $\epsilon = +$ and is of $\dim > 1$ if $\epsilon = -$ (see \cite{Wa91} or \cite[\S 2.17]{GanP}); this holds since $\rho_0$ is trivial on $\SL(2, \CC)$.

	Secondly, $\pi^{\mathrm{GI}, \wedge}_{\psi, \pm, -} = \pi_{\phi, \pm, -}$ will follow once we know $\pi_{\phi, \pm, -}$ is supercuspidal. By the tower property, it suffices to show $\Theta_{W_1, V_1^{\pm \epsilon_a}, \bpsi_a}\left( (\sigma_{0, \pm\epsilon_a} \otimes \nu_a)^{\mp\epsilon_0 \chi_a(-1)} \right) = 0$. To conclude, apply \cite[(A.2)]{GI21} to see that the signs force the $\Theta$-lift to $\Mp(W_1)$ to vanish.
\end{proof}

\begin{lemma}
	Let $\psi = (\chi_a \boxtimes r(1) \boxtimes r(2)) \oplus (\chi_b \boxtimes r(1) \boxtimes r(2)) \in \Psi(\tilde{G})^\star$, where $a, b$ are distinct modulo $F^{\times 2}$, so that $\EuScript{S}_\psi^\vee \simeq \bmu_2^2$. Then $\pi^{\mathrm{GI}, \wedge}_{\psi, \epsilon, \eta} = \pi_{\phi, \epsilon, \eta}$ for all $\epsilon, \eta \in \{\pm\}$.
\end{lemma}
\begin{proof}
	We have
	\begin{align*}
		\pi^{\mathrm{GI}}_{\psi, +, +} & = J_{\tilde{B}}\left( \chi_a |\cdot|^{1/2} \boxtimes \chi_b |\cdot|^{1/2} \right), & \pi_{\phi, +, +} & = \widetilde{\mathrm{St}}_{\bpsi}\left( \chi_a, \widetilde{\mathrm{st}}_{\chi_b, \bpsi} \right), \\
		\pi^{\mathrm{GI}}_{\psi, +, -} & = J_{\tilde{P}_1}\left( \chi_a |\cdot|^{1/2} \boxtimes \omega^-_{W_1, \bpsi_b} \right), & \pi_{\phi, +, -} & = \widetilde{\mathrm{St}}_{\bpsi}\left( \chi_a, \omega^-_{W_1, \bpsi_b} \right), \\
		\pi^{\mathrm{GI}}_{\psi, -, +} & = J_{\tilde{P}_1}\left( \chi_b |\cdot|^{1/2} \boxtimes \omega^-_{W_1, \bpsi_a} \right), & \pi_{\phi, -, +} & = \widetilde{\mathrm{St}}_{\bpsi}\left( \chi_b, \omega^-_{W_1, \bpsi_a} \right), \\
		\pi^{\mathrm{GI}}_{\psi, -, -} & = \pi_{\phi, -, -}, & \pi_{\phi, -, -} & = \Theta_{W_2, V_1^-, \bpsi_b}\left( \nu_{ab}^{\chi_{ab}(-1)} \right).
	\end{align*}
	Strictly speaking, $J_{\tilde{B}}\left( \chi_a |\cdot|^{1/2} \boxtimes \chi_b |\cdot|^{1/2} \right)$ should be understood as the Langlands quotient of the inducing datum $(\chi_a \times \chi_b) |\det|^{1/2}$ from $\tilde{P}_2$.

	Since $\omega^-_{W_1, \bpsi_c}$ is supercuspidal for all $c \in F^{\times}$, the cases $(+, -)$ and $(-, +)$ can be settled in a way similar to the first case in the proof of Lemma \ref{prop:Mp4-proof-1}.
	
	For the case $(-, -)$, it suffices to show $\pi_{\phi, -, -}$ is supercuspidal; by the tower property, it suffices to show $\Theta_{W_1, V_1^-, \bpsi_b}\left( \nu_{ab}^{\chi_{ab}(-1)} \right) = 0$. Indeed:
	\begin{itemize}
		\item the L-parameter for the character $\nu_{ab}$ of $\SO(V_1^-)$ is $\chi_{ab} \boxtimes r(2)$,
		\item the character $\chi_{ab} = \chi_a \chi_b$ is either ramified, or is unramified and maps Frobenius to $-1$,
	\end{itemize}
	thus \eqref{eqn:epsilon-SL} implies $\chi_{ab}(-1) \epsilon\left( \nu_{ab}, \bpsi \right) = 1$; now apply \cite[(A.2)]{GI21}.

	As for $(+, +)$, by \cite[Lemma C1 (ii)]{GI21} there is a short exact sequence
	\[ 0 \to \widetilde{\mathrm{St}}_{\bpsi}\left(\chi_a, \widetilde{\mathrm{st}}_{\chi_b, \bpsi} \right) \to I_{\tilde{P}_1}\left( \chi_a |\cdot|^{1/2} \boxtimes \widetilde{\mathrm{st}}_{\chi_b, \bpsi} \right) \to J_{\tilde{P}_1}\left( \chi_a |\cdot|^{1/2} \boxtimes \widetilde{\mathrm{st}}_{\chi_b, \bpsi} \right) \to 0. \]
	Commuting $(\cdot)^\wedge$ and parabolic induction as in \cite[(A.2.3)]{KMSW}, we get
	\[ 0 \to \widetilde{\mathrm{St}}_{\bpsi}\left(\chi_a, \widetilde{\mathrm{st}}_{\chi_b, \bpsi} \right)^\wedge \to I_{\tilde{P}_1^-}\left( \chi_a |\cdot|^{1/2} \boxtimes \omega^+_{W_1, \bpsi_b} \right) \to J_{\tilde{P}_1}\left( \chi_a |\cdot|^{1/2} \boxtimes \widetilde{\mathrm{st}}_{\chi_b, \bpsi} \right)^\wedge \to 0. \]

	The middle term is contained in $I_{\tilde{B}^-}\left( \chi_a |\cdot|^{1/2} \boxtimes \chi_b |\cdot|^{1/2} \right)$; the latter has a unique irreducible subobject, necessarily equal to $\widetilde{\mathrm{St}}_{\bpsi}\left(\chi_a, \widetilde{\mathrm{st}}_{\chi_b, \bpsi} \right)^\wedge$. On the other hand, $J_{\tilde{B}}\left( \chi_a |\cdot|^{1/2} \boxtimes \chi_b |\cdot|^{1/2} \right)$ is the image of the standard intertwining operator $I_{\tilde{B}}(\cdots) \to I_{\tilde{B}^-}(\cdots)$. Hence $\widetilde{\mathrm{St}}_{\bpsi}\left(\chi_a, \widetilde{\mathrm{st}}_{\chi_b, \bpsi} \right)^\wedge \simeq J_{\tilde{B}}\left( \chi_a |\cdot|^{1/2} \boxtimes \chi_b |\cdot|^{1/2} \right)$.
\end{proof}

\begin{lemma}
	Let $\psi = \rho \boxtimes r(1) \boxtimes r(2) \in \Psi(\tilde{G})$ where $\rho$ is a $2$-dimensional orthogonal irreducible representation of $\Weil{F}$, so that $\EuScript{S}_\psi^\vee \simeq \bmu_2$. Then $\pi^{\mathrm{GI}, \wedge}_{\psi, \epsilon} = \pi_{\phi, \epsilon}$ for all $\epsilon \in \{\pm\}$.
\end{lemma}
\begin{proof}
	Let $\tau$ be the supercuspidal representation of $\GL(2, F)$ parametrized by $\rho$. We have
	\begin{align*}
		\pi^{\mathrm{GI}}_{\psi, +} & = J_{\tilde{P}_2}\left( \tau|\det|^{1/2} \right), & \pi_{\phi, +} & = \widetilde{\mathrm{St}}_\bpsi(\tau), \\
		\pi^{\mathrm{GI}}_{\psi, -} & = \pi_{\phi, -}, & \pi_{\phi, -} & = \text{some supercuspidal}. 
	\end{align*}

	The case $\epsilon = -$ follows from supercuspidality.

	For the case $\epsilon = +$, by \cite[(A.2.3)]{KMSW}, the supercuspidality of $\tau$ and \cite[Lemma C2 (ii)]{GI21}, we see $\pi_{\phi, +}^\wedge \hookrightarrow I_{\tilde{P}_2^-}(\tau |\det|^{1/2})$. The inducing datum is $P_2^-$-negative, hence $\pi_{\phi, +}^\wedge$ is the unique irreducible subobject. But this subobject is also the image of the standard intertwining operator from $I_{\tilde{P}_2}(\tau |\det|^{1/2})$, hence $\pi_{\phi, +}^\wedge = J_{\tilde{P}_2}(\tau |\det|^{1/2})$.
\end{proof}

\begin{lemma}\label{prop:Mp4-proof-2}
	Let $\psi = 2(\chi_a \boxtimes r(1) \boxtimes r(2)) \in \Psi(\tilde{G})^\star$, so that $\EuScript{S}_\psi^\vee \simeq \bmu_2$. Then
	\begin{equation*}
		\pi^{\mathrm{GI}, \wedge}_{\psi, \pm} = \pi_{\phi, \pm}.
	\end{equation*}
\end{lemma}
\begin{proof}
	The arguments below are mainly based on \cite[\S 3.2.3 (b)]{HM10}. Denote by $R_{\tilde{P}}$ the Jacquet functor along a parabolic subgroup $\tilde{P} \subset \tilde{G}$, and by $r_{\tilde{P}}$ its effect on Grothendieck groups.

	Decompose $I_{\tilde{P}_2}(\chi_a \mathrm{St}_{\GL(2)})$ into $\tau_1 \oplus \tau_2$, where $\tau_1$ and $\tau_2$ are tempered irreducibles. They are taken so that $\tau_1, \tau_2$ correspond to $\rho_1, \rho_3$ in \cite[p.126]{HM10}, respectively. By the description of parabolically induced L-packets (see \cite[p.6]{Luo20}), we have $\Pi_\phi = \{\tau_1, \tau_2\}$.
	
	In \cite{HM10}, these constituents are shown to satisfy
	\begin{align*}
		r_{\tilde{P}_1} \left[ \tau_1 \right] & = \left[ \chi_a |\cdot|^{1/2} \boxtimes \omega^+_{W_1, \bpsi_a} \right], \\
		r_{\tilde{P}_1} \left[ \tau_2 \right] & = 2 \left[ \chi_a|\cdot|^{1/2} \boxtimes \widetilde{\mathrm{st}}_{\bpsi, \chi_a} \right] + \left[ \chi_a |\cdot|^{1/2} \boxtimes \omega^+_{W_1, \bpsi_a} \right].
	\end{align*}
	So $\tau_1 \not\simeq \tau_2$. The second equality implies
	\[ r_{\tilde{B}}\left[ \tau_2 \right] = 2 \left[ \chi_a |\cdot|^{1/2} \boxtimes \chi_a |\cdot|^{1/2} \right] + \left[ \chi_a |\cdot|^{1/2} \boxtimes \chi_a |\cdot|^{-1/2} \right], \]
	hence $(\tau_2)^\wedge \not\simeq \tau_2$ by looking at exponents. In \cite[Proposition 3.12]{HM10}, they are characterized with the terminology of p.112 of \textit{loc.\ cit.}\ as
	\begin{equation}\label{eqn:tau-12}
		\tau_1 = \lrangle{\chi_a |\cdot|^{1/2}; \omega^+_{W_1, \bpsi_a}}, \quad \tau_2 = \lrangle{\chi_a |\cdot|^{1/2}, \chi_a |\cdot|^{1/2}; \mathbf{1}}.
	\end{equation}

	The tables in \cite[C.1.4]{GI21} give
	\begin{align*}
		\pi^{\mathrm{GI}}_{\psi, +} & = J_{\tilde{B}}\left( \chi_a |\cdot|^{1/2} \boxtimes \chi_a |\cdot|^{1/2} \right), \\
		\pi^{\mathrm{GI}}_{\psi, -} & = J_{\tilde{P}_1}\left( \chi_a|\cdot|^{1/2} \boxtimes \widetilde{\mathrm{st}}_{\bpsi, \chi_a} \right);
	\end{align*}
	the first Langlands quotient should be understood as that of $(\chi_a \times \chi_a) |\det|^{1/2}$ from $\tilde{P}_2$. They are distinct and non-tempered.

	Let us show $(\tau_1)^\wedge \simeq \pi^{\mathrm{GI}}_{\psi, -}$. Apply the second adjunction to
	\[ \chi_a |\cdot|^{1/2} \boxtimes \omega^+_{W_1, \bpsi_a} \simeq R_{\tilde{P}_1}(\tau_1) \]
	to obtain a nonzero homomorphism $I_{\tilde{P}_1^-} \left( \chi_a|\cdot|^{1/2} \boxtimes \omega^+_{W_1, \bpsi_a} \right) \to \tau_1$. Applying $(\cdot)^\wedge$, we obtain a nonzero homomorphism $I_{\tilde{P}_1}\left( \chi_a|\cdot|^{1/2} \boxtimes \widetilde{\mathrm{st}}_{\bpsi, \chi_a} \right) \to (\tau_1)^\wedge$ which identifies $(\tau_1)^\wedge$ with the Langlands quotient.

	Next, let us show $(\tau_2)^\wedge \simeq \pi^{\mathrm{GI}}_{\psi, +}$. Indeed, $\tau_1$, $\tau_2$, $\pi^{\mathrm{GI}}_{\psi, +}$, $\pi^{\mathrm{GI}}_{\psi, -}$ are distinct and all have cuspidal support represented by $\chi_a |\cdot|^{1/2} \boxtimes \chi_a |\cdot|^{1/2}$. By \cite[Proposition 3.12]{HM10}, there are exactly $4$ irreducible representations with this support, permuted by $(\cdot)^\wedge$. Since $(\tau_1)^\wedge \simeq \pi^{\mathrm{GI}}_{\psi, -}$ and $(\tau_2)^\wedge \not\simeq \tau_2$, we must have $(\tau_2)^\wedge \simeq \pi^{\mathrm{GI}}_{\psi, +}$.

	We claim that $\tau_1$ is not $\bpsi$-generic. In view of Theorem \ref{prop:Luo} (iv) and $\Pi_\phi = \{\tau_1, \tau_2 \}$, this will imply $\tau_1 = \pi_{\phi, -}$ and $\tau_2 = \pi_{\phi, +}$, thus complete the proof.
	
	Embed $\tau_1$ into $I_{\tilde{P}_1}\left( \chi_a |\cdot|^{1/2} \boxtimes \omega^+_{W_1, \bpsi_a} \right)$ by \eqref{eqn:tau-12}. Note that $\psi \in \Psi(\tilde{G})^\star$ entails $\chi_a \neq \mathbf{1}$. By a comparison with the non-split short exact sequence
	\[ 0 \to (\text{non-generic irreducible}) \to I_{\tilde{P}_1}\left( \chi_a |\cdot|^{1/2} \boxtimes \omega^+_{W_1, \bpsi_a} \right) \to \pi^{\mathrm{GI}}_{\psi, +} \to 0 \]
	of \cite[Lemma C1 (iii)]{GI21}, we infer that $\tau_1$ is the non-generic constituent of $I_{\tilde{P}_2}(\chi_a \mathrm{St}_{\GL(2)})$.
\end{proof}

The discussions above exhaust the parameters in the table \cite[C.1.4]{GI21} which belong to $\Psi(\tilde{G})^\star$.

\begin{proof}[Proof of Lemma \ref{prop:Mp4-lemma}]
	We may assume $\psi \in \Psi(\tilde{G})^{\star}$ is non-tempered, and examine the aforementioned table. Lemma \ref{prop:Mp4-proof-0} settles the case of principal parameters. On the other hand, Corollary \ref{prop:anti-tempered-chi} and Lemma \ref{prop:tilde-mu-star} imply
	\[ \pi_{\psi, \chi} = (\pi_{\phi, \chi \mu_\psi})^\wedge. \]
	Thus it suffices to show $\pi^{\mathrm{GI}, \wedge}_{\psi, \chi} = \pi_{\phi, \chi\mu_\psi}$ for the remaining parameters. In view of Lemmas \ref{prop:Mp4-proof-1}--\ref{prop:Mp4-proof-2}, it suffices to show $\mu_\psi = \mathbf{1}$ for all $\psi \in \Psi(\tilde{G})^\star$.
	
	\index{Jord@$\mathrm{Jord}$}
	For every $\psi' \in \Psi_{\mathrm{gp}}(\tilde{G})$, let $\mathrm{Jord}(\psi')$ be the multi-set of triples $(\rho, a, b)$ in the decomposition of $\psi'$ into simple summands $\rho \boxtimes r(a) \boxtimes r(b)$. Recall from \S\ref{sec:anti-tempered-SO} that $\mu_\psi$ equals Xu's character $\epsilon^{\mathrm{M}/\mathrm{MW}}_\psi$ described in \cite[\S 1]{Xu21}. Specifically, as $\psi$ is anti-tempered of good parity, we shall take
	\begin{itemize}
		\item the auxiliary data $\upzeta_{1, b} := -1$ for all $(\rho, 1, b) \in \mathrm{Jord}(\psi)$;
		\item an admissible total order $>_\psi$ on $\mathrm{Jord}(\psi)$, whose condition simplifies to $b > b' \implies (\rho, 1, b) >_\psi (\rho, 1, b')$ in the anti-tempered case.
	\end{itemize}
	See \textit{loc.\ cit.}\ for a full explanation.

	\begin{itemize}
		\item For $\psi$ in any of the following cases
		\[
			\zeta \boxtimes r(1) \boxtimes r(4), \;
			(\zeta \boxtimes r(1) \boxtimes r(2)) \oplus (\zeta' \boxtimes r(1) \boxtimes r(2)), \;
			\rho \boxtimes r(1) \boxtimes r(2), \;
			2 \chi \boxtimes r(1) \boxtimes r(2),
		\]
		where $\chi$, $\zeta$ and $\zeta'$ are quadratic characters, $\zeta \neq \zeta'$, we have $\epsilon^{\mathrm{M}/\mathrm{MW}}_\psi = \mathbf{1}$ since in its definition \cite[p.1094]{Xu21}, we are always in the case ``$a + b$ is odd''.
		
		\item For $\psi = (\rho_0 \boxtimes r(1) \boxtimes r(1)) \oplus (\zeta \boxtimes r(1) \boxtimes r(2))$, we have $\epsilon^{\mathrm{M}/\mathrm{MW}}_\psi = \mathbf{1}$. Indeed, $(\rho_0, 1, 1) <_\psi (\zeta, 1, 2)$; in the notation of \textit{loc.\ cit.}\ we have $\epsilon^{\mathrm{M}/\mathrm{MW}}_\psi(\rho_0, 1, 1) = (-1)^{m+n}$ where $m = n = 0$, and $\epsilon^{\mathrm{M}/\mathrm{MW}}_\psi(\zeta, 1, 2) = 1$ since $1 + 2$ is odd.
	\end{itemize}
	This completes the proof of Lemma \ref{prop:Mp4-lemma}.
\end{proof}

\begin{remark}
	\index{epsilon-M-W-psi@$\epsilon^{\mathrm{M}/\mathrm{W}}_\psi$}
	The character $\epsilon^{\mathrm{M}/\mathrm{W}}_\psi = \epsilon^{\mathrm{MW}/\mathrm{W}}_\psi \epsilon^{\mathrm{M}/\mathrm{MW}}_\psi$ in \cite[(1.4)]{Xu21} also plays an important role in applications. For anti-tempered $\psi \in \Psi_{\mathrm{gp}}(\tilde{G})$, we have $\epsilon^{\mathrm{MW}/\mathrm{W}}_\psi = \mathbf{1}$. Indeed, the set $\mathcal{Z}_{\mathrm{MW}/\mathrm{W}}(\psi)$ in \cite[p.1094]{Xu21} is empty since $a = a' = 1$ for all $(\rho, a, b), (\rho', a', b') \in \mathrm{Jord}(\psi)$, hence $\epsilon^{\mathrm{MW}/\mathrm{W}}_\psi = \mathbf{1}$ by \textit{loc.\ cit.} This holds for all $n$. 
\end{remark}

\printindex

\printbibliography[heading=bibintoc]

\vspace{1em}
\begin{flushleft} \small
	W.-W. Li: Beijing International Center for Mathematical Research / School of Mathematical Sciences, Peking University. No.\ 5 Yiheyuan Road, Beijing 100871, People's Republic of China. \\
	E-mail address: \href{mailto:wwli@bicmr.pku.edu.cn}{\texttt{wwli@bicmr.pku.edu.cn}}
\end{flushleft}

\end{document}